%% file: XK_naturality.tex
\documentclass[11pt,reqno]{amsart}
\usepackage{bm,amsmath,amsthm,amssymb,mathtools,verbatim,amsfonts,tikz-cd,mathrsfs,xfrac}
\usepackage{enumitem}

\usepackage{adjustbox}

\usepackage[hidelinks,colorlinks=true,linkcolor=blue, citecolor=black,linktocpage=true]{hyperref}
\usepackage{graphicx}
\usepackage[alphabetic]{amsrefs}

\usepackage[top=1.4in, bottom=1.2in, left=1.4in, right=1.4in,marginpar=1in]{geometry}

\usetikzlibrary{matrix,arrows,decorations.pathmorphing}

\usepackage{chngcntr}
\counterwithin{figure}{section}

\newtheorem{thm}{Theorem}[section]
\newtheorem{prop}[thm]{Proposition}

\newtheorem{cor}[thm]{Corollary}
\newtheorem{lem}[thm]{Lemma}

\theoremstyle{definition}
\newtheorem{define}[thm]{Definition}

\theoremstyle{remark}
\newtheorem{rem}[thm]{Remark}
\newtheorem{example}[thm]{Example}

\newcommand{\ve}[1]{\boldsymbol{\mathbf{#1}}}

\newcommand{\R}{\mathbb{R}}

\newcommand{\Z}{\mathbb{Z}}
\newcommand{\N}{\mathbb{N}}
\newcommand{\Q}{\mathbb{Q}}
\newcommand{\C}{\mathbb{C}}

\renewcommand{\d}{\partial}
\renewcommand{\subset}{\subseteq}

\renewcommand{\tilde}{\widetilde}
\renewcommand{\bar}{\overline}

\newcommand{\iso}{\cong}

\newcommand{\then}{\Rightarrow}

\DeclareMathOperator{\can}{{can}}

\DeclareMathOperator{\Diff}{{Diff}}

\DeclareMathOperator{\End}{{End}}

\DeclareMathOperator{\gr}{{gr}}

\DeclareMathOperator{\Hom}{{Hom}}

\DeclareMathOperator{\id}{{id}}

\DeclareMathOperator{\im}{{im}}

\DeclareMathOperator{\join}{{Join}}

\DeclareMathOperator{\Spin}{{Spin}}

\DeclareMathOperator{\Sw}{{Sw}}

\DeclareMathOperator{\Sym}{{Sym}}

\DeclareMathOperator{\Sq}{{Sq}}

\DeclareMathOperator{\Tw}{{Tw}}

\newcommand{\std}{\mathrm{std}}
\newcommand{\sw}{\mathrm{sw}}

\newcommand{\lk}{\mathrm{lk}}

\newcommand{\bA}{\mathbb{A}}
\newcommand{\bB}{\mathbb{B}}

\newcommand{\bD}{\mathbb{D}}
\newcommand{\bE}{\mathbb{E}}
\newcommand{\bF}{\mathbb{F}}

\newcommand{\bH}{\mathbb{H}}
\newcommand{\bI}{\mathbb{I}}
\newcommand{\bJ}{\mathbb{J}}

\newcommand{\bT}{\mathbb{T}}
\newcommand{\bU}{\mathbb{U}}

\newcommand{\bX}{\mathbb{X}}

\newcommand{\bZ}{\mathbb{Z}}

\newcommand{\cA}{\mathcal{A}}
\newcommand{\cB}{\mathcal{B}}
\newcommand{\cC}{\mathcal{C}}
\newcommand{\cD}{\mathcal{D}}
\newcommand{\cE}{\mathcal{E}}
\newcommand{\cF}{\mathcal{F}}
\newcommand{\cG}{\mathcal{G}}
\newcommand{\cH}{\mathcal{H}}

\newcommand{\cK}{\mathcal{K}}
\newcommand{\cL}{\mathcal{L}}
\newcommand{\cM}{\mathcal{M}}

\newcommand{\cS}{\mathcal{S}}

\newcommand{\cX}{\mathcal{X}}
\newcommand{\cY}{\mathcal{Y}}

\newcommand{\frA}{\mathfrak{A}}
\newcommand{\frB}{\mathfrak{B}}

\newcommand{\frI}{\mathfrak{I}}

\newcommand{\fro}{\mathfrak{o}}

\newcommand{\frs}{\mathfrak{s}}

\newcommand{\scB}{\mathscr{B}}
\newcommand{\scC}{\mathscr{C}}

\newcommand{\scE}{\mathscr{E}}

\newcommand{\scH}{\mathscr{H}}

\newcommand{\scO}{\mathscr{O}}

\newcommand{\scT}{\mathscr{T}}
\newcommand{\scU}{W}
\newcommand{\scV}{Z}

\newcommand{\cCFL}{\mathcal{C\!F\!L}}
\newcommand{\cCFK}{\mathcal{C\hspace{-.5mm}F\hspace{-.3mm}K}}

\newcommand{\CF}{\mathit{CF}}

\newcommand{\HF}{\mathit{HF}}

\newcommand{\CFK}{\mathit{CFK}}

\newcommand{\xs}{\ve{x}}
\newcommand{\ys}{\ve{y}}
\newcommand{\zs}{\ve{z}}
\newcommand{\ws}{\ve{w}}

\newcommand{\ps}{\ve{p}}
\newcommand{\qs}{\ve{q}}
\newcommand{\as}{\ve{\alpha}}
\newcommand{\bs}{\ve{\beta}}
\newcommand{\gs}{\ve{\gamma}}
\newcommand{\ds}{\ve{\delta}}

\renewcommand{\a}{\alpha}
\renewcommand{\b}{\beta}
\newcommand{\g}{\gamma}
\newcommand{\dt}{\delta}
\newcommand{\veps}{\varepsilon}

\usepackage{leftidx}

\newcommand{\SL}{\mathit{SL}}

\DeclareMathOperator{\Cone}{{Cone}}

\newcommand{\Ss}[1]{\scriptstyle{#1}}
\newcommand{\Sss}[1]{\scriptscriptstyle{#1}}

\numberwithin{equation}{section}

\newcommand{\llsquare}{[\hspace{-.5mm}[}
\newcommand{\rrsquare}{]\hspace{-.5mm}]}

\newcommand\co{\colon}

\newcommand{\vopp}{{\tilde{v}}}

\allowdisplaybreaks

\def\dl {\bunderline{d}}

\newcommand{\bunderline}[1]{\underline{#1\mkern-2mu}\mkern2mu }

\def\du {\bar{d}}

\newcommand{\barcap}{\mathbin{\bar{\cap}}}

\newcommand{\diff}{\mathrm{diff}}
\newcommand{\inv}{\mathrm{inv}}

\newenvironment{usethmcounterof}[1]{%
  \renewcommand\thethm{\ref{#1}}\thm}{\endthm\addtocounter{thm}{-1}}

\title{The link surgery formula and equivariant surgeries}

\author{Kristen Hendricks}
\thanks{KH was partially supported by NSF grant DMS-2019396.}
\address{Department of Mathematics, Rutgers University, New Brunswick, NJ, USA}
\email{kristen.hendricks@rutgers.edu}
\author{Abhishek Mallick}
\address{Department of Mathematics, Dartmouth College,
Hanover, NH, USA}
\email{abhishek.mallick@dartmouth.edu}
\author{Matthew Stoffregen}
\address{Department of Mathematics, Michigan State University, East Lansing, MI, USA}
\email{stoffre1@msu.edu}
\thanks{MS was partially supported by NSF grant DMS-2203828.}
\author{Ian Zemke}
\address{Department of Mathematics\\University of Oregon\\  Eugene, OR, USA}
\email{izemke@uoregon.edu}
\thanks{IZ was partially supported by NSF grant DMS-2204375 and a Sloan Research Fellowship.}

\begin{document}
\maketitle

\begin{abstract}
We prove an equivariant version of the Heegaard Floer link surgery formula. As a special case, this gives an equivariant knot surgery formula for equivariant knots in $S^3$. Our proof goes by way of a naturality theorem for certain bordered modules described by the last author. As a sample application, we prove the kernel of the forgetful map from the equivariant homology cobordism group to the homology cobordism group contains a $\Z^\infty$-summand.
\end{abstract}

\tableofcontents

\section{Introduction}

Heegaard Floer homology is an invariant of 3-manifolds due to Ozsv\'{a}th and Szab\'{o} \cite{OSDisks, OSProperties}. The version we are most interested in, called the minus version, associates a finitely generated graded module $\HF^-(Y,\frs)$ over $\bF[U]$ to each pair $(Y, \frs)$, where $Y$ is a closed oriented $3$-manifold and $\frs$ is a $\Spin^c$ structure on $Y$. Here $\mathbb F$ is the field of two elements and $U$ is a variable of degree $-2$. We typically write $\HF^-(Y)$ for the direct sum of $\HF^-(Y,\frs)$ over all $\frs\in \Spin^c(Y)$.

Symmetries of 3-manifolds are an important topic in low-dimensional topology. Given a diffeomorphism $\phi\colon Y\to Y$, there is an induced endomorphism
\[
\HF^-(\phi)\colon \HF^-(Y)\to \HF^-(Y).
\] 
Understanding this diffeomorphism map is important for a number of questions about 3-manifolds and 4-manifolds; studying it has been used to produce many topological applications (e.g.\cite{AKS,DHM-Corks,DMS:equivariant, dai20242, levine2023new}) in topics like exotic 4-manifolds, exotic surfaces and knot concordance.

In order to produce topological applications, it is important to have tools which can effectively perform computations of the maps $\HF^-(\phi)$. In this paper, we consider the case that $Y$ is described as Dehn surgery on a link $L\subset S^3$ and $\phi$ is induced by a symmetry on $L$. By the well-known Lickorish–Wallace theorem, every 3-manifold is surgery on a link in $S^3$. Moreover, work of Sakuma \cite{sakuma2001surgery} implies that any finite order orientation preserving diffeomorphism on a 3-manifold is surgery on a particular type of symmetric link, called a periodic link. Another well-known result, commonly known as the Montesinos trick \cite{Mdbc}, implies that any double-branched cover of a knot is surgery on a symmetric link, called a strongly invertible link, where the covering involution is induced from the symmetry of the link. This equivariant surgery description of the double branched cover of knots has been influential in many of the aforementioned applications.


 This motivates us to describe a Dehn surgery formula for $\HF^{-}(\phi)$. Our main results are stated in terms of the Heegaard Floer Dehn surgery formulas. These formulas were developed by Ozsv\'{a}th and Szab\'{o} \cite{OSIntegerSurgeries} for knots, and subsequently expanded to links by Manolescu and Ozsv\'{a}th \cite{MOIntegerSurgery}. Given a link $L\subset S^3$ with integral framing $\Lambda$, the authors construct a chain complex $\cC_{\Lambda}(L)$ which is chain homotopy equivalent to the completion of $\HF^-(S^3_{\Lambda}(L))$ at the ideal $(U)\subset \bF[U]$. We let $\ve{\HF}^-(S^3_{\Lambda}(L))$ denote this completion. These Dehn surgery formulas were extended by the last author to arbitrary Morse framed links in closed, oriented 3-manifolds \cite{ZemExactTriangle}. The complex $\cC_{\Lambda}(L)$ is built from the link Floer homology \cite{OSLinks} of the link $L$, as well as some additional data; it is frequently more computable than the Floer complex $\CF^-(S^3_{\Lambda}(L))$.

Under very general assumptions, if $\phi\colon (Y,L)\to (Y,L)$ is a diffeomorphism of pairs which preserves the framing $\Lambda$, then there is an induced diffeomorphism $\Phi\colon Y_{\Lambda}(L)\to Y_{\Lambda}(L)$. In this paper, we show that $\phi$ induces a diffeomorphism map on the link surgery complex
\[
\cC(\phi)\colon \cC_{\Lambda}(Y,L)\to \cC_{\Lambda}(Y,L).
\]
 We prove that, under very general hypotheses, the isomorphism $H_* \cC_{\Lambda}(Y,L)\iso \ve{\HF}^-(Y_{\Lambda}(L))$ intertwines the maps $\cC(\phi)$ and $\HF^-(\Phi)$. We additionally show that our equivariant surgery formula holds more generally for symmetries of Morse framed links in arbitrary 3-manifolds.
 
 For the purposes of defining a map on the link surgery complex associated to a diffeomorphism of pairs $\phi\colon (Y,L)\to (Y,L)$, we assume that $\phi$ is \emph{strongly framed}. Roughly speaking, this means that a tubular neighborhood of $L$ is parametrized with a diffeomorphism to a union of solid tori, and that $\phi$ respects this parametrization.  See Definition~\ref{def:strong-framing} for a precise definition.


In the case that $K$ is a knot in $S^3$, we are able to give several ``local'' versions of the surgery formula, described below in Section~\ref{subsec:local}, which allow for efficient computations. These local versions are modeled on analogous local formulas for the involutive knot surgery formula \cite{HHSZExact}. We also illustrate our knot surgery formula in a number of example computations.

When $K$ is a knot in a closed 3-manifold $Y$, we follow standard conventions and write $\bX_n(Y,K)$ for the surgery complex (denoted $\cC_n(Y,K)$ above). If $\phi\colon (Y,K)\to (Y,K)$ is a diffeomorphism of pairs, we will write 
\begin{equation}
\bX(\phi)\colon \bX_n(Y,K)\to \bX_n(Y,K)
\label{eq:intro-X(phi)}
\end{equation}
for the induced map.

\begin{thm}
\label{thm:equivariant-knot-surgery-formula}
Let   $\phi\colon (Y,K)\to (Y,K)$ be a diffeomorphism of strongly framed knots and let $\Phi$ denote the induced diffeomorphism on $Y_n(K)$. The map $\bX(\phi)$ from Equation~\eqref{eq:intro-X(phi)} is intertwined with $\ve{\CF}(\Phi)\colon \ve{\CF}^-(Y_n(K))\to \ve{\CF}^-(Y_n(K))$ by a natural homotopy equivalence $\Gamma\colon \ve{\CF}^-(Y_n(K))\to \bX_n(Y,K)$, i.e.,
\[
\bX(\phi)\circ \Gamma\simeq \Gamma\circ  \ve{\CF}(\Phi).
\]
\end{thm}

\begin{rem} In Section~\ref{sec:knots} we show that if $Y=S^3$, the map $\bX(\phi)$ is determined completely by the naturality map on the knot Floer complex $\cCFK(\phi)\colon \cCFK(K)\to \cCFK(K)$. This map will also sometimes be denoted $\phi_K$.
\end{rem}

\begin{rem} An earlier result of the second author \cite{Mallick:surgery} establishes a version of the equivariant knot surgery formula when $n$ is sufficiently large (i.e. $n\ge 2g(K)-1$). See also \cite{DMZ_Corks}*{Lemma~2.3}.
\end{rem}

\subsection{Local formulas} \label{subsec:local} In the setting of actions on various flavors of Heegaard Floer homology, \emph{local equivalence} refers to an equivalence relation weaker than equivariant chain homotopy equivalence, whose details depend on the exact flavor of Heegaard Floer theory and type of action involved. In general, (equivariant) homology cobordisms and (equivariant) concordances induce local equivalences rather than chain homotopy equivalences. Moreover, much of the utility of surgery formulas for actions on Heegaard Floer homology has come not from studying the full surgery complex, but from replacing it with a tractable representative of its local equivalence class. In the context of surgeries and symmetries on 3-manifolds, various notions of local equivalence have been studied by many authors; see, e.g.,  \cite{HMZConnectedSum,HHSZExact,DHM-Corks,DMS:equivariant} for several examples.

For the special cases of strongly invertible and periodic symmetries on knots in $S^3$, we have tidy local formulas. In the case of strongly invertible knots, these take the following form, which closely resembles the local formulas of \cite[Theorem 1.6]{HHSZExact} for the involutive variant of Heegaard Floer homology. For the cases of periodic complexes, our local formulas take on a different form.

 Note that local equivalence is a relationship for pairs, here called $\phi$-complexes, consisting of chain complexes together with endomorphisms; in the cases below these are the chain complexes $\CF^-(S^3_n(K), [s])$, where $[s]$ denotes a $\Spin^c$ structure on the surgery in the conventional identification with $\bZ/n$. The endomorphisms are the induced maps $\CF^-(\phi)$, which in this context are called $\phi$ for brevity. We compare these to pairs consisting of suitable subcomplexes $A_s(K)$, treated as $\bF[U]$ complexes, of $\cCFK(K)$ together with the endomorphism $\phi_K$. Further discussion of the algebra of local equivalence may be found in Section~\ref{sec:local-algebra}.
 
 When $K\subset S^3$ is a strongly invertible knot and $n\in \Z$, with symmetry $\phi$, the induced symmetry on $S^3_n(K)$ acts on $\Spin^c(S^3_n(K))$ by conjugation. Therefore there are either one or two $\Spin^c$ structures on $S^3_n(K)$ which are fixed by the induced action, depending on whether $n$ is odd or even.  We prove the following:

\begin{thm}\label{thm:local-class}
 Suppose $K$ is a strongly-invertible knot in $S^3$, and $n >0$ is an integer.
 \begin{enumerate}
 \item The $\phi$-complex $(\CF^-(S^3_n(K),[0]),\phi)$ is locally equivalent to $(A_0(K),\phi_K)$, shifted upward in grading by $d(L(n,1),[0])$. 
 \item  The $\phi$-complex $(\CF^-(S^3_{2n}(K),[n]), \phi)$ is locally equivalent to the complex
  \[
 \begin{tikzcd}[column sep={1cm,between origins},labels=description] 
A_{-n}(K)
	\ar[dr, "h"]
& & A_n(K) \ar[dl,"v"]\\
& B_n(K)
\end{tikzcd}
\]
 shifted upward in grading by $d(L(2n,1),[n])$ with the involution coming from truncating the surgery complex.
 \end{enumerate}
\end{thm}

We may further refine the second statement.

\begin{thm}\label{thm:local-class-even-si}
Suppose $K$ is a strongly invertible knot in $S^3$, and $n>0$ is an integer. The $\phi$-complex $(\CF^-(S^3_{2n}(K),[n]),\phi)$ is locally equivalent to the complex
  \[
 \begin{tikzcd}[column sep={1cm,between origins},labels=description] 
A_{n}(K)
	\ar[dr, "v"]
& & A_n(K) \ar[dl,"v"]\\
& B_n(K)
\end{tikzcd}
\]
 shifted upward in grading by $d(L(2n,1),[n])$, where the involution which swaps the two copies of $A_n$, and fixes $B_n$. 
 \end{thm}

Similar results, stated in Theorem~\ref{thm:local-class-si-rational}, hold for rational surgeries. From this, similarly to \cite[Proposition 1.7]{HHSZExact} and the well-known formula of Ni and Wu for the nonequivariant case \cite{NiWu}, one can derive an invariant for the correction terms of the surgery, which is stated in Section~\ref{sec:local-si}, after reviewing correction terms are reviewed in Section~\ref{sec:local-algebra}.

In the case of periodic knots, we have the following, which has a slightly different form. 

\begin{thm}\label{thm:periodic-local-class}
 Suppose $K$ is a $p$-periodic knot in $S^3$ with $p$ even, and $n>0$ is an integer. If $-n/2<s\le n/2$, then the $\phi$-complex $(\CF^-(S^3_n(K),[s]),\phi)$ is locally equivalent to $(A_s(K),\phi_K)$, shifted upward in grading by $d(L(n,1),[s])$.
\end{thm}

\begin{rem}
\label{rem:odd-order-triviality} In fact the theorem above is only interesting if $p$ is even. If $(C,\phi)$ is a $\phi$-complex and $\phi^{p}\simeq \id$, where $p$ is odd, then $(C,\phi)$ is locally equivalent to $(C,\id)$. Indeed, we can define $F=\sum_{i=0}^{p-1} \phi^i$. This map satisfies $F \phi\simeq F \id$ and $\phi F\simeq F \id$ and therefore gives a local map from $(C,\phi)$ to $(C,\id)$ and from $(C,\id)$ to $(C,\phi)$. 
\end{rem}

Again similar results, stated in Theorem~\ref{thm:periodic-local-class-rational} hold for rational surgeries. We also derive a formula for the correction terms of the surgery, stated in Section~\ref{sec:periodic-local}. Note that Theorem~\ref{thm:periodic-local-class}(2) and its analog for rational surgeries differ from the corresponding formulas for involutive and strongly-invertible actions. 

In Section~\ref{subsec:example}, we illustrate our formulas by studying $\sfrac{+1}{2}$ surgery on the figure-eight knot with several of its natural symmetries.  The involutive local class of the surgery is trivial, as is the local class of the surgery with the action induced by either of the two inequivalent strong inversions on the figure eight knot, but the local class of the complex together with the action induced by the periodic symmetry on the knot is nontrivial.

\subsection{Applications}

We describe one application of the techniques of this paper to study the equivariant homology cobordism group.

Let $\Theta_{\Z}^3$ denote the \emph{homology cobordism group}. We recall that this is the set of homology 3-spheres $Y$, modulo the relation that $Y_0\sim Y_1$ if there is an integer homology cobordism $W$ from $Y_0$ to $Y_1$, to wit, a compact 4-manifold $W$ such that $\d W=Y_1\cup -Y_0$ and such that the inclusions $H_*(Y_i;\Z)\to H_*(W;\Z)$ are isomorphisms. 

There is a refinement $\Theta_\Z^{\mathrm{diff}}$ of the homology cobordism group for 3-manifolds equipped with a diffeomorphism, called the \emph{equivariant homology cobordism group}. This group is generated by pairs $(Y,\phi)$ where $Y$ is disjoint union of integer homology 3-spheres and $\phi\colon Y\to Y$ is an orientation preserving diffeomorphism which preserves each component of $Y$. Two pairs $(Y_0,\phi_0)$ and $(Y_1,\phi_1)$ are equivalent if there is a pair $(W,\Phi)$ such that $W$ is a cobordism from $Y_0$ to $Y_1$, and $\Phi\colon W\to W$ is a diffeomorphism which extends $\phi_0\sqcup \phi_1$. Here $W$ has $H_2(W; \mathbb{Z})=0$, and we require that $\Phi_{*}$ acts by the identity on $H_1(W , \partial W; \mathbb{Z})$. The group operation on $\Theta_{\Z}^{\diff}$ is disjoint union.  If one more restrictively requires that $\phi$ and its extensions be involutions, one obtains a corresponding group\footnote{In \cite{DHM-Corks}, the authors write $\Theta_{\Z}^{\tau}$ for what we write $\Theta_{\Z}^\inv$.} $\Theta_{\Z}^{\mathrm{inv}}$. There is a natural group homomorphism from $\Theta_{\Z}^{\mathrm{inv}}$ to $\Theta_{\Z}^{\mathrm{diff}}$. For more on the  motivation of this definition of $\Theta_{\Z}^{\mathrm{diff}}$, we refer readers to \cite[Section 2]{DHM-Corks}.

\begin{rem}For all the examples considered in this article, we will work with a connected 3-manifold with a diffeomorphism that comes with a fixed point. In this case, readers may think of the group operation between two such elements in $\Theta_{\Z}^{\diff}$ as equivariant connected sum, see \cite[Section 2]{DHM-Corks} for more discussion.
\end{rem}

It is not hard to see that every element of $\Theta_\Z^{\mathrm{diff}}$ has a connected representative, and that there is a forgetful map
  \[
 \cF\colon \Theta_{\Z}^{\mathrm{diff}}\to \Theta_{\Z}^3.
 \]
 In this paper, we study $\ker(\cF)$.   It is known that $\ker(\cF)$ contains a $\Z^\infty$-subgroup, by \cite{DHM-Corks}*{Theorem~1.3}. In this paper, we prove the following strengthening of that result:
 
 \begin{thm}\label{thm:application}
 The kernel of the forgetful map $\cF$ contains a $\Z^\infty$-summand.
 \end{thm}
 
 There is similarly a forgetful map from $\Theta_{\Z}^{\inv}$ to $\Theta_{\Z}^3$, and our argument also shows that the kernel of the this map contains a $\Z^\infty$-summand.
 
 The 3-manifolds we study to prove Theorem~\ref{thm:application} are $S^3_{+1}(2 T_{2n,2n+1}\# -2T_{2n,2n+1})$ for odd $n$, equipped with a symmetry introduced in Section~\ref{sec:application-section}. Similar 3-manifolds have been studied before in related contexts, as in \cite{HHSZ-Infinite,DMS:equivariant}. The obstruction we use to prove Theorem~\ref{thm:application} is the $\phi$-local class of a pair $(Y,\phi)\in \Theta_\Z^{\mathrm{diff}}$ from \cite{DHM-Corks}. This amounts to understanding the induced action of $\CF(\phi)$ on $\CF^-(Y)$ up to local equivalence. Algebraically, this construction is similar to the notion of iota-local equivalence studied in \cite{HMZConnectedSum}. Additionally, we use the algebraic formalism of \emph{almost-iota complexes} (in our case denoted almost-phi complexes) developed in \cite{DHSThomcob}, wherein a $\Z^\infty$ summand of $\Theta_\Z^3$ is constructed; our strategy is similar to theirs.
 
 \begin{rem} Note also that by Remark~\ref{rem:odd-order-triviality}, our Theorem~\ref{thm:application} gives a $\Z^\infty$-summand in the quotient of $\Theta_{\Z}^{\diff}$ by the subgroup spanned by 3-manifolds with odd order diffeomorphisms.
 \end{rem}

\subsection{Naturality and the link surgery formula}

We also extend our proof of Theorem~\ref{thm:equivariant-knot-surgery-formula} to the link surgery complex. If $L\subset Y$ is a link with Morse framing $\Lambda$, and $\phi\colon (Y,L)\to (Y,L)$ is a strongly framed diffeomorphism which preserves $\Lambda$, then we describe a chain map
\[
\cC(\phi)\colon \cC_{\Lambda}(L)\to \cC_{\Lambda}(L). 
\]
We prove the following:

\begin{thm}
\label{thm:equivariant-surgery-links}
Let $\phi\colon (Y,L)\to (Y,L)$ be a diffeomorphism of strongly framed links and $\Phi \colon Y_{\Lambda}(L)\to Y_{\Lambda}(L)$ be the induced diffeomorphism. On the level of homology, the natural isomorphism $H_* \cC_{\Lambda}(Y,L)\iso \ve{\HF}^-(Y_{\Lambda}(L))$ intertwines the map $\cC(\phi)$ with the map 
\[
\ve{\HF}(\Phi)\colon \ve{\HF}^-(Y_{\Lambda}(L))\to \ve{\HF}^-(Y_{\Lambda}(L)).
\]
\end{thm}

\begin{rem}
We in fact prove a chain level refinement of Theorem~\ref{thm:equivariant-surgery-links}. However, it is somewhat complicated by fact that Heegaard Floer homology is a functor of based 3-manifolds. On the chain level, $\cC_{\Lambda}(Y,L)$ is a completion of a free module over $\bF[U_1,\dots, U_\ell]$, with one $U_i$ variable for each link component. The chain map $\cC(\phi)$ interchanges the different $U_i$ variables if $\phi$ interchanges the corresponding link components. We will show that there is a natural homotopy equivalence between $\cC_{\Lambda}(L)$ and a multi-pointed version of $\ve{\CF}^-(Y_{\Lambda}(L))$ and $\cC_{\Lambda}(Y,L)$, and that this homotopy equivalence intertwines $\ve{\CF}(\Phi)$ with $\cC(\phi)$. Since this version of $\ve{\CF}^-(\Phi)$ permutes some of the $U_i$ variables, it is hard to identify on the chain level with the ordinary map on the singly pointed versions of Heegaard Floer homology, which have a single $U$ variable. We show that this subtlety vanishes on homology, leaving us with the statement above.
\end{rem}

The main technical result which underlies our proofs of Theorems~\ref{thm:equivariant-knot-surgery-formula} and ~\ref{thm:equivariant-surgery-links} is an extension of results on the naturality of Heegaard Floer homology due to Juh{\'a}sz, Thurston, and the last author \cite{JTNaturality} to the link surgery formula. To achieve this new naturality result, we consider a slightly restricted set of Heegaard diagrams with which to compute the link surgery formula, called \emph{meridional complete systems}. See Section~\ref{sec:meridional-diagrams}. We prove the following: 

\begin{thm}
\label{thm:naturality-intro} If $L$ is a strongly framed link in a 3-manifold $Y$, then the link surgery formula, computed with meridional complete systems of Heegaard diagrams, is natural. That is, if $L\subset Y$ is a strongly framed link (with underlying Morse framing $\Lambda$), then for each meridional complete system $\scH$ there is a model of the surgery formula $\cC_{\Lambda}(\scH)$. For each pair of meridional complete systems $\scH,$ $\scH'$, there is a homotopy equivalence $\Psi_{\scH\to \scH'}$, well defined up to chain homotopy, which satisfies the following:
\begin{enumerate}
\item $\Psi_{\scH'\to \scH''}\circ \Psi_{\scH\to \scH'}\simeq \Psi_{\scH\to \scH''}$;
\item $\Psi_{\scH\to \scH}\simeq \id$.
\end{enumerate}
\end{thm}

\begin{rem} Unfortunately, the restriction to meridional complete systems in the previous theorem does not seem well suited to studying the $\Spin^c$ conjugation action $\iota$ (cf. \cite{HMInvolutive}) in terms of the link surgery formula. 
\end{rem}

Finally, we note that all of the results described above naturally extend to the bordered perspective on the link surgery formula developed in \cite{ZemBordered, ZemExactTriangle}. Therein, an associative algebra $\cK$ was described. Given a knot $K\subset Y$ with Morse framing $\lambda$, the last author described a type-$D$ module $\cX_{\lambda}(Y,K)^{\cK}$ which naturally encodes the surgery formula for $(Y,K)$. There is also a type-$A$ module for the solid torus ${}_{\cK} \cD$, and the last author showed that
\[
\bX_{\lambda}(Y,K)\simeq \cX_{\lambda}(Y,K)^{\cK}\boxtimes {}_{\cK} \cD.
\]
Our techniques for Theorems~\ref{thm:equivariant-knot-surgery-formula}, ~\ref{thm:equivariant-surgery-links} and~\ref{thm:naturality-intro} extend to prove naturality of the type-$D$ modules $\cX_{\lambda}(Y,K)^{\cK}$. Similar statements hold for the link surgery formula. See Theorem~\ref{thm:naturality-type-D}. We refer the reader to Section~\ref{sec:background-surgery-algebra} for more background on this perspective. 

We remark that Guth and Kang \cite{GKNaturality} have recently proven a naturality statement for the hat flavor of Lipshitz, Ozsv\'{a}th and Thurston's bordered Floer homology \cite{LOTBordered}.

\subsection{Further directions}

There are some additional directions that could be useful for applications, which we do not address, but hope to address in future work.

\begin{enumerate}
\item The link surgery formula has natural gluing formulas, which can be interpreted in terms of either gluing bordered manifolds with torus boundary together or in terms of connected sums of knots. See \cite{ZemBordered}*{Appendix~A} for the topological interpretation of connected sums and gluing, and see \cite{ZemBordered}*{Section~12} for several connected sum formulas, interpreted as gluing theorems. It would be natural to try to prove that the naturality maps behave tensorially under such gluings. 
\item It would also be interesting to extend our naturality theorem to non-meridional systems. This could have applications to involutive refinements of the link surgery formula. 
\end{enumerate}

\subsection{Organization}

This paper is organized as follows. In Sections~\ref{sec:topological-background} and~\ref{sec:algebraic-background}, we review topological and algebraic background. In Section~\ref{sec:background-knot-surgery} we review the knot surgery formula, with an emphasis on notation. In Section~\ref{sec:equivariant-mapping-cone-statement}, we state the equivariant knot surgery formula and consider several simple examples. In Section~\ref{sec:local-formulas}, we prove several local equivalence class formulas. In Section~\ref{sec:application-section} we prove Theorem~\ref{thm:application} and prove our application to the equivariant homology cobordism group. In Section~\ref{sec:background-link-surgery} we review background on the link surgery formula. In Section~\ref{sec:equivariant-link-surgery-statement}, we state the equivariant link surgery formula. In Section~\ref{sec:pointed-diffeos}, we prove several results about basepoint swapping maps on Heegaard Floer homology, which come up naturally when considering symmetries of links which swap link components. In Sections~\ref{sec:homology-action-homotopies} and~\ref{sec:hypercubes-of-surgery-complexes}, we develop tools for building homotopy commutative diagrams of link surgery complexes. Finally, in Section~\ref{sec:naturality} we prove naturality of the link surgery formula. In Section~\ref{sec:proof-equivariant-surgery-formula}, we complete our proof of the equivariant link surgery formula.

\subsection{Acknowledgments} We thank Irving Dai, Matt Hedden, and Jen Hom for helpful conversations, in some cases over many years. Portions of this project were carried out while the first and fourth authors were in attendance at the ``2024 MATRIX-IBSCGP Workshop on symplectic and low-dimensional topology.'' Portions were also carried out at the 2025 Georgia International Topology Conference, where all four authors were in attendance. We thank the organizers for their hospitality.

\section{Topological background}
\label{sec:topological-background}

\subsection{Diffeomorphisms of links and link complements}

In this section, we describe the class of diffeomorphisms of pairs $(Y,L)$ for which we will prove functoriality of the Heegaard Floer link surgery formula.


\begin{define} 
We say that a link $L\subset Y$ is \emph{strongly framed} if each component of $L$ has a tubular neighborhood equipped with a diffeomorphism with $S^1\times D^2$.
\end{define}

\begin{rem} 
Note that a strong framing on a link induces a Morse framing. However, a strong framing contains more information than a Morse framing. For example, a strongly framed link is equipped with a preferred identification of the underlying link with a disjoint union of several copies of $S^1$, whereas a Morse framed link is not. 
%
\end{rem}

We regard the disk $D^2$ as the subset $\{z:|z|\le 1\}\subset \C$ of the complex plane, with $S^1 = \d D^2$. We refer to the map
\[
e\colon S^1\times D^2\to S^1\times D^2\]
given by
\[
e(x_1,x_2)=(\bar x_1, \bar x_2)\]
as the \emph{elliptic involution}. Note that $e$ is orientation preserving. If we instead identify $S^1$ with $\R/\Z$, then the map $e$ restricts to the map $x\mapsto -x$ on $\d (S^1\times D^2)$.

\begin{define}
\label{def:strong-framing}
 A \emph{diffeomorphism of strongly framed links} $\phi\colon (Y_1,L_1)\to (Y_2,L_2)$ consists of a diffeomorphism $\phi\colon Y_1\to Y_2$ which maps $L_1$ to $L_2$ and which satisfies one of two conditions on each link component $K\subset L_1$:
\begin{enumerate}
\item $\phi$ maps the framing of $K$ to the framing of $\phi(K)$.
\item $\phi$ maps the framing of $K$ to the framing of $\phi(K)$ composed with the elliptic involution.
\end{enumerate}
\end{define}

We say that a diffeomorphism of strongly framed links $\phi$ preserves the orientation of a component $K\subset L_1$ if the first condition is satisfied, and we say that $\phi$ reverses the orientation of $K$ if the second condition is satisfied. We are mostly interested in the case that $(Y_1,L_1)=(Y_2,L_2)$.

\begin{rem}
 Not all diffeomorphisms of a link complement $Y\setminus \nu(L)$ can be described via a diffeomorphism of strongly framed links  of $L$. For example the complement of the Hopf link $H$ is diffeomorphic to $\bT^2\times I$, so any element of $\SL_2(\Z)$ induces a diffeomorphism of the link complement. Most of these diffeomorphisms are not induced by any diffeomorphism of the Hopf link.  Indeed, a diffeomorphism of $S^3\setminus H$ that extends over $S^3$ must send meridians to meridians.
\end{rem}

\subsection{Modifying non-strong diffeomorphisms of links}

A general diffeomorphism $\phi\colon (Y,L)\to (Y,L)$ will not be a diffeomorphism of strongly framed links, as in Definition~\ref{def:strong-framing}. We will need to perform an isotopy of $\phi$, supported in a neighborhood of $L$, to obtain a strongly framed diffeomorphism. This choice may affect the resulting endomorphism of $\cX_{\Lambda}(L)^{\cL}$. Any two choices can be related, up to isotopy, by composition with Dehn twists on the boundary of a tubular neighborhood of $L$.

We note that a Dehn twist on a 2-torus which is parallel to the boundary of the solid torus $S^1\times D^2$ is isotopic to the identity rel boundary in the solid torus. Therefore, in general the choice of strongly framed representative of $\phi$ does not affect the resulting diffeomorphism on the Dehn surgered manifold. Nonetheless, the choice of strongly framed representative of a diffeomorphism $\phi$ of $(Y,L)$ may affect the resulting map on Heegaard Floer homology, because we view the solid torus as containing a basepoint, and Heegaard Floer homology is a functor of pointed 3-manifolds.

\begin{rem} Let $K\subset Y$ be a knot with Morse framing $\lambda$. Suppose that $\phi$ is a strongly framed diffeomorphism of $(Y,K)$ and $\phi'$ is obtained by performing an $(n\mu+m \lambda)$-framed Dehn twist to a boundary of a tubular neighborhood of $K$. Let $\Phi$ and $\Phi'$ be the two induced diffeomorphisms on $(Y_{\lambda}(K),w)$, where $w$ is a basepoint along the core of the surgery solid torus. Then $\Phi'$ is obtained from $\Phi$ by composing with the point pushing diffeomorphism which moves $w$ along the core of the surgery solid torus $n$ times. 
\end{rem}

\subsection{Equivariant knots} \label{sec:equivariant-knots}

We now recall several important types of symmetries on knots that we consider in this paper. For our purposes, a \emph{symmetric knot} $(K,\phi)$ is a pair where $K\subset S^3$ is a knot and $\phi\colon (S^3,K)\to (S^3,K)$ is a diffeomorphism of pairs.

\begin{define}
 We say a symmetric knot $(K,\phi)$ is \emph{strongly invertible} if $\phi$ preserves the orientation of $S^3$, reverses the orientation of $K$, and satisfies $\phi^2=\id$ on $S^3$. 
\end{define}

The connected sum operation on strongly invertible knots is not canonically defined. Rather, the connected sum operation is only well defined if we specify additional information. If $\phi$ is a strong inversion, then the fixed set $F\subset S^3$ of $\phi$ is an unknot by the Smith conjecture. This unknot intersects the knot $K$ in two points. A \emph{directed half axis} consists of a choice of a component of $F\setminus K$, equipped with an orientation. There is a well-defined connected sum operation on strongly invertible knots equipped with directed half axes. Furthermore, strongly invertible knots equipped with directed half axes form an equivariant concordance group, which we denote $\tilde{\cC}$ \cite{Sakumainvertible}. This group is known to be non-abelian \cite{DiPrisa:nonabelian} and contain the ordinary concordance group  in its center \cite{Sakumainvertible}.

If $(K,\phi)$ is strongly invertible, then we may pick a strong framing of $K$ with respect to which $\phi$ acts by the elliptic involution on a tubular neighborhood of $K$.

\begin{define} If $p\in \N$,  we say that a symmetric knot $(K,\phi)$, is \emph{$p$-periodic} if $\phi$ has order $p$ (that is, $\phi^p=\id$ and $\phi^i\neq \id$ if $i\in \{1,\dots, p-1\}$), $\phi$ is orientation preserving on $S^3$ and $K$, and $\phi$ has fixed set an unknot disjoint from $K$.
\end{define}

If $(K,\phi)$ is $p$-periodic, then $\phi$ will not fix any strong framing of $K$. Typically we will need to compose $\phi$ with  a twist $\rho$ of order $-j/p$ on a tubular neighborhood of $K$, where $\phi$ acts on $K\iso S^1$ by multiplication by $e^{2\pi i j/p}$. This ensures that $\rho\circ \phi$ preserves a strong framing of $K$.

\section{Algebraic background}
\label{sec:algebraic-background}

\subsection{Hypercubes and hyperboxes}

\label{sec:hypercubes}

We recall the formalism of hypercubes and hyperboxes, following the terminology of Manolescu and Ozsv\'{a}th \cite{MOIntegerSurgery}. Here, we write $\bE_n$ for $\{0,1\}^n$.

\begin{define} An $n$-dimensional \emph{hypercube of chain complexes} $\cC=(C_{\veps}, D_{\veps,\veps'})$ consists of a collection of vector spaces $C_{\veps}$ as well as a collection of linear maps $D_{\veps,\veps'}\colon C_{\veps}\to C_{\veps'}$ so that for all pairs $\veps$ and $\veps''$ with $\veps\le \veps''$, we have
\[
\sum_{\veps'|
\veps\le \veps'\le \veps''} D_{\veps',\veps''}\circ D_{\veps,\veps'}=0. 
\]
\end{define}

We will also make use of Manolescu and Ozsv\'{a}th's notion of a \emph{hyperbox}. If $\ve{d}=(d_1,\dots, d_n)\in \N_{>0}^n$ is a collection of positive integers, we write 
\[
\bE(\ve{d})=\{0,1,\dots, d_1\}\times \cdots \times \{0,1,\dots, d_n\}.
\]

\begin{define}
 If $\ve{d}=(d_1,\dots, d_n)$ is a tuple of positive integers, then a \emph{hyperbox of size $\ve{d}$} consists of a collection of vector spaces $(C_{\veps})_{\veps\in \bE(\ve{d})}$ as well as a collection of maps $D_{\veps,\veps'}\colon C_{\veps}\to C_{\veps'}$ whenever $\veps\le \veps'$ and $|\veps'-\veps|_{L^\infty}\le 1$. Furthermore, if $\veps\le \veps''$ are points such that $|\veps''-\veps|_{L^\infty}\le 1$, then the following compatibility condition is satisfied
\[
\sum_{\veps'| \veps\le \veps'\le \veps''} D_{\veps',\veps''}\circ D_{\veps,\veps'}=0.
\]
\end{define}

Manolescu and Ozsv\'{a}th define an operation on hyperboxes called \emph{compression} \cite{MOIntegerSurgery}*{Section~5.2}. This operation has input a hyperbox $(C_{\veps}, D_{\veps,\veps'})_{\veps\in \bE(\ve{d})}$ of size $\ve{d}=(d_1,\dots, d_n)$, and has output equal to a hypercube $(C'_{\veps}, D'_{\veps,\veps'})_{\veps\in \bE_n}$. Furthermore
\[
C'_{(\veps_1,\dots, \veps_n)}= C_{(d_1\veps_1,\dots, d_n \veps_n)}. 
\]
 We refer the reader to \cite{MOIntegerSurgery}*{Section~5.2} for more background on this construction.

It is natural to consider more generally hypercubes in the Fukaya category:

\begin{define} We say that $\cL_{\a}=(\as_{\veps}, \theta_{\veps,\veps'})_{\veps,\veps'\in \bE_n}$ is a \emph{hypercube of alpha attaching curves} if each $\as_{\veps}$ is a set of attaching curves on $\Sigma$, and $\theta_{\veps',\veps}\in \CF^-(\Sigma,\as',\as)$ is a choice of chain for each $\veps<\veps'$. Furthermore, the following compatibility condition is satisfied whenever $\veps<\veps'$:
\[
\sum_{\veps=\veps_1<\cdots<\veps_n=\veps'} f_{\a_{\veps_n},\dots, \a_{\veps_1}}(\theta_{\veps_{n}, \veps_{n-1}},\dots, \theta_{\veps_2,\veps_1})=0,
\]
where $f_{\a_{\veps_n},\dots, \a_{\veps_1}}$ counts rigid holomorphic $n$-gons.
\end{define}

Hypercubes of beta attaching curves are defined similarly, except that the chain $\theta_{\veps,\veps'}$ lies in $\CF^-(\Sigma,\bs,\bs')$.

\subsection{Type-$A$ and $D$ modules}

We now recall the algebraic formalism of type-$D$ and $A$ modules of Lipshitz, Ozsv\'{a}th and Thurston \cite{LOTBordered, LOTBimodules}.

We begin with type-$A$ modules, which are simply $A_\infty$-modules. More concretely, suppose that $A$ is an associative algebra over a ring $\ve{i}$. We assume that $\ve{i}$ has characteristic 2, and we write $\mu_2$ for multiplication on $A$. A left type-$A$ module ${}_A X$ consists of a left $\ve{i}$-module $X$ equipped with $\ve{i}$-linear maps
\[
m_{n+1}\colon \underbrace{A\otimes_{\ve{i}} \cdots \otimes_{\ve{i}} A}_{n} \otimes_{\ve{i}} X\to X
\]
which satisfy the compatibility condition that
\[
\begin{split}
&\sum_{j=0}^{n} m_{n-j+1}(a_n,\dots, m_{j+1}(a_j,\dots, a_1,x))\\
+&\sum_{i=1}^{n-1}m_n(a_n,\dots, a_{i+1}a_i,\dots, a_1,x)=0.
\end{split}
\]

A \emph{right type-$D$ module}, denoted $X^{\cA}=(X,\delta^1)$, consists of a right $\ve{i}$-module $X$, equipped with a map
\[
\delta^1\colon X\to X\otimes_{\ve{i}} A.
\]
The map $\delta^1$ is required to satisfy
\[
(\bI_X\otimes \mu_2)\circ (\delta^1\otimes \bI_A)\circ \delta^1=0.
\]

We recall from \cite{LOTBordered}*{Section~2.4} that there is a convenient model of the derived tensor product between a type-$D$ module and a type-$A$ module, called the \emph{box tensor product}. If $X^A$ is a type-$D$ module and ${}_{A} Y$ is a type-$A$ module, then there is a chain complex
\[
X^A\boxtimes {}_A Y
\]
defined as follows. The underlying vector space is $X\otimes_{\ve{i}} Y$ (where $A$ is an algebra over $\ve{i}$, which is typically an idempotent ring). We write $\delta^n$ for the $n$-fold composition of $\delta^1$, viewed as a map
\[
\delta^n\colon X\to X\otimes \underbrace{A\otimes \cdots \otimes A}_n.
\]
We define the differential on $X^A\boxtimes {}_A Y$ as the sum over $n\in \N$ of the following compositions:
\[
\begin{tikzcd}
X\otimes Y\ar[r, "\delta^n\otimes \bI"]& X\otimes \underbrace{A\otimes \cdots \otimes A}_n\otimes Y \ar[r, "\bI\otimes m_{n+1}"] & X\otimes Y.
\end{tikzcd}
\]

We additionally will need Lipshitz, Ozsv\'{a}th and Thurston's notion of a $DA$-bimodule. We refer the reader to \cite{LOTBimodules}*{Section~2} for an extensive introduction, and here recall briefly the definition. Suppose that $A$ and $B$ are algebras over rings $\ve{i}$ and $\ve{j}$. Then a $DA$-bimodule ${}_A X^B$ is an $(\ve{i},\ve{j})$-bimodule $X$ equipped with $(\ve{i},\ve{j})$-linear maps
\[
\delta_{j+1}^1\colon \underbrace{A\otimes\cdots \otimes A}_j\otimes X\to X\otimes B,
\]
which satisfy the $A_\infty$-associativity relation:
\[
\begin{split}
&\sum_{j=0}^{n}(\bI_X\otimes \mu_2) \circ (\delta_{n-j+1}^1\otimes \bI_B)\circ \delta_{j+1}^1 \\
+&\sum_{i=1}^{n-1} \delta_{n}^1\circ(\bI_{A^{\otimes n-i-1}}\otimes \mu_2\otimes \bI_{A^{\otimes i-1}})=0.
\end{split}
\]

\subsection{Twisted complexes}

We now recall the notion of a twisted complex. The reader may consult \cite{Bondal-Kapranov-TC} or \cite{KontsevichICM} for original references, or \cite{SeidelFukaya}*{Section~3k,l} for an exposition.

We suppose that $\cA$ is a (possibly non-unital) $A_\infty$-category. We now describe two $A_\infty$-categories. The first is the \emph{additive enlargement} $\Sigma \cA$ of $\cA$, and the second is the $A_\infty$-category of \emph{twisted complexes} over $\cA$, denoted $\Tw(\cA)$. Throughout, we assume that the Hom spaces of $\cA$ are vector spaces over $\bF=\Z/2$.

We begin by defining the \emph{additive enlargement} $\Sigma \cA$. Objects consist of formal (finite) direct sums of pairs $(X_i,V_i)$, where $X_i$ is an object of $\cA$ and $V_i$ is a graded vector space. A morphism between two objects consists of a matrix of pairs $(\theta_{ij}, f_{ij})$ where $\theta_{ij}\colon X_i\to X_j$ is a morphism in $\cA$ and $f_{ij}\colon V_i\to V_j$ is a linear map.  Composition in $\Sigma \cA$ is defined in the expected way, combining composition of the linear maps $f_{ij}$ with composition in the category $\cA$. It is not hard to show that $\Sigma \cA$ is an $A_\infty$-category. 

A \emph{twisted complex} in $\cA$ consists of a pair $(\cX,\delta)$ where $\cX$ is an object of $\Sigma \cA$ and $\delta\colon \cX\to \cX$ is a morphism which satisfies the Mauer-Cartan equation
\[
\sum_{n\ge 1} \mu_{n}^{\Sigma \cA}(\underbrace{\delta,\dots, \delta}_n)=0.
\]
Note that one needs some additional assumption on $(\cX,\delta)$ to ensure that the above sum involves only finitely many elements. A common assumption (cf. \cite{SeidelFukaya}) is to ask that $\delta$ be strictly upper triangular with respect to some ordering of its basis objects. This is sufficient for hypercubes of attaching curves.

\begin{example}
If $\cF$ is the Fukaya category of $\Sym^g(\Sigma)$, then $\Tw(\cF)$ contains the hypercubes of attaching curves considered in Section~\ref{sec:hypercubes}. Furthermore, if $\cA$ and $\cB$ are two hypercubes of attaching curves, then $\CF^-(\cA,\cB)$ is identical to the subcomplex of $\Hom_{\Tw(\cF)}(\cA,\cB)$ consisting of morphisms which strictly increase the cube filtrations.
\end{example}

\section{Background on the knot surgery formula}

\label{sec:background-knot-surgery}

In this section, we recall some background on Heegaard Floer homology and the knot surgery formula of Ozsv\'{a}th and Szab\'{o} \cite{OSIntegerSurgeries}.

\subsection{Heegaard Floer homology and knot Floer homology}

We begin by briefly recalling the basics of Heegaard Floer homology and knot Floer homology, mostly to establish notation.

Heegaard Floer homology is due to Ozsv\'{a}th and Szab\'{o} \cite{OSDisks,OSProperties}.  To a closed oriented 3-manifold $Y$ equipped with a $\Spin^c$ structure $\frs$, the theory associates a finitely generated graded $\bF[U]$ module, denoted $\HF^-(Y,\frs)$. Here $\bF$ is the field of two elements, and $U$ is a variable of degree $-2$. For our purposes, we are interested in the completed module 
\[
\ve{\HF}^-(Y,\frs):=\HF^-(Y,\frs)\otimes_{\bF[U]} \bF\llsquare U\rrsquare.
\]
The completed modules are non-trivial for only finitely many $\frs\in \Spin^c(Y)$, and therefore it is natural to consider
\[
\ve{\HF}^-(Y):=\bigoplus_{\frs\in \Spin^c(Y)} \ve{\HF}^-(Y,\frs).
\]

Knot Floer homology is a refinement of the above theory due to Ozsv\'{a}th and Szab\'{o} \cite{OSKnots} and Rasmussen \cite{RasmussenKnots}. There are a number of equivalent algebraic perspectives on the invariant. We will consider the description as a free chain complex over a two variable polynomial ring $\bF[W,Z]$. We denote this version of knot Floer homology by $\cCFK(Y,K)$. When $Y$ is a rational homology sphere, the chain complex $\cCFK(Y,K)$ is finitely generated. More generally, it decomposes as a direct sum over elements of $\Spin^c(Y)$ of finitely generated chain complexes.

When $Y$ is a rational homology 3-sphere, there are two Maslov gradings $\gr_{\ws}$ and $\gr_{\zs}$ on $\cCFK(Y,K)$. These have the  property that
\[
(\gr_{\ws},\gr_{\zs})(W)=(-2,0)\quad \text{and} \quad (\gr_{\ws}, \gr_{\zs})(Z)=(0,-2).
\]
The \emph{Alexander grading} $A$ is defined to be
\[
A=\frac{1}{2}(\gr_{\ws}-\gr_{\zs}). 
\]
Note that
\[
(\gr_{\ws},\gr_{\zs},A)(\d)=(-1,-1,0).
\]
In this paper, we write $U$ for the product $WZ$. Note that $W$ and $Z$ shift the Alexander grading, whereas $U$ preserves the Alexander grading. 

When $Y$ is an integer homology 3-sphere and $s\in \Z$, we write
\[
A_s(Y,K)\subset \cCFK(Y,K)
\]
for the subspace in Alexander grading $s$. The subspace $A_s(Y,K)$ is preserved by the differential and the action of $U$, but not by the actions of $W$ or $Z$.

\begin{rem} Many conventions exist in the literature, but it is more common to write $U$ and $V$, or alternately $\mathcal{U}$ and $\mathcal{V}$, for the variables we write $W$ and $Z$ for, respectively.
\end{rem}

\subsection{The mapping cone formula}

In this section, we recall Ozsv\'{a}th and Szab\'{o}'s mapping cone formula for a null-homologous knot $K$ in $Y$ with integer framing $\lambda$. Ozsv\'{a}th and  Szab\'{o} describe a chain complex $\bX_{\lambda}(Y,K)$ which computes $\ve{\CF}^-(Y_{\lambda}(K))$. We focus our exposition on the case that $Y$ is an integer homology 3-sphere.

The complex $\bX_{\lambda}(Y,K)$ decomposes as a mapping cone:
\[
\bX_{\lambda}(Y,K)=\Cone(v+h_\lambda\colon \bA(Y,K)\to \bB(Y,K)).
\]
The complexes $\bA(Y,K)$ and $\bB(Y,K)$ take the form
\[
\begin{split}
\bA(Y,K)&=\prod_{s\in \Z} \bm{A}_s(Y,K)\iso \ve{\cCFK}(Y,K)
\\
 \bB(Y,K)&=\prod_{s\in \Z}
 \bm{B}_s(Y,K)\iso \prod_{s\in \Z} \ve{\CF}^-(Y).
\end{split}
\]
In the above, $\ve{\cCFK}(Y,K)$ denotes the completion $\cCFK(Y,K)\otimes_{\bF[W,Z]} \bF\llsquare W,Z\rrsquare$, and $\ve{\CF}^-(Y)$ likewise denotes $\CF^-(Y)\otimes_{\bF[U]} \bF\llsquare U\rrsquare$. In particular, each $B_s(Y,K)$ denotes a copy of $\ve{\CF}^-(Y)$.  Similarly, we have $\bm{A}_s(Y,K)=A_s(Y,K)\otimes_{\bF[U]}\bF \llsquare U \rrsquare$ and $\bm{B}_s(Y,K)=B_s(Y,K)\otimes_{\bF[U]}\bF \llsquare U \rrsquare$.  

It is frequenty helpful to think of $\bB(Y,K)$ as a completion of 
\[
\CF^-(Y)\otimes_{\bF}  \bF[T,T^{-1}],
\]
where we set $B_s(Y,K)=\CF^-(Y)\otimes T^s$. 

The maps $v$ and $h_{\lambda}$ may be described as follows. In the construction of the complex $\cCFK(Y,K)$, we pick two basepoints $w,z\in K$. To construct $v$ and $h_{\lambda}$ we additionally pick a basepoint $p$ along $K$ in the complement of $w$ and $z$. For concreteness, we pick $p$ in the subarc of $K$ oriented from $z$ to $w$. This choice corresponds to the choice of an alpha-parallel arc system from \cite{ZemBordered}.

The map $v$ is obtained by first composing the inclusion map 
\[
i_Z\colon \cCFK(K)\to Z^{-1} \cCFK(K)
\]
with a canonical isomorphism
\[
\theta_w\colon Z^{-1}\cCFK(K)\iso \bigoplus_{s\in \Z} \CF^-(Y,w).
\]
The isomorphism $\theta_w$ is an isomorphism of chain complexes over $\bF[U]$, where $U$ acts on $\cCFK(K)$ by $U=WZ$. The map $\theta_w$ preserves the Alexander grading, assuming that we view the Alexander grading on the right hand side as given by the index $s$. There is a similar map
\[
\theta_z\colon W^{-1} \cCFK(K)\iso \bigoplus_{s\in \Z} \CF^-(Y,z).
\]

We then push $w$ to the point $p$ along a subarc of $K$. There is an induced diffeomorphism of $Y$, which we denote $f_{w\to p}$. The map $v$ is then given by the composition
\[
v=f_{w\to p}\circ \theta_w\circ i_Z.
\]
By construction, the map $v$ sends $A_s(Y,K)$ to $B_s(Y,K)$.

The map $h_{\lambda}$ is defined similarly, except that the roles of $Z$ and $W$ swapped, and furthermore $h_{\lambda}$ is defined to increase the Alexander grading by $\lambda$ by multiplying by an additional factor of $T^{\lambda}$. We write
\[
h_{\lambda}=T^{ \lambda} \circ f_{z\to p} \circ \theta_{z} \circ i_W.
\]
 By construction, the map $h_{\lambda}$ sends $A_{s}(Y,K)$ to $B_{s+\lambda}(Y,K)$.

Finally, the mapping cone formula $\bX_{\lambda}(Y,K)$ also requires a particular choice of algebraic completions. There are several ways of performing the completions. One version of completion is to first work over the power series ring $\bF\llsquare U\rrsquare$ and then to replace each of the direct sums with direct products. This is the perspective from \cite{MOIntegerSurgery}. Another choice is to complete the above descriptions with respect to the ideal $(U)$. Both of these perspectives are considered in \cite{ZemExactTriangle}. Therein the completions with direct products and power series is referred to as the \emph{chiral completion}, while completing at $(U)$ is referred to as the \emph{$U$-adic completion}. The surgery formula holds for both theories, but has slightly different properties. However, the results of this paper since all hold regardless of the choice of completions, as a result of which we will typically ignore the choice of completion in our exposition. 

\begin{rem} In \cite{OSIntegerSurgeries}, Ozsv\'{a}th and Szab\'{o} give a different description of $\bA(Y,K)$. They consider a $\bF[U,U^{-1}]$-chain complex $\CFK^\infty(Y,K)$, which is generated by tuples $[\xs,i,j]$ where $A(\xs)=j-i$. The variable $U$ acts by $U\cdot [\xs,i,j]=[\xs,i-1,j-1]$. They then define the subcomplex $A_s(Y,K)$ to be the $\bF[U]$-subcomplex of $\CFK^\infty(Y,K)$ spanned by generators $[\xs,i,j]$ for which $i\le 0$ and $j\le s$. To translate between versions, observe that the element $[\xs,i,j]$ corresponds in our notation to $\xs W^{-i}Z^{-j}$, and that the map $Z^{s}$ gives an isomorphism from the Ozsv{\'a}th-Szab{\'o} description of $A_s(Y,K)$ to our description of $A_s(Y,K)$. 
\end{rem}

\subsection{$\sigma$-basic models}
\label{sec:sigma-basic-models}

In this section, we describe some maximally asymmetric descriptions of the mapping cone formula, which are helpful for computations. These models will be familiar to the reader acquainted with the original Heegaard Floer surgery formula; we call these $\sigma$-basic models. The terminology is derived from \cite{MOIntegerSurgery} notion of a \emph{basic system of Heegaard diagrams} which are the underlying Heegaard diagrams used in the construction.

The $\sigma$-basic models are obtained by selecting the point $p$ to be the same as the point $w$. When we do this, the map $f_{w\to p}$ becomes the identity, and the map $v$ becomes $\theta_w \circ i_Z$. Typically, we also make the identification
\[
\bB(K)=Z^{-1} \cCFK(K),
\]
which is given by the map $\theta_w$. With respect to this identification, $v$ is equal to the map for localization
\[
v=i_Z\colon \cCFK(Y,K)\to Z^{-1} \cCFK(Y,K).
\]

With respect to the choice $p=w$, the map $h_\lambda$ takes the form
\[
h_\lambda=\theta_{z}^{-1} \circ T^\lambda\circ f_{z\to w} \circ \theta_w \circ i_W.
\]
The map
\[
\cF:=\theta_z^{-1}\circ T^\lambda\circ f_{z\to w} \circ \theta_w\colon W^{-1} \cCFK(K)\to Z^{-1} \cCFK(K)
\]
is typically referred to as the \emph{flip-map}. With the above notation, we can write
\[
h_\lambda=\cF\circ i_W.
\]

Sometimes the subspace of $W^{-1} \cCFK(K)$ in Alexander grading $s$ is denoted $\tilde{B}_{s}(Y,K)$. The map $\cF$ sends $\tilde{B}_s(Y,K)$ to $B_{s+\lambda}(Y,K)$. When $Y=S^3$, both $\tilde{B}_s(Y,K)$ and $B_{s+\lambda}(Y,K)$ are homotopy equivalent to $\bF\llsquare U\rrsquare$, so the map $\cF$ is determined up to chain homotopy. In computations, we can therefore pick any convenient choice for $\cF$. Frequently, one takes $\cF_s$, the restriction of $\cF$ to $\tilde{B}_s(Y,K)$, to be the map 
\begin{equation}
\cF_s=Z^{2s+\lambda} \iota_K
\end{equation}
where $\iota_K$ is the skew-graded and skew-equivariant map on $\cCFK(K)$ defined by the first author and Manolescu \cite{HMInvolutive}. (This map is discussed further in Section~\ref{sec:local-algebra}.)

\subsection{The surgery algebra}
\label{sec:background-surgery-algebra}
We now recall the reinterpretation from \cite{ZemBordered} of the surgery formula in terms of modules over an algebra. We first recall the algebra, denoted $\cK$. It is an algebra over the idempotent ring with two idempotents:
\[
\ve{I}=\ve{I}_0\oplus \ve{I}_1.
\]
Here $\ve{I}_\veps\iso \bF$ for $\veps=0,1$. We write $i_{\veps}$ for the generator of $\ve{I}_{\veps}$. We then set
\[
\ve{I}_{0}\cdot \cK\cdot \ve{I}_0=\bF[W,Z],\quad  \ve{I}_0\cdot \cK\cdot \ve{I}_1=0
\]
\[
\ve{I}_1\cdot \cK\cdot \ve{I}_0=\bF[U,T,T^{-1}]\otimes_{\bF} \langle \sigma,\tau \rangle\quad \text{and} \quad \ve{I}_1\cdot \cK\cdot \ve{I}_1=\bF[U,T,T^{-1}].
\]
The relations in the algebra are given as follows. Let
\[
\phi^\sigma,\phi^\tau \colon \bF[W,Z]\to \bF[U,T,T^{-1}]
\]
be the maps given by
\[
\phi^\sigma(W^i Z^j)=U^iT^{j-i}\quad \text{and} \quad \phi^\tau(W^i Z^j)=U^j T^{j-i}.
\]
If $a\in \ve{I}_0\cdot \cK\cdot \ve{I}_0$, then we have the relations
\[
\sigma \cdot a=\phi^\sigma(a)\cdot \sigma\quad \text{and} \quad \tau \cdot a=\phi^\tau(a)\cdot \tau.
\]

If $Y$ is a closed 3-manifold and $K\subset Y$ is a knot with Morse framing $\lambda$, then there is a type-$D$ module $\cX_{\lambda}(Y,K)^{\cK}$. For a rationally null-homologous knot, this can be computed by reformulating the data of the surgery formula of Ozsv\'{a}th and Szab\'{o} \cite{OSIntegerSurgeries, OSRationalSurgeries}, as we now describe. As a vector space, $\cX_{\lambda}(Y,K)^{\cK}\cdot \ve{I}_0$ is the $\bF$-span of a free $\bF[W,Z]$-basis of $\cCFK(Y,K)$.  As a vector space, $\cX_{\lambda}(Y,K)^{\cK}\cdot \ve{I}_1$ is the $\bF$-span of a free $\bF[U]$-basis of $\CF^-(Y)$.

The differential $\delta^1$ on $\cX_{\lambda}(Y,K)^{\cK}$ is constructed as follows. The terms of the differential weighted by elements of $\ve{I}_0\cdot \cK\cdot \ve{I}_0$ encode the differential of $\cCFK(Y,K)$. The terms weighted by $\ve{I}_1\cdot \cK\cdot \ve{I}_1$ encode the differential of $\CF^-(Y)$. Finally, the terms weighted by $\sigma$ and $\tau$ encode $v$ and $h_{\lambda}$, respectively. In more detail, if $\xs$ is a generator of $\cCFK(Y,K)$, and $v(\xs)$ has a summand of $\ys U^i T^j$, then $\delta^1(\xs)$ has a summand of $\ys\otimes  U^i T^j\sigma$. Similarly if $h_{\lambda}(\xs)$ has a summand of $\ys U^i T^j$, then $\delta^1(\xs)$ has a summand of $\ys\otimes U^i T^j \tau$. 

It is also possible to define the type-$D$ module $\cX_{\lambda}(Y,K)^{\cK}$ for homologically essential knots $K\subset Y$. We refer the reader \cite{ZemExactTriangle} for the details of this case.

The mapping cone chain complex $\bX_\lambda(Y,K)$ may be recovered from the above type-$D$ module by tensoring the type-$A$ module ${}_{\cK} \cD$ for a 0-framed solid torus:
\[
\bX_{\lambda}(Y,K)\iso \cX_{\lambda}(Y,K)^{\cK}\boxtimes {}_{\cK} \cD.
\]

We now recall the type-$A$ module  ${}_{\cK} \cD$ in more detail. As vector spaces, we have
\[
i_0\cdot \cD\iso \bF[W,Z]\quad \text{and} \quad i_1\cdot \cD\iso \bF[U,T,T^{-1}].
\]
The actions of $i_0\cdot \cK\cdot i_0$ and $i_1\cdot \cK\cdot i_1$ are given by polynomial multiplication. Additionally, the action of $i_1 \cdot \cK \cdot i_0$ is given by
\[
m_2(\sigma, W^i Z^j)=U^{i} T^{j-i}\quad \text{and} \quad m_2(\tau, W^iZ^j)=U^j T^{j-i}.
\]
The module ${}_{\cK} \cD$ has $m_j=0$ for $j>2$.

\subsection{The elliptic involution}
\label{sec:elliptic-involution}

We now recall that the elliptic involution of the torus induces a bimodule over the surgery algebra. We denote this bimodule by ${}_{\cK}[E]^{\cK}$. We recall that this bimodule is rank 1 over the idempotent ring $\ve{I}$, and has structure map
\[
\delta_2^1(a,1)=1\otimes E(a),
\]
where $E\colon \cK\to \cK$ is the algebra morphism specified by
\[
E(W^iZ^j)=W^j Z^i,\quad E(U)=U, \quad E(\sigma)=\tau,\quad E(\tau)=\sigma, \quad E(T)=T^{-1}.
\]

\begin{rem} The bimodule ${}_{\cK} [E]^{\cK}$ is described in \cite{ZemBordered}*{Remark~6.2}. In \cite{CZZSatellites}*{Theorem~1.6}, it is shown that in the chiral topology, this bimodule is homotopy equivalent to the $DA$-bimodule associated to the mapping cylinder of the elliptic involution on $\bT^2$.
\end{rem}

In Lemma~\ref{lem:change-orientation}, below, we will we see that the effect of reversing the orientation of a knot component of a link $L\subset Y$ is to tensor with the bimodule ${}_{\cK} [E]^{\cK}$.

There is a morphism
\[
\varpi\colon {}_{\cK}[E]^{\cK} \boxtimes {}_{\cK} \cD\to {}_{\cK} \cD,
\]
given by
\[
\varpi(W^iZ^j)=W^jZ^i\quad \text{and} \quad \varpi(U^iT^j)=U^i T^{-j}. 
\]

\section{The equivariant knot surgery formula}
\label{sec:equivariant-mapping-cone-statement}

In this section, we state our equivariant surgery formulas in the case of knots. Additionally we present several examples. The goal of this section is to be a user's guide to the techniques.

\subsection{Main statement}

We now introduce our equivariant surgery formulas. We first state them in terms of the surgery algebra formulation from \cite{ZemBordered}. In the subsequent section we recast them in terms of the mapping cone formula and carry out several simple examples.

Let  $
\phi\colon (Y,K)\to (Y,K)$ be a strongly framed diffeomorphism.  The statement of the equivariant surgery theorem takes one of two forms, depending on whether the diffeomorphism preserves or reverses the orientation of $K$. When $\phi$ is orientation preserving, there is an induced type-$D$ morphism
\[
\cX(\phi)\colon \cX_{\lambda}(Y,K)^{\cK}\to \cX_{\lambda}(Y,K)^{\cK}.
\]
If instead $\phi$ is orientation reversing, then there is an induced diffeomorphism
\[
\cX(\phi)\colon \cX_{\lambda}(Y,K)^{\cK}\to \cX_{\lambda}(Y,K)^{\cK}\boxtimes{}_{\cK} [E]^{\cK}.
\]

\begin{rem} Naturality of the link surgery complexes, proven in Section~\ref{sec:naturality}, implies that the maps $\cX(\phi)$ are well-defined up to chain homotopy.
\end{rem}

The equivariant knot surgery formula takes the following form:

\begin{thm}
\label{thm:equivariant-knot-surgery} Let $K\subset Y$ be a strongly framed knot with Morse framing $\lambda$ and let $\phi\colon (Y,K)\to (Y,K)$ be a diffeomorphism of strongly framed knots. Let $\Phi\colon Y_{\lambda}(K)\to Y_{\lambda}(K)$ denote the induced diffeomorphism. There is a homotopy equivalence
\[
\Gamma\colon \bX_{\lambda}(K)=\cX_{\lambda}(Y,K)^{\cK}\boxtimes {}_{\cK} \cD\to \ve{\CF}^-(Y_{\lambda}(K))
\]
satisfying the following:
\begin{enumerate}
\item If $\phi$ preserves the orientation of $K$, then $\Gamma$ intertwines the map $\ve{\CF}^-(\Phi)$ with the map $\cX(\phi)\boxtimes \bI_{\cD_0}$.
\item If $\phi$ reverses the orientation of $K$, then $\Gamma$ intertwines the map $\ve{\CF}^-(\Phi)$ with the composition shown below:
\[
\begin{tikzcd}[column sep=1.5cm]
{\cX_{\lambda}(Y,K)^{\cK}\boxtimes {}_{\cK} \cD}
	\ar[r, "{\cX(\phi)\boxtimes \bI}"]
&
{\cX_{\lambda}(Y,K)^{\cK} \boxtimes {}_{\cK} [E]^{\cK} \boxtimes {}_{\cK} \cD}
	\ar[r, "{\bI\boxtimes \varpi}"]
&
{ \cX_{\lambda}(Y,K)^{\cK}\boxtimes {}_{\cK} \cD}.
\end{tikzcd}
\]
\end{enumerate}
The statements also hold if we view ${}_{\cK} \cD$ as a bimodule ${}_{\cK} \cD_{\bF[U]}$. 
\end{thm}

\subsection{The equivariant mapping cone formula}

We now translate the statement of Theorem~\ref{thm:equivariant-knot-surgery} into a more traditional statement about the mapping cone formula. We focus on the case of a knot $K$ in an integer homology 3-sphere $Y$. 

We first recall the more basic isomorphism
\[
\bX_{\lambda}(Y,K)\iso \cX_{\lambda}(Y,K)\boxtimes {}_{\cK} \cD.
\]
In this case, we write
\[
\bX_{\lambda}(Y,K)\iso \Cone(v+h_{\lambda}\colon \bA(Y,K)\to \bB(Y,K)).
\]
Here, as previously, 
\[
\bA(Y,K)\iso \cCFK(Y,K)\iso \prod_{s\in \Z} A_s(K)\quad \text{and} \quad \bB(Y,K)\iso \prod_{s\in \Z} B_s(Y,K)\iso \prod_{s\in \Z} \CF^-(Y).
\]
In particular we recall that $A_s(K)$ denotes the $\bF[U]$-subcomplex of $\cCFK(K)$ in Alexander grading $s$.  The map $v$ maps $A_s(K)$ to $B_s(K)$ while $h_{\lambda}$ maps $A_s(K)$ to $B_{s+\lambda}(K)$.

We now describe the structure of the map $\bX(\phi)$ on the mapping cone formula. As in the statement of Theorem~\ref{thm:equivariant-knot-surgery}, the precise form depends on whether $\phi$ preserves or reverses the orientation of $K$. 

We consider first the case that $\phi$ preserves the orientation of $K$. In this case, the map $\bX(\phi)$ takes the form of the vertical direction in the following diagram.
\begin{equation}
\begin{tikzcd} \bX_{\lambda}(Y,K) \ar[d, "\bX(\phi)"] \\ \bX_{\lambda}(Y,K)
\end{tikzcd}
=
\begin{tikzcd} \bA(Y,K)
	\ar[d, "\phi_0"] 
	\ar[r, "v+h_{\lambda}"]
	\ar[dr, "g+j"]
& 
\bB(Y,K)
	\ar[d, "\phi_1"]
	\\
\bA(Y,K)
	\ar[r, "v+h_{\lambda}"]
& \bB(Y,K)
\end{tikzcd}
\label{eq:X(phi)-schematic}
\end{equation}
In the above:
\begin{enumerate}[label=($p$-\arabic*), ref=$p$-\arabic*]
\item The map $\phi_0$ is an $\bF[W,Z]$-equivariant chain map, and maps $A_s$ to $A_s$.
\item The map $\phi_1$ is an $\bF[U,T,T^{-1}]$-equivariant chain map, and maps $B_s$ to $B_s$.
\item The map $g$ sends $A_s$ to $B_s$. It satisfies
\[
\d(g)=v \circ \phi_0+\phi_1\circ v.
\]
The map $g$ satisfies $g(a\cdot x)=\phi^\sigma(a)\cdot g(x)$ for $a\in \bF[W,Z]$ and $x\in \bA(Y,K)$.
\item The map $j$ sends $A_s$ to $B_{s+\lambda}$. It satisfies
\[
\d(j)=h_{\lambda}\circ \phi_0+\phi_1\circ h_{\lambda}. 
\]
The map $j$ satisfies $j(a\cdot x)=\phi^\tau(a)\cdot j(x)$ for $a\in \bF[W,Z]$ and $x\in \bA(Y,K)$.
\end{enumerate}

We now consider the case when $\phi$ reverses the orientation of the knot $K$. We recall some terminology first. If $f$ is a map between $\bF[W,Z]$-modules, we say $f$ is \emph{skew-equivariant} if $f(Wx)=Z f(x)$ and $f(Zx)=Wf(x)$ for all $x$. Similarly if $f$ is a map between $\bF[U,T,T^{-1}]$-modules, we say that $f$ is \emph{skew-equivariant} if $f(U x)=U  f(x)$ and $f(T^i x)=T^{-i} f(x)$, for all $x$ and $i\in \Z$.

 When $\phi$ reverses the orientation of the knot $K$, the map $\bX(\phi)$ takes the same form as in Equation~\eqref{eq:X(phi)-schematic}. In this case, we have:
\begin{enumerate}[label=($n$-\arabic*), ref=$n$-\arabic*]
\item The map $\phi_0$ is an $\bF[W,Z]$-skew-equivariant chain map, and maps $A_s$ to $A_{-s}$.
\item The map $\phi_1$ is an $\bF[U,T,T^{-1}]$-skew-equivariant chain map, and maps $B_s$ to $B_{-s+\lambda}$.
\item The map $g$ sends $A_s$ to $B_{-s+\lambda}$. It satisfies
\[
\d(g)=h_{\lambda} \circ \phi_0+\phi_1\circ v.
\]
The map $g$ satisfies $g(a\cdot x)=(E\circ\phi^\sigma)(a)\cdot g(x)$ for $a\in \bF[W,Z]$ and $x\in \bA(Y,K)$.
\item The map $j$ sends $A_s$ to $B_{-s}$. It satisfies
\[
\d(j)=v\circ \phi_0+\phi_1\circ h_{\lambda}. 
\]
The map $j$ satisfies $j(a\cdot x)=(E\circ \phi^\tau)(a)\cdot j(x)$ for $a\in \bF[W,Z]$ and $x\in \bA(Y,K)$.
\end{enumerate}

\subsection{Knots in $S^3$}
\label{sec:knots}

As a general principle, the structural properties of the surgery formula allow for efficient computations for knots in $S^3$. As an example, the entire surgery complex $\bX_{\lambda}(K)$ is determined by  $\cCFK(K)$ when $K\subset S^3$. This principle also extends to refinements of the surgery formula, such as the involutive knot surgery formula \cite{HHSZExact}.

We begin by introducing some terminology concerning gradings. Note that the idempotent 0 subspace of $\cX_n(K)^{\cK}$ can naturally be given a $(\gr_{\ws},\gr_{\zs})$-bigrading. We additionally view the idempotent 1 subspace of $\cX_n(K)^{\cK}$ as having a $(\gr,A)$-bigrading, where, $\gr$ denotes the Maslov grading and $A$ denotes the Alexander grading. In the special case of a knot in $S^3$, the idempotent 1 subspace of $\cX_n(K)^{\cK}$ can be chosen to have a single generator $\xs$, for which we declare $\gr(\xs)=A(\xs)=0$. We set $\gr(U)=-2$, $A(U)=0$ and $A(T^{\pm 1})=\pm 1$.

\begin{rem}Note that when using the $\sigma$-basic construction of $\cX_{n}(K)^{\cK}$, the $\sigma$-labeled terms of the differential preserve the Alexander grading, while the $\tau$-labeled terms shift the Alexander grading by $n$. In general, it is helpful to relax this assumption, and consider modules where the $\sigma$-labeled differentials shift the Alexander grading by some $k\in \Z$, while the $\tau$-labeled differentials shift the Alexander grading by $n+k$. For example, the two type-$D$ modules below are both homotopy equivalent to the $+1$ framed invariant of the unknot:
\[
\begin{tikzcd} \xs_1 \ar[r, "\sigma+T \tau"] & \xs_0
\end{tikzcd}
\qquad \text{and} \qquad
\begin{tikzcd} \xs_1 \ar[r, "T^{-1}\sigma+\tau"] & \xs_0
\end{tikzcd}.
\]
\end{rem}

We make the following definition about morphisms between type-$D$ modules for knots in $S^3$:

\begin{define} Suppose that $K,K'\subset S^3$ are knots, and $\cX_n(K)^{\cK}$ and $\cX_n(K')^{\cK}$ are their modules representing their bordered invariants. Suppose that the $\sigma$ arrows of $\cX_n(K)^{\cK}$ shift the Alexander grading by $A_\sigma$, and the $\sigma$ arrows of $\cX_n(K')^{\cK}$ shift the Alexander grading by $A_\sigma'$. We say that a map
\[
f^1\colon \cX_n(K)^{\cK} \to \cX_n(K')^{\cK}
\]
is \emph{grading preserving} if the following hold:
\begin{enumerate}
\item The component which preserves idempotent 0 preserves the $(\gr_{\ws},\gr_{\zs})$-bigrading.
\item The component which preserves idempotent 1 preserves the Maslov grading $\gr$, and shifts the Alexander grading by $A_\sigma'-A_\sigma$.
\item The components which are weighted by $\tau$ map $\gr_{\zs}$-grading $n$ to $\gr$-grading $n+1$. Furthermore, they shift the Alexander grading by $A_\sigma'-A_\sigma+n$.
\item The components which are weighted by $\sigma$ map $\gr_{\ws}$-grading $n$ to $\gr$-grading $n+1$. Furthermore, these components shift the Alexander grading by $A_\sigma'-A_\sigma$. 
\end{enumerate}
\end{define}

\begin{prop}\label{prop:X(phi)-from-CFK(phi)} Let $K,K'\subset S^3$ be knots, and let
\[
f^1\colon \cCFK(K)^{\bF[W,Z]}\to \cCFK(K')^{\bF[W,Z]}
\]
be a grading preserving chain map. Then, up to chain homotopy, there is a unique grading preserving type-$D$ morphism
\[
\cX(f^1)\colon \cX_n(K)^{\cK}\to \cX_n(K')^{\cK}
\]
whose induced map on $\cCFK$ is $f^1$. 
\end{prop}

\begin{proof} We first recast the standard proof that the complex $\bX_{\lambda}(K)$ is determined by the complex $\cCFK(K)$. We phrase this in terms of the type-$D$ module $\cX_{\lambda}(K)^{\cK}$. We define the idempotent 0 subspace of $\cX_{\lambda}(K)^{\cK}$ to be $\cCFK(K)$ (viewed as a type-$D$ module over $\bF[W,Z]$), and we view idempotent 1 as consisting of a single generator, with vanishing differential. Let us write $[1]^{\bF[U,T,T^{-1}]}$ for this complex (with one generator and vanishing differential). In idempotent 1, there is a Maslov grading, which we denote by $\gr$, and an Alexander grading $A$.
 We give the generator of the module $[1]$ Maslov and Alexander grading 0.

 The terms weighted by $\sigma$ in the differential of $\cX_{\lambda}(K)^{\cK}$ are determined by a homotopy equivalence
\[
v^1\colon \cCFK(K)^{\bF[W,Z]}\boxtimes {}_{\bF[W,Z]} [\phi^\sigma]^{\bF[U,T,T^{-1}]}\to [1]^{\bF[U,T,T^{-1}]},
\]
and the map $h_{\lambda}$ is determined by a homotopy equivalence
\[
h^1\colon \cCFK(K)^{\bF[W,Z]}\boxtimes {}_{\bF[W,Z]} [\phi^\tau]^{\bF[U,T,T^{-1}]}\to [1]^{\bF[U,T,T^{-1}]}.
\]
The map $v^1$ intertwines the $\gr_{\ws}$ grading of $\cCFK$ with the $\gr$ grading, and is homogeneously graded with respect to the Alexander grading. Write $A_{\sigma}$ for the shift in the Alexander grading of $v^1$.

 The map $h^1$ intertwines $\gr_{\zs}$ with $\gr$, and is homogeneously graded with respect to the Alexander grading. We write $A_\tau$ for the shift in Alexander grading of $h^1$. By construction
 \[
 A_\tau-A_\sigma=n.
 \]
  We observe that $[1]^{\bF[U,T,T^{-1}]}$ has a unique self-homotopy equivalence in each Alexander grading. Therefore, if $A_\sigma$ is fixed, the maps $v^1$ and $h^1$ are then determined up to chain homotopy. The modules obtained from different choices of $A_\sigma$ are related by a homotopy equivalence.
   This implies that the type-$D$ module $\cX_{\lambda}(K)^{\cK}$ is itself determined by $\cCFK(K)$ up to homotopy equivalence.

We now move on to the type-$D$ morphisms $\cX(f^1)$.  There are two cases to consider. The first is that $f^1$ is not null-homotopic after localizing at $W$ or $Z$. The second is that $f^1$ becomes null-homotopic after localizing at $Z$ or $W$. (It is straightforward to see that $f^1$ is null-homotopic after localizing at $W$ if and only $f^1$ is null-homotopic after localizing at $Z$). 

We consider first the case that $f^1$ is not-null-homotopic after localizing at $W$ or $Z$. We now define the type-$D$ morphism $\cX(f^1)$ as follows. We define the component of $\cX(f^1)$ which preserves idempotent 0 to be $\cX(f^1)$. Write
\[
A_\sigma, A_\tau,  A_\sigma',A_\tau'\in \Z
\]
for the shifts in the $\sigma$ and $\tau$ arrows of $\cX_n(K)^{\cK}$ and $\cX_n(K')^{\cK}$, respectively.

We define the component of $\cX(f^1)$ which preserves idempotent 1 to be multiplication by $T^{A'_\sigma-A_\sigma}$
 We define the $\sigma$ weighted component of $\cX(f^1)$ by picking a suitably graded map $g^1$ mapping the following diagram a hypercube of chain complexes:
\[
\begin{tikzcd}[column sep=1.5cm, row sep=1.5cm]
\cCFK(K)\boxtimes {[\phi^\sigma]}
	\ar[d, "v^1"]
	\ar[r, "f^1\boxtimes \bI"] 
	\ar[dr, "g^1"]
&
	\cCFK(K')\boxtimes {[\phi^\sigma]} 
	\ar[d, "v^1"]
\\
{[1]}^{\bF[U,T,T^{-1}]} 
	\ar[r, "T^{A_\sigma'-A_\sigma}"] &
{[1]}^{\bF[U,T,T^{-1}]}
\end{tikzcd}
\]
We can see that $g^1$ exists because the complexes at each vertex of the cube are homotopy equivalent to $[1]^{\bF[U,T,T^{-1}]}$, and $v^1\circ f^1\boxtimes \bI$ and $T^{A_\sigma'-A_\sigma}\circ v^1$ are homotopy equivalences with the same grading. Note that any two choices of $g^1$ which are suitably graded are chain homotopic since there are no non-zero endomorphisms of $[1]^{\bF[U,T,T^{-1}]}$ which increase the Maslov grading by $1$.  The $\tau$ weighted components of $\cX(f^1)$ are constructed similarly.  It is straightforward to see that any two suitably graded choices give chain homotopic type-$D$ morphisms.

Next, we consider the case that $f^1$ becomes null-homotopic after localizing at $W$ or $Z$. In this case, the map
\[
f^1\boxtimes \bI \colon \cCFK(K)\boxtimes [\phi^\sigma]\to \cCFK(K)\boxtimes [\phi^\tau]
\]
is null-homotopic, and similar when we replace $\phi^\sigma$ with $\phi^\tau$. We define the $\sigma$ weighted component of $\cX(f^1)$ by picking a map $g^1$ which makes the following diagram a hypercube:
\[
\begin{tikzcd}[column sep=1.5cm, row sep=1.5cm]
\cCFK(K)\boxtimes {[\phi^\sigma]}
	\ar[d, "v^1"]
	\ar[r, "f^1\boxtimes \bI"] 
	\ar[dr, "g^1"]
&
	\cCFK(K)\boxtimes {[\phi^\sigma]} 
	\ar[d, "v^1"]
\\
{[1]}^{\bF[U,T,T^{-1}]} 
	\ar[r, "0"] &
{[1]}^{\bF[U,T,T^{-1}]}
\end{tikzcd}.
\]
We define the $\tau$-weighted components of $\cX(f^1)$ similarly.  As before, any two choices are chain homotopic.
\end{proof}

One immediate corollary is that if $K\subset S^3$ is a knot, then the type-$D$ module $\cX_n(K)^{\cK}$ is uniquely determined up to homotopy equivalence by $\cCFK(K)$. Another helpful corollary is the following extension for morphisms:

\begin{cor}
\label{cor:CFK-to-XK}
Let $\phi\colon (S^3,K)\to (S^3,K)$ be a diffeomorphism of a strongly framed knot $K\subset S^3$. Then the type-$D$ morphism $\cX(\phi)$ is determined up to chain homotopy by the induced chain map
\[
\cCFK(\phi)\colon \cCFK(K)\to \cCFK(K).
\]
\end{cor}
\begin{proof}
When $\phi$ preserves the orientation of the knot $K$, then $\cCFK(\phi)$ is grading preserving so the claim follows immediately from Proposition~\ref{prop:X(phi)-from-CFK(phi)}.

When $\phi$ reverses the orientation of $K$, then $\cCFK(\phi)$ is skew-equivariant and skew-graded. Recall that a skew equivariant morphism $f\colon \cCFK(K)\to \cCFK(K)$ is the same as a type-$D$ morphism 
\[
f^1\colon \cCFK(K)^{\bF[W,Z]}\to \cCFK(K)^{\bF[W,Z]} \boxtimes {}_{\bF[W,Z]} [E_0]^{\bF[W,Z]}. 
\]
We can naturally define a grading on $\cX_n(K)^{\cK}\boxtimes {}_{\cK} [E]^{\cK}$ as follows. If $\xs$ is a generator in $\cX_n(K)$, we write $\bar{\xs}$ for the corresponding generator of $\cX_n(K)\boxtimes [E]$.  If $\xs_0$ is a generator of $\cX_n(K)\cdot i_0$ (i.e. a basis element of $\cCFK(K)^{\bF[W,Z]}$), we declare the corresponding generator $\bar{\xs}_0$ of $\cX_n(K)\boxtimes [E]\cdot i_0$ to have bigrading $(\gr_{\zs}(\xs_0), \gr_{\ws}(\xs_0))$. If $\xs_1$ is a generator of $\cX_n(K)\cdot i_1$, we declare the corresponding generator $\bar{\xs}_1$ of $\cX_n(K)\boxtimes [E]\cdot i_1$ to have grading $(\gr,A)(\bar{\xs}_1)=(\gr(\xs_1),-A(\xs_1))$. With these conventions, the map $\cCFK(\phi)$ is a grading preserving type-$D$ morphism from $\cCFK(K)^{\bF[W,Z]}$ to $\cCFK(K)^{\bF[W,Z]}\boxtimes {}_{\bF[W,Z]} [E_0]^{\bF[W,Z]}$, so Proposition~\ref{prop:X(phi)-from-CFK(phi)} extends to give a type-$D$ morphism
\[
\cX(\phi)\colon \cX_n(K)^{\cK}\to \cX_n(K)^{\cK}\boxtimes {}_{\cK}[E]^{\cK},
\]
which is unique up to chain homotopy. 
\end{proof}

\begin{rem}
A similar idea is used to show that a model of the involutive surgery complex $(\bX_n(S^3,K),\iota_{\bX})$ is determined algebraically by the $\iota_K$-complex $(\cCFK(K),\iota_K)$. See \cite{HHSZExact}*{Section~3.5}. 
\end{rem}

\subsection{Examples}

We now explain how to use the equivariant surgery formula in practice. 

We begin by considering the strong inversion on the trefoil knot $T_{2,3}$. We will write $\phi$ for this diffeomorphism. Before we begin the computation, we explain a helpful convention when working with the elliptic bimodule. Namely, we write $\theta$ for the generator of ${}_{\cK} [E]^{\cK}$. We view $\theta$ as satisfying $a\cdot \theta= \theta\cdot E(a)$. Schematically,
\[
f^{1}\colon \cX^{\cK}\to \cX^{\cK}\boxtimes {}_{\cK} [E]^{\cK}
\]
is a type-$D$ morphism such that $\ys\otimes a$ is a summand of $f^1(\xs)$, then we will indicate in our diagrams with an arrow
\[
\begin{tikzcd}
 \xs \ar[r, "\theta a",red] & \ys.
\end{tikzcd}
\]

\begin{rem}
\label{rem:theta-algebra-element}
It is helpful to think of $\theta$ as an algebra element which satisfies 
\[
\theta W=Z \theta,\quad \theta Z=W\theta, \quad \theta T^{\pm 1}=T^{\mp 1} \theta, \quad \theta \tau=\sigma \theta,\quad  \theta \tau= \sigma \theta, \quad \text{and} \quad U \theta=\theta U.
\]
\end{rem}

It is straightforward to use the grading constraints on the map induced by the strong inversion to see that the map induced on the knot Floer complex switches $\xs$ and $\zs$ and fixes $\ys$.

Applying the procedure from Corollary~\ref{cor:CFK-to-XK} gives the following model for the induced type-$D$ module morphism on $\cX_{+1}(T_{2,3})^{\cK}$:
\[
\begin{tikzcd}[labels=description]
& \ys
	\ar[dl, "W"]
	\ar[dr, "Z"]
	\ar[loop above, red, looseness=30, "\theta"]
\\
\xs
	\ar[rr, leftrightarrow, "\theta",red]
	\ar[dr, "T\sigma+UT^2 \tau ",labels=left]
&& 
\zs
	\ar[dl, "UT^{-1}\sigma+\tau", labels=right]
\\[1.5cm]
&\ve{p}
	\ar[loop below,looseness=30, "T\theta",red]
\end{tikzcd}
\]

Using the above, we now compute the equivariant mapping cone $(\bX_{+1}(T_{2,3}), \bX(\phi))$. We will compute the induced diffeomorphism map $\bX(\phi)$ on the truncation taking the following shape:
\[
\begin{tikzcd}[column sep={1.25cm,between origins}, labels=description] A_{-1} \ar[dr, "h_1"]  & A_0\ar[d, "v"] \ar[dr, "h_{1}"] & A_1 \ar[d, "v"]\\
&B_0& B_1
\end{tikzcd}
\]

We recall that we view $\bX_{+1}(T_{2,3})$ as the box tensor product of $\cX_{+1}(T_{2,3})^{\cK}$ with the solid torus module ${}_{\cK} \cD$. This means that we can decompose elements of $A_s(T_{2,3})$ into sums of elementary tensors $\xs|a$, where $\xs$ is a free $\bF[W,Z]$-generator of $\cCFK(K)$, $a\in \bF[W,Z]$, and $A(\xs)+A(a)=s$. (Here we write $|$ for the tensor product). The modules $B_s(T_{2,3})$ have very simple $\bF[U]$-module generators; they are of the form $\ve{p}|T^s$ for $s\in \Z$.  With this notation, the truncation of $\bX_{+1}(T_{2,3})$ and the map $\bX(\phi)$ are shown below:
\[
\begin{tikzcd}[column sep={1.4cm,between origins}, labels=description]
& 
{\ys|W}
	\ar[dl, "1"]
	\ar[dr, "U"]
	\ar[rrrrrr,red, bend left, leftrightarrow, "1"]
&&&
\ys|1
	\ar[dl, "1"]
	\ar[dr, "1"]
	\ar[loop above, red, "1"]
&&& 
\ys |Z
	\ar[dl, "U"]
	\ar[dr, "1"]
\\
\xs |W^2
	\ar[drrrr, "U"]
&&
\zs|1
	\ar[drr, "1"]
&
\xs |W
	\ar[dr, "U"]
	\ar[drrrr, "U"]
	\ar[rr, leftrightarrow, red, "1"]
&&
\zs |Z
	\ar[dl, "U",crossing over]
	\ar[drr, "U"]
&
\xs|1
	\ar[dr, "1"]
	\ar[llll,red, bend left=18,crossing over, leftrightarrow, "1"]
&&
 \zs |Z^2
	\ar[dl, "U"]
	\ar[llllllll,crossing over,bend left=15, leftrightarrow,red, "1"]
\\[1.6cm]
&&&\,&
\ve{p}|T^0
	\ar[rrr, leftrightarrow, red, "1"]
&&\,& \ve{p}|T
\end{tikzcd}
\]

The red arrows denote $\bX(\phi)$ while the black arrows denote the differential of $\bX_{+1}(T_{2,3})$. We illustrate how $\bX(\phi)$ is computed with a few simple examples. We recall that $\bX(\phi)$ is defined as the composition:
\[
\begin{tikzcd}
\cX_{+1}(T_{2,3})^{\cK}\boxtimes {}_{\cK} \cD\ar[r, "\cX(\phi)\boxtimes \bI"] & \cX_{+1}(T_{2,3})^{\cK}\boxtimes {}_{\cK} [E]^{\cK} \boxtimes {}_{\cK} \cD \ar[r, "\bI\boxtimes \varpi"] & \cX_{+1}(T_{2,3})^{\cK} \boxtimes {}_{\cK} \cD. 
\end{tikzcd}
\]
As an example of how the computation is performed, $\bX(\phi)(\xs W^2)$ is computed as below:
\[
\begin{tikzcd}[column sep=2cm]
\xs | W^2
	\ar[r, "\cX(\phi)\boxtimes \bI",mapsto]
&
 \zs| \theta|W^2
	\ar[r, "\bI\boxtimes \varpi",mapsto]
&
\zs |Z^2.
\end{tikzcd}
\]

\section{Local formulas for equivariant surgeries on knots}
\label{sec:local-formulas}
In this section we establish formulas for the equivariant surgery complex on a knot, in analogy with the local formulas for the involutive surgery formula of \cite[Theorem 1.6 and Proposition 22.9]{HHSZExact}. The cases of strongly invertible and two-periodic knots are quite different, and we therefore address them separately.

\subsection{The algebra of local equivalence} \label{sec:local-algebra}

In this section we review the algebra of local equivalence for chain complexes with endomorphisms.

\subsubsection{Phi-complexes}

\begin{define} A \emph{generalized $\phi$-complex} (or \emph{generalized phi-complex}) is pair $(C,\phi)$ such that the following hold:
\begin{enumerate}
\item $C$ is a finitely-generated free $\mathbb F[U]$-complex equipped with a relative $\mathbb Z$-grading and absolute $\mathbb Q$-grading $\gr$, such that the differential $\partial$ has $\gr(\partial)=-1$ and the action of $U$ has $\gr(U)=-2$.
\item There is an isomorphism $U^{-1}H_*(C) \simeq \mathbb F[U, U^{-1}]$ up to an overall grading shift.
\item $\phi$ is a grading-preserving chain map which induces an isomorphism on $H_*(C)$.
\end{enumerate}
If in addition $\phi^2$ is chain homotopic to the identity, we say that $(C,\phi)$ is a \emph{$\phi$-complex} or a \emph{true $\phi$-complex}.
\end{define}

Note that a true $\phi$-complex is algebraically the same as an $\iota$-complex, in \cite{HMZConnectedSum}*{Definition~8.1}. 

Given two generalized $\phi$-complexes $(C, \phi)$ and $(C', \phi')$, a homogeneously-graded $\mathbb F[U]$-chain map $f \co C \to C'$ is said to be a \emph{$\phi$-homomorphism} if $f \circ \phi \simeq \phi' \circ f$. Two generalized $\phi$-complexes are said to be \emph{$\phi$-equivalent}, or \emph{strongly equivalent}, if there is a chain homotopy equivalence between them which is also a $\phi$-homomorphism. 

Given two generalized $\phi$-complexes, there is a `product' operation given by 
\begin{equation} \label{eq:group-relation-I} (C,\phi) \times (C', \phi') \simeq (C\otimes_{\mathbb F[U]} C', \phi \otimes \phi').
\end{equation}

\begin{define} Let $(C,\phi)$ and $(C',\phi')$ be generalized $\phi$-complexes. 
\begin{itemize}
\item A \emph{local map} from $(C,\phi)$ to $(C',\phi')$ is a grading-preserving $\phi$-homomorphism $F\co C\rightarrow C'$ which induces a graded isomorphism between $U^{-1}H_*(C)$ and $U^{-1}H_*(C')$.
\item We say that $(C, \phi)$ and $(C',\phi')$ are \emph{locally equivalent} if there is a local map from $(C,\phi)$ to $(C',\phi')$ and a local map from $(C',\phi')$ to $(C,\phi)$.
\end{itemize}
\end{define}

The set of local equivalence classes of true $\phi$-complexes forms an abelian group $\frI$, with product given by the relationship $\times$ in \eqref{eq:group-relation-I}. Inverses are given by dualizing both $C$ and $\phi$ with respect to $\mathbb F[U]$; we write $-(C,\phi)$ for this dual generalized $\phi$-complex. To a rational homology sphere $Y$ together with a $\Spin^c$ structure $\frs$ and a self-diffeomorphism $\phi$ of $Y$ preserving $\frs$, Heegaard Floer homology associates a $\phi$-complex $(\CF^-(Y), \CF(\phi))$. Ordinarily, a choice of basepoint would be required to associate a diffeomorphism map to $\CF^-(Y)$, but since $Y$ is a rational homology sphere, the choice of basepoint does not affect the diffeomorphism map by \cite{ZemGraphTQFT}*{Theorem~D}. See \cite{DHM-Corks}*{Section~4.1} for a detailed discussion. Where it will occasion no confusion we also write $\phi$ to denote $\CF(\phi)$. If $\phi$ is an orientation-preserving involution, $(\CF^-(Y), \phi)$ is a true $\phi$-complex. By \cite[Theorem 1.2]{DHM-Corks} this construction induces a homomorphism
\[ \Theta^{\mathrm{inv}}_{\Z} \rightarrow \frI.
\]
Similarly, there is a group consisting of local equivalence classes of generalized $\phi$-complexes, which is the target of a homomorphism from $\Theta_{\Z}^{\mathrm{diff}}$.

There is an additional, weaker, equivalence relation between iota-complexes, introduced in \cite{DHSThomcob} (see also \cite[Section 3.3]{HHSZExact}).
 
 \begin{define}[\cite{DHSThomcob}*{Definition 3.1}]
 	Let $C_1$ and $C_2$ be free, finitely generated chain complexes over $\bF[U]$,  such that each $C_i$ has an absolute $\Q$-grading and a relative $\Z$-grading with respect to which $U$ has grading $-2$. Two grading-preserving $\bF[U]$-module homomorphisms
 	\[ 
 	f, g 
 	\colon
 	C_1 \rightarrow C_2
 	\]
 	are \emph{homotopic mod $U$}, denoted $f \simeq g \mod U$, if there exists an $\bF[U]$-module homomorphism $H \colon C_1 \rightarrow C_2$ such that $H$ increases grading by one and
 	\[
 	f + g + H \circ \d + \d \circ H \in \im U.
 	\]
 \end{define}

\begin{define}[\cite{DHSThomcob}*{Definition 3.2}]
 An \emph{almost phi-complex} (or \emph{almost $\phi$-complex}) $\cC=(C,\bar{\phi})$ consists of the following data:
 \begin{itemize}
 \item A free, finitely-generated chain complex $C$ over $\bF[U]$ with an absolute $\mathbb Q$ grading and relative $\mathbb Z$ grading, with
 \[
U^{-1} H_*(C)\iso \bF[U,U^{-1}]. 
 \]
 \item A grading-preserving $\bF[U]$-module homomorphism $\bar{\phi}\colon C\to C$ such that
 \[
\bar{\phi}\circ \d+\d \circ \bar{\phi}\in \im U\qquad \text{and} \qquad \bar{\phi}^2\simeq \id \mod U. 
 \]
 \end{itemize}
 \end{define}
 
Of course, any phi-complex induces an almost phi-complex. The definition of tensor product of almost phi-complexes is the same as equation~\eqref{eq:group-relation-I}.

In analogy with the terminology above, an \emph{almost $\phi$-homomorphism} from $(C_1, \bar{\phi}_1)$ to $(C_2, \bar{\phi}_2)$ is a homogeneously-graded, $\bF[U]$-equivariant chain map $f \colon C_1 \rightarrow C_2$ such that $f \circ \bar{\phi}_1 \simeq \bar{\phi}_2 \circ f \mod U.$ We then have the following new relation between almost $\phi$-complexes.

\begin{define}[\cite{DHSThomcob}*{Definition~3.5}]Suppose $(C_1,\bar{\phi}_1)$ and $(C_2,\bar{\phi}_2)$ are almost $\phi$-complexes.
\begin{enumerate}
\item An \emph{almost local map} from $(C_1,\bar{\phi}_1)$ to $(C_2,\bar{\phi}_2)$ is a grading-preserving almost $\phi$-homomorphism $F\colon C_1\to C_2$, which induces an isomorphism from $U^{-1} H_*(C_1)$ to $U^{-1} H_*(C_2)$.
\item We say that $(C_1,\bar{\phi}_1)$ are $(C_2,\bar{\phi}_2)$ are \emph{almost locally equivalent} if there is an almost local map from $(C_1,\bar{\phi}_1)$ to $(C_2,\bar{\phi}_2)$, as well as an almost local map from $(C_1,\bar{\phi}_1)$ to $(C_2,\bar{\phi}_2)$.
\end{enumerate}
\end{define} 

Using the definitions above, one may construct an almost local equivalence group $\widehat{\frI}$ of almost $\phi$-complexes. It is a non-trivial result that $\widehat{\frI}$ can be parametrized explicitly \cite{DHSThomcob}*{Theorem~6.2}, as we now describe.  To a sequence $(a_1, b_2, a_3, b_4, \dots, a_{2m-1}, b_{2m})$, where $a_i \in \{ \pm\}$ and $b_i \in \Z \setminus \{0\}$, we may associate an almost $\phi$-complex 
 \[
 C(a_1, b_2, a_3, b_4, \dots, a_{2m-1}, b_{2m}),
 \]
  called the \emph{standard complex of type $(a_1, b_2, a_3, b_4, \dots, a_{2m-1}, b_{2m})$}, as follows.  The standard complex is freely generated over $\bF[U]$ by $t_0, t_1, \dots, t_{2m}$. For each symbol $a_i$, we introduce an $\omega=(1+\phi)$-arrow between $t_{i-1}$ and $t_i$ as follows:
 \begin{itemize}
 	\item If $a_i = +$, then $\omega t_{i} = t_{i-1}$.
 	\item If $a_i = -$, then $\omega t_{i-1} = t_i$. 
 \end{itemize}
 For each symbol $b_i$, 
 we introduce a $\d$-arrow between $t_{i-1}$ and $t_i$ as follows:
 \begin{itemize}
 	\item If $b_i>0$, then $\d t_{i} = U^{|b_i|} t_{i-1}$.
 	\item If $b_i<0$, then $\d t_{i-1} = U^{|b_i|} t_{i}$.
 \end{itemize}
 
In computations with standard complexes, it will frequently be convenient to represent the group operation with $+$ instead of $\otimes$. The dual of the standard complex $C(a_1, b_2, a_3, b_4, \dots, a_{2m-1}, b_{2m})$ is the standard complex $-C(a_1, b_2, a_3, b_4, \dots, a_{2m-1}, b_{2m}) = C(-a_1, -b_2, -a_3, -b_4, \dots, -a_{2m-1}, -b_{2m})$, where if $a_i$ is $+$ then $-a_i$ is $-$ and vice versa.

Every element of $\widehat{\frI}$ is locally equivalent to a unique standard complex \cite{DHSThomcob}*{Theorem~6.2}.  Thus, in spite of $\widehat{\frI}$ being infinitely-generated, its elements are easy to describe.

Using this characterization, Dai, Hom, Stoffregen, and Truong define a sequence of homomorphisms $\varphi_n \co \hat{\frI} \to \mathbb Z$ for $n>1$, defined as follows. Given $(C,\bar{\phi})$ an almost $\phi$-complex, let $C(a_1, b_2, a_3, b_4, \dots, a_{2m-1}, b_{2m})$ be a standard complex representing it. Then
\[ \varphi_n(C,\bar{\phi}) = (\text{\# of parameters } b_i =n) - (\text{\# of parameters } b_i =-n).\]

We conclude by recalling two numerical invariants associated to a generalized $\phi$-complex, called its \emph{correction terms}. We follow the formulation \cite[Lemma 2.12]{HMZConnectedSum}, which is a reformulation of the original definitions from \cite{HMInvolutive}. The construction was originally for true phi-complexes; the observation that it generalizes is esse

\begin{define}
Let $(C,\phi)$ be a generalized phi-complex. Then $\dl(C,\phi)$ is the maximum grading of a homogeneous cycle $a \in C$ such that $[U^{n}a] \neq 0$ for all positive $n$ and furthermore there exists an element $b$ such that $\partial b= (\mathrm{id} + \phi)a$.
\end{define}

\begin{define}
Let $(C,\phi)$ be a generalized phi-complex. Consider triples $(x,y,z)$ consisting of elements of $C$, with at least one of $x$ or $y$ nonzero such that $\partial y = (\mathrm{id} + \phi)x$, $\partial z = U^{m} x$, for some $m \geq 0$ and $[U^n(U^{m}y + (\mathrm{id} + \phi)z)] \neq 0$ for all $n \geq 0$. If $x \neq 0$, assign this triple the value $\gr(x)+1$; if $x=0$, assign this triple the value $\gr(y)$. Then $\du(C,\phi)$ is defined as the maximum of these grading values across all valid triples $(x,y,z)$.
\end{define}

The two correction terms are exchanged with a negative sign under dualizing; in particular, $\dl((C,\phi))=-\du(-(C,\phi))$ and vice versa. We may treat the correction terms as topological invariants as follows.

\begin{define}
If $Y$ is a rational homology sphere with an orientation-preserving diffeomorphism $\phi$ and $\frs \in \Spin^c(Y)$ is fixed by $\phi$ then 
\[\dl^{\phi}(Y,\frs) = \dl(\CF^-(Y, \frs), \phi) \qquad \qquad \qquad \du^{\phi}(Y,\frs) = \du(\CF^-(Y, \frs), \phi).\] \end{define}

\subsubsection{$\phi_K$-complexes} 
\label{sec:phi_K-complexes}

Let $(C_K, \d)$ be a bigraded, free, finitely generated complex over $\bF[W,Z]$, where the two terms in the bigrading are $\gr(x) = (\gr_{\ws}(x), \gr_{\zs}(x))$. We say $C_K$ is an \emph{abstract knot complex} if the gradings are such that $\gr(\d)=(-1,-1)$ while $\gr(W)=(-2,0)$ and $\gr(Z)=(0,-2)$, and moreover $C_K \otimes \bF[W,Z, W^{-1},Z^{-1}]$ is chain homotopy equivalent to $\bF[W,Z,W^{-1},Z^{-1}]$.  We refer to  $A=\tfrac{1}{2}(\gr_{\ws}-\gr_{\zs})$ as the \emph{Alexander grading}.

Given an abstract knot complex, there are two naturally associated maps
 \[
 \Phi,\Psi\colon C_K\to C_K,
 \]
 as follows. We write $\d$ as a matrix with respect to a free $\bF[W,Z]$-basis of $C_K$. We define $\Phi$ to be the endomorphism obtained differentiating each entry of this matrix with respect to $W$. We define $\Psi$ to be the endomorphism obtained by differentiating each entry with respect to $Z$. The maps $\Phi$ and $\Psi$ are independent of the choice of basis, up to $\bF[W,Z]$-equivariant chain homotopy \cite{ZemConnectedSums}*{Corollary~2.9}. These maps naturally appear in the context of knot Floer homology, see \cites{SarkarMaslov, ZemQuasi, ZemCFLTQFT}; in particular the \emph{Sarkar map} $\xi_K$ associated to the complex is given by the formula
\[ \xi_K = \mathrm{Id} + \Psi \circ \Phi.\]
It is a standard fact that $\xi_K^2 \simeq \mathrm{Id}$.

We say an $\bF$-linear map $F\colon C_K\to C_K'$ is \emph{skew-$\bF[W,Z]$-equivariant} if 
 \[
 F\circ Z=W\circ F\quad \text{and} \quad F\circ W=Z\circ F.
 \]
 Similarly, we say a map $F$ is \emph{skew-graded} if $F$ interchanges $\gr_{\ws}$ and $\gr_{\zs}$.
 

As discussed in Section~\ref{sec:equivariant-knots}, there are two types of symmetries arising from knots that we wish to consider. We begin with the strongly invertible symmetries.

\begin{define} A strongly invertible $\phi_K$-complex consists of a pair $(C_K, \phi_K)$ such that
\begin{itemize}
\item $C_K$ is an abstract knot complex;
\item $\phi_K$ is a skew-graded, skew-$\bF[W,Z]$-equivariant chain map such that $\phi_K^2 \simeq \mathrm{Id}$.
\end{itemize}
\end{define}

A \emph{$\phi_K$-homomorphism} from $(C_K,\phi_K)$ to $(C_{K'},\phi_{K'})$ is a chain map $f\colon C_K\to C_{K'}$ such that
\[
f\circ \phi_K \simeq \phi_{K'}\circ f.
\]
A $\phi_K$-homomorphism is called a \emph{local map} if it induces an isomorphism from $U^{-1} H_*(C_K)$ to $U^{-1} H_*(C_{K'})$. We say that two $\phi_K$-complexes are locally equivalent if there are local maps in both directions. 

There is a composition operation
\begin{equation} (C_K, \phi_K) \times (C_{K'}, \phi_{K'}) = (C_K\otimes_{\bF[W,Z]} C_{K'}, \phi_K \otimes \phi_K').\end{equation}
which preserves local equivalence. Given this, one may define a group $\frI_K$ of local equivalence classes of strongly-invertible $\phi_K$-complexes, again with inverses given by dualizing.

If $K \subseteq S^3$ is a strongly invertible knot, then by \cite[Theorem 1.7]{DMS:equivariant}, if $\phi_K \co \cCFK(K) \rightarrow \cCFK(K)$ is the map induced by the symmetry, the pair $(\cCFK(K), \phi_K)$ has the structure of a strongly invertible $\phi_K$-complex. Indeed, the authors show \cite[Theorem 1.8]{DMS:equivariant} there is a homomorphism
\[\widetilde{\cC} \rightarrow \frI_K.\]
where as in Section~\ref{sec:equivariant-knots} the group $\widetilde{\cC}$ is the concordance group of directed strongly invertible knots.

We now turn to the case of periodic symmetries. 

\begin{define} A $p$-periodic $\phi_K$-complex consists of a pair $(C_K, \phi_K)$ such that
\begin{itemize}
\item $C_K$ is an abstract knot complex;
\item $\phi_K$ is a grading-preserving, $\bF[W,Z]$-equivariant chain map such that $\phi_K^p \simeq \xi_K$.
\end{itemize}
\end{define}

In the same fashion as before, a \emph{local map} from $(C_K,\phi_K)$ to $(C_{K'},\phi_{K'})$ consists of a chain map $f\colon C_K\to C_{K'}$ such that
\[
f\circ \phi_K \simeq \phi_{K'}\circ f,
\]
and we say that two $p$-periodic $\phi_K$-complexes are locally equivalent if there are local maps in both directions.

A straightforward extension of the work of \cite[Section 8]{DMS:equivariant} shows that if $K$ is a $p$-periodic knot, then if $\phi_K \co \cCFK(K) \rightarrow \cCFK(K)$ is the map induced by the symmetry, the pair $(\cCFK(K), \phi_K)$ has the structure of a $p$-periodic $\phi_K$-complex; equivariant concordance preserves the local equivalence class of the complex. As there is no notion of equivariant connect sum for periodic knots, we do not attempt to form a group of complexes in this case.

Before continuing we recall an important property of the maps induced by periodic symmetries on knots. For a knot $K$ in $S^3$, let $\iota_K$ be the endomorphism induced by the conjugation action on $\cCFK(K)$ defined by the first author and Manolescu \cite[Section 6.1]{HMInvolutive}. The map $\iota_K$ is a skew-graded $\bF[W,Z]$-skew equivariant map such that $\iota_K^2 \simeq \xi_K$, and therefore $\iota_K^4 \simeq \mathrm{Id}$. We have the following (in \cite[Theorem 8.1]{DMS:equivariant}, this is stated for $p=2$; the proof is the same for general $p$).

\begin{lem} \cite[Theorem 8.1]{DMS:equivariant} \label{lem:periodic-commutes} Let $K$ be a $p$-periodic knot. Then the map $\phi_K$ induced by the action of $K$ on $\cCFK(K)$ has
\[\phi_K \circ \iota_K \simeq \iota_K \circ \phi_K.\]
\end{lem}

When $K$ is instead strongly invertible, a more complicated relation is satisfied between $\iota_K$ and $\phi_K$. See \cite[Theorem 1.7]{DMS:equivariant}.


\subsubsection{Almost $\phi_K$-complexes}

It will be helpful for our purposes to have a notion of \emph{almost $\phi_K$-complexes}. This is a straightforward amalgamation of the notions described above. We will view $U$ as being an element of $\bF[W,Z]$ by setting $U=WZ$. We make the following definition:
\begin{define} A \emph{almost $\phi_K$-complex} consists of a pair $(C,\bar{\phi}_K)$ satisfying the following:
\begin{enumerate}
\item   $C$ is an abstract knot complex.
\item $\bar{\phi}_K\colon C\to C$ is an $\bF[W,Z]$-skew equivariant map such that 
\[
\bar{\phi}_K\circ \d+\d \circ \bar{\phi}_K=0\mod U, \quad \text{and} \quad \bar{\phi}_K^2+\id \simeq 0 \mod U.
\]
 Furthermore, we assume that $\bar{\phi}_K$ interchanges $\gr_{\ws}$ and $\gr_{\zs}$.
\end{enumerate}
\end{define}

A \emph{local map} from $(C_1,\bar{\phi}_1)$ to $(C_2,\bar{\phi}_2)$ consists of a grading preserving, $\bF[W,Z]$-equivariant chain map $f\colon C_1\to C_2$ such that
\[
f\circ \bar{\phi}_1+\bar{\phi}_2\circ f\simeq 0\mod U,
\]
and such that $f$ induces an homotopy equivalence from $U^{-1} C_1$ to $U^{-1} C_2$. 
As in the previous cases, we say that two almost $\phi_K$-complexes are \emph{locally equivalent} if there are local maps in both directions.

If $(C,\bar{\phi}_K)$ is an almost $\phi_K$-complex, then there is an almost-$\phi$ complex
\[
(A_0(C),\bar{\phi}_K)
\]
where $A_0(C)\subset C$ denotes the subspace in Alexander grading 0.

The following is straightforward to prove:
\begin{lem}
\label{lem:almost-phi_K-local-to-almost-phi-local}
The map
\[
 (C,\bar{\phi}_K)\mapsto (A_0(K), \bar{\phi}_K)
\]
is well-defined on almost $\phi_K$-local classes.
\end{lem}

\subsection{Local formulas for strongly-invertible knots} \label{sec:local-si}In this section we prove Theorem~\ref{thm:local-class} and Theorem~\ref{thm:local-class-even-si}. We recall the first theorem. 

\begin{usethmcounterof}{thm:local-class}
 Suppose $K$ is a strongly-invertible knot in $S^3$, and $n >0$ is an integer.
 \begin{enumerate}
 \item The $\phi$-complex $(\CF^-(S^3_n(K),[0]),\phi)$ is locally equivalent to $(A_0(K),\phi_K)$, shifted upward in grading by $d(L(n,1),[0])$. Here, $[0]$ is the self-conjugate $\Spin^c$ structure of $S^3_n(K)$ corresponding to $[0]\in \Z/n$ under the standard correspondence $\Spin^c(S^3_n(K))\iso \Z/n$.
 \item  The $\phi$-complex $(\CF^-(S^3_{2n}(K),[n]), \phi)$ is locally equivalent to the complex
  \[
 \begin{tikzcd}[column sep={1cm,between origins},labels=description] 
A_{-n}(K)
	\ar[dr, "h"]
& & A_n(K) \ar[dl,"v"]\\
& B_n(K)
\end{tikzcd}
\]
 shifted upward in grading by $d(L(2n,1),[n])$, with the involution obtained by truncating the involution $\bX(\phi_K)$ on the surgery complex.
 \end{enumerate}
\end{usethmcounterof}

Our strategy for proving Theorem~\ref{thm:local-class} is similar to  \cite{HHSZExact}*{Section~3}, though we use slightly different notation. We use the $\sigma$-basic notation from Section~\ref{sec:sigma-basic-models}, so that we identify 
\[
\bA(K)=\cCFK(K)\quad \text{and} \quad \bB(K)=Z^{-1} \cCFK(K).
\]
The map $v$ is the canonical inclusion, while $h_{n}$ is $Z^{n} \cF \circ i_W$, where $i_W$ is the inclusion from $\cCFK(K)$ to $W^{-1}\cCFK(K)$. Following standard notation, we write $\tilde{v}$ for the map $i_W$. 

We also write $B_s(K)\subset Z^{-1}\cCFK(K)$ for the subcomplex in Alexander grading $s$, while we write $\tilde{B}_s(K)\subset W^{-1} \cCFK(K)$ for the subcomplex in Alexander grading $s$.

 Let $K$ be a strongly invertible knot, so that the induced map $\phi_K$ on $\cCFK(K)$ is skew-graded.  
 We pick a flip map $\cF\colon \tilde{B}_s\to B_{s}$ and set
 \[
 h_{n}=Z^{n} \cF \tilde{v}.
 \]
 Note that
 \begin{equation}
 \cF \circ W^i= Z^{-i} \circ\cF \quad \text{and} \quad \cF\circ Z^i=U^i Z^i \circ \cF.
 \end{equation}
 
 With respect to these choices, the map $\bX(\phi_K)$ takes the following form:
 \[
 \begin{tikzcd}[column sep=1cm, row sep=1cm] \bA(K) \ar[r, "\phi_{\bA}"]\ar[d, "v+h_n"] \ar[dr, "H_\phi\tilde v"] & \bA(K) \ar[d, "v+h_n"]
 \\
 \bB(K) \ar[r, "\phi_{\bB}"] & \bB(K)
 \end{tikzcd}
 \]
 The map $\phi_{\mathbb A}$ sends $A_s$ to $A_{-s}$ via
\[ \phi_{\mathbb A} = \phi_K \]
%
and the map $\phi_{\mathbb B}$ sends $B_s$ to $B_{n-s}$ via
\[\phi_{\mathbb B} = Z^{n} \cF\phi_K. \]
We observe firstly that
\[
\phi_{\bB}v_s+h_{n,-s} \phi_{\bA}=0.
\]

The map  $H_\phi$ maps $\tilde{\bB}(K)$ to $\bB(K)$, and decomposes over Alexander gradings as
 \[
 H_\phi = \sum_{s\in \Z} H_{\phi,s}.
 \] 
The map $H_{\phi,s}$ sends $\widetilde{B}_s $ to $ B_{-s}$ and has the property that $H_{\phi,s}\vopp_s \co A_s \rightarrow B_{-s}$ is a null-homotopy of
\[
\phi_{\mathbb B}  h_{n,s} + v_{-s} \phi_{\mathbb A}= \cF \phi_K \cF \tilde{v}+ v\phi_K.
\]
Note that such a null-homotopy exists because the above map can be factored as
\[
(\cF \phi_K \cF+  \phi_K) \tilde{v}
\]
and the map $\cF \phi_K \cF+  \phi_K$ is zero on homology, since $H_*(\tilde{B}_s)\iso H_*(B_s)=\bF[U]$.

Note that for computational purposes, we can choose each $H_{\phi,s}$ separately (for each $s$), to be any suitably graded null-homotopy of $(\cF \phi_K \cF+  \phi_K) \tilde{v}$. Compare \cite{HHSZExact}*{}.

\begin{proof}[Proof of Theorem~\ref{thm:local-class}] Our argument is essentially the same as the proof of \cite[Proposition 3.21(a)]{HHSZExact}, wherein a local equivalence is constructed between the involutive mapping cone formula $(\bX_{+1}(K),\iota_{\bX})$ and the $\iota$-complex $(A_0(K),\iota_K)$. For the argument, it is helpful to 
set
 \[
 \cF=Z^{2s} \phi_K
 \]
 so that
 \[
 h_{n}=Z^{2s+n} \phi_K \tilde{v}.
\]

The map  in \cite{HHSZExact} from $\bX_{+1}(K)$ to $A_0(K)$ is projection, which commutes with the involution $\bX(\phi_K)$ on the nose. A map $F$, in the other direction, is described in \cite[Proposition 3.21(a)]{HHSZExact}, wherein it is proven that $F$ commutes with the conjugation involution up to chain homotopy.

In our present setting, the projection still gives a local map from $(\bX_{+1}(K), \bX(\phi_K))$ to $(A_0(K),\phi_K)$. We can construct an analogous map $F$ by replacing $\iota_K$ with $\phi_K$. In our present notation, when $n=1$, the map $F$ takes the following form:
\begin{equation} \label{fig:local-map-F-schematic-1st}
\begin{tikzcd}[row sep=2.5cm]
\bX_{+1}(K) \langle k\rangle \ar[d, "F"] \\ \bX_{+1}(K) \langle k+1\rangle
\end{tikzcd}
=
 \begin{tikzcd}[column sep={1cm,between origins},labels=description] 
 &&A_{-k}
 	\ar[dr,dashed, "h"]
 	\ar[dd,"\id"]
 	\ar[dddl, bend right=5, "Jv"]
 	\ar[ddll,bend right=10, "W^{2k+1}\phi_K"]
 &
 &\cdots & & 
 A_{k}
 	\ar[dl,dashed,"v" ]
 	\ar[dd,"\id"]
 	\ar[ddrr,bend left=10, "Z^{2k+1}\phi_K"]
&\, &
 \\ 
 &&&
B_{-k+1}
 	\ar[dd,"\id"]
&\cdots 
& B_{k}
	\ar[dd,"\id"]
&\, 
\\[1cm]
A_{-k-1}
	\ar[dr,"h",dashed]&
&A_{-k}
  	\ar[dl, "v",dashed]
  	\ar[dr,dashed, "h"]
 &
 &\cdots & & 
A_{k}
	\ar[dl,dashed,"v"]
	\ar[dr,dashed, "h"]&&
    A_{k+1}\ar[dl,dashed, "v"]
 \\ 
& B_{-k}&& B_{-k+1} &\cdots & B_{k}&&B_{k+1} 
 \end{tikzcd}
\end{equation}
The map $J\colon B_{-k}\to B_{-k}$ is chosen so that $Jv$ is a null-homotopy of
\[
hW^{2k+1}  \phi_K+v=(\phi_K^2+\id)v
\]
When $n>1$, the map $F$ is defined by a similar diagram. We leave the details to the reader.

The proof that $F\circ \bX(\phi_K)+\bX(\phi_K)\circ F$ is null-homotopic follows from the same argument as in \cite[Proposition 3.21(a)]{HHSZExact}. The argument therein is obtained by explicitly writing out $F\circ \bX(\phi_K)+\bX(\phi_K)\circ F$, and factoring components through either $v$ or $\tilde{v}$. One uses the fact that $B_s\simeq \tilde{B}_s\simeq \bF[U]$ for all $s$ to show that certain maps factoring through $v$ or $\tilde{v}$ are null-homotopic. We leave the details to the reader since they are identical to \cite[Proposition 3.21(a)]{HHSZExact}.
 \end{proof}
 
 \begin{rem} Our map $F$ in Equation~\eqref{fig:local-map-F-schematic-1st} looks slightly different than the map in \cite{HHSZExact}*{Proposition~3.21}. For example the maps on the left and right sides are $W^{2k+1} \phi_K$ and $Z^{2k+1}\phi_K$, instead of $U^{k+1}\iota_K$ and $U^k \iota_K$, respectively. The difference in appearance between these two terms is due to our different conventions on the mapping cone formula. Recall that we define $A_s(K)$ to be the Alexander grading $s$ summand inside of $\cCFK(K)$, whereas in \cite{HHSZExact} one defines $A_s(K)$ to be the span of generators $[\xs,i,j]$ where $A(\xs)=j-i$, $i\le 0$ and $j\le s$. Thinking of $[\xs,i,j]$ as $\xs W^{-i} Z^{-j}$, we note that multiplication by $Z^s$ gives an isomorphism from our $A_s(K)$ to the version of $A_s(K)$ considered in \cite{HHSZExact}. This isomorphism intertwines the two definitions of the map $F$, defined above.
 \end{rem}

We now turn our attention to Theorem \ref{thm:local-class-even-si}, which we now restate:

\begin{usethmcounterof}{thm:local-class-even-si}
Suppose $K$ is a strongly invertible knot in $S^3$, and $n>0$ is an integer. The $\phi$-complex $(\CF^-(S^3_{2n}(K),[n]),\phi)$ is locally equivalent to the complex
  \[
 \begin{tikzcd}[column sep={1cm,between origins},labels=description] 
A_{n}(K)
	\ar[dr, "v"]
& & A_n(K) \ar[dl,"v"]\\
& B_n(K)
\end{tikzcd}
\]
 shifted upward in grading by $d(L(2n,1),[n])$, where the involution which swaps the two copies of $A_n$, and fixes $B_n$. 
\end{usethmcounterof}

\begin{proof}[Proof of Theorem~\ref{thm:local-class-even-si}] As in the proof of Theorem~\ref{thm:local-class}, we choose the flip map to be
\[
\cF=Z^{2s}\phi_K.
\]
We then have $h_{2n,s} = Z^{2s+2n} \phi_K \widetilde{v}_s$ so that the map $\phi_{\mathbb A}$ sends $A_s \rightarrow A_{-s}$ via $\phi_K$ and the map $\phi_{\mathbb B}$ sends $B_s \rightarrow B_{2n-s}$ via $Z^{2s+2n}\phi_K^2$.

As above, we let $H_{\phi,s}: \widetilde{B}_s \rightarrow B_{-s}$ be a choice of map such that $H_{\phi,s}\vopp_s : A_s \rightarrow B_{-s}$ is a null-homotopy between $\phi_{\mathbb B} h_{2n,s} + v_{-s} \phi_{\mathbb A}$. We set $H_\phi=\sum_{s\in \Z} H_{\phi,s}$ and
\[
\bX(\phi_K)=\phi_{\bA}+\phi_{\bB}+H_\phi.
\]

We now choose a null-homotopy $M$ of $\id+\phi_K^2$ on $\cCFK(K)$, and refer to its restrictions to various subcomplexes also as $M$.

\begin{rem} For the reader comparing this proof to the arguments of \cite[Section 3.7]{HHSZExact}, we note that the availability of the homotopy $M$ on $\cCFK(K)$ substantially simplifies what follows; in particular, in the case of the involutive surgery formula, it was necessary to construct homotopies between $\id + \iota_K^2$ and the identity on each of $A_s$ and $B_s$ and relate them appropriately. \end{rem}

The truncation of the second part of Theorem \ref{thm:local-class} goes through for this choice sans incident, so that the surgery complex is locally equivalent to the truncation
  \[
 \begin{tikzcd}[column sep={1cm,between origins},labels=description] 
A_{-n}
	\ar[dr, "h"]
& & A_n \ar[dl,"v"]\\
& B_n
\end{tikzcd}
\]
with involution coming from truncating the surgery complex. Now, following the structure of \cite[Proposition 3.21(2)]{HHSZExact}, we construct two chain maps $F$ and $G$ as shown in \eqref{eq:fg}. Here, the solid arrows are $F$ and $G$. Recall that we use $M$ for the restriction of the global homotopy, so that $[\partial, M] = \mathrm{id} + \phi_K^2$ as maps from $B_n \rightarrow B_n$.

\begin{equation} \label{eq:fg}
F=
\begin{tikzcd} [column sep={1.5cm,between origins},labels=description] 
A_{-n}
	\ar[dr, "h",dashed]
	\ar[dd," \phi_K"]
& & A_n \ar[dl,"v",dashed]
	\ar[dd, "\id"]	
\\
& B_n
	\ar[dd, "\id"]
\\
A_{n}
\ar[dr, "v",dashed]
& & A_n \ar[dl,"v",dashed]\\
& B_n
\end{tikzcd} \qquad G=
\begin{tikzcd} [column sep={1.5cm,between origins},labels=description] 
A_{n}
	\ar[dr, "v",dashed]
	\ar[dd,"\phi_K"]
	\ar[dddr, "Mv"]
& & A_n \ar[dl,"v",dashed]
	\ar[dd, "\id"]	
\\
& B_n
	\ar[dd, "\id"]
\\
A_{-n}
\ar[dr, "h",dashed]
& & A_n \ar[dl,"v",dashed]\\
& B_n
\end{tikzcd}
\end{equation}
We now compute $[F,\bX(\phi_K)]$ and $[G,\bX(\phi_K)]$, obtaining \eqref{eq:fgcommutator}. In these diagrams, the two copies of $A_n$ in the bottom complex are drawn in reversed order: the copy which was previously on the right is now on the left, and vice versa, to simplify the picture.
\begin{equation} \label{eq:fgcommutator} 
[F,\bX(\phi_K)]=
\begin{tikzcd} [column sep={1.5cm,between origins},labels=description] 
A_{-n}
	\ar[dr, "h",dashed]
	\ar[dddr,"H_{\phi,-n}\vopp"]
& & A_n \ar[dl,"v",dashed]
	\ar[dd, "\id+\phi_K^2"]	
\\
& B_n
	\ar[dd, "\id+\phi_K^2"]
\\
A_{n}
	\ar[dr, "v",dashed]
& & A_n \ar[dl,"v",dashed]\\
& B_n
\end{tikzcd} \qquad [G,\bX(\phi_K)]=
\begin{tikzcd} [column sep={1.5cm,between origins},labels=description] 
A_{n}
	\ar[dr, "v",dashed]
	\ar[dd,"\id+\phi_K^2"]
	\ar[dddr,"Lv"]
& & A_n \ar[dl,"v",dashed]
	\ar[dddl,"Mv"]
\\
& B_n
\ar[dd, "\id+\phi_K^2"]
\\
A_{n}
\ar[dr, "v",dashed]
& & A_{-n} \ar[dl,"h",dashed]\\
& B_n
\end{tikzcd}
\end{equation}
Here, using the fact that $\vopp \phi_K = \phi_K v$, the map $L$ is given by
\[ L = H_{\phi,-n}\phi_K + \phi^2_K M.\]
We now construct null-homotopies of these two commutators. These will be of the following form.

\begin{equation} \label{eq:fghomotopies}
H_F=
\begin{tikzcd} [column sep={1.5cm,between origins},labels=description] 
A_{-n}
	\ar[dr, "h",dashed]
	\ar[dddr,"\alpha"]
& & A_n \ar[dl,"v",dashed]
	\ar[dd, "M"]	
\\
& B_n
	\ar[dd, "M"]
\\
A_{n}
	\ar[dr, "v",dashed]
& & A_n \ar[dl,"v",dashed]\\
& B_n
\end{tikzcd}
\qquad H_G=
\begin{tikzcd} [column sep={1.5cm,between origins},labels=description] 
A_{n}
	\ar[dr, "v",dashed]
	\ar[dd,"M"]
	\ar[dddr,"\gamma"]
& & A_n 
	\ar[dl,"v",dashed]
\\
& B_n
	\ar[dd,"M"]
\\
A_{n}
	\ar[dr, "v",dashed]
& & A_{-n} \ar[dl,"h",dashed]\\
& B_n
\end{tikzcd}
\end{equation}
We now check that we can produce a suitable $\alpha$ and $\gamma$ such that $H_F$ and $H_G$ are nullhomotopies of $[F,\bX(\phi_K)]$ and $[G,\bX(\phi_K)]$ respectively. We see that the relation $[F,\phi] = [H_F, \partial_{\bX}]$ is equivalent to the facts that $\mathrm{id}+ \phi_K^2 = [\partial, M]$ on all of $A_n$, $A_{-n}$, and $B_n$  as appropriate, we have $Mv=vM$, and furthermore that 
\[ M h + H_{\phi,-n} \widetilde{v} = [\partial, \alpha] \]
Since $h = Z^{2s+2n} \phi_K \vopp$, the left-hand side is a +1 graded chain map into $B_s$ which factors through $\vopp$, hence nullhomotopic. So an appropriate choice of $\alpha$ exists. Likewise the relation $[G,\phi] = [H_G, \partial_{\bX}]$ is equivalent to the fact that $\mathrm{id}+ \phi_K^2 = [\partial, M]$ throughout, that $Mv=vM$, and furthermore that
\begin{align*}
(L+M)v = [\partial, \gamma].
\end{align*}
Again, because the left side factors through $v$ and is a $+1$-graded chain map, a suitable choice of $\gamma$ exists. This concludes the argument. \end{proof}

\begin{rem} We remark that Theorem~\ref{thm:local-class} extends without complication when $\phi_K$ is a strongly framed diffeomorphism which reverses the orientation of $K$. Our proof of Theorem~\ref{thm:local-class-even-si} does not extend, since we used non-trivially that $\phi_K^2\simeq \id$ on $\cCFK(K)$.
\end{rem}

The situation of rational surgeries is similar. If $p,q>0$ are coprime, following the convention of \cite{OSRationalSurgeries} and \cite{NiWu} for identifying $\Spin^c$ structures of $S^3_{p/q}(K)$ with $\mathbb Z/p$, the $\Spin^c$ structures preserved by the action of $\phi$ are $\left[ \frac{q-1}{2}\right]$ if $q$ is odd and $\left[ \frac{p+q-1}{2}\right]$ if one of $p$ or $q$ is even. We then have the following.

\begin{thm}\label{thm:local-class-si-rational}
 Suppose $K$ is a strongly-invertible knot in $S^3$, and $p,q >0$ are coprime integers.
 \begin{enumerate}
 \item If $q$ is odd, the $\phi$-complex $\left(\CF^-\left(S^3_{p/q}(K),\left[ \frac{q-1}{2}\right]\right),\phi\right)$ is locally equivalent to $(A_0(K),\phi_K)$, shifted upward in grading by the $d$-invariant of $L(p,q)$ in the corresponding $\Spin^c$ structure.
 \item If one of $p$ or $q$ is even, the $\phi$-complex $\left(\CF^-\left(S^3_{p/q}(K),\left[\frac{p+q-1}{2}\right]\right), \phi\right)$ is locally equivalent to the complex
  \[
 \begin{tikzcd}[column sep={1cm,between origins},labels=description] 
A_{0}(K)
	\ar[dr, "v"]
& & A_0(K) \ar[dl,"v"]\\
& B_0(K)
\end{tikzcd}
\]
 shifted upward in grading by the $d$-invariant of $L(p,q)$ in the corresponding $\Spin^c$ structure, with the involution obtained by exchanging the two factors of $A_0(K)$ and acting by the identity on $B_0(K)$.
 \end{enumerate}
\end{thm}

The proof is identical mutatis mutandis to that of Theorems~\ref{thm:local-class} and \ref{thm:local-class-even-si}. Note that this result strongly resembles \cite[Proposition 22.9]{HHSZExact}, except that we use a different convention for the identification of the $\Spin^c$ structures; for more on this, see \cite[Pages 204-205]{HHSZExact}. 

From the above, we extract the following formula for the correction terms of rational surgeries on strongly invertible knots, which follows from the local class computations above in the same way as \cite[Proposition 1.7]{HHSZExact}. If $\phi$ is a symmetry on a knot, following \cite{DMS:equivariant} we set 
\begin{align*}
\underline{V}_0^{\phi}(K) &= -\frac{1}{2}\underline{d}(A_0(K), \phi_K) \\
\overline{V}_0^{\phi}(K) &= -\frac{1}{2}\overline{d}(A_0(K), \phi_K).
\end{align*}

\begin{prop}\label{thm:local-class-si-rational}
 Suppose $K$ is a strongly-invertible knot in $S^3$, and $p,q >0$ are coprime integers.
 \begin{enumerate}
 \item If $q$ is odd, then
 \begin{align*}
\dl^{\phi}\left(\CF^-\left(S^3_{p/q}(K),\left[ \frac{q-1}{2}\right]\right)\right) &= d\left(L(p,q),\left[ \frac{q-1}{2}\right]\right) - 2\underline{V}_0^{\phi}(K) \\
	\du^{\phi}\left(\CF^-\left(S^3_{p/q}(K),\left[ \frac{q-1}{2}\right]\right)\right) &= d\left(L(p,q),\left[ \frac{q-1}{2}\right]\right) - 2\overline{V}_0^{\phi}(K). 		\end{align*}
\item 	If one of $p$ or $q$ is even, then 
\begin{align*}
\dl\left(\CF^-\left(S^3_{p/q}(K),\left[\frac{p+q-1}{2}\right]\right) \right) &= d\left(S^3_{p/q}(K),\left[\frac{p+q-1}{2}\right]\right) \\
\overline{d}\left(\CF^-\left(S^3_{p/q}(K),\left[\frac{p+q-1}{2}\right]\right)\right) &= d\left(L(p,q),\left[\frac{p+q-1}{2}\right]\right).
\end{align*}	
\end{enumerate}
In particular,
\[\dl^{\phi}(S^3_{+1}(K)) = -2\underline{V}_0^{\phi}(K) \qquad \qquad \du^{\phi}(S^3_{+1}(K)) = -2\overline{V}_0^{\phi}(K).\]
\end{prop}

\subsection{Local formulas for periodic knots} \label{sec:periodic-local} 

In this section we prove Theorem \ref{thm:periodic-local-class}, which we restate below.


\begin{usethmcounterof}{thm:periodic-local-class}
Suppose $K$ is a $p$-periodic knot in $S^3$ with $p$ even, and $n>0$ is an integer. If $-n/2< s\le n/2$, then the $\phi$-complex $(\CF^-(S^3_n(K),[s]),\phi)$ is locally equivalent to $(A_s(K),\phi_K)$ shifted upward in grading by $d(L(n,1),[s])$.
\end{usethmcounterof}

Let $K$ be $p$-periodic, so that the induced map $\phi_K$ on $\cCFK(K)$ is grading preserving. We choose the flip map to be 
\[
\mathcal F_s = Z^{2s} \iota_K.
\]
The map $\phi_{\mathbb A}$ sends $A_s$ to itself via $\phi_K$. We also pick the map $\phi_{\mathbb B}$ on $B_s$ to be $\phi_K$.
We let  $L\colon \cCFK(K)\to \cCFK(K)$ be a skew equivariant and skew graded map which satisfies
\[
\iota_K \circ \phi_K+\phi_K\circ \iota_K=[\d, L].
\]
Note that $L$ maps $A_s$ to $A_{-s}$. 

We take the map $H_{s,\phi}\co \tilde{B}_s \rightarrow B_{s+n}$ to be $Z^{2s+n} L_s \tilde{v}$, so that $H_{s,\phi}$ is a null-homotopy of 
\[
h_{n,s} \phi_{\bA}+\phi_{\bB} h_{n,s}=Z^{2s+n}(\iota_K \phi_K + \phi_K\iota_K)\tilde{v},
\]
 We set $H_{\phi} = \sum_{s\in \Z} H_{s,\phi}$, so that the total involution on the surgery complex is $\bX(\phi_K) = \phi_{\mathbb A} + H_{\phi} + \phi_{\mathbb B}$.

 \begin{proof}[Proof of Theorem \ref{thm:periodic-local-class}] Once again, we use a similar truncation map to the one defined in \cite[Proposition 3.21(a)]{HHSZExact}. We will again prove that this is a $\phi$-local equivalence by using the fact that any chain map of positive grading which maps between complexes which are homotopy equivalent to $\bF[U]$ must be nullhomotopic. The map $\phi_K$ shares fewer structural properties with $\iota_K$, so we spell out the details here. Additionally, the truncation we use is slightly different than in the case of the conjugation symmetry.
 
 Let $n>0$ and suppose that $-n/2<s\le n/2$. We recall that $\CF^-(S^3_n(K),[s])$ is homotopy equivalent to the summand of $\bX_n(K)$ generated by $A_t(K)$ and $B_t(K)$ with $t\equiv s \mod n$. We write $\bX_n(K, [s])$ for this complex. If $a\le b$ are integers, we will write $\bX_n(K,[s])\langle a,b\rangle$ for the quotient complex of $\bX_n(K)$ spanned by
  \[
  A_t(K) \text{ for } a\le t\le b\quad \text{and} \quad B_j(K) \text{ for } a<t\le b.
  \]
 We write $\bX_n(K,[s])\langle a,b\rangle$ for the summand of the above generated by $A_t$ and $B_t$ which additionally satisfy $t\equiv s\mod n$. We typically assume that $a,b\equiv s\mod n$ when considering $\bX_n(K,[s])\langle a,b\rangle$.
 
 We will show that if $b\ge -n/2$, then there is an inclusion map
 \[
 F\colon \bX_n(K,[s])\langle a,b\rangle \to \bX_n(K,[s])\langle a,b+n\rangle,
 \]
 and this map is local. Similarly, we will show that if $a\le n/2$, then there is an inclusion map
 \[
G\colon \bX_n(K,[s])\langle a,b\rangle \to \bX_n(K,[s])\langle a-n,b\rangle 
 \]
 which is also a local map. When $-n/2\le s\le n/2$, we can compose these local maps to give a map from $\bX_n(K)\langle s,s\rangle=A_s(K)$ to $\bX_n\langle a,b\rangle$ for $a\ll 0$ and $b\gg 0$. Note that projection gives a map in the opposite direction, which always commutes with $\bX(\phi_K)$ on the nose.

  Construction of these local maps follows from essentially the same reasoning as in the previous theorem. For completeness, we will describe the construction of $E:=F\circ G$, which increases both ends of the truncation. This map is shown below:
  \begin{equation} \label{fig:local-map-F0-schematic}
  E= \begin{tikzcd}[column sep={1cm,between origins},labels=description] 
   &&A_{a}
   	\ar[dr,dashed, "h"] 
   	\ar[dd,"\id"]
   	\ar[dddl, bend right=5, "Jv"]
   	\ar[ddll,bend right=25, "W^{-2a+n}\iota_K"]
   &
   &\cdots & & 
   A_{b}
   	\ar[dl,dashed,"v" ]
   	\ar[dd,"\id"]
   	\ar[ddrr,bend left=25, "Z^{2b+n}\iota_K"]
  &\, &
   \\ 
   &&&
  B_{a+n}
   	\ar[dd,"\id"]
  &\cdots 
  & B_{b}
  	\ar[dd,"\id"]
  &\, 
  \\[1cm]
  A_{a-n}
  	\ar[dr,"h",dashed]&
  &A_{a}
    	\ar[dl, "v",dashed]
    	\ar[dr,dashed, "h"]
   &
   &\cdots & & 
  A_{b}
  	\ar[dl,dashed,"v"]
  	\ar[dr,dashed, "h"]&&
    A_{b+n}\ar[dl,dashed, "v"]
   \\ 
  & B_{a}&& B_{a+n} &\cdots & B_{b}&&B_{b+n} 
   \end{tikzcd}
  \end{equation}
 In the above, dashed arrows denote the internal differential and solid arrows denote the map itself. The map $J$ is a choice of nullhomotopy of $\id+\iota_K^2$ on $B_{a}$. Note that $E$ can be defined for
 \[
 b\ge -n/2\quad \text{and} \quad a\le n/2,
 \]
 since this is where the powers of $W$ and $Z$ will be nonnegative. These constraints on $a$ and $b$  coincide with the constraints described above for the maps $F$ and $G$.
 
  Now we observe that $E\bX(\phi_K) + \bX(\phi_K) E =0$ on all summands of $\bX_{n}(K)\langle a,b\rangle$ except for $A_{a}$ and $A_b$. On these summands, we can expand out the expressions:

  \begin{equation}
  [E,\bX(\phi_K)]= \begin{tikzcd}[column sep={1cm,between origins},labels=description] 
   &&A_{a}
   	\ar[dr,dashed, "h"]
   	\ar[dddl, bend right=5, "\substack{L\tilde v \iota_K\\ +[J,\phi_K] v }"]
   	\ar[ddll,bend right=10, "{W^{-2a+n}[\iota_K,\phi_K]}", labels=left]
   &
   &\cdots & & 
   A_{b}
   	\ar[dl,dashed,"v" ]
   	\ar[ddrr,bend left=10,labels=right, "{Z^{2b+n}[\iota_K,\phi_K]}"]
   	\ar[dddr, bend left=5, " Z^{2b+n} L \tilde{v}"]
  &\, &
   \\ 
   &&&
  B_{a+n}
  &\cdots 
  & B_{b}
  &\, 
  \\[1cm]
  A_{a-n}
  	\ar[dr,"h",dashed]&
  &A_{a}
    	\ar[dl, "v",dashed]
    	\ar[dr,dashed, "h"]
   &
   &\cdots & & 
  A_{b}
  	\ar[dl,dashed,"v"]
  	\ar[dr,dashed, "h"]&&
      A_{b+n}\ar[dl,dashed, "v"]
   \\ 
  & B_{a}&& B_{a+n} &\cdots & B_{b}&&B_{b+n} 
   \end{tikzcd}
  \end{equation}
 
 A null-homotopy of $[E,\bX(\phi_K)]$ is given by the following diagram: 
  \begin{equation}
   \begin{tikzcd}[column sep={1cm,between origins},labels=description] 
   &&A_{a}
   	\ar[dr,dashed, "h"]
   	\ar[dddl, bend right=5, "\a v"]
   	\ar[ddll,bend right=10, "{W^{-2a+n}L}", labels=left]
   &
   &\cdots & & 
   A_{b}
   	\ar[dl,dashed,"v" ]
   	\ar[ddrr,bend left=10,labels=right, "{Z^{2b+n}L}"]
  &\, &
   \\ 
   &&&
  B_{a+n}
  &\cdots 
  & B_{b}
  &\, 
  \\[1cm]
  A_{a-n}
  	\ar[dr,"h",dashed]&
  &A_{a}
    	\ar[dl, "v",dashed]
    	\ar[dr,dashed, "h"]
   &
   &\cdots & & 
  A_{b}
  	\ar[dl,dashed,"v"]
  	\ar[dr,dashed, "h"]&&
      A_{b+n}\ar[dl,dashed, "v"]
   \\ 
  & B_{a}&& B_{a+n} &\cdots & B_{b}&&B_{b+n} 
   \end{tikzcd}
  \end{equation}
  In the above, $\a$ is a $+2$ graded map from $B_{a}$ to itself, which is a null-homotopy of
  \[
   [\iota_K, L] +[J,\phi_K] 
  \]
  The above map is null-homotopic because it is a chain map which has $+1$ Maslov grading. This concludes the proof.
  \end{proof}

As previously, similar formulas hold for rational surgeries, as follows. Let $K$ be $p$-periodic, so that we may consider $n/m$ surgery for $n$ and $m$ positive and coprime as usual. 

\begin{thm}\label{thm:periodic-local-class-rational}
Suppose $K$ is an $p$-periodic knot in $S^3$ with $p$ even, and $n,m>0$ are coprime integers. If $0 \leq s\le n-1$, then the $\phi$-complex $(\CF^-(S^3_{n/m}(K),[s]),\phi)$ is locally equivalent to $(A_{\lfloor \frac{s}{m} \rfloor}(K),\phi_K)$, shifted upward in grading by $d(L(p,q), [s])$.
\end{thm}

We extract the following computation of the correction terms. Again following \cite{DMS:equivariant}, for a periodic symmetry $\phi$ let
\begin{align*} \underline{V}^{\phi}_s(K) &= -\frac{1}{2}\dl(A_s, \phi_K) \\
			\overline{V}^{\phi}_s(K) &= -\frac{1}{2}\dl(A_s, \phi_K) \end{align*}
\begin{prop}Suppose $K$ is a $p$-periodic knot in $S^3$ with $p$ even, and $n,m>0$ are coprime integers. If $0 \leq s\le n-1$, then 
\begin{align*}
\dl^{\phi}(\CF^-(S^3_{n/m}(K),[s]) &= d(L(n,m),[s]) - 2\underline{V}^{\phi}_{\lfloor \frac{s}{m} \rfloor}(K) \\
\du^{\phi}(\CF^-(S^3_{n/m}(K),[s]) &= d(L(n,m),[s]) - 2\overline{V}^{\phi}_{\lfloor \frac{s}{m} \rfloor}(K). 
\end{align*}
\end{prop}

\subsection{Example: $\sfrac{1}{2}$-surgery on the figure-eight knot} \label{subsec:example}

We now consider the example of $S^3_{\sfrac{1}{2}}(4_1)$. The figure eight knot admits two strong inversions relating by mirroring, whose actions on $\cCFK(4_1)$ we will denote $\sigma_K$ and $\sigma'_K$. Moreover it has a two-periodic symmetry, whose action on $\cCFK(4_1)$ we denote $\phi_K$. Each of these induces a symmetry on $\CF^-(S^3_{\sfrac{1}{2}}(4_1))$, denoted $\sigma$, $\sigma'$, and $\phi$ respectively. More specifically, following the computations of \cite[Example 2.26]{DMS:equivariant} and \cite[Theorem 1.7 and Figure 3]{Mallick:surgery}, we see that if $\cCFK(4_1)$ is generated over $\bF[W, Z]$ by $x_0, a, b, c, e$ as in Figure~\ref{fig:4_1},  we have that
\begin{align*}
\iota_K x_0 &= x_0 + e   & \qquad \sigma_K x_0 &= x_0 +e  & \qquad \sigma'_K x_0 &= x_0   & \qquad \phi_K x_0 &= x_0+e \\
\iota_K a &= a+ x_0    & \qquad \sigma_K a &= a    & \qquad \sigma'_K a &= a+e   &  \qquad \phi_K a &= a+x_0 \\
\iota_K b &= c    & \qquad \sigma_K b &= c    & \qquad \sigma'_K b &= c   &  \qquad \phi_K b &= b \\
\iota_K c &= b    & \qquad \sigma_K c &= b    & \qquad \sigma'_K c &= b   &  \qquad \phi_K c &= c \\
\iota_K e &= e    & \qquad \sigma_K e &= e    & \qquad \sigma'_K e &= e   &  \qquad \phi_K e &= e \\
\end{align*}

\begin{figure}[h]
\includegraphics{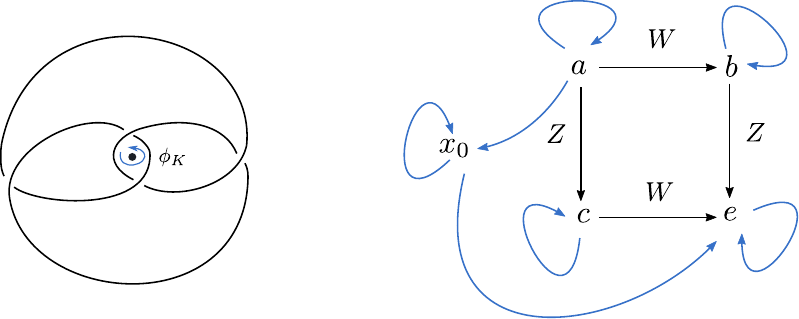}
\caption{Left: The periodic involution $\phi_K$ on the knot $4_1$ with the axis of symmetry coming out the page. Right: The knot Floer complex of $4_1$ with the induced action of $\phi_K$.} 
\label{fig:4_1}
\end{figure}

\noindent Proposition 3.21(2) of \cite{HHSZExact} shows that $(\CF^-(S^3_{\sfrac{1}{2}}(4_1)), \iota_K)$ is involutively locally equivalent to
  \begin{equation} \label{eqn:vee}
 \begin{tikzcd}[column sep={1cm,between origins},labels=description] 
A_{0}
	\ar[dr, "v"]
& & A_0 \ar[dl,"v"]\\
& B_0
\end{tikzcd}
\end{equation}
where $B_0$ may be taken to be a copy of $\mathbb F[U]$; the involution exchanges the two copies of $A_0$ via the identity and acts by the identity on $B_0$. Now, recall that $V_0(4_1)=0$, implying in particular that after replacing $B_0$ with $\mathbb F[U]$ we have that $v$ is surjective. Indeed, there exists $a \in A_0$ such that $\partial a =0$ and $U^n[a] \neq 0$ for any $n$ with $v(a) = 1 \in \mathbb F[U]$. Let $a_0$ be the copy of $a$ in the left-hand copy of $A_0$ and $a_1$ be the copy in the right-hand $A_0$ Of necessity $[a] \notin \im(U)$. Then there is a local equivalence between $(\mathbb F[U], \mathrm{Id})$ and \eqref{eqn:vee} via sending $f(1) = a_0 + a_1$ in one direction and projecting onto the summand $\{U^n(a_0+a_1)\}$ in the other direction. By Theorem~\ref{thm:local-class-si-rational} the same is true for $\sigma_K$ and $\sigma_K'$. Theorem~\ref{thm:periodic-local-class-rational}, however, $(\CF^-(S^3_{\sfrac{1}{2}}(4_1)), \phi_K)$ is locally equivalent to $(A_0, \phi_K)$, which is not locally trivial.

\section{An application to the equivariant homology cobordism group}
\label{sec:application-section}

In this section, we prove the following:

\begin{thm}\label{thm:Z-infinity-summand}
The kernel of the forgetful maps from $\Theta^{\mathrm{inv}}_{\Z}$ and $\Theta^{\mathrm{diff}}_{\Z}$ to $\Theta^3_\Z$ both contain $\Z^\infty$-summands.
\end{thm}
This is stated as Theorem~\ref{thm:application} in the introduction.

The summand we describe is generated by certain symmetries on $S_{+1}^3(K_n)$ for odd $n>1$, where
\[
K_n:=2 T_{2n,2n+1}\# -2T_{2n,2n+1}.
\]
These symmetries were studied in \cite{DMS:equivariant}, though the local class could not be completely computed due to the lack of a surgery formula for $+1$ surgeries. 
 
\subsection{Background}

We now describe some background on the knot Floer homology of equivariant knots. 

We recall from Section~\ref{sec:equivariant-knots} that a diffeomorphism $\phi\colon (S^3,K)\to (S^3,K)$ is called a strong inversion if $\phi$ preserves the orientation of $S^3$, reverses the orientation of $K$, and satisfies $\phi^2=\id$ on $S^3$.

Given a strong inversion $\phi$, there is an induced map on knot Floer homology. To define this map, we assume that basepoints $w$ and $z$ are interchanged by $\phi$. The map $\phi$ induces a map on the knot Floer complex
\[
\cCFK(\phi)\colon \cCFK(K)\to \cCFK(K)
\]
which interchanges $W$ and $Z$.  Alternatively, we can think of $\phi_*$ as a type-$D$ module morphism map from
\[
\cCFK(\phi) \colon \cCFK(K)^{\bF[W,Z]}\to \cCFK(K)^{\bF[W,Z]}\boxtimes {}_{\bF[W,Z]} [E_0]^{\bF[W,Z]},
\]
where $E_0\colon \bF[W,Z]\to \bF[W,Z]$ is the algebra morphism $E_0(W^i Z^j)=W^j Z^i$. 

We will typically abuse notation and write $\phi$ for $\cCFK(\phi)$.

We recall from Section~\ref{sec:phi_K-complexes} that Dai and the second and third author proved a connected sum formula for the map induced by a strong inversion \cite{DMS:equivariant}*{Theorem~4.1}. They prove that if $(K_1,\phi_1)$ and $(K_2,\phi_2)$ are strongly invertible knots in $S^3$, then the strong inversion $\phi_\#$ on the connected sum has induced map 
\[
\phi_\#\simeq \phi_1\otimes \phi_2
\]
with respect to a homotopy equivalence $\cCFK(K_1\#K_2)\simeq \cCFK(K_1)\otimes \cCFK(K_2)$, for all of the equivariant connected sums of $K_1,K_2$.

Next, we consider the summand swapping strong inversion $\phi_{\sw}$ on $K\# rK$, where $rK$ denotes the string reversal of $K$. We abuse notation and write $\phi_{\sw}$ both for the diffeomorphism, and the induced map on the knot Floer complex. If $K$ is a knot, there is a strong inversion on $K\# r K$ which swaps the two summands, as in Figure~\ref{fig:41}. Here $rK$ denotes $K$ with its string orientation reversed. 

\begin{figure}[h]
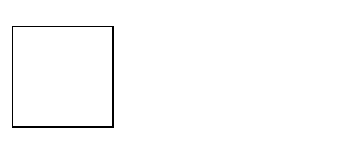
\caption{The summand swapping strong inversion $\phi_{\sw}$ on $K\# r K$.}
\label{fig:41}
\end{figure}

 Dai, and the second and third author \cite{DMS:equivariant} compute the induced map by the swapping involution. The computation is similar in flavor to earlier work of Juh\'{a}sz and the last author, which computes the map by the periodic swapping involution on $K\# K$ \cite{JZStabilizationDistance}*{Theorem~8.2}.

The map $\phi_{\sw}$ is as follows. Firstly, write $\bar{\cCFK}(K)$ for $\cCFK(K)^{\bF[W,Z]}\boxtimes {}_{\bF[W,Z]} [E_0]^{\bF[W,Z]}$; i.e., switch $W$ and $Z$ in the complex $\cCFK(K)$. We write
\[
\Sw\colon \cCFK(K)\otimes \bar{\cCFK}(K)\to \bar{\cCFK}(K)\otimes \cCFK(K)
\]
for the map which switches the two factors. Then we define
\[
\Sq\colon \bar{\cCFK}(K)\otimes \cCFK(K)\to \cCFK(K)\otimes \bar{\cCFK}(K)
\]
for the $\bF[W,Z]$ skew-equivariant map which switches $W$ and $Z$ everywhere.

\begin{thm}[\cite{DMS:equivariant}*{Theorem~4.3}] \label{thm:swapp-computation}
There is a homotopy equivalence of $\cCFK(K\# r K)$ with $\cCFK(K)\otimes \bar{\cCFK}(K)$ which intertwines $\phi_{\sw}$ with 
\[
(\id\otimes \id+\Phi\otimes \Psi)\circ \Sq\circ \Sw.
\]
\end{thm}

\subsection{Local class computations}

In this section, we perform several local class computations. We begin by introducing some notation. We will write
\[
\cC_n=\cCFK(T_{2n,2n+1})
\]
and we write $\bar{\cC}_n$ for $\bar{\cCFK}(T_{2n,2n+1})$. In particular, this means that $\cC_n$ denotes the $\bF[\scU,\scV]$ complex in Figure \ref{fig:Cn} generated by elements $x_{k}$ such that $-2n+2 \leq k \leq 2n-2$ with $k$ even and $y_{\ell}$ such that $-2n+1 \leq \ell \leq 2n-1$ with $\ell$ odd, with nonzero differentials given by
\begin{align*}
\partial(x_{k}) = \scV^{c_{2n-1+k}}y_{k-1} + \scU^{c_{2n+k}}y_{k+1}
\end{align*}
\noindent determined by the symmetric sequence of positive integers $(c_1,c_2,\dots, c_{4n-3},c_{4n-2}) = (1, 2n-1, 2, 2n-2, \dots, 2n-2,2,2n-1,1)$.

\begin{figure}
\begin{tikzcd}[column sep={.8cm,between origins}, labels=description, row sep=1cm]
y_{1-2n}&& y_{3-2n}&& \cdots&& y_{-1}&& y_{1}&&\cdots&& y_{2n-3}&&y_{2n-1}\\
&x_{2-2n}\ar[ul, "\scV"]\ar[ur, "\scU^{2n-1}"]&&x_{4-2n}\ar[ul, "\scV^2"] \ar[ur, "\scU^{2n-2}"]&& \cdots\ar[ur] \ar[ul]  && x_{0} \ar[ul, "\scV^{n}"] \ar[ur, "\scU^{n}"]&& \cdots \ar[ur] \ar[ul]  && x_{2n-4} \ar[ul, "\scV^{2n-2}"] \ar[ur, "\scU^{2}"]&& x_{2n-2} \ar[ul, "\scV^{2n-1}"] \ar[ur, "\scU"]
\end{tikzcd}
\caption{The complex $\cC_n$.}\label{fig:Cn}
\end{figure}

We will study the swapping inversion on $T_{2n,2n+1}\# r T_{2n,2n+1}$. Since torus knots are reversible, we can view this as a strong inversion on $T_{2n,2n+1}\# T_{2n,2n+1}$.

We now describe a small model for the local class of
\[
(\cC_n\otimes \bar{\cC}_n, \phi_{\sw}).
\]
  We recall from the previous section that we can take $\phi_{\sw}$ to be $(\id\otimes \id+\Phi\otimes \Psi)\circ \Sq\circ \Sw$.
  
  We find it convenient to identify $\cC_n$ and $\bar{\cC}_n$.   There is an identification of $\cC_n$ with $\bar{\cC}_n$. Due to grading constraints, this identification is canonical, and identifies $x_i$ with $x_{-i}$ and $y_i$ with $y_{-i}$. Therefore, we can canonically identify
  \[
 \cC_n\otimes \bar{\cC}_n\iso \cC_n\otimes \cC_n.
  \]
Under this identification,   
\[
(\Sq\circ \Sw)(x_i|x_j)=x_{-j}|x_{-i}
\]
 and similarly for tensors involving $y_i$. 
  
    In Figure~\ref{fig:Yn}, we describe a subcomplex $\cY_n$ of $\cC_n\otimes \cC_n$. The map $\phi_{\sw}$ is reflection on this subcomplex, with the exception of $x_0x_0$, which is mapped to $x_0x_0+U^{n-1}y_{-1}y_1$. We now show $\cY_n$ admits a complement as a $\phi_K$-complex. We encourage the reader to compare the following to \cite{DMS:equivariant}*{Lemma~6.11}, which gives partial information about the local class.

  \begin{figure}[h]
  \[
  \begin{tikzcd}[column sep={1cm,between origins}, labels=description, row sep=1cm]
  y_{1-2n}y_{1-2n}
  	&& y_{1-2n}y_{3-2n}
  	&& y_{3-2n}y_{3-2n}
  	&& y_{3-2n}y_{5-2n}
  	&& \cdots
  	&& y_{-3}y_{-1}
  	&& y_{-1} y_{-1}
  	&&\,\\
  & y_{1-2n} x_{2-2n} 
  	\ar[ul, "\scV"]
  	\ar[ur, "\scU^{2n-1}"]
  && x_{2-2n}y_{3-2n}
  	\ar[ul, "\scV"]
  	\ar[ur, "\scU^{2n-1}"]
  	&&
  	y_{3-2n}x_{4-2n}
  		\ar[ul, "\scV^2"]
  		\ar[ur, "\scU^{2n-2}"]
  	&& \cdots
  		\ar[ur]
  		\ar[ul]
  	&& y_{-3}x_{-2}
  		 	\ar[ul]
  		 	\ar[ur, "\scU^{n+1}"]
  	&& x_{-2} y_{-1}
  		\ar[ul, "\scV^{n-1}"]
  		\ar[ur, "\scU^{n+1}"]
  	&& 
  y_{-1}	x_{0}
  		\ar[ul, "\scV^n"]
  		\ar[ur, "\scU^n"]
  \\\cdots
  y_{-1}y_{1}
  && 
  y_{1}y_{1}
  && 
  y_{1}y_{3}
  && 
  y_{3}y_{3}
  && 
  \cdots
  && 
  y_{2n-3}y_{2n-1}
  && 
  y_{2n-1}y_{2n-1}
  &&\,\\
  & 
  x_{0} y_{1}
  	\ar[ul, "\scV^n"]
  	\ar[ur, "\scU^n"]
  && 
  y_{1}x_{2}
  	\ar[ul, "\scV^{n+1}"]
  	\ar[ur, "\scU^{n-1}"]
  &&
  x_{2}y_{3}  
  	\ar[ul, "\scV^{n+1}"]\ar[ur, "\scU^{n-1}"]
  && \cdots \ar[ul] \ar[ur]
   && y_{2n-3}x_{2n-2}
   		\ar[ul]\ar[ur, "\scU"]
    && x_{2n-2} y_{2n-1}
    	\ar[ul, "\scV^{2n-1}"] \ar[ur, "\scU"]
    	\\[-.5cm]
  &&&&y_{1}y_{-1}+y_{-1}y_1&\,\\
  &&y_{-1}x_{0}+x_{0}y_{-1} \ar[urr, "\scU^n"]&&&& y_1x_{0}+x_{0}y_1 \ar[ull, "\scV^n"]\\
  &&&& x_{0}x_{0}\ar[ull, "\scV^n"]\ar[urr, "\scU^n"] \ar[uuuullll, bend left=50,red, "U^{n-1}"]
  \end{tikzcd}
  \]
  \caption{The subcomplex $\cY_n\subset \cX_n$.  The map $\phi_{\sw}$ is reflection plus the red arrow. The staircase complex in the top two rows is split in half in the figure.} \label{fig:Yn}
  \end{figure}
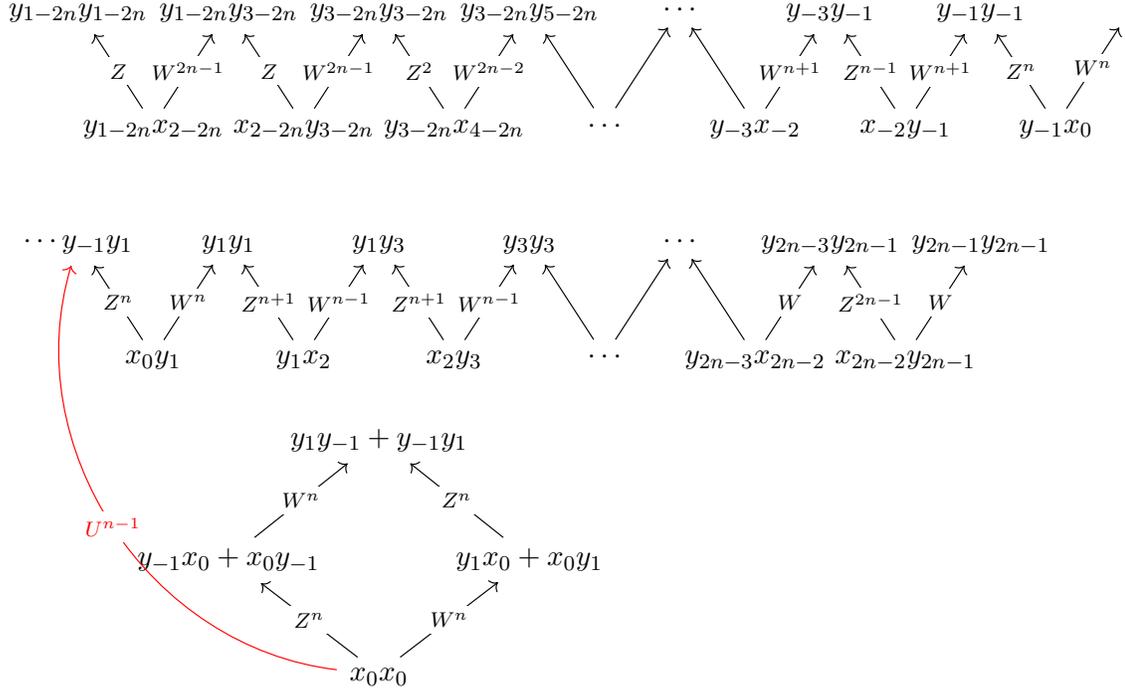

  \begin{prop}
  If $n$ is odd, then there is a local equivalence of $\phi_K$-complexes
  \[
  ( \cCFK(T_{2n,2n+1}\# T_{2n,2n+1}), \phi_{\sw})\sim_{\mathrm{loc}} \cY_n
  \]
  where $\cY_n$ is the subcomplex shown in Figure~\ref{fig:Yn} with the involution coming from restriction of $\phi_{\sw}$.
  \end{prop}
  \begin{proof} Our proof is organized as follows. We will first describe a decomposition of $\bF[W,Z]$-chain complexes
  \[
 \cC_n\otimes \cC_n\iso \cY_n\oplus \cG_n,
  \]
  for some free complex $\cG_n$ (described below). This splitting is not $\phi_{\sw}$-equivariant. Instead, if we let 
  \[
  I\colon \cY_n\to \cC_n\otimes \cC_n\quad \text{and} \quad \Pi\colon \cC_n\otimes \cC_n\to \cY_n
  \]
  denote the projection and inclusion maps arising from this splitting, then we will have
  \begin{equation}
  I\circ \phi_{\sw}=\phi_{\sw}\circ I\quad \text{and} \quad \Pi\circ \phi_{\sw}\simeq \phi_{\sw} \circ \Pi,
  \label{eq:splitting-local-equivalence}
  \end{equation}
  establishing the local equivalence in the statement.
  
    We begin by describing the summand $\cG_n$:
 \begin{enumerate}[label=($G$-\arabic*), ref=$G$-\arabic*]
    \item\label{G-1} If $i$ and $j$ are both odd and $i\neq \pm j$, then $y_iy_j+y_jy_i$ is a generator of $\cG_n$.
     \item\label{G-2} If $i$ is odd, $j$ is even, and $j\neq -i\pm 1$, then $y_ix_j+x_jy_i$ is a generator of $\cG_n$.
     
	\item \label{G-3} If $i>0$ is even, write $i=2n-2k$ for some $n> k\ge 1$. Then $\cG_n$ has a generator
      \[
      x_ix_i+k(2n-k)\scU^{k-1}\scV^{2n-k-1}y_{i+1}y_{i-1}.
      \]      
\item \label{G-4} If $i<0$ is even then $\cG_n$ has a generator $x_ix_i.$   

\item \label{G-5} If $i>0$ has $i=2m$ for $m$ odd, then $\cG_n$ has generators 
			\begin{enumerate}				
				\item \label{G-5a} $x_{-i}x_{i}$
				\item \label{G-5b} $x_i x_{-i}$
			\end{enumerate}

\item \label{G-6} If $i>0$ has $i=2m$ for $m$ even, then $\cG_n$ has generators
			\begin{enumerate}
			\item \label{G-6a} $x_{-i}x_{i} + U^{\lambda_{-i}}x_0x_0$
				\item \label{G-6b} $x_i x_{-i} + U^{\lambda_i} x_0x_0$
			\end{enumerate}
where $\lambda_{-i}$ and $\lambda_i$ denote the unique nonnegative integers for which the sums are homogeneously graded.
     \item\label{G-7} If $i$ and $j$ are even with $|i|< |j|$, then  
      \[
      x_ix_j+k_i(2n-k_j) \scU^{k_i-1} \scV^{2n-k_j-1} y_{i+1}y_{j-1}
      \] 
      is a generator of $\cG_n$, where $i=2n-2k_i$ and $j=2n-2k_j$. 
  \item \label{G-8} If $i$ and $j$ are even and $|i|>|j|$, then $x_ix_j$ is a generator of $\cG_n$.    
  
\item\label{G-9} If $i$ and $j$ are odd, $i<j-2$, and $i \neq -j$, then 
      \[
      y_iy_j+U^{\gamma_{i,j}}   y_{i+2}y_{j-2}
      \]
       is a generator of $\cG_n$, where $\gamma_{i,j}$ is the unique nonnegative integer so that the sum is homogeneously graded.

  \item\label{G-10} If $i>2$ is odd and $i=2m+1$ for $m$ odd, then the following are generators of $\cG_n$:
	\begin{enumerate}      
      \item
      \label{G-10a} $y_iy_{-i}+U^{\gamma_{i,-i}}  y_{i-2}y_{-i+2}$;
      \item \label{G-10b} $y_{-i}y_i + U^{\gamma_{-i,i}} y_{-i+2}y_{i-2}$.
      \end{enumerate}
  
    \item\label{G-11} If $i>2$ is odd and $i=2m+1$ for $m$ even, then the following are generators of $\cG_n$:
	\begin{enumerate}      
      \item
      \label{G-11a} $y_iy_{-i}+ U^{\gamma_{i,-i}}  y_{i-2}y_{-i+2} + U^{\delta_{i,-i}}(y_1y_{-1} + y_{-1}y_{1})$;
      \item \label{G-11b} $y_{-i}y_i + U^{\gamma_{-i,i}} y_{-i+2}y_{i-2} +  U^{\delta_{-i,i}}(y_1y_{-1} + y_{-1}y_{1})$
      where as usual the exponents are the nonnegative integers which produce a homogeneous term.
      \end{enumerate}

    \item\label{G-12} If $i$ and $j$ are even, $i<j$, and either $i\neq -j$ or $i \neq 2m$ for $m\neq 0$ even, the following are generators of $\cG_n$:
      \begin{enumerate}
      \item \label{G-12a}  $x_iy_{j+1}+\scU^{\nu_{i,j}}y_{i+1}x_j$;
      \item \label{G-12b} $y_{-j-1}x_{-i} + \scV^{\nu_{i,j}}x_{-j}y_{-i-1}$.
      \end{enumerate}
      Here, $\nu_{i,j}$ denotes the unique nonnegative integer so that the sums are homogeneously graded. 
      
	\item \label{G-13} For even $j>0$ and $j=2m$ for $m$ even, the following are generators of $\cG_n$:
		\begin{enumerate}
		 \item \label{G-13a}  $x_{-j}y_{j+1}+\scU^{\a_{-j,j}}y_{-j+1}x_j + \scU^{\epsilon_{-j,j}}\scV^{\eta_{-j,j}}(y_{-1}x_0 + x_0y_{-1})$;
      \item \label{G-13b} $y_{-j-1}x_{j} + \scV^{\a_{-j,j}}x_{-j}y_{j-1} + \scU^{\eta_{-j,j}}\scV^{\epsilon_{-j,j}}(y_1x_0 +x_0y_1)$.
		\end{enumerate}
	Here as usual the exponents are the unique nonnegative powers which make the term homogeneous.

       \item \label{G-14} For even $j>0$ and $j=2m$ for $m$ odd, the following are generators of $\cG_n$: 
  \begin{enumerate}
      \item \label{G-14a} $y_{j+1} x_{-j}+ \scV^{\a_{-j,j}} x_j y_{-j+1}$;
      \item \label{G-14b} $x_{j}y_{-j-1}+\scU^{\a_{-j,j}} y_{j-1}x_{-j}$.
      \end{enumerate}
  
    \item \label{G-15} For even $j>0$ and $j=2m$ for $m$ even, the following are generators of $\cG_n$: 
  \begin{enumerate}
      \item \label{G-15a} $y_{j+1} x_{-j}+ \scV^{\a_{-j,j}} x_j y_{-j+1}+ \scU^{\epsilon_{j,-j}}\scV^{\eta_{j,-j}}(y_{-1}x_0 + x_0y_{-1})$;
      \item \label{G-15b} $x_{j}y_{-j-1}+\scU^{\a_{-j,j}} y_{j-1}x_{-j} + \scU^{\eta_{j,-j}}\scV^{\epsilon_{j,-j}}(y_1x_0+x_0y_1)$
      \end{enumerate}
 
      \item\label{G-16} For even $j>0$, the following are generators of $\cG_n$:
      \begin{enumerate}
      \item\label{G-16a} $y_{j-1}x_{-j}+\scU^{\a_{-j,j-2}}x_{j-2}y_{-j+1}$;
      \item\label{G-16b}  $x_jy_{-j+1}+\scV^{\a_{-j,j-2}} y_{j-1}x_{-j+2}$.  
  		\end{enumerate}
  
     \end{enumerate}

\vskip 3mm

\noindent We start by showing that $\cX_n = \cY_n +\cG_n$. We will check that each generator of $\cX_n$ can be written as a sum of generators in $\cY_n$ and $\cG_n$. We first show that all elements $y_iy_j$, for $i$ and $j$ both odd, are in $\cY_n+\cG_n$. Since $y_iy_i$ is in $\cY_n$, we may assume that $i\neq j$. For the case that $i \neq -j$, by adding an element of type \eqref{G-1} we can assume that $i<j$. Then after adding some number of elements of type \eqref{G-9}, we reduce either to the case of $y_iy_i$ or $y_iy_{i+2}$, which are in $\cY_n$. Now for the case of $y_iy_{-i}$. If $i>0$ then adding elements of type \eqref{G-10a} and \eqref{G-11a} has the effect of reducing this situation to either $y_{1}y_{-1}$ or $y_{-1}y_1$; similarly if $i<0$ adding elements of type \eqref{G-10b} and \eqref{G-11b} reduces this to either $y_{1}y_{-1}$ or to $y_{-1}y_1$, both of which are in $\cY_n$. We now consider elements of the form $x_iy_j$ or $y_jx_i$. Of course if $|i-j|=1$, then either $x_iy_j$ is in $\cY_n$, or $x_iy_j$ plus an element of the form \eqref{G-2} is in $\cY_n$, and similarly for $y_jx_i$. Now, consider an arbitary $x_iy_j$. If $i<j$, by adding elements of type \eqref{G-12a} and \eqref{G-13a} we can replace such terms with sums of $x_ky_{m}$ and $y_{m}x_k$ such that $|k-m|<|i-j|$, along with possibly instances of $y_{-1}x_0+x_0y_{-1}$, which is in $\cY_n$.  If $i>j$, and $j\neq -i\pm 1$, we can use \eqref{G-2} to switch to $y_jx_i$ and then reduce with elements of type \eqref{G-12b} and \eqref{G-13b}.  If $j=-i+1$ and $i>j$, then \eqref{G-16b} applies to reduce $|i-j|$.  If $j=-i-1$ and $i>j$, \eqref{G-14b} and \eqref{G-15b} apply (note $j<-2$ by hypothesis at this point) to reduce $|i-j|$.  So inductively it suffices to check that we can similarly replace terms $y_jx_i$ with $x_ky_m,x_my_k$ with $|k-m|<|i-j|$. In the event that $j \neq -i \pm 1$ this can be done by adding an element of the form \eqref{G-2} to obtain $x_iy_j$ and repeating the argument above. In the event that $j=-i \pm 1$, if $i<j$ we must add terms of the form \eqref{G-12b} or \eqref{G-13b}, and if $i>j$ we must add terms of the form \eqref{G-14a}, \eqref{G-16a}, or \eqref{G-15a} as appropriate. Either way we reduce to the case that $|i-j|=1$, showing that all elements $x_iy_j$ or $y_jx_i$ are in the span. Finally, each $x_ix_j$ is a sum of generators of types \eqref{G-3}-\eqref{G-8} with the above terms. Ergo indeed $\cX_n = \cY_n + \cG_n$.

\vskip 3mm

\noindent Now we check that the rank of $\cX_n$ is exactly the sum of the ranks of $\cY_n$ and $\cG_n$. One checks easily that
\[ \mathrm{rank}(\cX_n) = 16n^2-8n+1 \qquad \qquad \mathrm{rank}(\cY_n) = 8n+1.\]
It suffices to check that there are exactly $16n^2-16n$ generators in the generating set for $\cG_n$ described above. Now, there are $2n^2-2n$ generators of type \eqref{G-1}, $4n^2-6n+2$ generators of type \eqref{G-2}, and $4n^2-4n$ total generators of types \eqref{G-3}-\eqref{G-8}. There are then $2n^2-2n$ generators of types \eqref{G-9}-\eqref{G-11}, $4n^2-6n+2$ generators of type \eqref{G-12} and \eqref{G-13}, and $4n-4$ generators of types \eqref{G-14}-\eqref{G-16}. This implies that this generating set must be a basis.

\vskip 3mm

\noindent We now verify that $\cG_n$ is a subcomplex:

\begin{itemize}

\item On \eqref{G-1} we see that $\partial$ vanishes. 

\item The map $\partial$ sends elements of type \eqref{G-2} to either one element of type \eqref{G-1} or a sum of two elements of type \eqref{G-1}. 

\item Elements \eqref{G-3} and \eqref{G-4} are mapped to sums of \eqref{G-2}. 

\item Elements of type \eqref{G-5a} are mapped to a sum of an element of type \eqref{G-12a} and an element of type \eqref{G-12b}.  Elements of types \eqref{G-5b} are mapped to a sum of an element of type \eqref{G-14a} and an element of type \eqref{G-14b}. 

\item Elements of type \eqref{G-6a} are mapped to a sum of an element of type \eqref{G-13a} and an element of type \eqref{G-13b}. Elements of type \eqref{G-6b} are mapped to a sum of an element of type \eqref{G-15a} and an element of type \eqref{G-15b}.

\item Elements of type \eqref{G-7}  and \eqref{G-8} for which $i<j$ are mapped to a sum of an element of type \eqref{G-12a} and \eqref{G-12b}. In the case of elements with $j<i$, elements for which $|i+j|>2$ are mapped to a sum of four elements of type \eqref{G-2}, an element of type \eqref{G-12a}, and an element of type \eqref{G-12b}. Now we consider the cases in which $|i+j|=2$, still with $j<i$. Elements for which $i+j=-2$ are sent to a sum of two elements of type \eqref{G-2}, an element of type \eqref{G-12b}, and an element of type \eqref{G-16a}. Elements for which $i+j=2$ are sent to a sum of two elements of type \eqref{G-2}, an element of type \eqref{G-12a}, and an element of type \eqref{G-16b}. 

\item The differential vanishes on elements of type \eqref{G-9}-\eqref{G-11}.  

\item Elements of either type \eqref{G-12} are mapped to elements of type \eqref{G-9} or of type \eqref{G-10}. 

\item Elements of both types \eqref{G-13} are mapped to elements of type \eqref{G-11}. 

\item Elements of both types \eqref{G-14} are mapped to elements of type \eqref{G-10}.

\item Elements of both types \eqref{G-15} are mapped to elements of type \eqref{G-11}.

\item Elements of both types \eqref{G-16} are mapped to a sum of two elements of type \eqref{G-1} and an element of type \eqref{G-9}. 

\end{itemize}

\vskip 3mm

\noindent We now verify Equation~\eqref{eq:splitting-local-equivalence}, i.e. that $\Pi$ and $I$ homotopy commute with $\phi_{\sw}$. Note that $I\circ \phi_{\sw}=\phi_{\sw}\circ I$, because $\cY_n$ is preserved by $\phi_{\sw}$. For the relation $\phi_{\sw}\circ \Pi\simeq \Pi\circ \phi_{\sw}$,  we will show that $\cG_n$ is nearly preserved by $\phi_{\sw}$, and we will show how to construct a null-homotopy of $\phi_{\sw}\circ \Pi+ \Pi\circ \phi_{\sw}$. We observe the following:
\begin{itemize}

\item Elements of type \eqref{G-1} are sent to other elements of type \eqref{G-1} under $\phi_{\sw}$. 

\item Elements of type \eqref{G-2} are sent to other elements of type \eqref{G-2}.

\item Elements of types \eqref{G-3} and \eqref{G-4} are interchanged.

\item Elements of both types \eqref{G-5} are fixed by $\phi_{\sw}$.

\item Elements of both types \eqref{G-6} are sent to a sum of themselves with some number of nonnegative $U$-powers of elements of type \eqref{G-10} and \eqref{G-11}, along with possibly a nonnegative $U$ power of $y_1y_{-1}+y_{-1}y_{1}$. (Note that $y_1y_{-1}+y_{-1}y_1$ is in $\cY_n$).

\item Elements of types \eqref{G-7} and \eqref{G-8} are interchanged.

\item Elements of type \eqref{G-9} are interchanged with other elements of type \eqref{G-9}.

\item Elements of type \eqref{G-10} and \eqref{G-11} are fixed.

\item For elements of type \eqref{G-12}-\eqref{G-16}, in each case elements of type (a) and type (b) are interchanged.

\end{itemize}

We note that the only generators of $\cG_n$ on which $\phi_{\sw}$ has a component mapping from $\cG_n$ to $\cY_n$ are those of type \eqref{G-6}. If after adding terms of type \eqref{G-10} and \eqref{G-11} to $\phi_{\sw}(x_{2m}x_{-2m})$ we obtain a summand of the form $U^t (y_{1}y_{-1}+y_{-1}y_1)$, then we claim that $t>n-1$. To see this, we observe that $U^{n-1}(y_1y_{-1}+y_{-1}y_1)$, $x_0x_0$, and $x_{2m}x_{-2m}$ all lie in Alexander grading 0, and $x_{2m}x_{-2m}$ lies in lower $\gr_{\ws}$ grading than $x_0x_0$. Since $U^{n-1}(y_1 y_{-1}+y_{-1}y_1)$ is in  $\phi_{\sw}(x_0x_0)$, we conclude that $t>n-1$, as claimed. We may therefore remove the offending term of $\phi_{\sw}\circ\Pi+\Pi\circ \phi_{\sw}$ via a homotopy $H$ which sends $H(x_{2m}x_{-2m})$ to $W^{t+1-n}Z^t (y_{-1}x_0 + x_0y_{-1})$.

It follows that $\Pi$ and $I$ are local equivalences, so $(\cCFK(T_{2n,2n+1}\# T_{2n,2n+1}), \phi_{\sw})$ is locally equivalent to $(\cY_n,\phi_{\sw})$, as claimed.
\end{proof}

We note that if we delete the box from $\cY_n$, then we obtain a staircase complex, for which we will write $\cD_n$. Note that $\cD_n$ coincides with $(\cCFK(T_{2n,4n+1}),\phi_{\std})$, where $\phi_{\std}$ is the map induced by the standard strong inversion (cf. \cite{HHSZ-Infinite}*{Proposition~3.1}).

\begin{cor}
\label{cor:almost-local-class-swap}If $n>1$ and $n$ is odd, the complex $(\cC_n\otimes \cC_n,\phi_{\sw})$ is almost locally equivalent to $(\cD_n,\phi_{\std})$.
\end{cor}
\begin{proof} The $\phi_K$-component which connects the box to the staircase is weighted by $U^{n-1}$. In the almost $\phi_K$-local class, this arrow plays no role, so $\cY_n$ decomposes as a direct sum of a staircase and a box, with $\phi_{\sw}$ preserving both summands modulo $U$. The claim follows. 
\end{proof}

We now recall another symmetry on $\cC_n\otimes \cC_n\iso \cCFK(T_{2n,2n+1}\# T_{2n,2n+1})$. We recall that $T_{2n,2n+1}$ is strongly invertible. We will write $\phi_{\std}$ for the induced map on $\cCFK(T_{2n,2n+1})$. Since $\cCFK(T_{2n,2n+1})$ is a staircase complex, the map $\phi_{\std}$ is constrained to be the map $x_i\mapsto x_{-i}$, $y_i\mapsto y_{-i}$, extended skew-equivariantly. The local class of this complex is computed in \cite{DMS:equivariant}. They describe a complex $(\cE_n,\phi)$ and prove
\begin{equation}
(\cC_n,\phi_{\std})\otimes(\cC_n,\phi_{\std})\sim \cE_n \label{eq:DMS-locally-equivalence-E}
\end{equation}
(We do not need to understand the particular form of the complex $\cE_n$). 

We recall additionally the box complex $(\cB_n,\phi)$. It has five generators, given by the following diagram:
\[
v\quad \oplus \quad \begin{tikzcd}[labels=description, column sep=1cm, row sep=1cm] 
r_{-1} \ar[r, "W^n"]
	& t
\\
 r_0 
 	\ar[u, "Z^n"] \ar[r, "W^n"]
 & r_1 \ar[u, "Z^n"]
\end{tikzcd}
\]
The bigradings are determined by setting $(\gr_{\ws}, \gr_{\zs})(v)=(\gr_{\ws}, \gr_{\zs})(t)=(0,0)$. 
The $\phi_K$-action is given by the following formula
\[
\phi_K(v)=v+t,\quad \phi_K(r_0)=r_0,\quad \phi_K(r_{-1})=r_1,\quad \phi_K(r_1)=r_{-1}, \quad \phi_K(t)=t.
\]

\begin{lem}
\label{lem:almost-phi-K-class-K_n} Let $K=T_{2n,2n+1}$ with $n>1$ odd. The $\phi_K$-almost local class of
\[
(\cCFK(K\#  K\# mrK\# mK), \phi_{\std}\# \phi_{\std}\#  \phi_{\sw})
\]
coincides with the box complex $(\cB_n,\phi)$. 
In the above, $mK$ denotes the mirror of $K$, while $mrK$ denotes the mirror of $K$ with the reversed knot orientation.
\end{lem}
\begin{proof}
The class is locally equivalent to
\[
(\cC_n\otimes \cC_n\otimes \cC_n^\vee \otimes \cC_n^\vee,  \phi_{\std}\otimes \phi_{\std}\otimes \phi_{\sw}^\vee).
\]
By Corollarly~\ref{cor:almost-local-class-swap} and Equation~\eqref{eq:DMS-locally-equivalence-E}, this is $\phi_K$-almost locally equivalent to
\[
(\cE_n, \phi)\otimes (\cD_n,\phi)^\vee
\]
which is locally equivalent to $(\cB_n,\phi)$ by \cite{DMS:equivariant}*{Lemma~6.13}.
\end{proof}

\subsection{Local class of surgeries}

We now compute the almost local class of $+1$-surgery on 
\[
(K_n,\phi):= (2T_{2n,2n+1}\# -2T_{2n,2n+1},\phi_{\std}\# \phi_{\std} \# \phi_{\sw}).
\]

\begin{prop}
\label{prop:almost-phi-class-surgery} The almost $\phi$-class of $(S_{+1}^3(K_n), \phi)$ is equal to a complex of the following form:
\[
\begin{tikzcd}
\xs
\ar[loop below, red, dashed, looseness=20] \ar[r, red,dashed]
  &
   \ys
   \ar[loop below, red, dashed, looseness=15] 
    &
    \zs \ar[l, "U^n"]
   \ar[loop below, red, dashed, looseness=20] 
\end{tikzcd}
\]
(The solid arrow denotes the differential, while the dashed arrows denotes $\phi$). 
\end{prop}
\begin{proof}
By Theorem~\ref{thm:local-class}, we have that $(\CF^-(S^3_{+1}(K_n),\phi))$ is locally equivalent to $(A_0(K_n),\phi)$. This, of course, also implies that the almost-$\phi$ class of $(\CF^-(S^3_{+1}(K_n),\bar{\phi}))$ is equal to the almost $\phi$ class $(A_0(K_n),\bar{\phi})$. By Lemma~\ref{lem:almost-phi-K-class-K_n}, the $\phi_K$ class of $(K_n,\phi)$ is equal to the box complex $(\cB_n,\phi)$. By Lemma~\ref{lem:almost-phi_K-local-to-almost-phi-local}, this implies that the almost $\phi$ class of $(A_0(K_n),\bar{\phi})$ is equal to the almost $\phi$ class of $(A_0(\cB_n),\bar{\phi})$. One computes directly that $(A_0(\cB_n),\bar{\phi})$ is homotopy equivalent as an almost-$\phi$ class to the complex in the statement, completing the proof. 
\end{proof}

\subsection{Proof of Theorem~\ref{thm:Z-infinity-summand}}

Finally, we describe our proof of Theorem~\ref{thm:Z-infinity-summand}:
\begin{proof}[Proof of Theorem~\ref{thm:Z-infinity-summand}]
By Proposition~\ref{prop:almost-phi-class-surgery}, the almost $\phi$-class of $(S^3_{+1}(K_n),\phi)$ is equal to the almost $\phi$ class $C(-1,n)$. This complex has $\varphi_n=1$ and $\varphi_i=0$ if $i\neq n$. It follows that the map
\[
\bigoplus_{n\ge 1} \varphi_n\colon \ker( \cF)\to \Z^\infty
\]
is a surjection, and that the pairs $(S^3_{+1}(K_n),\phi)$ span a $\Z^\infty$-summand.
\end{proof}

\section{Background on the link surgery formula}

\label{sec:background-link-surgery}

In this section, we present background on the link surgery formula of Manolescu and Ozsv\'{a}th \cite{MOIntegerSurgery}.

\subsection{Meridional surgery diagrams}
\label{sec:meridional-diagrams}

The link surgery formula requires a collection of several Heegaard diagrams for the underlying 3-manifold $Y$. There is a substantial amount of flexibility in the choice of these diagrams. However, for our purposes it is helpful to focus on a restricted class of Heegaard diagrams, which we call \emph{meridional} Heegaard link diagrams.

\begin{define} A Heegaard link diagram $(\Sigma,\as,\bs,\ws,\zs)$ is called a \emph{meridional Heegaard link diagram} for $(Y,L)$ if for each link component $K_i\subset L$, there is a distinguished disk $D_i\subset \Sigma$ such that $D_i$ contains $w_i$ and $z_i$, and such that there is a single curve $\b_{0,i}$ which divides $D_i$ into two half disks, one containing $w_i$ and the other $z_i$. Furthermore, no other attaching curves intersect $D_i$.  See Figure~\ref{fig:naturality_28}.
\end{define}

\begin{figure}[h]
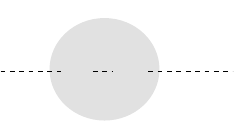
\caption{The distinguished disk $D_i$ containing the intersection of the shadow $S_i$ and the curve $\b_{0,i}$.}
\label{fig:naturality_28}
\end{figure}

On a meridional Heegaard link diagram, we can pick an embedded knot shadow $S_i$ for each link component. Since there is a canonical short arc connecting $w_i$ and $z_i$ which passes through $\b_{0,i}$, picking a knot shadow amounts to picking an arc from $w_i$ to $z_i$ which is disjoint from the beta curves. Note that such a shadow determines a Morse framing on the component $K_i$, namely the framing which is parallel to $T\Sigma$.

There is a natural interpretation of meridional Heegaard link diagrams in terms of sutured manifolds, as we now describe. See \cite{JDisks} for more background on sutured manifolds in the context of Heegaard Floer theory. Firstly, we consider the sutured manifold 
\[
(M(L),\g^\circ),
\] defined as follows. We define  $M:=Y\setminus \nu(L)$ and we define the sutures $\g$ by picking a single suture on each component of $\d M$. We pick each suture $\g$ to bound a small disk in $\d M$. We declare $R_-(M)\subset \d M$ to be the collection of small disks bounded by $\g$, and we declare $R_+(M)\subset \d M$ to be the complement of $R_-(M)$. Note that $R_+(M)$ is a disjoint union of $|L|$ punctured tori. On $\d M$ we place a meridian $\mu_i$ and a longitude $\lambda_i$ for each link component, so that there is a unique intersection of $\mu_i$ and $\lambda_i$, which occurs in the disk component of $R_-(M)$.

We now pick a sutured Heegaard diagram $(\Sigma^\circ,\as,\bs^\circ)$ for $(M(L),\g^\circ)$. The surface $\Sigma^{\circ}$ has $\ell$ boundary components, and furthermore $|\bs^\circ|=|\as|-\ell$, where $\ell=|L|$.  

To obtain a meridional Heegaard diagram for $(Y,L)$ from this description of the sutured manifold, we use the descending vector field of a Morse function to project each $\mu_i$ to a collection of $\ell$ arcs on $\Sigma^\circ$. These arcs have boundary on $\d \Sigma^\circ$, and also are pairwise disjoint and also disjoint from the $\bs^\circ$ curves. We fill each component of $\d \Sigma^\circ$ with a disk $D_i$, and in each component we add two basepoints $w_i$ and $z_i$. This gives $(\Sigma,\ws,\zs)$. We extend the projections of the $\mu_i$ over these disks to obtain the curves $\b_{0,i}$. We set $\bs_0=\bs^\circ\cup \{\b_{0,1},\dots, \b_{0,\ell}\}$.

Additionally, observe that we can similarly project each $\lambda_i$ to the Heegaard surface, as above to obtain a knot shadow $S_i$. Write $\cS=\{S_1,\dots, S_\ell\}$ for the set of shadows.

Given a meridional Heegaard diagram with shadows, we may construct a Heegaard diagram $(\Sigma,\as,\bs_{\Lambda},\ws)$ for $Y_{\Lambda}(L)$ by replacing $\b_{0,i}$ with a small isotopic translate of the knot shadow $\b_{\Lambda,i}$.

To construct the link surgery complex, we wind each curve $\b_{0,i}$ parallel to $\b_{\lambda,i}$ to form an attaching curve $\b_{1,i}$. For $\veps\in \bE_\ell$ we let $\bs_{\veps}$ denote the union 
\[
\bs_{\veps}=\bs^\circ\cup \{\b_{0,i}: \veps_i=0\} \cup \{\b_{1,i}: \veps_i=1\}
\]
with some curves translated slightly to obtain admissibility. We assume that the winding of each $\b_{1,i}$ is sufficient so that the diagram $(\Sigma,\bs_{\veps}, \bs_{\veps'}, \ps)$ is weakly admissible for each collection $\ps\subset \ws\cup \zs$ such that $\ps$ contains exactly one basepoint from each component of $L$.

\begin{define}
\label{def:meridional-surgery-diagram}
We refer to the collection $(\Sigma,\as,\bs_{\Lambda}, \{\bs_{\veps}\}_{\veps\in \bE_n}, \cS)$ described above as a \emph{meridional surgery diagram} for $(Y,L,\Lambda)$.
\end{define}
Since the main input for a meridional surgery diagram for $(Y,L,\Lambda)$ is a Heegaard diagram for $(M(L),\g^\circ)$, we obtain the following set of moves between any two meridional surgery diagrams:

\begin{prop}
\label{prop:meridional-diagram-moves}
 Any two meridional surgery diagrams for $(Y,L,\Lambda)$ may be related by a sequence of the following moves:
\begin{enumerate}
\item Handleslides and isotopies amongst the $\as$ curves, not crossing over the distinguished disks $D_i$. 
\item Handleslides of the curves $\b_{0,i}$ over the curves of $\bs^\circ$, as well as isotopies of these curves which are stationary in the distinguished disks.
\item Isotopies of the curves $\bs_{\veps}$ which do not cross over the distinguished disks $D_i$, as well as handleslides over the curves in $\bs^{\circ}$.
\item Index $(1,2)$-stabilization, to wit taking the connected sum of $\Sigma$ with a genus 1 surface, adding a new curve $\a$ to $\as$, as well as adding a curve $\b$ (or a small translate of $\b$) to $\bs^\circ$, such that furthermore $|\a\cap \b|=1$.
\item Simultaneous isotopies of the shadows $S_i$ and the curves $\b_{\Lambda,i}$, which are stationary in the distinguished disks $D_i$, such that $S_i$ and $\b_{\Lambda,i}$ remain disjoint during the isotopy. Similarly, simultaneous handleslides of $\b_{\Lambda,i}$ and $S_i$ across a curve in $\bs^\circ$. 
\item Ambient isotopy of the surface $\Sigma^\circ$ in the sutured manifold $(M(L),\g^\circ)$.
\end{enumerate}
\end{prop}

The proof follows from standard techniques, so we omit it.

\subsection{Meridional complete systems and the link surgery complex}
\label{sec:meridional-complete-systems}

In this section, we discuss the construction of the link surgery complex. We suppose that $\cH=(\Sigma, \as, \bs_{\Lambda}, \{\bs_{\veps}\}_{\veps\in \bE_\ell}, \ws, \zs, \cS)$ is a meridional surgery diagram.  To construct the link surgery complex, we need to pick some additional data, consisting of Floer chains between the various $\bs_{\veps}$. There are several ways to organize these chains. Manolescu and Ozsv\'{a}th \cite{MOIntegerSurgery} describe these as a collection of hypercubes, indexed by oriented sublinks $\vec{M}$ of $L$, which satisfy certain compatibly conditions. For our purposes, we take the perspective from \cite{ZemExactTriangle} and arrange the necessary data as a single hypercube of Lagrangians with local systems, which simplifies the notation. 

We recall the general framework of Lagrangians with local systems. If $L$ is a Lagrangian, a \emph{local system} consists of a vector space $E$ and a representation of the fundamental groupoid $\Pi(L)$ on $E$, i.e. a functor from $\Pi(L)$ to $\End(E)$ (where the latter is viewed as a category with one object).  We write $L^E$ for a Lagrangian equipped with a local system. A morphism between Lagrangians with local systems $L_1^{E_1}$ and $L_2^{E_2}$ consists of a pair $(\xs, \phi)$ where $\xs\in L_1\cap L_2$, and $\phi\colon E_1\to E_2$ is a linear map. 

Holomorphic polygon counting maps can be defined for Lagrangians with local systems. Given Lagrangians $L_1^{E_1},\dots, L_{n+1}^{E_{n+1}}$ and inputs $\langle \xs_1,\phi_1\rangle,\dots, \langle \xs_n,\phi_n\rangle$, one defines the polygon counting maps by counting rigid holomorphic polygons with inputs $\xs_1,\dots, \xs_n$. The linear map associated to a polygon class $\psi$ in the output is the natural composition of $\rho_{n+1}\circ \phi_n\circ \rho_{n}\circ \cdots \circ \phi_1\circ \rho_1$, where $\rho_{i}$ is the monodromy along the boundary component of $\psi$ on the Lagrangian $L_i$.

In the setting of the link surgery formula, we consider local systems of $\bF[U_1,\dots, U_\ell]$-modules. We begin by defining the $\bF[U]$-modules $E_0=\bF[W,Z]$ and $E_1=\bF[U,T,T^{-1}]$. Using the knot shadows, we equip $\bs_{\veps}$ with the local system
\begin{equation}\label{eq:local-system}
E_{\veps}:=E_{\veps_1}\otimes_{\bF}\cdots \otimes_{\bF} E_{\veps_\ell}.
\end{equation}
The monodromy maps for each factor of $E_0$ are given by $Z_i^{\# \d_{\b_\veps}(\psi)\cap S_i}$.  The monodromy maps for $E_1$ are given by $T_i^{\# \d_{\b_\veps}(\psi)\cap S_i}$. Additionally a holomorphic polygon is weighted by an overall factor of $U_1^{n_{w_1}(\psi)}\dots U_\ell^{n_{w_\ell}(\psi)}$. We write $\bs_{\veps}^{E_\veps}$ when viewing $\bs_{\veps}$ as being equipped with this local system.  (\emph{A priori} the differential could involve negative powers of $Z_i$, however these are counteracted by positive powers of $U_i$, and one can see that the cumulative effect of all powers only involves monomials in $W_i$, $Z_i$ and $U_i$ with nonnegative exponent; see \cite{ZemExactTriangle}*{Appendix~A}).

We can construct a hypercube of attaching curves with local systems $\scB_{\Lambda}=(\bs_{\veps},\Theta_{\veps,\veps'})_{\veps\in \bE_\ell}$. Given $\veps<\veps'$, write $I_{\veps,\veps'}\subset \{1,\dots, \ell\}$ for the set of indices $i$ where $\veps_i< \veps_i'$.   Each of the morphisms $\Theta_{\veps,\veps'}$ decomposes as a sum of morphisms
\[
\Theta_{\veps,\veps'}=\sum_{\scO \colon I_{\veps,\veps'}\to \{\sigma,\tau\} } \Theta_{\veps,\veps'}^{\scO}.
\]
In the above, $\scO$ is a function from the set $I_{\veps,\veps'}$ to the set of symbols $\{\sigma,\tau\}$. The chains $\Theta_{\veps,\veps'}$ are constructed in \cite{ZemExactTriangle}*{Section~3.5} when $(\Sigma, \{\bs_{\veps}\}_{\veps\in \bE_\ell})$ is suitably simple, and in \cite{ZemExactTriangle}*{Proof of Theorem 9.1 (1)} for more general sets of attaching curves.

We now review aspects of the construction of $\scB_{\Lambda}$. We will refer to a set $\qs\cup \ws\cup \zs$ as a \emph{complete collection of basepoints} if $\qs$ contains exactly one basepoint from each link component. If $|\veps'-\veps|_{L^1}=1$ and $\veps<\veps'$, then we define 
\[
\Theta_{\veps,\veps'}=\Theta_{\veps,\veps'}^{\sigma}+\Theta_{\veps,\veps'}^{\tau}.
\]
The generators $\Theta_{\veps,\veps'}^{\sigma}$ are described as follows. We focus firstly on the case when $\ell=1$. In this case, the diagram $(\Sigma,\bs_0,\bs_1,\ws,\zs)$ is a stabilization of the genus 1 diagram $(\bT^2,\b_0,\b_1,w,z)$ shown in Figure~\ref{fig:naturality:42}.

\begin{figure}[h]
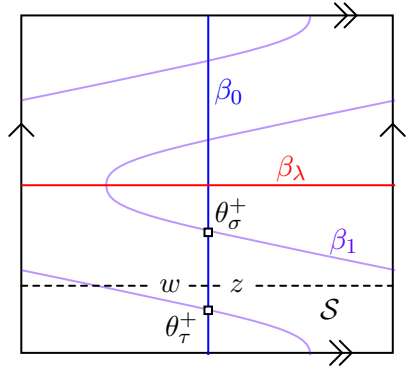
\caption{The genus 0 diagram $(\bT^2,\b_0,\b_1,w,z)$, the shadow $\cS$, and the intersection points $\theta_\sigma^+$ and $\theta_\tau^+$.}
\label{fig:naturality:42}
\end{figure}

The intersection point $\theta_\sigma^+$ is the top graded cycle with $\frs_{w}(\xs)$ torsion. The intersection point $\theta_\tau^+$ is the top graded cycle with $\frs_{z}(\xs)$ torsion. We set 
\[
\Theta^\sigma_{0,1}=\langle \theta_\sigma^+, \phi^\sigma\rangle\quad \text{and} \quad \Theta^\tau_{0,1}=\langle \theta_\tau^+ ,\phi^\tau\rangle.
\]
For more general diagrams $(\Sigma,\bs_{\veps},\bs_{\veps'},\ws,\zs)$, where $\veps'=\veps+e_i$, with $e_i$ the standard $i$-th basis vector, we consider the subcomplex 
\[
\ve{\CF}^-_{\phi^{\sigma_i}}\left(\bs_{\veps}^{E_{\veps}},\bs_{\veps'}^{E_{\veps'}}, \frs_{\sigma_i}\right)\subset\ve{\CF}^-\left(\bs_{\veps}^{E_{\veps}},\bs_{\veps'}^{E_{\veps'}}\right)
\] generated by Floer morphisms $\langle \xs, f\rangle$ where $\xs$ has $\frs_{\qs}(\xs)$ torsion for each complete collection $\qs$ such that $w_i\in \qs$ and $f$ is a multiple of
\[
\phi^{\sigma_i}=\id\otimes \cdots\otimes  \phi^\sigma\otimes \cdots \otimes \id.
\]
In the above equation, $\phi^\sigma$ is in the $i$-th tensor component. It is straightforward to see that $\frs_{\qs}(\xs)=\frs_{\qs'}(\xs)$ if $w_i\in \qs,\qs'$ and $\xs\in \bT_{\b_{\veps}}\cap \bT_{\b_{\veps'}}$. From \cite{ZemExactTriangle}*{Section~8.2}, $\ve{\CF}^-_{\phi^{\sigma_i}}\left(\bs_{\veps}^{E_{\veps}},\bs_{\veps'}^{E_{\veps'}}, \frs_{\sigma_i}\right)$ has an Alexander and Maslov grading, and we select $\Theta_{\veps,\veps'}^{\sigma_i}$ to be a cycle representing the top degree element of homology in Alexander grading 0 (cf. \cite{ZemExactTriangle}*{Lemma~8.15 and Remark~8.16}). The higher length chains of $\scB_{\Lambda}$ are constructed by an inductive filling construction. See \cite{ZemExactTriangle}*{Proof of Theorem 9.1 (1)} for more details.

 With the above set in place, the link surgery complex is merely the Floer complex
 \[
 \cC_{\Lambda}(Y,L):=\ve{\CF}^-(\as,\scB_{\Lambda}).
 \]
 
 We refer to the data of the diagram $(\Sigma, \as, \{\bs_{\veps}\}_{\veps\in \bE_\ell}, \ws, \zs, \cS)$ as well as the chains in the hypercube $\scB_{\Lambda}$ as a \emph{meridional complete system of Heegaard diagrams}.

\subsection{The isomorphism with the closed 3-manifold invariant}

Given a meridional surgery diagram $\cH=(\Sigma,\as,\bs_{\Lambda},(\bs_{\veps})_{\veps\in \bE_\ell}, \ws, \zs, \cS)$ we can write down a canonical map
\[
\Gamma\colon \ve{\CF}^-(Y_{\Lambda}(L))\to \cC_{\Lambda}(Y,L),
\]
as follows. The Heegaard diagram $(\Sigma,\as,\bs_{\Lambda},\ws)$ represents $Y_{\Lambda}(L)$. Furthermore, there is a canonical generator of $\ve{\CF}^-(\bs_{\Lambda}, \scB_{\Lambda})$ of top degree. This can be seen in several ways. Firstly, using the surgery formula, the complex $\ve{\CF}^-(\bs_{\Lambda}, \scB_{\Lambda})$ computes $\ve{\CF}^-(\#^{g-\ell} S^1\times S^2,\ws)$, which is well-known to admit a canonical generator in the torsion $\Spin^c$ structure. More generally, we can construct a morphism 
\[
\Theta_{\Gamma}\colon \bs_{\Lambda}\to \scB_{\Lambda}
\] by defining a morphism of twisted complexes as follows. We define the cycle $\bs_{\Lambda}\to \bs_{0,\dots, 0}$ to be the generator of 
\[
\begin{split}
\ve{\HF}^-(\bs_{\Lambda}, \bs_{0,\dots, 0}^{E_{0,\dots, 0}})\iso& H_*\cCFK ((S^1\times S^2)^{\# g-\ell}, \bU_n)
\\
\iso &\frac{\bF\llsquare W_1,Z_1,\dots, W_\ell,Z_\ell\rrsquare}{U_i=U_j}\otimes (\bF\oplus \bF)^{\otimes g-\ell}.
\end{split}
\]
Here $\bU_n$ denotes an $n$-component unlink. The length 2 relations of the resulting hypercube $\Cone(\bs_{\Lambda}\to \bs_{0,\dots, 0})$ hold up to homology by the model computation in \cite{ZemExactTriangle}*{Lemma~5.1}. From here, the standard filling construction for hypercubes (see \cite{MOIntegerSurgery}*{Lemma~8.6}) can be used to construct the higher length chains.

The hypercube relations for $\Cone(\bs_{\Lambda}\to \scB_{\Lambda})$ can also be verified more directly if we pick particularly simple isotopy classes of the attaching curves $\bs_{\veps}$ in $\scB_{\Lambda}$. This is the approach taken in \cite{ZemExactTriangle}; see Section~5 and Theorem~6.1 therein.

In \cite{ZemExactTriangle}*{Theorem~6.1}, we prove that $\Theta_{\Gamma}$ is an isomorphism of twisted complexes of Lagrangians. Therefore the map
\[
\Gamma:=\mu_2^{\Tw}(-,\Theta_{\Gamma})\colon \ve{\CF}^-(\as,\bs_{\Lambda})\to \cC_{\Lambda}(Y,L)
\]
is a homotopy equivalence of chain complexes.

\subsection{The link surgery formula and the elliptic involution}

In this section, we describe the effect on the link surgery formula of changing the string orientation of a link:

\begin{lem}
\label{lem:change-orientation} If $L\subset Y$ is an oriented and framed link, and $r_i L$ is obtained by reversing the orientation of the $i$-th component of $L$, then there is a canonical isomorphism
\[
\cX_{\Lambda}(Y,r_i L)^{\cK_1\otimes \cdots \otimes \cK_\ell}\iso \cX_{\Lambda}(Y,L)^{\cK_1\otimes \cdots \otimes \cK_\ell}\boxtimes {}_{\cK_i} [E]^{\cK_i}.
\]
Furthermore, this isomorphism is natural, that is, it commutes with the naturality maps for the link surgery complex. (In the above, each $\cK_i$ denotes a copy of the algebra $\cK$).
\end{lem}
\begin{proof} A Heegaard link diagram $(\Sigma,\as,\bs,\ws,\zs)$ for $(Y,r_i L)$ is obtained from a Heegaard link diagram for $(Y,L)$ by switching the labeling of the $w_i$ and $z_i$ basepoints on $\Sigma$. We can also perform this manipulation to the link surgery complex, and the effect is to change the labeling of the $\sigma$-generators $\theta^{\sigma}$ and the $\theta^{\tau}$, generators. From this discussion, we obtain the canonical isomorphism in the statement. Commutation with the naturality maps follows because the isomorphism is merely the result of relabeling.
\end{proof}

\section{The equivariant link surgery formula}
\label{sec:equivariant-link-surgery-statement}

\subsection{Statement of the equivariant link surgery formula}

Firstly, if $(Y,L,\Lambda)$ and $(Y',L',\Lambda')$ are framed links and 
\[
\phi\colon (Y,L,\Lambda)\to (Y',L',\Lambda')
\]
is a diffeomorphism of pairs which is orientation-preserving on each link component, then naturality of the link surgery complexes will imply that there is a morphism of type-$D$ modules
\[
\cX(\phi)\colon \cX_{\Lambda}(Y,L)^{\cL_\ell}\to \cX_{\Lambda'}(Y',L')^{\cL_\ell},
\]
well-defined up to chain homotopy.

We now specialize to the case where $\phi$ is a symmetry of $(Y,L,\Lambda)$, though we now allow the case where $\phi$ might reverse the orientations of some link components, and also may permute the components. Firstly, let
\[
\varrho \colon \{1,\dots, \ell\}\to \{1,\dots, \ell\}
\]
denote the permutation of the components of $L$ induced by the $\phi$.

Next, let $J\subset L$ denote the images of the components of $L$ where $\phi$ is orientation reversing. 
%

 We define a  bimodule
\[
{}_{\cL_\ell} [\bI_{\varrho^{-1}}]^{\cL_\ell}
\]
to be the bimodule associated to the algebra morphism from $\cL_\ell$ to $\cL_\ell$ which permutes the factors of $\cL_\ell$ according to the permutation $\varrho^{-1}$. That is, if $a_i$ is an algebra element in the $i$-th tensor factor of $\cL_{\ell}$, we set
\[
\delta_2^1(a_i,1)=1\otimes a_{\varrho^{-1}(i)}.
\]

Note that $\phi$ induces a diffeomorphism of strongly framed links from $(Y,L)$ to $(Y,r_J L)$. Consequently, there is an induced diffeomorphism map $\cX(\phi)$ on type-$D$ modules. Since $\phi$ permutes the components of $L$, this takes the form of a map
\[
\cX(\phi)\colon \cX_{\Lambda}(Y,L)^{\cL_\ell}\to \cX_{\Lambda}(Y,r_J L)^{\cL_\ell}\boxtimes {}_{\cL_\ell}[\bI_{\varrho^{-1}}]^{\cL_\ell}.
\]

Using Lemma~\ref{lem:change-orientation}, we may repackage this as a type-$D$ morphism
\[
\cX(\phi)\colon \cX_{\Lambda}(Y,L)^{\cL_\ell}\to \cX_{\Lambda}(Y,L)^{\cL_\ell}\boxtimes {}_{\cL_\ell}[\scE_\phi]^{\cL_\ell}
\]
where $\scE_\phi$ is the following bimodule. Firstly, define ${}_{\cL_\ell} [E_J]^{\cL_\ell}$ to be the external tensor product of the elliptic bimodule ${}_{\cK_i} [E]^{\cK_i}$, ranging over $i\in J$, and the identity bimodule ${}_{\cK_i}[\bI]^{\cK_i}$ for $i\not \in J$. Then we set
\[
{}_{\cL_\ell}[\scE_\phi]^{\cL_\ell}:={}_{\cL_\ell} [E_J]^{\cL_\ell}\boxtimes {}_{\cL_\ell}[\bI_{\varrho^{-1}}]^{\cL_\ell}.
\]

\subsection{Statement of the equivariant surgery formula}

We now state our equivariant surgery formula. We suppose that $\phi$ is a symmetry of the framed link $(Y,L,\Lambda)$. As before, we let $J\subset L$ denote the images under $\phi$ of the components of $L$ on which $\phi$ is orientation reversing. In the previous section, we described a map
\[
\cX(\phi)\colon \cX_{\Lambda}(Y,L)^{\cL_\ell}\to \cX_{\Lambda}(Y,L)^{\cL_\ell}\boxtimes {}_{\cL_\ell} [\scE_\phi]^{\cL_\ell}. 
\]

We will write ${}_{\cL_\ell} \cD^\ell$ for the external tensor product of $\ell$ copies of the solid torus module ${}_{\cK} \cD$ defined in Section~\ref{sec:background-surgery-algebra}.

We observe firstly that there is a natural isomorphism of type-$A$ modules
\[
\varrho\colon {}_{\cL_\ell} [\bI_{\varrho^{-1}}]^{\cL_\ell}\boxtimes {}_{\cL_\ell} \cD^{\ell}\to {}_{\cL_\ell} \cD^{\ell},
\]
which  permutes the factors of $\cD^{\ell}$ according to $\varrho$. The map $\varrho$ has only $\varrho_1$ non-trivial, and has $\varrho_j=0$ for $j>1$.  For example
\[
\varrho_1(W_i)=W_{\varrho(i)}
\]
and similarly for other elements of $\bI_{\varrho^{-1}}\boxtimes \cD^\ell$.

Additionally, there is a natural map
\[
\varpi_{J}\colon {}_{\cL_n} [E_J]^{\cL_\ell}\boxtimes {}_{\cL_\ell}\cD^{n}\to {}_{\cL_\ell}\cD^{\ell},
\]
which is the external tensor product of $|J|$ copies of the morphism $\varpi$ from Section~\ref{sec:elliptic-involution}. 

We write
\[
\Omega_\phi\colon {}_{\cL_\ell} [\scE_\phi]^{\cL_{\ell}}\boxtimes {}_{\cL_\ell} \cD^\ell\to {}_{\cL_\ell} \cD^\ell
\]
for the composition of $\varrho$ and $\varpi_J$.

The equivariant surgery formula is the following:

\begin{thm} Let $(L,\Lambda)$ be a Morse framed link in a 3-manifold $Y$, and let $\phi\colon(Y,L,\Lambda)\to (Y,L,\Lambda)$ be a diffeomorphism of strongly framed links. Let $\Phi\colon Y_{\Lambda}(L)\to Y_{\Lambda}(L)$ denote the diffeomorphism induced by $\phi$. Then there is a homotopy equivalence
\[
\Gamma\colon \ve{\CF}^-(Y_{\Lambda}(L))\to \cC_{\Lambda}(Y,L)
\]
which intertwines $\ve{\CF}^-(\Phi)$ with the map $\cC(\phi)$, defined as the composition of the following two maps:
\[
\begin{tikzcd}[column sep=.7cm]
\cX_{\Lambda}(Y,L)^{\cL_\ell}\boxtimes {}_{\cL_\ell} \cD^{\ell} \ar[r, "\cX(\phi)\boxtimes \bI"] & \cX_{\Lambda}(Y,L)^{\cL_\ell}\boxtimes {}_{\cL_\ell} [\scE_\phi]^{\cL_n}\boxtimes {}_{\cL_\ell} \cD^{\ell}\ar[r, "\bI\boxtimes \Omega_\phi"]&\cX_{\Lambda}(Y,L)^{\cL_\ell}\boxtimes {}_{\cL_\ell} \cD^{\ell}.
\end{tikzcd}
\]
\end{thm} 

\subsection{Alexander gradings and cube points}

We now describe some of the algebraic structure of the map $\cC(\phi)$, focusing on its behavior with regards to the cube points, and with regards to the Alexander grading. Firstly, the map $\cC(\phi)$ splits as a sum of maps
\[
\cC(\phi)=\sum_{\vec{M}\subset L} F^{\vec{M}}.
\]
In the above, the sum  is over all oriented sublinks of $\vec{M}\subset L$ (with orientations potentially differing from $L$). The type-$D$ morphism $\cX(\phi)$ similarly splits as a sum
\[
\cX(\phi)=\sum_{\vec{M} \subset L} f^{\vec{M}}.
\]

If $\vec{M}\subset L$, write $a_{\vec{M}}$ for the product of $\sigma_i$ for $i$ such that $+K_i\subset \vec{M}$ and $\tau_i$ for $i$ such that $-K_i\subset \vec{M}$.  The map $f^{\vec{M}}$ has only algebra outputs which are weighted by elements of the form products of $b\cdot a_{\vec{M}}\cdot \theta$ where $b$ is an algebra element  which does not shift idempotent, and $\theta$ denotes the generator of the module $[\scE_\phi]$.  Here, we are thinking of the tensor factor of $\theta$ in the same way as in Remark~\ref{rem:theta-algebra-element}. 


We write $\varrho$ for the permutation on the set of link components which is induced by $\phi$. We can view $\varrho$ as also acting on the points of the cube $\bE_\ell$ via the formula
\[
\varrho(\veps_1,\dots, \veps_\ell)=(\veps_{\varrho^{-1}(1)},\dots, \veps_{\varrho^{-1}(\ell)}).
\]

If $\veps\in \bE_\ell$, we will write $L_{\veps}\subset L$ for the sublink consisting of components $K_i$ which have $\veps_i=0$. If $M\subset L$, write $\delta_M\in \{0,1\}^\veps$ for the indicator function on the components of $M$.

The map $F^{\vec{M}}$ sends $\cC_{\veps}\subset \cC_{\Lambda}(L)$ to $\cC_{\varrho(\veps)+\delta_{M}}$, where $\delta_{M}\in \{0,1\}^\ell$ denotes the indicator function on the components of $M$ (i.e. $\delta_M(i)=1$ if and only if $K_i\subset M$).

We now consider the Alexander grading of the map $F^{\vec{M}}$ and $f^{\vec{M}}$. We first recall the Alexander gradings on the link surgery hypercube. On cube point $\veps$, the link surgery complex $\cC_{(0,\dots, 0)}\subset \cC_{\Lambda}(L)$ is a completion of the link Floer complex $\cCFL(L)$. This has a canonical absolute $\bH(L)$-valued Alexander grading. At a cube point $\veps\in \bE_\ell$, the complex $\cC_{\veps}$ can be identified with $\cCFL(L_\veps)\otimes \bigotimes_{ \{i|\veps_i=1\}} \bF[T_i^{\pm 1}]$. Here, $\cCFL(L_{\veps})$ denotes a link Floer complex with $|L|-|L_\veps|$ extra free basepoints (which each contribute a $U_i$ variable). This complex admits a relative $\bH(L_\veps)\times \Z^{|L|-|L_\veps|}$ valued Alexander grading. We can normalize this to be an absolute $\bH(L)$-valued Alexander grading by requiring all of the surgery maps $\Phi^{+K_i}$ to preserve the Alexander grading. We call this the \emph{$\sigma$-normalized} Alexander grading (cf. \cite{ZemExactTriangle}*{Section~7.2}). With this convention, the structure map $\Phi^{\vec{M}}$ sends Alexander grading $\ve{s}$ to Alexander grading $\ve{s}+\Lambda_{\vec{M}}$, where $\Lambda_{\vec{M}}$ consists of 
\[
\Lambda_{\vec{M}}=\sum_{\{i| -K_i\subset \vec{M} \}} \Lambda_i,
\]
where 
\[
\Lambda_{i}=(\lk(K_1,K_i),\dots, \lambda_i, \dots, \lk(K_\ell,K_i))\in \Z^\ell.
\]
Note that if we view $H_1(S^3\setminus L)$ as $\Z^\ell$, then $\Lambda_i$ is equal to the class of the longitude of the component $K_i$.

We will also write $\phi(L\setminus L_{\veps})$ for the image of the components of $L\setminus L_{\veps}$ under $\phi$, equipped with the push-forward orientation from $L$.


Additionally, if $J\subset L$, we write $R_J\colon \bH(L)\to \bH(L)$ for the map which sends $\ve{s}=(s_1,\dots, s_\ell)$ to $\ve{s}'=(s_1',\dots, s_\ell')$ where $s_i'=s_i$ if $K_i\not \subset J$ and $s_i'=-s_i$ if $K_i\subset J$.

\begin{lem} Let $(L,\Lambda)$ be an integrally framed link in $S^3$, and let $\phi$ be a diffeomorphism of strongly framed links. Let $J\subset L$ denote the sublink consisting of the image under $\phi$ of the components on which $\phi$ is orientation reversing. Write $P$ for permutation of the components of $L$, induced by $\phi$. Suppose that $\veps,\veps'\in \bE_\ell$,  $M\subset L$ and that $\veps'=\varrho(\veps)+\delta_M$. The map $F^{\vec{M}}_{\veps,\veps'}$ sends the subspace of $\cC_\veps(L)$ in Alexander grading $\ve{s}$ to the subspace of  $\cC_{\veps'}(L)$ which lies in Alexander grading 
\[
(R_J\circ \varrho)(\ve{s})+\Lambda_{\phi(L\setminus L_\veps)}+\Lambda_{\vec{M}}.
\]
\end{lem}

\begin{proof} By construction, the map $F^{\vec{M}}_{\veps,\veps'}$ will have the same shift in Alexander grading as $\Phi^{\vec{M}}_{\varrho(\veps),\veps'}\circ F^{\emptyset}_{\veps,\varrho(\veps)}$. By definition, $\Phi^{\vec{M}}_{\varrho(\veps),\veps'}$ shifts the $\sigma$-normalized Alexander grading by $\Lambda_{\vec{M}}$. Therefore it suffices to show the claim when $\vec{M}=\emptyset$. We begin by verifying the claim when $\veps=(0,\dots, 0)$. In this case, the claim is just that Alexander grading $\ve{s}$ is mapped to Alexander grading $(R_J\circ \varrho)(\ve{s})$. This follows from the definition of the absolute Alexander grading on link Floer homology; see, for example, \cite{ZemAbsoluteGradings}*{Proposition~8.3} for a very similar argument. For general $\veps$, we can use the following commutative diagram of gradings to compute the grading of $F^{\emptyset}_{\veps,\varrho(\veps)}$:
\[
\begin{tikzcd}[column sep=2cm]
\cC_{\vec{0}}(\ve{s}) \ar[r, "F^\emptyset_{\vec{0},\vec{0}}"] \ar[d, "\Phi^{L\setminus L_{\veps}}"]
 & \cC_{\vec{0}}((R_J \circ \varrho)(\ve{s})) \ar[d, "\Phi^{\phi(L\setminus L_{\veps})}"]
 \\
\cC_{\veps}(\ve{s}) \ar[r,dashed, "F^\emptyset_{\veps,\varrho(\veps)}"] &\cC_{\varrho(\veps)}\left((R_J\circ \varrho)(\ve{s})+\Lambda_{\phi(L\setminus L_{\veps})}\right).
\end{tikzcd}
\]
This completes the proof.
\end{proof}

%

\section{Diffeomorphisms of multi-pointed 3-manifolds}
\label{sec:pointed-diffeos}

In this section, we discuss diffeomorphisms of 3-manifolds with multiple basepoints.  Ozsv\'{a}th and Szab\'{o} proved that Heegaard Floer homology is invariant under adding basepoints to a 3-manifold \cite{OSLinks}*{Proposition~6.5}. In this section, we prove that on the level of homology, the diffeomorphism maps are invariant under adding basepoints.

\begin{thm}\label{thm:multi-pointed-diffeos} Suppose that $Y$ is an integer homology 3-sphere and $\phi\colon Y\to Y$ is an orientation preserving diffeomorphism. Let $\phi_0$ and $\phi_1$ be diffeomorphisms which are isotopic to $\phi$ as unpointed diffeomorphisms, such that $\phi_0$ and $\phi_1$ preserve non-empty collections of basepoints $\ws_0,\ws_1\subset Y$ setwise. Then there is an isomorphism
\[
\HF^-(Y,\ws_0)\iso \HF^-(Y,\ws_1)
\] 
which intertwines the maps $\HF(\phi_0)$ and $\HF(\phi_1)$. 
\end{thm}

Theorem~\ref{thm:multi-pointed-diffeos} does not give us full access to the diffeomorphism map $\CF(\phi)$ on $\CF^-(Y)$. Nonetheless we remark in Section~\ref{sec:correction-terms} that the induced map $\HF(\phi)$ on $\HF^-(Y)$ and $\HF^-(-Y)$ determines the correction terms $\underline{d}^\phi(Y)$ and $\overline{d}{}^\phi(Y)$. 

\subsection{Stablization}

We recall from \cite{OSLinks}*{Section~4} that Heegaard Floer homology $\HF^-(Y)$ can naturally be defined for multi-pointed 3-manifolds. Indeed, if $\ws$ is a nonempty collection of basepoints, there is a chain complex $\CF^-(Y,\ws)$ which is free over the ring $\bF[U_1,\dots, U_n]$, where $|\ws|=n$. Ozsv\'{a}th and Szab\'{o} prove that the homology $\HF^-(Y,\ws)$ is invariant of the choice $\ws$ \cite{OSLinks}*{Proposition~6.5}. They further show that if $w_{n+1}\not \in \ws$, then
\begin{equation}
\CF^-(Y,\ws\cup \{w_{n+1}\})\simeq \Cone\left(\begin{tikzcd}[column sep=2cm] \CF^-(Y,\ws)[U_{n+1}] \ar[r, "U_n+U_{n+1}"] & \CF^-(Y,\ws)[U_{n+1}] \end{tikzcd}\right),
\label{eq:decompose-mapping-cone}
\end{equation}
which implies in particular that the homology of $\CF^-(Y,\ws)$ coincides with the homology of $\CF^-(Y,\ws\cup \{w_0\})$. Furthermore, all of the $U_i$ have the same action on homology.

It is helpful to formalize the above by observing that there is a natural chain map from $\CF^-(Y,\ws)$ to $\CF^-(Y,\ws\cup \{w_{n+1}\})$ which sends an element of $\CF^-(Y,\ws)$ into the right copy of $\CF^-(Y,\ws)$ in the mapping cone description of Equation~\eqref{eq:decompose-mapping-cone}. This map is easily seen to be the same as the \emph{free-stabilization} map
\[
S_{w_{n+1}}^+\colon \CF^-(Y,\ws)\to \CF^-(Y,\ws\cup \{w_{n+1}\})
\]
described in \cite{ZemGraphTQFT}*{Section~6}.
In particular, we see that $S_{w_{n+1}}^+$ is an isomorphism on homology.

\begin{rem}
In \cite{ZemGraphTQFT}, there is also a free destabilization map $S_{w_{n+1}}^-$, defined in the opposite direction. This map is zero on homology for the complexes above. (It becomes non-zero if we set some of the $U_i$ equal before taking homology).
\end{rem}

\begin{lem}\label{lem:stabilization-diffeo} If $\phi$ is an orientation preserving diffeomorphism of $Y$ which fixes a collection of basepoints $\ws\subset Y$ setwise. If $w_{n+1}\in Y\setminus \ws$ then
\[
\HF(\phi)\circ S_{w_{n+1}}^+\simeq S_{\phi(w_{n+1})}^+\circ \HF(\phi).
\]
\end{lem}

The above follows from the naturality of the maps $S_{w_{n+1}}^+$, proven in \cite{ZemGraphTQFT}*{Section~6}.

\subsection{Correction terms}
\label{sec:correction-terms}

In this section, we prove the following result, which is a straightforward adaptation of the ideas of \cite{HMInvolutive}*{Section 4} and \cite[Lemma 2.12]{HMZConnectedSum}. Given a diffeomorphism $\phi \colon Y \rightarrow Y$, let $-\phi \colon -Y \rightarrow -Y$ be the same map on the three-manifold with its orientation reversed.

\begin{prop} The map $\HF(\phi)\colon \HF^-(Y)\to \HF^-(Y)$ determines $\underline{d}{}^{\phi}(Y)$ and the map $\HF(-\phi) \colon \HF^-(-Y) \to \HF^-(-Y)$ determines the correction term $\overline{d}{}^\phi(Y)$.
\end{prop}

\begin{proof} Since $\overline{d}{}^{\phi}(Y) = -\underline{d}{}^{-\phi}(-Y)$, it suffices to show that the map $\HF(\phi)$ determines $\underline{d}{}^{\phi}(Y)$. Recall that in terms of the chain complex $\CF^-(Y)$ and the map $\CF(\phi)$, the lower correction term is the maximum grading of a cycle $v$ in $\CF^-(Y)$ with the property that $U^n[v]\neq 0$ for any $n \geq 0$, and there is a $w$ in $\CF^-(Y)$ such that $\partial w = (1+\CF(\phi))v$; this is the analog of \cite[Lemma 2.12(a)]{HMZConnectedSum}. This is equivalent to maximum grading of a homology class $[v]$ in $\HF^{-}(Y)$ such that $U^n[v]\neq 0$ for any $n \geq 0$ and $(1+\HF(\phi))[v]=0$, and in particular determined by $\HF^-(Y)$ and $\HF(\phi)$.
\end{proof}

\subsection{Basepoint swapping maps}

In this section, we prove a technical result about basepoint moving maps in Heegaard Floer theory. We will later use the results of this section to show that if $\phi$ is an automorphism of $(Y,L)$ which fixes a framing $\Lambda$ on $L$, then the induced map $\cC(\phi)$ on $\cC_{\Lambda}(L)$ computes the map for a diffeomorphism of $Y_{\Lambda}(L)$ which is isotopic to the diffeomorphism induced by $\phi$ on the singly pointed version of Heegaard Floer homology $\HF^-(Y_{\Lambda}(L))$.

In particular, we show the following:

\begin{thm}
\label{thm:swap-diffeo-on-homology}
Let $(Y,\ws)$ be a multi-pointed 3-manifold with $|\ws|>1$, and suppose that $\lambda$ is an arc connecting two distinct basepoints in $\ws$. Write $\Sw_\lambda\colon (Y,\ws)\to (Y,\ws)$ for the diffeomorphism which switches the endpoints of $\lambda$ and is supported in a neighborhood of $\lambda$. Then $(\Sw_\lambda)_*$ induces the identity map on $\HF^-(Y,\ws)$, where $\HF^-(Y,\ws)$ denotes the version of Floer homology with $|\ws|$ distinct $U_i$ variables.
\end{thm}

To prove the above, we compute the diffeomorphism map $(\Sw_{\lambda})_*$ on $\CF^-(Y,\ws)$ using the graph TQFT from \cite{ZemGraphTQFT}. The map will typically be non-trivial on the chain level, but we will give an argument to show that it is the identity on homology.

 In \cite{ZemGraphTQFT}*{Proposition~14.24}, the last author proves that as an endomorphism of
\[
\frac{\CF^-(Y,\ws)}{U_1=\cdots=U_n},
\]
we have
\begin{equation}
(\Sw_{\lambda})_*\simeq \id+(\Phi_{w_1}+\Phi_{w_2}) \circ A_{\lambda}.\label{eq:basepoint-swapping-colors}
\end{equation}
Here, $A_{\lambda}$ denotes the relative homology action, and $\Phi_{w_i}$ are the basepoint actions. The definitions of these maps are recalled below. In our present setting, the above formula is not useful since we are not setting all of the $U_i$ to be equal. Instead, we will prove an analog of Equation \eqref{eq:basepoint-swapping-colors} which holds when all of the basepoints have distinct variables; our new formula appears in Proposition~\ref{prop:compute-Sw-map} below.

 We recall the definitions of the maps appearing in Equation~\eqref{eq:basepoint-swapping-colors}, which we will also require for its generalization. The map $A_{\lambda}$ is defined by representing $\lambda$ as an immersed arc on the Heegaard diagram, with $\d \lambda=\{w_1,w_2\}$. One then sets
\[
A_{\lambda}(\xs):=\sum_{\ys\in \bT_{\a}\cap \bT_{\b}} \sum_{\substack{\phi\in \pi_2(\xs,\ys)\\ \mu(\phi)=1}} \#(\d_{\a}(\phi)\cap \lambda) \# (\cM(\phi)/\R) U_1^{n_{w_1}(\phi)}\cdots U_n^{n_{w_n}(\phi)} \ys.
\] 
The map $\Phi_{w_i}$ is defined via the formula
\[
\Phi_{w_i}(\xs):=U_i^{-1}\sum_{\ys\in \bT_{\a}\cap \bT_{\b}} \sum_{\substack{\phi\in \pi_2(\xs,\ys)\\ \mu(\phi)=1}} n_{w_i}(\phi) \# (\cM(\phi)/\R) U_1^{n_{w_1}(\phi)}\cdots U_n^{n_{w_n}(\phi)} \ys.
\]
 We recall from \cite{ZemGraphTQFT}*{Lemmas~5.1,~5.5} that if $\d \lambda=\{w_1,w_2\}$, then
\begin{equation}
\begin{split} \d(A_\lambda)&=U_{1}+U_{2}\\
A_{\lambda}^2&\simeq U_1\simeq U_2.
\end{split}
\label{eq:basic-prop-A_lambda}
\end{equation}
Another relation that is helpful is the following:
\[
[A_{\lambda},A_{\lambda'}]\simeq\sum_{w_i\in \d \lambda\cap \d \lambda'} U_{i},
\]
for which we refer the reader to \cite{ZemGraphTQFT}*{Lemma~5.4}.

For our proof of Theorem~\ref{thm:swap-diffeo-on-homology}, we need to consider somewhat more complicated ways of identifying variables. We recall that if $R$ is a commutative ring, an $R$-\emph{coloring} (or just \emph{coloring} if $R$ is implicit) of $\ws$ is a map $\sigma\colon \ws\to R$. In this case, there is a version of the Floer complex $\CF^-(Y,\ws^\sigma)$ which is defined as a free $R$-module over $\bT_{\a}\cap \bT_{\b}$. The differential counts pseudoholomorphic disks of index 1 weighted by $\sigma(w_1)^{n_{w_1}(\phi)}\cdots \sigma(w_n)^{n_{w_n}(\phi)}$, where $\ws=\{w_1,\dots, w_n\}$.

We now describe a variation on the map $\Phi_{w_i}$. Suppose $\sigma,\sigma'\colon \ws\to R$ are colorings which differ only at a single basepoint $w_n\in \ws$. We place no restrictions on $\sigma(w_n)$ and $\sigma'(w_n)$. Write $X_1,\dots, X_{n-1}$ for $\sigma(w_1),\dots, \sigma(w_{n-1})$ respectively, and write $X_n$ and $Y_n$ for $\sigma(w_n)$ and $\sigma(w_n')$, respectively. In this situation, there is a map
\[
\Omega_{w_n}^{\sigma,\sigma'}\colon \CF^-(Y,\ws^\sigma)\to \CF^-(Y,\ws^{\sigma'})
\] 
defined by the equation
\[
\begin{split}
&\Omega_{w_n}^{\sigma,\sigma'}(\xs)\\
=&\sum_{\ys\in \bT_{\a}\cap \bT_{\b}} \sum_{\substack{\phi \in \pi_{2}(\xs,\ys)\\ \mu(\phi)=1}} \# (\cM(\phi)/\R) X_1^{n_{w_1}(\phi)}\cdots X_{n-1}^{n_{w_{n-1}}(\phi)} \frac{X_{n}^{n_{w_n}(\phi)}+Y_{n}{}^{n_{w_n}(\phi)}}{X_n+Y_{n}},
\end{split}
\]
extended $R$-equivariantly. This is the analog of a map appearing in \cite{MOIntegerSurgery}*{Section~13}.

\begin{lem}
\label{lem:Omega-chain-maps}
If $\sigma$ and $\sigma'$ are colorings which differ only at $w_n$, then the map $\Omega_{w_n}^{\sigma,\sigma'}$ is a chain map.
\end{lem}
\begin{proof}
The key algebraic relation is that if $a,b\ge 0$, then
\begin{equation}
\frac{X_n^{a+b}+Y_n^{a+b}}{X_n+Y_n}=X_n^a \frac{X_n^b+Y_n^b}{X_n+Y_n}+ \frac{X_n^a+Y_n^a}{X_n+Y_n} Y_n^b,
\label{eq:generalized-Leibniz-rule}
\end{equation}
which is straightforward to verify. To verify that $\Omega_{w_n}^{\sigma,\sigma'}$ is a chain map,
 we count the ends of moduli spaces of index 2 classes of flowlines weighted by the algebra elements
\[
X_1^{n_{w_1}(\phi)}\cdots X_{n-1}^{n_{w_{n-1}}(\phi)} \frac{X_{n}^{n_{w_n}(\phi)}+Y_{n}{}^{n_{w_n}(\phi)}}{X_n+Y_{n}}.
\]
Using Equation~\eqref{eq:generalized-Leibniz-rule}, we see that an index 2 class breaking into two index 1 classes corresponds to
\[
\d \circ \Omega_{w_n}^{\sigma,\sigma'}+\Omega_{w_n}^{\sigma,\sigma'}\circ \d.
\]
There are additionally boundary degenerations in the ends. For each $w_i$, there is a unique class of index 2 alpha degenerations covering $w_i$, and similarly a unique class of index 2 beta degenerations covering $w_i$. Modulo 2, each of these has a unique representative by \cite{OSLinks}*{Theorem~5.5}. Furthermore, the algebraic weight of each of these classes is $U_i$ if $i\neq n$, and $1$ if $i=n$. Therefore the total count of such boundary degenerations cancels, so we are left with $[\d, \Omega_{w_n}^{\sigma,\sigma'}]=0$. 
\end{proof}

In \cite{ZemGraphTQFT}*{Section~6}, there are \emph{free-stabilization and destabilization}  maps
\[
\begin{split}
S_{w}^+&\colon \CF^-(Y,\ws^\sigma)\to \CF^-(Y,(\ws\cup \{w\})^{\sigma'}) \quad \text{and} \\
S_w^- &\colon \CF^-(Y,(\ws\cup \{w\})^{\sigma'})\to \CF^-(Y,\ws^\sigma). 
\end{split}
\]
These are chain maps for any choices of colorings $\sigma$ and $\sigma'$ which coincide on $\ws$. 

The same argument as in \cite{ZemGraphTQFT}*{Lemma~14.14} shows that if $|\ws|>1$, then
\begin{equation}
\Omega_{w_n}^{\sigma,\sigma'}\simeq S^+_{w_n}\circ S_{w_n}^-. \label{eq:Omega=SwSw}
\end{equation}

Another relation which is quite useful is the ``trivial strand relation'' \cite{ZemGraphTQFT}*{Lemma~7.10}: if $\ws$ is a collection of basepoints, $w\not\in \ws$ and $\lambda$ is a path from $w$ to a basepoint $w'\in \ws$, then
\begin{equation}
S^-_{w} A_{\lambda} S_{w}^+\simeq \id,
\label{eq:trivial-strand-relation}
\end{equation}
as an endomorphism of $\CF^-(Y,\ws^{\sigma})$. Here, we color $w$ and $w'$ with the same color, so that $A_{\lambda}$ is a chain map. 

Additionally, suppose that $\sigma_0$ and $\sigma_1$ are colorings of $\ws$ which differ only at $w\in \ws$. Suppose that $w'\not \in \ws$, and that $\sigma_0'$ and $\sigma_1'$ are extensions of $\sigma_0,\sigma_1$ to $\ws \cup \{w'\}$. Then we have the relation
\begin{equation}
S_{w'}^{+}\circ \Omega_{w}^{\sigma_0,\sigma_1}\simeq \Omega_{w}^{\sigma_0',\sigma_1'}\circ S_{w'}^+
\label{eq:[Sw-Omega]}
\end{equation}
by the same argument as \cite{ZemGraphTQFT}*{Lemma~14.17}. A similar relation holds with $S_{w'}^-$ replacing $S_{w'}^+$.

Next, we consider arbitrary colorings $\sigma$ and $\sigma'$. Note that we can consider the ``identity'' map
\[
\bI^{\sigma,\sigma'}\colon \CF^-(Y,\ws^\sigma)\to \CF^-(Y,\ws^{\sigma'}),
\]
defined by 
\[
\bI(a\cdot \xs)=a\cdot \xs
\]
for any $a\in R$ and $\xs\in \bT_{\a}\cap \bT_{\b}$. Note that $\bI^{\sigma,\sigma'}$ is not a chain map. Instead, we have the following:

\begin{lem}\label{lem:del-identity} Suppose that $\sigma$ and $\sigma'$ are arbitrary colorings on $\ws$. Write $X_1,\dots, X_n$ for the colors of $\sigma$ and $Y_1,\dots, Y_n$ for the colors of $\sigma'$. Write $\sigma_i$ for the coloring which uses colors $Y_1,\dots, Y_i,X_{i+1},\dots, X_n$. Then
\[
\d(\bI^{\sigma,\sigma'})=\sum_{i=1}^n (X_i+Y_i)\cdot \bI^{\sigma_i,\sigma_n} \circ \Omega_{w_i}^{\sigma_{i-1},\sigma_i}\circ \bI^{\sigma_0,\sigma_{i-1}}.
\]
\end{lem}
\begin{proof} If $\sigma$ and $\sigma'$ differ at a single basepoint $w_i$, then we easily compute that
\[
\d(\bI^{\sigma,\sigma'})=(X_i+Y_i)\Omega_{w_i}^{\sigma,\sigma'}.
\]
For the general case, we write
\[
\bI^{\sigma,\sigma'}=\bI^{\sigma_{n-1},\sigma_n}\circ \cdots \circ \bI^{\sigma_0,\sigma_1}
\]
and apply the Leibniz rule for the morphism differential. This completes the proof.
\end{proof}

\begin{rem}
\label{rem:sum-is-chain-map} The following variation on Lemma~\ref{lem:del-identity} is also helpful. Suppose that $\ws=\ws_0\sqcup \ws_1$, and that $\sigma,\sigma'$ are two colorings on $\ws$. Suppose further that $\sigma|_{\ws_0}=\sigma'|_{\ws_0}$, and that $\sigma|_{\ws_1}\equiv X$ and $\sigma'|_{\ws_1}\equiv Y$, for some $X,Y\in R$. Write $\ws_1=\{w_1,\dots, w_n\}$. Then one can argue similarly to Lemma~\ref{lem:del-identity} that the map
\[
\sum_{i=1}^n  \bI^{\sigma_i, \sigma_n}\circ \Omega_{w_i}^{\sigma_{i-1}, \sigma_i} \circ \bI^{\sigma_0, \sigma_{i-1}}
\]
is a chain map. Here, $\sigma_i$ is the coloring which assigns $w_{i+1},\dots, w_n$ the color $X$, and $w_1,\dots, w_i$ the color $Y$.
\end{rem}

We recall from \cite{ZemGraphTQFT}*{Lemma~14.16} that if $\lambda$ is an arc from $w_i$ to $w_j$ (for $w_i\neq w_j$) and we set $U_i=U_j$, then
\[
A_\lambda\circ \Phi_{w_i}+\Phi_{w_i} \circ A_{\lambda}+\bI\simeq 0.
\]
When we do not give $w_i$ and $w_j$ the same color, we instead have the following:

\begin{lem}\label{lem:[A,Omega]} Suppose that $\d \lambda=\{w_i,w_n\}$, with $w_i \neq w_n$, and $\sigma$ and $\sigma'$ are colorings which differ only on $w_n$. Then
\[
A_{\lambda}^{\sigma'}\circ \Omega_{w_n}^{\sigma,\sigma'}+\Omega_{w_n}^{\sigma,\sigma'} \circ A_{\lambda}^{\sigma}+\bI^{\sigma,\sigma'}\simeq 0. 
\]
\end{lem}
\begin{proof}
The homotopy is constructed by counting moduli spaces of index 2 disks weighted by
\[
\#(\d_{\a}(\phi)\cap \lambda)X_1^{n_{w_1}(\phi)}\cdots X_{n-1}^{n_{w_n}(\phi)}\frac{X_n^{n_{w_n}(\phi)}+ Y_n^{n_{w_n}(\phi)}}{X_n+Y_n}
\]
We count ends of moduli spaces as in Lemma~\ref{lem:Omega-chain-maps}. The main difference occurs when counting boundary degenerations. From the above weights, all beta boundary degenerations are weighted by 0, and furthermore exactly one alpha boundary degeneration counts non-trivially (the one covering $w_n$). This boundary degeneration is weighted by 1, which contributes the term $\bI^{\sigma,\sigma'}$ to the relation.
\end{proof}

We now prove the following:

\begin{lem}\label{lem:complicated[A,Omega]}
 Let $\ws=\{w_1,w_2,w'\}$, and let $\sigma_0$, $\sigma_1$ and $\sigma_2$ be the colorings:
\begin{enumerate}
\item $\sigma_0(w_1)=X$, $\sigma_0(w')=X$, $\sigma_0(w_2)=Y$,
\item $\sigma_1(w_1)=Y$, $\sigma_1(w')=X$, $\sigma_1(w_2)=Y$,
\item $\sigma_2(w_1)=Y$, $\sigma_2(w')=X$, $\sigma_2(w_2)=X$.
\end{enumerate}
Suppose that $\lambda_2$ is a path from $w'$ to $w_2$.  Then
\begin{equation}
A_{\lambda_2}^{\sigma_2}(\Omega_{w_2}^{\sigma_1,\sigma_2}\bI^{\sigma_0,\sigma_1}+ \bI^{\sigma_1,\sigma_2}\Omega_{w_1}^{\sigma_0,\sigma_1})\simeq (\Omega_{w_2}^{\sigma_1,\sigma_2}\bI^{\sigma_0,\sigma_1}+ \bI^{\sigma_1,\sigma_2}\Omega_{w_1}^{\sigma_0,\sigma_1})A_{\lambda_2}^{\sigma_0}+\bI^{\sigma_0,\sigma_2}.
\label{eq:relation-A-lambda-Omega-Omega}
\end{equation}
(Note also that both sides are chain maps by Lemmas~\ref{lem:del-identity}, Remark~\ref{rem:sum-is-chain-map} and the Leibniz rule).
\end{lem}
\begin{proof}The homotopy counts holomorphic disks $\phi$ which are weighted by the quantity
\begin{equation}
\#(\d_{\a}(\phi)\cap \lambda_2) X^{n_{w'}(\phi)} \left(Y^{n_{w_1}(\phi)} \frac{X^{n_{w_2}(\phi)}+Y{}^{n_{w_2}(\phi)}}{X+Y} +\frac{X^{n_{w_1}(\phi)}+Y^{n_{w_1}(\phi)}}{X+Y}Y^{n_{w_2}(\phi)} \right).
\label{eq:homotopy-A-Omega-Omega}
\end{equation}
We now show that this gives a homotopy as stated; we observe firstly that it is equal to
\[
\#(\d_{\a}(\phi)\cap \lambda_2)X^{n_{w'}(\phi)}\frac{X^{n_{w_1}(\phi)}Y^{n_{w_2}(\phi)}+X^{n_{w_2}(\phi)} Y^{n_{w_1}(\phi)}}{X+Y}.
\]

We consider a class of index 2 disks $\psi$ which decomposes as a composition $\phi'*\phi$ of two index 1 disks. In this case, we compute that
\begin{equation}
\begin{split}
&\frac{X^{n_{w_1}(\phi'+\phi)}Y^{n_{w_2}(\phi'+\phi)}+X^{n_{w_2}(\phi'+\phi)}Y^{n_{w_1}(\phi'+\phi)}}{X+Y}\\
=&X^{n_{w_2}(\phi')}Y^{n_{w_1}(\phi')}Y^{n_{w_1}(\phi)} \frac{X^{n_{w_2}(\phi)}+Y^{n_{w_2}(\phi)}}{X+Y}
\\
&+X^{n_{w_2}(\phi')}Y^{n_{w_1}(\phi')} \frac{X^{n_{w_1}(\phi)}+Y^{n_{w_1}(\phi)}}{X+Y}Y^{n_{w_2}(\phi)}\\
&+Y^{n_{w_2}(\phi')} \frac{X^{n_{w_1}(\phi')}+Y^{n_{w_1}(\phi')}}{X+Y}X^{n_{w_1}(\phi)} Y^{n_{w_2}(\phi)}
\\
&+Y^{n_{w_1}(\phi')} \frac{X^{n_{w_2}(\phi')}+Y^{n_{w_2}(\phi')}}{X+Y} X^{n_{w_1}(\phi)}Y^{n_{w_2}(\phi)}.
\end{split}
\label{eq:sum-weights-A-Omega-rel}
\end{equation}
The relation above is useful in proving that the map $\Omega_{w_2}^{\sigma_1,\sigma_2}\bI^{\sigma_0,\sigma_1}+ \bI^{\sigma_1,\sigma_2}\Omega_{w_1}^{\sigma_0,\sigma_1}$ is a chain map. To get the chain homotopy described in Equation~\eqref{eq:relation-A-lambda-Omega-Omega}, we multiply Equation~\eqref{eq:sum-weights-A-Omega-rel} by 
\[
\# \d_{\a}(\phi'+\phi)\cap \lambda_2 = \# \d_{\a}(\phi')\cap \lambda_2+\# \d_{\a}(\phi)\cap \lambda_2
\]
 and collect terms. This gives all of the terms in Equation~\eqref{eq:relation-A-lambda-Omega-Omega} (as well as a homotopy) except for $\bI^{\sigma_0,\sigma_2}$. The term $\bI^{\sigma_0,\sigma_2}$ is contributed by boundary degenerations: For each $\xs\in \bT_{\a}\cap \bT_{\b}$ and each basepoint $w$ there are two classes of index 2 boundary degenerations which cover $w$, one alpha degeneration and one beta degeneration. The count of representatives of each class is 1 modulo 2 by \cite{OSLinks}*{Theorem~5.5}. Only the alpha degeneration covering $w_2$ has non-zero  weight in Equation~\eqref{eq:homotopy-A-Omega-Omega}, and the weight of this degeneration is 1. Therefore Equation~\eqref{eq:relation-A-lambda-Omega-Omega} follows.
\end{proof}

A very similar, but slightly simpler, argument shows the following:

\begin{lem}
\label{lem:[Omega,A]-disjoint}
 Suppose that $\sigma$ and $\sigma'$ are colorings of $\ws$ which differ only on a basepoint $w\in \ws$, and that $\lambda$ is an arc which connects two basepoints $w_1,w_2\in \ws$, and $w\neq w_1,w_2$. Then
\[
\Omega_{w}^{\sigma,\sigma'}\circ A_{\lambda}^{\sigma}\simeq A_{\lambda}^{\sigma'}\circ \Omega_{w}^{\sigma,\sigma'}. 
\] 
\end{lem}
\begin{proof} The chain homotopy counts index 1 disks weighted by 
\[
(\# \d_{\a}(\phi) \cap \lambda)\bigg(\prod_{\substack{w_i\in \ws\\ w_i\neq w}} X_{i}^{n_{w_i}(\phi)}\bigg) \frac{X^{n_{w}(\phi)} +Y^{n_{w}(\phi)}}{X+Y},
\]
where $X$ and $Y$ are the variables for $w$ in $\sigma$ and $\sigma'$, respectively. From here, the argument proceeds as in Lemma~\ref{lem:[A,Omega]}. 
\end{proof}

We are now set up to state and prove our computation of the map $(\Sw_\lambda)_*$. To simplify the notation, we focus on the case that $\ws=\{w_1,w_2\}$, though our computation of $\Sw_\lambda$ extends to the case when there are extra basepoints, not moved by the swapping diffeomorphism, with only minor change. Note that there are two equivalent ways to view $\Sw_{\lambda}$. The first is a map of $\bF[U_1,U_2]$-modules which interchanges the actions of $U_1$ and $U_2$. The second way is to replace $U_1$ and $U_2$ by two variables $X$ and $Y$, and consider $\Sw_{\lambda}$ as an $\bF[X,Y]$-equivariant map from $\CF^-(Y,\ws^\sigma)$ to $\CF^-(Y,\ws^{\sigma'})$ for two different colorings, $\sigma$ and $\sigma'$, of $\ws$. Although the former approach is more natural aesthetically, the latter is more convenient for the proof, so we adopt this approach. Henceforth, we work with colorings mapping to the ring $R=\bF[X,Y]$.  We now prove an algebraic formula for the basepoint swapping map:

\begin{prop}\label{prop:compute-Sw-map} Let $\ws=\{w_1,w_2\}$. Let $\sigma$, $\sigma'$ and $\sigma''$ be the colorings:
\begin{enumerate}
\item $\sigma(w_1)=X$, $\sigma(w_2)=Y$.
\item $\sigma'(w_1)=Y$, $\sigma'(w_2)=Y$.
\item $\sigma''(w_1)=Y$, $\sigma''(w_2)=X$.
\end{enumerate}
The map
\[
(\Sw_{\lambda})_*\colon \CF^-(Y,\ws^\sigma)\to \CF^-(Y,\ws^{\sigma''})
\] 
satisfies
\[
(\Sw_{\lambda})_*\simeq \bI^{\sigma,\sigma''}+(\bI^{\sigma',\sigma''}\circ \Omega_{w_1}^{\sigma,\sigma'}+\Omega_{w_2}^{\sigma',\sigma''}\circ \bI^{\sigma,\sigma'})\circ A_{\lambda}^\sigma.
\]
\end{prop}

\begin{rem} Before embarking on the proof, we point out some basic algebraic subtleties. Care must be taken when working with compositions of maps which are not all chain maps. For example if $f\simeq f'$ and $g\simeq g'$, but $f$ and $g$ are not chain maps, it might not be the case that $f\circ g\simeq f'\circ g'$. Instead, we must apply the Leibniz rule for compositions. One special case that is important for us is the following. Suppose that $f_n,\dots, f_{i+1},f_{i-1},\dots, f_1$ are chain maps, and that $f_i$ is a map (possibly not a chain map). Suppose that $f_i\simeq f_i'$ for some map $f_i'$. Then there is a chain homotopy
\[
f_n\circ \cdots \circ f_i\circ \cdots\circ f_1\simeq f_n\circ \cdots \circ f_i'\circ \cdots\circ f_1.
\]
\end{rem}

\begin{proof}[Proof of Proposition~\ref{prop:compute-Sw-map}] Following \cite{ZemGraphTQFT}*{Proposition~14.24}, we add an extra basepoint $w'$. We also pick arcs $\lambda_1$ from $w_1$ to $w'$ and $\lambda_2$ from $w'$ to $w_2$ so that
\[
\lambda=\lambda_2* \lambda_1.
\]
See Figure~\ref{fig:naturality-32}.
\begin{figure}[h]
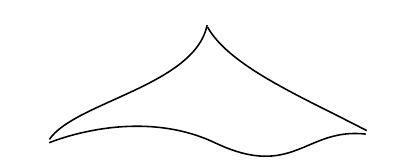
\caption{The arcs $\lambda_1$ and $\lambda_2$. }
\label{fig:naturality-32}
\end{figure}

 Using \cite{ZemGraphTQFT}*{Theorem~14.11} we can factor $(\Sw_{\lambda})_*$ as
\[
(S_{w'}^- A_{\lambda_2} S_{w_2}^+) (S_{w_2}^- A_{\lambda} S_{w_1}^+) (S_{w_1}^- A_{\lambda_1} S_{w'}^+).
\]
To make some manipulations easier to follow, we record the colorings in the superscripts. We write $\sigma_0$, $\sigma_1$ and $\sigma_2$ for the three colorings of $\{w_1,w_2,w'\}$ below:
\begin{enumerate}
\item $\sigma_0(w)=X$, $\sigma_0(w')=X$, $\sigma_0(w_2)=Y$,
\item $\sigma_1(w)=Y$, $\sigma_1(w')=X$, $\sigma_1(w_2)=Y$,
\item $\sigma_2(w)=Y$, $\sigma_2(w')=X$, $\sigma_2(w_2)=X$.
\end{enumerate}
Using the relation in Equation~\eqref{eq:Omega=SwSw}, we rewrite the above expression for $(\Sw_{\lambda})_{*}$ as
\[
S_{w'}^- A_{\lambda_2}^{\sigma_2} \Omega_{w_2}^{\sigma_1,\sigma_2} A_{\lambda}^{\sigma_1} \Omega_{w_1}^{\sigma_0,\sigma_1} A_{\lambda_1}^{\sigma_0} S_{w'}^+
\]

We write $A_{\lambda}^{\sigma_1}=A_{\lambda_1}^{\sigma_1}+A_{\lambda_2}^{\sigma_1}$ and the above becomes
\[
S_{w'}^- A_{\lambda_2}^{\sigma_2} \Omega_{w_2}^{\sigma_1,\sigma_2} (A_{\lambda_1}^{\sigma_1}+A_{\lambda_2}^{\sigma_1}) \Omega_{w_1}^{\sigma_0,\sigma_1} A_{\lambda_1}^{\sigma_0} S_{w'}^+.
\]
We write this as
\[
S_{w'}^- A_{\lambda_2}^{\sigma_2} \Omega_{w_2}^{\sigma_1,\sigma_2} A_{\lambda_1}^{\sigma_1} \Omega_{w_1}^{\sigma_0,\sigma_1} A_{\lambda_1}^{\sigma_0} S_{w'}^+
+S_{w'}^- A_{\lambda_2}^{\sigma_2} \Omega_{w_2}^{\sigma_1,\sigma_2} A_{\lambda_2}^{\sigma_1} \Omega_{w_1}^{\sigma_0,\sigma_1} A_{\lambda_1}^{\sigma_0} S_{w'}^+
\]
Using Lemmas~\ref{lem:[A,Omega]} and~\ref{lem:[Omega,A]-disjoint}, we can move the $A_{\lambda_1}^{\sigma_1}$ in the left summand rightward within its term, and we can move the $A_{\lambda_2}^{\sigma_2}$ in the right summand leftward within its term. We obtain that the above map is chain homotopic to
\[
\begin{split}
&S_{w'}^- (A_{\lambda_2}^{\sigma_2})^2 \Omega_{w_2}^{\sigma_1,\sigma_2} \Omega_{w_1}^{\sigma_0,\sigma_1} A_{\lambda_1}^{\sigma_0} S_{w'}^+\\
+&S_{w'}^- A_{\lambda_2}^{\sigma_2} \Omega_{w_2}^{\sigma_1,\sigma_2} \Omega_{w_1}^{\sigma_0,\sigma_1} (A_{\lambda_1}^{\sigma_0})^2 S_{w'}^+\\
+&S_{w'}^- A_{\lambda_2}^{\sigma_2} (\Omega_{w_2}^{\sigma_1,\sigma_2}\bI^{\sigma_0,\sigma_1}+ \bI^{\sigma_1,\sigma_2}\Omega_{w_1}^{\sigma_0,\sigma_1}) A_{\lambda_1}^{\sigma_0} S_{w'}^+
\end{split}
\]
By Equation~\eqref{eq:basic-prop-A_lambda}, the above is homotopic to
\[
\begin{split}
&X S_{w'}^- \Omega_{w_2}^{\sigma_1,\sigma_2} \Omega_{w_1}^{\sigma_0,\sigma_1} A_{\lambda_1}^{\sigma_0} S_{w'}^+\\
+&X S_{w'}^- A_{\lambda_2}^{\sigma_2} \Omega_{w_2}^{\sigma_1,\sigma_2} \Omega_{w_1}^{\sigma_0,\sigma_1} S_{w'}^+\\
+&S_{w'}^- A_{\lambda_2}^{\sigma_2} (\Omega_{w_2}^{\sigma_1,\sigma_2}\bI^{\sigma_0,\sigma_1}+ \bI^{\sigma_1,\sigma_2}\Omega_{w_1}^{\sigma_0,\sigma_1}) A_{\lambda_1}^{\sigma_0} S_{w'}^+
\end{split}
\]
We observe that
\begin{equation}
XS_{w'}^- \Omega_{w_2}^{\sigma_1,\sigma_2} \Omega_{w_1}^{\sigma_0,\sigma_1} A_{\lambda_1}^{\sigma_0} S_{w'}^+ +X S_{w'}^- A_{\lambda_2}^{\sigma_2} \Omega_{w_2}^{\sigma_1,\sigma_2} \Omega_{w_1}^{\sigma_0,\sigma_1} S_{w'}^+\simeq 0 \label{eq:cancelling-terms}
\end{equation}
because 
\[
S_{w'}^- \Omega_{w_2}^{\sigma_1,\sigma_2} \Omega_{w_1}^{\sigma_0,\sigma_1} A_{\lambda_1}^{\sigma_0} S_{w'}^+\simeq \Omega_{w_2}^{\sigma',\sigma''} \Omega_{w_1}^{\sigma,\sigma'} S_{w'}^- A_{\lambda_1}^{\sigma_0}S_{w'}^+\simeq  \Omega_{w_2}^{\sigma',\sigma''} \Omega_{w_1}^{\sigma,\sigma'}
\]
by Equations~\eqref{eq:trivial-strand-relation} and \eqref{eq:[Sw-Omega]}, and similarly
\[
 S_{w'}^- A_{\lambda_2}^{\sigma_2} \Omega_{w_2}^{\sigma_1,\sigma_2} \Omega_{w_1}^{\sigma_0,\sigma_1} S_{w'}^+\simeq  \Omega_{w_2}^{\sigma',\sigma''} \Omega_{w_1}^{\sigma,\sigma'}.
\]
These two terms cancel, giving Equation~\eqref{eq:cancelling-terms}.

We will perform the following manipulation: 
\begin{equation}
\begin{split}
&S_{w'}^- A_{\lambda_2}^{\sigma_2} (\Omega_{w_2}^{\sigma_1,\sigma_2}\bI^{\sigma_0,\sigma_1}+ \bI^{\sigma_1,\sigma_2}\Omega_{w_1}^{\sigma_0,\sigma_1}) A_{\lambda_1}^{\sigma_0} S_{w'}^+
\\
\simeq &S_{w'}^-  \left((\Omega_{w_2}^{\sigma_1,\sigma_2}\bI^{\sigma_0,\sigma_1}+ \bI^{\sigma_1,\sigma_2}\Omega_{w_1}^{\sigma_0,\sigma_1}) A_{\lambda_2}^{\sigma_0}A_{\lambda_1}^{\sigma_0}+ \bI^{\sigma_0,\sigma_2}A_{\lambda_1}^{\sigma_0}\right) S_{w'}^+
\\
\simeq & S_{w'}^-  \left((\Omega_{w_2}^{\sigma_1,\sigma_2}\bI^{\sigma_0,\sigma_1}+ \bI^{\sigma_1,\sigma_2}\Omega_{w_1}^{\sigma_0,\sigma_1}) A_{\lambda}^{\sigma_0}A_{\lambda_1}^{\sigma_0}+ \bI^{\sigma_0,\sigma_2}A_{\lambda_1}^{\sigma_0}\right) S_{w'}^+\\
= & S_{w'}^-  \left((\Omega_{w_2}^{\sigma_1,\sigma_2}\bI^{\sigma_0,\sigma_1}+ \bI^{\sigma_1,\sigma_2}\Omega_{w_1}^{\sigma_0,\sigma_1}) A_{\lambda}^{\sigma_0}+ \bI^{\sigma_0,\sigma_2}\right)A_{\lambda_1}^{\sigma_0} S_{w'}^+\\
\simeq &((\Omega_{w_2}^{\sigma',\sigma''} \bI^{\sigma,\sigma'}+\bI^{\sigma', \sigma''} \Omega_{w_1}^{\sigma,\sigma'})A_\lambda^{\sigma}+\bI^{\sigma,\sigma''}) S_{w'}^-A_{\lambda_1}^{\sigma_0} S_{w'}^+\\
\simeq  &(\Omega_{w_2}^{\sigma',\sigma''} \bI^{\sigma,\sigma'}+\bI^{\sigma', \sigma''} \Omega_{w_1}^{\sigma,\sigma'})A_\lambda^{\sigma}+\bI^{\sigma,\sigma''}.
\end{split}
\label{eq:final-basepoint-swapping-equation}
\end{equation}

We now explain the above manipulation. Going from the first line to the second is Lemma~\ref{lem:complicated[A,Omega]}. Going from the second to the third is accomplished by adding
\[
\begin{split}
&S_{w'}^- (\Omega_{w_2}^{\sigma_1,\sigma_2}\bI^{\sigma_0,\sigma_1}+ \bI^{\sigma_1,\sigma_2}\Omega_{w_1}^{\sigma_0,\sigma_1}) A_{\lambda_1}^{\sigma_0}A_{\lambda_1}^{\sigma_0} S_{w'}^+\\
\simeq &(\Omega_{w_2}^{\sigma',\sigma''}\bI^{\sigma,\sigma'}+ \bI^{\sigma',\sigma''}\Omega_{w_1}^{\sigma,\sigma'}) S_{w'}^-  (A_{\lambda_1}^{\sigma_0})^2 S_{w'}^+\\
\simeq& X (\Omega_{w_2}^{\sigma',\sigma''}\bI^{\sigma,\sigma'}+ \bI^{\sigma',\sigma''}\Omega_{w_1}^{\sigma,\sigma'}) S_{w'}^-  S_{w'}^+\\
\simeq& 0.
\end{split}
\]
Returning to Equation~\eqref{eq:final-basepoint-swapping-equation}, we observe line four is obtained by factoring line 3. Line 5 is obtained by commuting $S_{w'}^-$ past the other terms. The final line follows from Equation~\eqref{eq:trivial-strand-relation}.  Having established Equation~\eqref{eq:final-basepoint-swapping-equation}, the proof is complete.
\end{proof}

Finally, we complete the proof of Theorem~\ref{thm:swap-diffeo-on-homology}.

\begin{proof}[Proof of Theorem~\ref{thm:swap-diffeo-on-homology}]
The statement of Theorem~\ref{thm:swap-diffeo-on-homology} concerns a map with no colorings, whereas most of the above results consider basepoints with colorings. To translate between the statements, we will now let $X=U_1$ and $Y=U_2$. We note that the map in the statement is the composition
\[
\begin{tikzcd}
\CF^-(\cH,\ws^{\sigma})\ar[r, "(\Sw_\lambda)_*"] &\CF^-(\cH,\ws^{\sigma'}) \ar[r, "\scT"]& \CF^-(\cH,\ws^{\sigma})
\end{tikzcd}
\]
where $\scT$ is the map which sends $\xs U_1^i U_2^j$ to $\xs U_1^j U_2^i$. Note that $\CF^-(\cH,\ws^\sigma)$ is the same as the complex for which we would normally write $\CF^-(\cH,\ws^\sigma)$. 

  By Proposition~\ref{prop:compute-Sw-map}, the map labeled $(\Sw_{\lambda})_*$ in the above composition is homotopic to
\[
(\Sw_\lambda)_*\simeq  \bI^{\sigma,\sigma''}+(\bI^{\sigma',\sigma''}\circ \Omega_{w_1}^{\sigma,\sigma'}+\Omega_{w_2}^{\sigma',\sigma''}\circ \bI^{\sigma,\sigma'})\circ A_{\lambda}^\sigma.
\]
We consider the map. Our goal is to show that $\scT\circ (\Sw_{\lambda})_*$ is chain homotopic to the identity map. To this end, note that $\scT\circ \bI^{\sigma,\sigma''}$ is just the map which switches $U_1$ and $U_2$, but is the identity on intersection points. Therefore we compute that
\[
(\scT\circ \bI^{\sigma,\sigma''}+\id)(U_1^i U_2^j \xs)=(U_1^i U_2^j+U_1^jU_2^i )\xs=\left(U_1^i \frac{U_1^j+U_2^j}{U_1+U_2}+ U_1^j\frac{U_1^i+U_2^i}{U_1+U_2} \right)(U_1+U_2)\xs.
\]
We define a (non-$\bF[U_1,U_2]$-equivariant) map
\[
j\colon \CF^-(\cH,\ws^\sigma)\to \CF^-(\cH,\ws^\sigma)
\]
via the formula
\[
j(U_1^iU_2^j\xs)=\left(U_1^i \frac{U_1^j+U_2^j}{U_1+U_2}+ U_1^j\frac{U_1^i+U_2^i}{U_1+U_2} \right)\xs.
\]
Note that $j$ satisfies
\begin{equation}
(U_1+U_2)\circ j=j\circ (U_1+U_2)=(\scT\circ \bI^{\sigma,\sigma''}+\id).
\end{equation}
In the above equation, we are writing $(U_1+U_2)$ for the endomorphism $\xs\mapsto (U_1+U_2)\cdot \xs$.

We define a chain homotopy $h\colon \CF^-(\cH,\ws)\to \CF^-(\cH,\ws)$ via the formula
\[
h=j\circ A_{\lambda}^{\sigma}.
\]
We see by direct computation that
\[
\d(j)=\scT\circ(\bI^{\sigma',\sigma''}\circ \Omega_{w_1}^{\sigma,\sigma'}+\Omega_{w_2}^{\sigma',\sigma''}\circ \bI^{\sigma,\sigma'}).
\]
Therefore, applying the Leibniz rule and using the above relations, we get
\[
\begin{split} \d(h)&=j\circ \d(A_\lambda^\sigma)+\d(j)\circ A_\lambda^\sigma\\
&= j\circ (U_1+U_2)+\scT\circ(\bI^{\sigma',\sigma''}\circ \Omega_{w_1}^{\sigma,\sigma'}+\Omega_{w_2}^{\sigma',\sigma''}\circ \bI^{\sigma,\sigma'})\circ A_{\lambda}^\sigma\\
&=(\scT\circ \bI^{\sigma,\sigma''}+\id)+\scT\circ(\bI^{\sigma',\sigma''}\circ \Omega_{w_1}^{\sigma,\sigma'}+\Omega_{w_2}^{\sigma',\sigma''}\circ \bI^{\sigma,\sigma'})\circ A_{\lambda}^\sigma\\
&\simeq \id+\scT\circ (\Sw_\lambda)_*,
\end{split}
\]
completing the proof.
\end{proof}

\subsection{Proof of Theorem~\ref{thm:multi-pointed-diffeos}}

We now complete the proof of Theorem~\ref{thm:multi-pointed-diffeos}:

\begin{proof}[Proof of Theorem~\ref{thm:multi-pointed-diffeos}] By assumption, $\phi_0$ and $\phi_1$ are isotopic as unpointed diffeomorphisms. After an isotopy, we may assume that $\phi_0$ and $\phi_1$ both preserve $\ws_0\cup \ws_1$ setwise.  Note that it suffices to prove the case when $\ws_0\subset \ws_1$, since we can apply the result in this case repeatedly to prove the general case. Therefore, we assume that $\ws_1=\ws_0\sqcup \ws'$ for some collection $\ws'$ of basepoints.

Since a pair of diffeomorphisms which are isotopic relative to $\ws_0$ induce the same map on $\HF^-(Y,\ws_0)$, we may assume that $\phi_0$ preserves $\ws'$ pointwise. We first define an isomorphism
\[
S\colon \HF^-(Y,\ws_0)\to \HF^-(Y,\ws_0\cup \ws')
\]
as the composition of the free-stabilization maps $S_{w_i}^+$ ranging over $w_i\in \ws'$. Using \cite{OSLinks}*{Proposition~6.5}, one can easily show that $S$ is an isomorphism. Indeed, they show that $(Y,\ws_0\cup \{w_i\})$ admits a Heegaard diagram where
\[
\CF^-(Y,\ws_0\cup \{w_i\})\iso\Cone(\CF^-(Y,\ws_0)[U_i]\times \theta^-\xrightarrow{U_1+U_i} \CF^-(Y,\ws_0)\times \theta^+),
\]
where $U_i$ is the variable for $w_i$, $U_1$ is the variable for some basepoint in $\ws_0$. Here, $\theta^+$ and $\theta^-$ are just notation. The Heegaard diagram that Ozsv\'{a}th and Szab\'{o} consider can also be used to compute the free-stabilization map $S_{w_i}^+$, which takes the form $S_{w_i}^+(\xs)=\xs\times \theta^+$ (extended equivariantly). This map is obviously a quasi-isomorphism.

Since $\phi_0$ preserves $\ws'$ pointwise, we may apply Lemma~\ref{lem:stabilization-diffeo} to see that
\[
S\circ \HF(\phi_0)=\HF(\phi_0)\circ S.
\]

Given the above discussion, we may now view $\phi_0$ and $\phi_1$ as both acting on the pointed manifold $(Y,\ws_0\cup \ws')$. Since $\phi_0$ and $\phi_1$ are isotopic as unpointed diffeomorphisms, it follows that $\phi_1$ is isotopic to $R\circ \phi_0$, where $R$ is a composition of basepoint swapping maps. Theorem~\ref{thm:swap-diffeo-on-homology} implies that $\HF(R\circ \phi_0)=\HF(\phi_0)$ as endormophisms of $\HF^-(Y,\ws_0\cup \ws')$. This concludes the proof, since we have shown that
\[
S\circ \HF(\phi_0)=\HF(\phi_0)\circ S=\HF(\phi_1)\circ S.
\]
\end{proof}

\begin{cor} Suppose that $\phi\colon (Y,L,\Lambda)\to (Y,L,\Lambda)$ is a diffeomorphism of strongly framed links. Write $\Phi\colon Y_{\Lambda}(L)\to Y_{\Lambda}(L)$ for the corresponding diffeomorphism.  Write $\cC(\phi)\colon \cC_{\Lambda}(L)\to \cC_{\Lambda}(L)$ for the associated chain endomorphism of the link surgery complex, and write $\ve{\HF}(\Phi)\colon \ve{\HF}^-(Y_{\Lambda}(L))\to \ve{\HF}^-(Y_{\Lambda}(L))$ for the map induced by the diffeomorphism $\Phi$ on the singly pointed version of Heegaard Floer homology. Then there is an isomorphism of pairs
 \[
 (H_* \cC_{\Lambda}(L), \cC(\phi))\iso (\ve{\HF}^-(Y_{\Lambda}(L)), \ve{\HF}(\Phi)).
 \]
\end{cor}

\section{Homology actions, twisted coefficients and $\Spin$-bordism }
\label{sec:homology-action-homotopies}

If $\g\subset \Sigma$ is a closed curve, there is a map
\[
\frA_{\g}\colon \CF^-(\cA,\cB)\to \CF^-(\cA,\cB),
\]
which we will informally refer to as the \emph{hypercube homology action}. This map was first defined in \cite{HHSZNaturality}*{Section~7.2}. 

In this section, we study the naturality of the hypercube homology action. Specifically, we study its behavior under changing $\g$ on $\Sigma$.

Before we begin, we note that constructing the endomorphism $\frA_{\g}$ is related to constructions of twisted coefficients on the hypercube chain complex $\CF^-(\cA,\cB)$. The work of this section will establish several naturality results for certain twisted Floer homologies.

Our main use for these results is to define canonical transition maps on the link surgery formula for changing the shadows on a Heegaard surface.

\subsection{The hypercube homology action}
\label{sec:closed-curves-action}

Let $\cA$ and $\cB$ two hypercubes of attaching curves on $\Sigma$. We begin by defining the map 
\[
\frA_{\g}\colon \CF^-(\cA,\cB)\to \CF^-(\cA,\cB)
\]
 when $\g$ is a smooth 1-cycle on $\Sigma$. The construction is described in \cite{HHSZNaturality}*{Section~7.2}. Let 
 \[
 \phi\in \pi_2(\theta_{\a_{\veps_n}, \a_{\veps_{n-1}}},\dots, \xs,\dots, \theta_{\b_{\nu_{k-1}}, \b_{\nu_{k}}}, \ys)
 \]
be a class of polygons. The domain of $D(\phi)$ is an integral 2-chain on $\Sigma$. The boundary of $D(\phi)$ may naturally be written as the union of two 1-chains, $\d_{\a}(\phi)$ and $\d_{\b}(\phi)$. We orient these using the convention that
\[
\d D(\phi)=\d_{\b}(\phi)+\d_{\a}(\phi)
\]
and 
\[
\d \d_{\b}(\phi)=-\d \d_{\a}(\phi)=\xs-\ys.
\] 
We set  $\#\d_{\a}(\phi)\cap \g$ to be the  intersection number of $\g$ with the components of $\d D(\phi)$ which lie in a curve of $\cA$. We define $\#\d_{\b}(\phi)\cap \g$ similarly.
 The hypercube morphism $\frA_{\g}$ is gotten by summing over holomorphic polygons which are weighted by $\# \d_{\a}(\phi) \cap \g$. Similarly the morphism $\frB_{\g}$ is obtained by summing over holomorphic polygons which are weighted by $\# \d_{\b}(\phi)\cap \g$. 
 
It is helpful to interpret the above slightly differently. We can think of $\g$ as determining a formal endomorphism of each $\as_{\veps}$. We will write $F_{\g}$ for this formal endomorphism. Summing over all $\veps$, we also obtain a formal endomorphism $F_{\g}\colon \cL_{\a}\to \cL_{\a}$. Then the morphism $\frA_{\g}$ can be described formally as
\[
\frA_{\g}=\mu_2^{\Tw}(F_{\g},-).
\] 
To make sense of the above equation, we need to have an interpretation of 
\[
\mu_{n+1}(\xs_1,\dots, \xs_i,\g, \xs_{i+1},\dots, \xs_n)
\]
 when $\xs_i$ are ordinary Floer chains. If $(\Sigma,\ds_0,\dots, \ds_n)$ is a Heegaard multi-diagram, then we define $\mu_{n+1}(\xs_1,\dots, \xs_{i}, \g, \xs_{i+1}, \dots, \xs_n)$ as counting holomorphic polygons with inputs $(\xs_1,\dots, \xs_n)$ of index $3-n$ which are weighted by $\#(\d_{\dt_i}(\phi)\cap \g)$.

\begin{rem} It is clear from the definition that $\frA_\g=\frB_\g$ since $\#\d_{\a}(\phi)\cap \g= \# \d_{\b}(\phi) \cap \g$ (with mod 2 coefficients). 
\end{rem}

\begin{prop}\label{prop:homology-actions-chain-maps}
 If $\cA$ and $\cB$ are hypercube of Lagrangians, then 
\[
\frA_\g \colon \CF^-(\cA,\cB)\to \CF^-(\cA,\cB)
\]
is a chain map. 
\end{prop}

The proof, which is laid out in \cite{HHSZNaturality}*{Section~7.2}, follows from a straightforward count of the ends of moduli spaces.

\subsection{Hypercubes with twisted coefficients}
\label{sec:twisted-coefficients-def}
As is standard, the construction of the homology action $\frA_\g$ from the previous section adapts to give a version of Floer homology with twisted coefficients. If we are given a class $\g\in H_1(Y)$, we may view $\bF[T,T^{-1}]$ as an $\bF[H^1(Y)]$-module, by defining 
\[
e^{\eta}\cdot T^\a=T^{\a+\eta(\g)}. 
\]

Suppose $\cA,\cB$ are hypercubes of attaching curves on $\Sigma$ and $\g\subset \Sigma$ is a closed 1-chain, transverse to the curves of $\cA$ and $\cB$, and $\g$ is disjoint from the intersection points of $\cA$ and $\cB$. We may form a twisted Floer complex $\underline{\CF}^-_\g(\cA,\cB)$ as follows. The underlying vector space is the same as $\CF^-(\cA,\cB)\otimes_{\bF} \bF[T,T^{-1}]$. The differential counts the same curves as $\CF^-(\cA,\cB)$, except with an extra power of
\[
T^{-\# \d_{\a}(\psi)\cap \g}=T^{\# \d_{\b}(\psi)\cap \g}.
\]

The discussion extends in a straightforward manner to the case that we have a collection of 1-cycles $\ve{\g}=(\g_1,\dots, \g_n)$ on $\Sigma$. In this case, the underlying vector space $\underline{\CF}^-_{\ve{\g}}(\cA,\cB)$ is 
\[
\CF^-(\cA,\cB)\otimes_{\bF} \bF[T_1,T_1^{-1},\dots T_n,T_n^{-1}].
\]
The differential is again similar to that of $\CF^-(\cA,\cB)$, except with an extra weight of
\[
T_1^{\# \d_{\b}\cap \g_1}\cdots T_n^{\# \d_{\b}\cap \g_n}.
\]

\subsection{A relative Chern class formula}\label{subsec:relative-chern}

We now describe a helpful Chern class formula, which will be needed for proving certain naturality results for Floer complexes with twisted coefficients.

Let $(\Sigma,\as,\bs,\ws)$ be a Heegaard diagram for $Y$. We suppose that $\g\subset \Sigma$ is an immersed, oriented 1-manifold. We view $\g$ as a map from a disjoint union of circles to $\Sigma$. The bundle $\g^*TY$ can be canonically trivialized, by declaring, for example, that $t=(v_1,v_2,v_3)$ is a trivialization where $v_1$ is an oriented tangent for $\g$, $v_2\in T\Sigma$ is the oriented normal to $v_1$, and $v_3$ is normal to $\Sigma$, compatible with the orientation of $Y$.

We now suppose that
\[
i\colon \scC\to Y
\]
is a smooth map and $\g=i|_{\d \scC}$, where we still assume that the image of $\g$ lies in $\Sigma$. As above, the bundle $i^* TY|_{\d \scC}$ admits a canonical trivialization, for which we write $t$.

We assume that $\g$ is disjoint from the intersection points of $\as$ and $\bs$, and also disjoint from the basepoints $\ws$. Given $\xs\in \bT_{\a}\cap \bT_{\b}$, the construction from \cite{OSDisks}*{Section~2.6}, constructs a non-vanishing vector field $v_{\ws}(\xs)$ on $Y$. By construction $v_{\ws}(\xs)$ is normal to $\Sigma$ along $\g$. Taking the orthogonal complement of $v_{\ws}(\xs)$ gives an oriented 2-plane field $\xi_{\ws}(\xs)$ on $Y$, and the trivialization $t$ of $TY|_{\g}$ naturally restricts to a trivialization of $\xi_{\ws}(\xs)$ along $\g$, for which we also write $t$.

Note that an almost complex structure on $\xi_{\ws}(\xs)$ induces a relative $\Spin^c$ structure on the bundle $\xi_{\ws}(\xs)$ via the map $U(1)\to \Spin^c(2)$, and consequently a relative $\Spin^c$ structure $\frs_{\ws}(\xs,t)$ on $i^* TY\iso i^* (\xi_{\ws}(\xs)\oplus \underline{\R})$, where $\underline{\R}$ is the trivial bundle. The relative Chern classes are identified, so
\[
c_1(\frs_{\ws}(\xs,t))=c_1(\xi_{\ws}(\xs),t).
\]

We are now interested in the case that $\scC$ is induced by a relative 2-chain $C$ on  $(\Sigma,\as,\bs,\ws)$. In this case, we assume that $C$ is an integral domain on $\Sigma$ which has boundary 
\[
\d C=\g+A+B,
\]
where $A$ is an integral combination of curves in $\as$, and $B$ is an integral combination of curves in $\bs$. In this case, $C$ can be canonically capped off with disks along the curves of $A$ and $B$ to get a surface $\hat{C}$. Smoothing corners gives a smooth map
\[
i\colon \hat{C}\to Y.
\]

\begin{lem}
\label{lem:relative-Chern-class}
If $\g$, $C$ and $\hat{C}$ are as above, then
\[
\langle c_1(\frs_{\ws}(\xs,t)), [\hat{C}, \d \hat{C}]\rangle =e(C)+2\bar{n}_{\xs}(C)-2n_{\ws}(C). 
\]
Here $e(C)$ denotes the Euler measure of $C$, and $\bar{n}_{\xs}(C)=\sum_{x_i\in \xs} \bar{n}_{x_i}(C)$ where $\bar{n}_{x_i}(C)$ is the average multiplicity in the four regions adjacent to $x_i$. 
\end{lem}
\begin{proof} The proof is essentially the same as Ozsv\'{a}th and Szab\'{o}'s proof in the case that $\g$ is empty \cite{OSProperties}*{Proposition~7.5}. Away from small neighborhoods of $\xs$ and $\ws$, the 2-plane field $\xi_{\ws}(\xs)$ coincides with $T\Sigma$. We excise small disks from $C$ near these points. As in the proof of Gauss-Bonnet formula, we can compute the Chern class by picking a real valued connection on $\xi_{\ws}(\xs)$ which is compatible with some Riemannian metric, and then integrating the curvature 2-form over the surface. This induces a complex valued connection on $\xi_{\ws}(\xs)$. Away from $\xs$ and $\ws$, the curvature of $\xi_{\ws}(\xs)$ coincides up to an overall constant with the Gaussian curvature of $\Sigma$ with the chosen metric. Integrating as in the proof of the Gauss-Bonnet theorem (cf. \cite{MilnorStasheff}*{Appendix~C}), and applying a simple model computation in a neighborhood of $\xs$ and $\ws$, as in the proof of \cite{OSProperties}*{Proposition~7.5}, yields the stated formula.
\end{proof}

We now consider immersed 1-chains $\g\subset Y$ equipped with trivializations $t$ of $TY|_{\g}$. Recall that a $\Spin$ structure on a real vector bundle $\xi\to X$ is equivalent to a homotopy class of trivialization on the 1-skeleton which extends to the 2-skeleton, as long as the fiber dimension is 3 or greater \cite{milnor-spin}. Therefore, we note that a trivialization on $TY|_\g$ is equivalent to a $\Spin$ structure on $TY|_{\g}$, since $\g$ is has no cells of dimension 2 or greater. 

Suppose that $\g$ is an immersed 1-cycle in $Y$, and $t$ is a trivialization of $\g^* TY$. We say that $(\g,t)$ is $\Spin$ \emph{nullhomologous} if there is a smooth map $i\colon \scC^2\to Y$ from a surface with boundary to $Y$ such that $i|_{\d \scC}=\g$ and such that the bundle $(i^* TY, t)$ on $\scC$ is relatively $\Spin$. Note that the set of $\Spin$ 1-chains modulo $\Spin$ cobordism is a group, which we denote by $H_1^{\Spin}(Y)$. The addition operation is disjoint union.   Since $TY$ is trivializable, a framing $s$ of $Y$ induces an isomorphism 
\[
I_s\colon H_1(Y)\oplus \Z/2\to H_1^{\Spin}(Y).
\]
The first coordinate of $I_s$ enhances a homology class to a framed homology class by using the restriction of $s$. The image of $\Z/2$ comes from a null-bordant $1$-cycle with the framing opposite the framing of $Y$.  We note that $H_1^{\Spin}(Y)$ can also be described in terms of the normal bundle to $\gamma$: let $L_1(Y)$ denote the collection of immersed $1$-cycles $\gamma$ with spin structure on the normal bundle $N\gamma$, up to bordisms along which the spin structure on the normal bundle extends.  The group $L_1(Y)$ is isomorphic to $\Omega_1^{\Spin}(Y)$, by an isomorphism depending on a spin structure on $Y$; see \cite[Corollary 3.4]{Atiyah-bordism}. More precisely, $L_1(Y)$ is the \emph{twisted} bordism group $\Omega_1^{\Spin}(Y,TY)$; since $TY$ is parallelizable, this latter group is noncanonically isomorphic to the untwisted bordism group.  An element $(\gamma,t)$ of $L_1(Y)$ determines a spin structure on $T\gamma\oplus N\gamma\cong \gamma^*TY$, using the trivial spin structure on $T\gamma$; this induces a morphism $L_1(Y)\to H_1^{\Spin}(Y)$, which turns out to be an isomorphism.

As noted at the start of this section, any immersed $1$-manifold $\gamma$ in a Heegaard surface $\Sigma\subset Y$ can be canonically equipped with a trivialization on $TY|_\g$. On the other hand, given a pair $(\g,s)$ where $\g$ is an immersion of $S^1$ in $Y$, and $s$ is a trivialization of $\g^* TY$, we can always modify $\g$ by a regular homotopy so that the image of $\g$ lies in $\Sigma$ and the trivialization induced by the immersion in $\Sigma$ coincides with $s$ up to homotopy. To see this, observe that we can always immerse $\g$ in $\Sigma$ by a simple topological argument. The surface-induced trivialization can then be modified by adding kinks to $\g$, as in Figure~\ref{fig:40}. 

\begin{figure}[h]
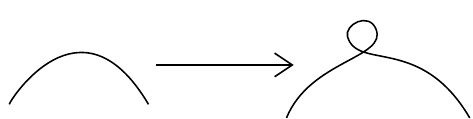
\caption{Adding a kink to $\g$ on $\Sigma$ to change the induced trivialization of $\g^* TY$.}
\label{fig:40}
\end{figure}

  We now show that $\Spin$ bordism can be detected easily using Heegaard diagrams:

\begin{lem}
\label{lem:spin-even} Let $(\Sigma,\as,\bs,\ws)$ be a Heegaard diagram, and let $\g$ be an immersed 1-cycle on $\Sigma$ which represents the trivial homology class. The following are equivalent:
\begin{enumerate}
\item $\g$ is $\Spin$ nullhomologous in $Y$.
\item\label{itm:spin-even-2} There is an integral 2-chain $C\subset \Sigma$ such that $\d C=\g+A+B$, where $A$ (resp. $B$) is an integral combination of curves in $\as$ (resp. $\bs$), such that
\[
e(C)+2\bar{n}_{\xs}(C)-2n_{\ws}(C)\equiv 0 \pmod 2.
\]
\item $\g$ is null-homologous in $Y$ and for all integral 2-chains $C\subset \Sigma$ with $\d C=\g+A+B$, as above, we have
\[
e(C)+2\bar{n}_{\xs}(C)-2n_{\ws}(C)\equiv 0 \pmod 2.
\]
\end{enumerate}
\end{lem}
\begin{proof} We show $(2) \iff (3)$. We observe that if $C$ and $C'$ are two integral 2-chains as in the statement, then $C-C'$ determines a homology class $h$ in $H_2(Y)$. By Ozsv\'{a}th and Szab\'{o}'s Chern class formula \cite{OSProperties}*{Proposition~7.5}, we observe that
\[
\left(e(C)+2\bar{n}_{\xs}(C)-2n_{\ws}(C)\right)-\left(e(C')+2\bar{n}_{\xs}(C')-2n_{\ws}(C')\right)=\langle c_1(\frs_{\ws}(\xs)),h\rangle
\]
which is always even.

We now show $(1)\then (2).$ Suppose $\g$ is $\Spin$ nullhomologous. Since $\g$ is null-homologous,  there is a relative 2-chain $C$ on $\Sigma$ with $\d C=\g+A+B$. By Lemma~\ref{lem:relative-Chern-class} we have that
\[
\langle c_1(\frs_{\ws}(\xs,t)), [\hat C, \d \hat{C}]\rangle =e(C)+2\bar{n}_{\xs}(C)-2n_{\ws}(C).
\]
We recall the general fact that a real bundle $\xi\to X$ is $\Spin$ if and only if $w_2(\xi)=0$. Furthermore, we recall that if $\xi$ is equipped with a $\Spin^c$ structure $\frs$, when $c_1(\frs)$ reduces to $w_2(\xi)$ modulo 2.  Moreover, although there are potentially many $\Spin^c$ structures on $i^*TY$ (rel boundary), it suffices to check the modulo $2$ reduction for any particular $\Spin^c$ structure $\frs$. Since $i^* TY$ is equipped with a $\Spin^c$ structure, we therefore conclude that $i^* TY$ is $\Spin$ if and only if the quantity $e(C)+2\bar{n}_{\xs}(C)-2n_{\ws}(C)$ is even.   We note, in particular, that this quantity does not depend on $\xs$ modulo $2$, as can also be seen directly.  

The above argument also clearly implies that $(3)\then  (1)$, since if $C$ bounds $\g+A+B$ and $e(C)+2\bar{n}_{\xs}(C)-2n_{\ws}(C)$ is even, then there is a surface, homologous to $\hat{C}$ rel boundary, with the property that $i^* TY$ is $\Spin$.
\end{proof}

\begin{cor} Suppose $\g$ and $\g'$ are oriented, immersed curves on a Heegaard diagram $(\Sigma,\as,\bs,\ws)$ and let $t$ and $t'$ be the framings induced by $T\Sigma$. If $(\g,t)$ and $(\g',t')$ are homotopic in $Y$ as framed curves and $C$ is a 2-chain on $\Sigma$ with $\d C=\g-\g'+A+B$, where $A$ and $B$ are integral combinations of the curves in $\as$ and $\bs$, respectively, then
\[
e(C)+2\bar{n}_{\xs}(C)-2n_{\ws}(C)\equiv 0\pmod 2. 
\]
\end{cor}

The main point of interest for us will be the case that $\g$ and $\g'$ are embedded in $\Sigma$, and are isotopic as framed knots in $Y$.

\begin{example} We consider two basic examples. One is a flat unknot $\g_0$ embedded on $\Sigma$, and another is a figure eight $\g_0'$, which is immersed on $\Sigma$.  See Figure~\ref{fig:naturality_33}. We assume both curves are away from the attaching curves and basepoints on the Heegaard diagram. The curve $\g_0$ is not $\Spin$-nullhomologous since it bounds a disk on $\Sigma$ with Euler measure $\pm 1$. The curve $\g_0'$ is $\Spin$ null-homologous since it bounds a domain which has Euler measure 0. 
\end{example}

\begin{figure}
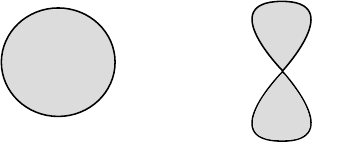
\caption{The curves $\g_0$ and $\g_0'$. The curve $\g_0$ is not $\Spin$-nullhomologous, while $\g_0'$ is. The shaded regions denote bounding 2-chains.}
\label{fig:naturality_33}
\end{figure}

In the following, for any $2$-chain $C\subset \Sigma$ satisfying Condition \ref{itm:spin-even-2} of Lemma \ref{lem:spin-even}, we say $C$ is a Spin-bounding $2$-chain (of $\g$).  More generally, for $\g\subset Y$ with $\Spin$ structure, and an immersed surface $\scC\subset Y$ with $\partial \scC=\g$ along which the $\Spin$ structure extends, we call $\scC$ a Spin bounding surface.

\subsection{On the invariance of the hypercube homology action}
\label{sec:invariance-homology-action}

In this section, we investigate the relation between hypercube homology action and $\Spin$-homology.

We suppose that $(\Sigma,\ws)$ is a pointed Heegaard surface in a 3-manifold $Y$, splitting $Y$ into the union of two handlebodies $U_\a$ and $U_{\b}$. We assue that $\cA$ and $\cB$ are both hypercubes of handleslide equivalent attaching curves on $Y$. We assume that each collection $\as$ in $\cA$ bounds compressing disks in $U_{\a}$ and each $\bs$ in $\cB$ bounds compressing disks for $U_{\b}$. We assume that $\g$ is an immersed and oriented 1-manifold on $\Sigma$ and we write $t$ for the induced trivialization of $\g^* TY$.

If $(\g,t)$ is $\Spin$ null-homologous in $Y$, we will describe a null-homotopy of $\frA_{\g}$. The null-homotopy will depend on a choice of relative 2-chain $\scC$ bounding $\g$. We assume that $\scC$ is an oriented surface with boundary and 
\[
 i\colon \scC\to Y
\] is a smooth map such that $i|_{\d \scC}=\g$. 
 With this data chosen, we will construct a map 
\[
H_{\scC}^{\ws}\colon \CF^-(\cA,\cB)\to \CF^-(\cA,\cB)
\]
satisfying
\[
\d(H_{\scC}^{\ws})=\frA_{\g}.
\]
The dependence on $\scC$ is explored in Lemma~\ref{lem:ambiguity-homotopy}.

We now construct the map $H_{\scC}^{\ws}$. For each $(\nu,\veps)$, we pick a relative 2-chain $C_{\nu,\veps}$ on $\Sigma$ such that
\[
\d C_{\nu,\veps}=\g+A_{\nu}+B_{\veps},
\]
where $A_{\nu}$ is an integer combination of curves from $\as_{\nu}$, and $B_{\veps}$ is an integer combination of curves from $\bs_{\nu}$. Recall that we write $\hat{C}_{\nu,\veps}$ for the 2-chain in $Y$ with boundary $\g$ obtained by capping $C_{\nu,\veps}$ along compressing disks for $\as_{\nu}$ and $\bs_{\nu}$. We assume that the capped off surface $\hat{C}_{\nu,\veps}$ is homologous to $\scC$.

We define the morphism
\[
H^{\ws}_{\scC}\colon \CF^-(\cA,\cB)\to \CF^-(\cA, \cB)
\]
by setting $(H^{\ws}_{\scC})_{(\veps,\nu),(\veps,\nu)}(\xs)$ to be 
\[
(H^{\ws}_{\scC})_{(\veps,\nu),(\veps,\nu)}(\xs):=\frac{e(C_{\veps,\nu})+2\bar{n}_{\xs}(C_{\nu,\veps})-2n_{\ws}(C_{\nu,\veps})}{2}\cdot \xs
\]
 for all $\xs\in \CF^-(\as_\nu, \bs_{\veps})$. The map $H^{\ws}_{\scC}$ has no higher length terms. Note that since $\g$ is $\Spin$-nullhomologous, the fraction in the above equation is an integer by Lemma~\ref{lem:spin-even}.

\begin{lem}
\label{lem:ambiguity-homotopy} Suppose that $\g$ is $\Spin$ nullhomologous and $\scC$ and $\scC'$ are two integral 2-chains in $Y$ with $\d \scC=\d \scC'=\g$. Then
\[
(H^{\ws}_{\scC'}-H^{\ws}_{\scC})(\xs)= \frac{\langle c_1(\frs_{\ws}(\xs)), [\scC'-\scC]\rangle}{2}\cdot \xs
\]
for an intersection point $\xs$. 
\end{lem}
\begin{proof} This follows immediately from the Chern class formula from  \cite{OSProperties}*{Proposition~7.5} and the definition of $H_{\scC}^{\ws}$. 
\end{proof}

In particular, if $(C_{\nu,\veps})_{(\nu,\veps)\in \bE_m\times \bE_n}$ and $(C_{\nu,\veps}')_{(\nu,\veps)\in \bE_m\times \bE_n}$ are two collections of relative 2-chains representing the class $\scC$, then the map $H_{\scC}^{\ws}$ defined with $C_{\nu,\veps}$ is identical to the map defined with $C_{\nu,\veps}'$.

\begin{prop}
\label{prop:nullhomotopy-homology}
Suppose that $\cA$ and $\cB$ are hypercubes of strongly equivalent attaching curves on $(\Sigma,\ws)$, determining a 3-manifold $Y$. Let $\g\subset \Sigma$ be an immersed curve which is $\Spin$-nullbordant in $Y$. Let $\scC$ a Spin bounding 2-chain in $Y$ which bounds $\g$. Then
\[
\d(H^{\ws}_{\scC})=\frA_{\g},
\]
as morphisms of hypercubes. 
\end{prop}

Before we give the proof, we describe a straightforward homological calculation: 
\begin{lem}
\label{lem:periodic-domain-tuples} Let $\cH=(\Sigma,\as_{m},\dots, \as_1, \bs_1,\dots, \bs_n,\ws)$ be a Heegaard tuple where each $\as_i$ and $\as_j$ are handleslide equivalent, and each $\bs_i$ and $\bs_j$ are handleslide equivalent. Let $D(\cH)$ denote the group of periodic domains, i.e. integral 2-chains $P$ on $\Sigma$ with boundary equal to a linear combination of  $\as_i$ and $\bs_i$ curves such that furthermore $n_{w_i}(P)=0$ for all $w_i\in \ws$. If $\gs,\ds\in \{\as_m,\dots, \as_0,\bs_0,\dots, \bs_n\}$, let $D(\cH_{\g,\dt})$ denote the space of periodic domains with boundary on $\gs$ and $\ds$. Then there is an isomorphism
\[
D(\cH)\iso D(\cH_{\a_m,\a_{m-1}})\oplus \cdots \oplus D(\cH_{\a_0,\b_0})\oplus D(\cH_{\b_0,\b_1})\oplus \cdots \oplus D(\cH_{\b_{n-1},\b_n}). 
\]
\end{lem}
\begin{proof} 
We first observe that if $\gs$ and $\gs'$ are handleslide equivalent, then $D(\cH_{\g,\g'})\iso \Z^{g+|\ws|-1}$. Furthermore, there is a concrete isomorphism. Given an element $P\in D(\cH_{\gs,\gs'})$, $\d P$ consists of an integer linear combination of $\gs$ and $\gs'$ curves. We define a map 
\[
L_{\g,\g'}\colon D(\cH_{\g,\g'})\to \Z^{g+|\ws|-1}
\]
to map a $P\in D(\cH_{\g,\g'})$ to the coefficients of the $\gs$ curves in $\d P$. Naturally, there is another map
\[
R_{\g,\g'}\colon D(\cH_{\g,\g'})\to \Z^{g+|\ws|-1},
\]
 obtained by mapping to the coefficients of the $\gs'$ appearing in $\d P$. We leave it to the reader to verify that both $L_{\g,\g'}$ and $R_{\g,\g'}$ are isomorphisms.

There is a canonical map
\begin{equation}
 D(\cH_{\a_m,\a_{m-1}})\oplus \cdots \oplus D(\cH_{\a_0,\b_0})\oplus D(\cH_{\b_0,\b_1})\oplus \cdots \oplus D(\cH_{\b_{n-1},\b_n})\to D(\cH)
\label{eq:map=sum}
\end{equation}
obtained by summing. To see that this is injective, we observe that if
\[
P_{\a_m,\a_{m-1}}+\cdots +P_{\b_{n-1},\b_n}=0,
\]
then $P_{\a_m,\a_{m-1}}$ must be zero since no other summands have boundary on $\as_m$, and the map $L_{\a_m,\a_{m-1}}$ is an isomorphism. Similarly $P_{\b_{n-1},\b_n}=0$. Repeating this procedure, we are left with $P_{\a_0,\b_0}=0$, showing that the aforementioned map was injective. Similarly, to see that the map in Equation~\eqref{eq:map=sum} is surjective, we take an arbitrary $P\in D(\cH)$. Since $R$ and $L$ are isomorphisms, we may add an element in $D(\cH_{\a_m,\a_{m-1}})$ so that $\d P$ has zero coefficients on $\as_m$. Repeating this procedure, we reduce to the case that $P$ has boundary only on $\as_0$ and $\bs_0$, which implies that $P\in D(\cH_{\a_0,\b_0})$, so the aforementioned map is surjective and the proof is complete.
\end{proof}

\begin{proof}[Proof of Proposition~\ref{prop:nullhomotopy-homology}]
For convenience, write $\as_i$ for $\as_{\nu_i}$ and $\bs_i$ for $\bs_{\veps_i}$. Similarly write $\theta_{i+1,i}^{\a}$ for $\theta_{\nu_{i+1},\nu_i}^\a$. We will write $C_{i,j}$ for $C_{\nu_i,\veps_j}$ as well.

Suppose
\[
\psi
\in\pi_2(\theta^{\a}_{m,m-1},\dots, \theta^{\a}_{1,0}, \xs, \theta_{0,1}^{\b},\dots, \theta_{n-1,n}^{\b}, \ys)
\]
is a class of polygons. Write 
\[
N_C^{\ws}(\xs)=\frac{e(C)+2\bar n_{\xs}(C)-2n_{\ws}(C)}{2}
\]
The main claim follows from the equation
\begin{equation}
N^{\ws}_{C_{\nu_m,\veps_n}}(\ys)-N^{\ws}_{C_{\nu_0,\veps_0}}(\xs)=-\#\d_{\b}(\psi)\cap \g,
\label{eq:main-claim-homology-action-null-homotopy}
\end{equation}
which we now prove.

 On the Heegaard multi-diagram $(\Sigma,\as_{m},\as_{\nu_0},\bs_{0}, \bs_{n})$, Lemma~\ref{lem:periodic-domain-tuples} implies that we may write 
\[
C_{\nu_m,\veps_n}-C_{\nu_0,\veps_0}=P_{\nu_0,\veps_0}+P_{\nu_m,\nu_0}^\a+P_{\veps_0,\veps_n}^{\b}
\]
where $P_{\nu_0,\veps_0}$ is a doubly periodic domain on $(\Sigma,\as_{0}, \bs_{0})$, and similarly for the other terms. Noting that $H_2(U_{\a})=H_2(U_{\b})=0$, where $U_{\a}$ and $U_{\b}$ are the alpha and beta handlebodies of $Y$, we conclude that
\[
[P_{\nu_0,\veps_0}]=0,
\]
because both $C_{\nu_m,\veps_n}$ and $C_{\nu_0,\veps_0}$ represent $\scC$. Note that as in the proof of Lemma~\ref{lem:ambiguity-homotopy} we have that
\[
N_{C_{\nu_0,\veps_0}+P_{\nu_0,\veps_0}}^{\ws}=N_{C_{\nu_0,\veps_0}}^{\ws}.
\]
By replacing  $C_{\nu_0,\veps_0}$ with $C_{\nu_0,\veps_0}+P_{\nu_0,\veps_0}$, we may assume that
\begin{equation}
C_{\nu_m,\veps_n}-C_{\nu_0,\veps_0}=P_{\nu_n,\nu_0}^{\a}+P_{\veps_0,\veps_m}^{\b}.
\label{eq:Cmn-C00}
\end{equation}

Lemma~\ref{lem:periodic-domain-tuples} also implies that there are periodic domains $P_{\nu_{i+1},\nu_i}^{\a}$ and $P_{\veps_i,\veps_{i+1}}^{\b}$ so that
\[
P_{\nu_n,\nu_0}^{\a}=P_{\nu_n,\nu_{n-1}}^{\a}+\cdots+P_{\nu_{1},\nu_0}^{\a}\quad \text{and} \quad P_{\veps_0,\veps_m}^\b=P_{\veps_0,\veps_1}^{\b}+\cdots+P_{\veps_{m-1}, \veps_m}^{\b}. 
\]

 If $\eta$ and $\tau$ are two 1-chains on $\Sigma$, we define
\[
\#\eta\barcap\tau \in \frac{1}{4} \Z
\]
as follows. We consider the four translations of $\eta$ in the $\mathrm{NE}$, $\mathrm{NW}$, $\mathrm{SW}$ and $\mathrm{SE}$ directions (where $\mathrm{N}$ means the direction of an oriented tangent to $\eta$). We take the ordinary intersection number of each of these with $\tau$, and then define $\# \eta \barcap \tau$ to be the average.

The Leibniz rule for intersections reads
\[
0=\#\d(\d_{\b_{i}}(\psi) \barcap C_{0,0})=\# (\d \d_{\b_i}(\psi)) \barcap C_{0,0}-\# \d_{\b_i}(\psi) \barcap \d C_{0,0}.
\]

We define closed 1-chains $A_s$ in $\as_s$ and $B_t$ in $\bs_t$ as
\[
\d P_{\nu_i,\nu_{i-1}}^{\a}=A_i-A_{i-1}\quad \text{and} \quad \d P_{\veps_{t-1}, \veps_t}^\b=B_t-B_{t-1}.
\]
Note that Equation~\eqref{eq:Cmn-C00} implies that
\[
\d C_{0,0}= \g+A_0+B_0\quad \text{and} \quad \d C_{m,n}=\g+A_m+B_n.
\]

Using the Leibniz rule for intersections applied to $\d_{\b}(\psi)\cap C_{0,0}$, we observe that
\begin{equation}
N_{C_{0,0}}^{\ws}(\ys)=N_{C_{0,0}}^{\ws}(\xs)- \# \d_{\b}(\psi)\cap \g-\# \d_{\b}(\psi)\barcap (A_0+B_0).
\label{eq:Relation-Ny-Nx}
\end{equation}
Next, since $\d D(\psi)= \d_{\b}(\psi)+\d_{\a}(\psi)$, and $A_0$ is a closed 1-cycle, we observe that
\begin{equation}
\d_{\b}(\psi)\barcap A_0=-\d_{\a}(\psi) \barcap A_0.
\label{eq:db=da}
\end{equation}
Combining Equations~\eqref{eq:Cmn-C00}, ~\eqref{eq:Relation-Ny-Nx} and ~\eqref{eq:db=da} we obtain that
\[
\begin{split}
N_{C_{n,m}}^{\ws}(\ys)=&N_{C_{0,0}}^{\ws}(\xs)- \# \d_{\b}(\psi)\cap \g\\
&+ \#\d_{\a}(\psi)\barcap A_0-\sum_{i=0}^{m-1} N_{P_{i+1,i}^\a}^{\ws}(\ys)\\
&-\# \d_\b(\psi) \barcap B_0-\sum_{i=0}^{n-1} N_{P_{i,i+1}^\b}^{\ws}(\ys).
\end{split}
\]
We claim that
\begin{equation}
\#\d_{\a}(\psi)\barcap A_0-\sum_{i=0}^{m-1} N_{P_{i+1,i}^\a}^{\ws}(\ys)=-\# \d_\b(\psi) \barcap B_0-\sum_{i=0}^{m-1} N_{P_{i+1,i}^\b}^{\ws}(\ys)=0, \label{eq:da(psi)capA_0}
\end{equation}
which will establish Equation~\eqref{eq:main-claim-homology-action-null-homotopy} and prove the main claim. We consider the left side of the above equation first. We will write $\d_{\a_m+\cdots +\a_i}(\psi)$ for the portion of $\d D(\psi)$ which lies on $\as_m,\dots, \as_i$. Note that $ \d_{\a_n+\dots+\a_i}(\psi)$ is a 1-chain with boundary $\ys-\theta_{i,i-1}^{\a}$. Applying the Leibniz rule for intersections to $\d_{\a_n+\dots+ \a_i}(\psi)\cap P_{i+1,i}^\a$ and summing, we obtain
\[
\begin{split}
&\# \d_{\a}(\psi)\barcap A_0-\sum_{i=0}^{m-1} N_{P_{i+1,i}^\a}^{\ws}(\ys)
\\
=&\# \d_{\a}(\psi)\barcap A_0 -\sum_{i=0}^{n-1} N^{\ws}_{P_{i+1},i}(\theta_{i+1,i}^{\a})-\sum_{i=0}^{n-1} \#\d_{\a_{n}+\cdots+\a_i}\barcap (A_{i+1}-A_i)
\end{split}
\]
We observe that the above equation vanishes because 
\[
\# \d_{\a}(\psi) \barcap A_0-\sum_{i=0}^{n-1} \# \d_{\a_n+\dots +\a_i} \barcap (A_{i+1}-A_i)
\]
is a telescoping sum and
\[
N_{P_{i+1,i}^{\a}}^{\ws}(\theta_{i+1,i}^{\a})=\frac{\langle c_1(\frs_{\ws}(\theta_{i+1,i}^{\a})), [P_{i+1,i}^{\a}]\rangle}{2}=0
\]
by the Chern class formula of Ozsv\'{a}th and Szab\'{o} \cite{OSProperties}*{Proposition~7.5} and the fact that the class $\theta_{i+1,i}^{\a}$ represent the torsion $\Spin^c$ structure on $(S^1\times S^2)^{\# g}$. The vanishing of the middle portion of Equation~\eqref{eq:da(psi)capA_0} follows from the same logic, completing the proof.
\end{proof}

\subsection{Change of 1-chain maps and twisted coefficients}
\label{sec:twisted-coefficients-Phi-w-C}

The construction in Section~\ref{sec:invariance-homology-action} of a null-homotopy of the homology action also extends to give homotopy equivalences between Floer complexes with twisted coefficients.

 We suppose that $\g\subset \Sigma$ is an immersed 1-chain. As discussed in Section~\ref{sec:twisted-coefficients-def}, there is a twisted Floer complex
\[
\underline{\CF}^-_{\g}(\cA,\cB,\ws),
\]
which is a free module over $\bF[T,T^{-1},U_1,\dots, U_n]$. We now suppose that $\g'$ is another immersed 1-chain which is $\Spin$ homologous to $\g$ in $Y$.

We assume that $\scC$ is a $\Spin$ bounding 2-chain with boundary $\g'-\g$.    We now define a map
\[
\Phi_{\scC}^{\ws}\colon \underline{\CF}^-_{\g}(\cA,\cB)\to \underline{\CF}^-_{\g'}(\cA,\cB)
\]
as follows. Let $\ve{d}$ (resp. $\ve{d}'$) denote the size of $\cA$ (resp. $\cB$). For each pair $(\nu,\veps)\in \bE(\ve{d})\times \bE(\ve{d'})$, we pick a relative 2-chain $C_{\nu,\veps}$ on  $\Sigma$ with boundary $\g'-\g+A_{\nu}+B_{\veps}$, where $A_{\nu}$ (resp. $B_{\veps}$) is an integral combination of curves from $\as_{\nu}$ (resp. $\bs_{\veps}$). All components of
$\Phi_{\scC}^{\ws}$ preserve the cube point $(\nu,\veps)$. For $\xs\in \bT_{\a_{\nu}}\cap \bT_{\b_{\veps}}$, we set
\[
\Phi_{\scC}^{\ws}(\xs)=T^{-N_{C_{\nu,\veps}}^{\ws}(\xs)} \cdot \xs,
\]
extended equivariantly over the $U_i$ and $T$ actions.
Here
\begin{equation}\label{eq:n-def}
N_{C}^{\ws}(\xs)=\frac{e(C)+2\bar{n}_{\xs}(C)-2n_{\ws}(C)}{2}.
\end{equation}

\begin{lem} 
The map $\Phi_{\scC}^{\ws}$ is a chain isomorphism.
\end{lem}
\begin{proof} The proof follows by modifying the proof of Proposition~\ref{prop:nullhomotopy-homology}, as follows. We suppose that $\psi
\in\pi_2(\theta^{\a}_{m,m-1},\dots, \theta^{\a}_{1,0}, \xs, \theta_{0,1}^{\b},\dots, \theta_{n-1,n}^{\b}, \ys)$. We consider Equation~\eqref{eq:main-claim-homology-action-null-homotopy}, with $\g$ replaced therein by $\g'-\g$. This reads
\[
N_{C_{\nu_m,\veps_n}}^{\ws}(\ys)-N_{\nu_0,\veps_0}^{\ws}(\xs)=-\# \d_{\b}(\psi)\cap \g'+\# \d_{\b}(\psi)\cap \g.
\]
Rearranging, we get
\[
-N_{\nu_0,\veps_0}^{\ws}(\xs)+ \# \d_{\b}(\psi)\cap \g'=\# \d_{\b}(\psi)\cap \g-N_{\nu_m,\veps_n}^{\ws}(\ys),
\]
which implies that $\Phi_{\scC}^{\ws}$ commutes with the hypercube structure maps.
\end{proof}

The above discussion extends easily to the case where we we have a collection of immersed 1-chains $\gs=\{\g_1,\dots, \g_n\}$ on $\Sigma$ and we twist by each one, working with a twisted Floer complex which has a $T_i$ variable for each $\g_i$. If $\gs'=\{\g_1',\dots, \g_n'\}$ are another collection of immersed 1-chains and $\scC_1,\dots, \scC_n$ are $\Spin$ bounding 2-chains with $\d \scC_i=\g_i'-\g_i$, then there is a similar chain isomorphism
\[
\Phi_{\scC_1,\dots, \scC_n}^{\ws}\colon \underline{\CF}^-_{\ve{\g}}(\cA,\cB)\to \underline{\CF}^-_{\ve{\g}'}(\cA,\cB).
\]

\section{Hypercubes of link surgery complexes}
\label{sec:hypercubes-of-surgery-complexes}

In this section, we describe how to construct hypercubes of link surgery complexes when the underlying Heegaard splitting is unchanged.

\subsection{Alexander gradings and local systems}
\label{sec:Alexander-grading}

We begin by describing an Alexander grading on certain morphism spaces, which will be important in constructing hypercubes. In particular, we will describe an Alexander grading on morphism spaces of the form $\ve{\CF}^-(\bs_\veps^{E_{\veps}}, \bs_{\veps'}^{E_{\veps'}})$ when $\bs_{\veps}$ and $\bs_{\veps'}$ are Lagrangians with local systems appearing in the construction of the surgery complex.

Let $(\Sigma,\ws,\zs)$ be a meridional Heegaard splitting for a link $(Y,L)$, produced via the construction in Section~\ref{sec:meridional-diagrams}. We write $D_i$ for the distinguished disk containing $w_i$ and $z_i$.

Suppose that $\bs_{\veps}$ and $\bs_{\veps'}$ are attaching curves on $\Sigma$, associated to cube points $\veps,\veps'\in \bE_n$. That is, $\bs_{\veps}$ is constructed from the beta curves of a sutured diagram for $M^\circ(L)$ by adding special meridional curves for the components $K_i\subset L$ with $\veps_i=0$, and by adding isotopic copies of these which are disjoint from the $D_i$ when $\veps_i=1$.

\begin{define}
\label{def:well-adapted}
We say that a link shadow $\cS=\{S_1,\dots, S_\ell\}$ on $\Sigma$ for $(Y,L)$ is \emph{adapted} to $\bs_{\veps}$ if each shadow $S_i$ passes through $w_i$ and $z_i$, and $S_i\setminus \{w_i,z_i\}$ decomposes as $\lambda_\b\cup \lambda_\a$, as follows:
\begin{enumerate}
\item  $\lambda_\a$ is an arc from $w_i$ to $z_i$ which is contained in $D_i$.
\item $\lambda_\b$ becomes isotopic (rel boundary) in $U_{\b}$ to $K_i\cap U_{\b}$ once we push the interior of $\lambda_{\b}$ into $U_{\b}$. 
\item If $\veps_i=0$, then $\lambda_\b$ is disjoint from $\bs_{\veps}$. 
\end{enumerate}
\end{define}

Suppose now that $\bs_{\veps}$ and $\bs_{\veps'}$ are two collections of attaching curves for $\veps\le \veps'\in \bE_\ell$. (Here we allow the case that $\veps=\veps'$, with the understanding that $\bs_{\veps}$ and $\bs_{\veps'}$ are different attaching curves). We let $\cS_{\veps}=(S_1,\ldots, S_\ell)$ and $\cS_{\veps'}=(S_1',\ldots,S_\ell')$ be two link shadows which are adapted to $\bs_{\veps}$ and $\bs_{\veps'}$, respectively. Using these link shadows, we can construct local systems $E_\veps$ and $E_{\veps'}$ over $\bs_{\veps}$ and $\bs_{\veps'}$, as in Equation~\eqref{eq:local-system}, respectively. This yields a Floer complex
\[
\ve{\CF}^-(\bs_{\veps}^{E_{\veps}}, \bs_{\veps'}^{E_{\veps'}}). 
\]
We now claim that if $\cS_{\veps}$ and $\cS_{\veps'}$ represent the same Morse framing on $L$, then $\ve{\CF}^-(\bs_{\veps}^{E_{\veps}}, \bs_{\veps'}^{E_{\veps'}})$ has a natural $\Z^n$-valued Alexander grading, which we define as follows. 

The Floer complex $\ve{\CF}^-(\bs_{\veps}^{E_{\veps}}, \bs_{\veps'}^{E_{\veps'}})$ is spanned by elements of the form $\langle \xs, \phi\rangle$ where $\phi\colon E_{\veps}\to E_{\veps'}$ is an $\bF[U]$-linear map, and $\xs\in \bT_{\b_{\veps}}\cap \bT_{\b_{\veps'}}$. We then complete with respect to the action of $U$.  Notice that $\Hom_{\bF[U]}(E_\veps, E_{\veps'})$ has a natural $n$-component Alexander grading, corresponding to the shift induced by $\phi$. Indeed, both $E_0$ and $E_1$ admit Alexander gradings by 
\[
A(W)=-1, \quad A(Z)=A(T)=1 \quad \text{and} \quad  A(U)=0,
\]
 from which $E_\veps$ inherits a $\mathbb{Z}^n$-valued Alexander grading; in particular, we have a notion of Alexander-homogeneous morphisms $\phi\colon E_\veps\to E_{\veps'}$.  We write $A_i(\phi)$ for the $i$-th component of the Alexander grading of $\phi$.

To define the Alexander grading, we need to make some auxiliary choices. We pick a $\Spin$ 2-chain $\scC_i$ in $Y$ which bounds $S_i'-S_i$  for each $i=1,\dots, n$. We refer to a collection $\qs$ which contains exactly one basepoint from each link component as a \emph{complete set of basepoints}. We pick any complete set of basepoints $\ve{q}\subset \ws\cup \zs$ so that $\frs_{\qs}(\xs)$ is torsion.

 We define an Alexander grading $A_i(\phi)$ for homogeneous $\phi$ as follows.   For $\phi$ a morphism that is homogeneous with respect to Alexander grading on $E_\veps,E_{\veps'}$, we then define the Alexander grading of $\langle \xs,\phi\rangle$ to be
\begin{equation}
A_i(\langle \xs, \phi\rangle)=A_i(\phi)+N_{C_i}^{\qs}(\xs) \label{eq:A_i-def-morphism-hypercubes}
\end{equation}
where $C_i$ is a relative periodic domain on $\Sigma$ which induces the homology class $\scC_i$ in $Y$, and $N_{C_i}^{\qs}(\xs)$ is as in~\eqref{eq:n-def}.

\begin{rem} In the context of the surgery formula, the condition that $\frs_{\qs}(\theta)$ is torsion can be understood as follows. Suppose that $\theta\in \bT_{\b_{\veps}}\cap \bT_{\b_{\veps'}}$ for $\veps\le \veps'$. Let $I_{\veps,\veps'}=\{i: \veps_i<\veps_i'\}$. There are $2^{|I_{\veps,\veps'}|}$ $\Spin^c$ classes of intersection points which contribute to the surgery hypercube. We identify these $\Spin^c$ structures with maps from $s\colon I_{\veps,\veps'}\to \{\sigma,\tau\}$. The condition that $\frs_{\qs}(\theta)$ is torsion is equivalent to  $w_i\in \qs$ whenever $s(i)=\sigma$, and $z_i\in \qs$ whenever $s(i)=\tau$.
\end{rem}

\begin{rem} Some care is required to define the term $n_{\qs}(C_i)$ appearing in the definition of $N_{C_i}^{\qs}(\xs)$ in Equation~\eqref{eq:A_i-def-morphism-hypercubes}, since $\d C_i$ intersects the basepoints $\qs$. We may assume that all of the intersections are transverse, and if $q\in \qs$, we define $n_{q}(C_i)$ appearing therein to be the average of multiplicities in the four regions adjacent to $q$.
\end{rem}

\begin{lem}\label{lem:alex-well-def} The Alexander grading $A_i$ is independent of the choice of $\qs$ such that $\frs_{\qs}(\theta)$ is torsion. 
\end{lem}
\begin{proof} The ambiguity in $\qs$ is as follows. For each $i$ such that $\veps_i<\veps'_i$, the choice of the basepoint on $K_i$ is fixed. (If $\theta$ and $\phi$ have labeling $\sigma_i$, then $w_i\in \qs$; if they have labeling $\tau_i$, then $z_i\in \qs$). If $\veps_i=\veps_i'$, then we can have either $w_i$ or $z_i$ in $\qs$. If $\veps_i=\veps_i'=1$, then we can have either $w_i$ or $z_i$, but since they are immediately adjacent, the choice is irrelevant. Therefore it suffices to consider the case that $\veps_i=\veps_i'=0$.  We will show that if $\veps_i=\veps_i'=0$, then
\[
N_{\scC_j}^{\qs_0\cup \{w_i\}}(\theta)=N_{\scC_j}^{\qs_0\cup \{z_i\}}(\theta),
\]
for all $j$.  If $j\neq i$ then 
\[
N_{\scC_j}^{\qs_0\cup \{z_i\}}(\theta)-N_{\scC_j}^{\qs_0\cup \{w_i\}}(\theta)=-(n_{z_i}-n_{w_i})(\scC_j).
\]
If $i\neq j$, $(n_{z_i}-n_{w_i})(\scC_j)$ represents the algebraic intersection number of the knot $K_i$ with the 2-chain $\scC_j$, which vanishes because both $\cS_{\veps}$ and $\cS_{\veps'}$ represent link shadows for $L$ and we can choose $\scC_j$ to geometrically represent an isotopy between these two shadows which is disjoint from the other knot components. We now consider the case that  if $i=j$. It is not hard to see using geometric reasoning similar to the previous case that the quantity $(n_{z_i}-n_{w_i})(\scC_i)$ is equal to the difference in the Morse framings between $\cS_{\veps}$ and $\cS_{\veps'}$, and therefore vanishes vanishes because  $\cS_{\veps}$ and $\cS_{\veps'}$ induce the same Morse framings. 
\end{proof}

We now consider the Alexander grading in the context of hypercubes. 

\begin{prop}
	\label{prop:grading-preserving}
Suppose that $\bs_{0}^{E_{\veps_0}},\dots, \bs_n^{E_{\veps_n}}$ is a sequence of attaching curves with local systems on a Heegaard splitting $(\Sigma,\ws,\zs)$ for an $\ell$-component link $L\subset Y$, $\veps_0\le \cdots \le \veps_n$ is a sequence of cube points in $\bE_\ell$, and that $\theta_{i,i+1}\colon \bs_i^{E_{\veps_i}}\to \bs_i^{E_{\veps_{i+1}}}$ is a sequence of homogeneously Alexander graded chains. Then the holomorphic polygon map satisfies
\[
A_j(\mu_n(\theta_{0,1},\dots, \theta_{n-1,n}))=A_j(\theta_{0,1})+\cdots+A_j(\theta_{n-1,n}).
\]
for each $j\in \{1,\dots, \ell\}$.
\end{prop}

\begin{proof} Let $\cS_0,\dots, \cS_n$ for the link shadows used in the constructions of the local systems $E_{\veps_i}$.  Write $S_0,\dots, S_n$ for the shadows for the link component $K_j$.  We pick, for each $i\in \{0,\dots, n-1\}$ a $\Spin$ 2-chain $\scC_{i,i+1}$ in $U_{\b_{i+1}}\cup- U_{\b_{i}}$ with boundary $S_{i+1}-S_i-B_i^{i,i+1}+B_{i+1}^{i,i+1}$.  We represent $\scC_{i,i+1}$ by a 2-chain $C_{i,i+1}$ on $\Sigma$ with boundary $S_{i+1}-S_i- B_i^{i,i+1}+B_{i+1}^{i,i+1}$ where $B_{i}^{i,i+1}$ is a linear combination of the curves in $\bs_i$ and $B_{i+1}^{i,i+1}$ is a linear combination of the curves in $\bs_{i+1}$. Since $\bs_0,\bs_{1},\dots, \bs_n$ are all isotopic (ignoring the basepoints), it follows that by adding domains with boundary equal to linear combinations of the curves $\bs_{0},\dots, \bs_n$, we can ensure that $B_i^{i-1,i}=B_{i}^{i,i+1}$ for all $i$. Therefore we write $B_i$ for $B_i^{i,i+1}=B_i^{i+1,i}$. This may change the 2-chains $C_{i,i+1}$ by adding periodic domains. However this does not change the gradings $A_{j}(\theta_{i,i+1})$ by Lemma~\ref{lem:alex-well-def} because $\frs_{\qs}(\theta_{i,i+1})$ is assumed to be torsion. Similarly, we may assume that the 2-chain used to compute the grading of an output of the polygon counting map is
	\[
	C_{0,n}:=C_{0,1}+\cdots+C_{n-1,n}.
	\]

We also write $\theta_{i,i+1}=\langle \xs_{i,i+1}, f_{i,i+1}\rangle$. We consider a polygon class $\psi\in \pi_2(\theta_{0,1},\dots, \theta_{n-1,n},\ys)$. Let $f\colon E_{\veps_0}\to E_{\veps_n}$ be the linear morphism which is output by the homology class of $n$-gons $\psi$.

We apply the Leibniz rule to $\d_{\b_0+\dots +\b_i}(\psi)\cap C_{i,i+1}$. This yields
\[
0 =\bar{n}_{C_{i,i+1}}(\theta_{i,i+1})-\bar{n}_{C_{i,i+1}}(\ys)+\d_{\b_0+\dots+\b_i}(\psi)\cap (S_{i+1}-S_i+B_{i+1}-B_i).
\]
We now sum the above expression over $i=0,\dots, n-1$ and then regroup terms to get
\[
0=\sum_{i=0}^{n-1} \bar{n}_{C_{i,i+1}}(\theta_{i,i+1})-\bar{n}_{C_{0,n}}(\ys)+\sum_{i=0}^{n-1} \#\d_{\b_i}(\psi)\cap (S_n-S_i+B_n-B_i)=0.
\]
We note that $\#\d_{\b_i}(\psi)\cap B_i=0$. Therefore we can rearrange the above equation to obtain
\begin{equation}
\begin{split}
\bar{n}_{C_{0,n}}(\ys)=&\sum_{i=0}^{n-1}\bar{n}_{C_{i,i+1}}(\theta_{i,i+1})-\sum_{i=0}^{n-1} \# \d_{\b_i}(\psi)\cap S_i\\
&+\sum_{i=0}^{n-1} \# \d_{\b_i}(\psi)\cap (S_n+B_n). 
\end{split}
\label{eq:regroup-Alexander-gradings}
\end{equation}
Note that $S_n+B_n$ is closed (as a 1-chain on $\Sigma$) and $\sum_{i=0}^{n-1}\# \d_{\b_i}(\psi)$ is homologous to $-\d_{\b_n}(\psi)$, so
\[
\sum_{i=0}^{n-1} \# \d_{\b_i}(\psi)\cap (S_n+B_n)=-\# \d_{\b_n}(\psi)\cap (S_n+B_n)=-\# \d_{\b_n}(\psi)\cap S_n.
\]
Combining this with Equation~\eqref{eq:regroup-Alexander-gradings}, we obtain
\begin{equation}
\sum_{i=0}^{n-1} \bar{n}_{C_{i,i+1}}(\theta_{i,i+1})=\bar{n}_{C_{0,n}}(\ys)+\sum_{i=0}^{n} \# \d_{\b_i}(\psi)\cap S_i. \label{eq:relation-sum-nC_itheta_i}
\end{equation}

 We easily compute that
\begin{equation}
A_j(f)=\# \d_{\b_0}(\psi)\cap S_0+A_j(f_{0,1})+ \cdots +A_j(f_{n-1,n})+\# \d_{\b_n}(\psi)\cap S_n.
\label{eq:A_j-output-formula}
\end{equation}

We compute therefore that if $\langle \ys, f\rangle$ is a summand of the polygon counting map, then
\[
\begin{split}
&\sum_{i=0}^{n-1} A_j(\langle \theta_{i,i+1},f_{i,i+1}\rangle)
\\
=&\sum_{i=0}^{n-1}\left(\frac{e(C_{i,i+1})+2\bar{n}_{\theta_{i,i+1}}(C_{i,i+1})-2n_{\qs}(C_{i,i+1})}{2}+A_j(f_{i,i+1})\right)\\
=&\frac{e(C_{0,n})+2\sum_{i=0}^{n-1} \bar{n}_{\xs_{i,i+1}}(C_{i,i+1})-2n_{\qs}(C_{0,n})}{2}+\sum_{i=0}^{n-1}A_j(f_{i,i+1}) \\
=&\frac{e(C_{0,n})+2 \bar{n}_{\ys}(C_{0,n})-2n_{\qs}(C_{0,n})}{2}+\sum_{i=0}^n \# \d_{\b_i}(\psi)\cap S_i+\sum_{i=0}^{n-1}A_j(f_{i,i+1})\\
=& A_j(\langle \ys, f\rangle ).
\end{split}
\]
To go from the first line to the second, we use the definition of $A_j$. To go from the second line to the third, we use the relation $C_{0,n}=C_{0,1}+\cdots+C_{n-1,n}$. To go from the third line to the fourth, we use Equation~\eqref{eq:relation-sum-nC_itheta_i}. To go from the fourth line to the final line, we use the definition of $A_j$ as well as Equation~\eqref{eq:A_j-output-formula}. This completes the proof.
\end{proof}

\subsection{Homology groups}
\label{sec:homology-groups}
We now consider the homology groups of the morphism spaces $\ve{\CF}^-(\bs_{\veps}^{E_{\veps}}, \bs_{\veps'}^{E_{\veps'}})$ when $\veps\le \veps'$. For our purposes, it is more helpful to compute the homology of a subcomplex, as we now describe. We let $M$ denote the set of components $K_i$ such that $\veps_i'>\veps_i$. Let $\fro(M)$ denote the set of orientations on $M$. If $\scO\in \fro(M)$, we let
\[
\phi^{\scO}_i\colon E_{\veps_i}\to E_{\veps'_i}
\]
denote the identity if $\veps_i=\veps_i'$, $\phi^\sigma$ if $\veps_i<\veps_i'$ and $K_i$ is oriented positively in $\scO$, and $\phi^\tau$ if $\veps_i<\veps_i'$ and $K_i$ is oriented negatively in $\scO$. We define
\[
\phi^{\scO}:=\phi_1^{\scO}\otimes \cdots \otimes \phi_n^{\scO}\colon E_{\veps}\to E_{\veps'}.
\]

Given $\scO$, write $\qs\subset \ws\cup \zs$ for a complete set of basepoints which contains $w_i$ if $\veps_i<\veps_i'$ and $K_i$ is oriented positively in $\scO$, and which contains $z_i$ if $\veps_i<\veps_i'$ and $K_i$ is oriented negatively in $\scO$. We will write $\ve{\CF}^-_{\phi^{\scO}}(\bs_{\veps}^{E_\veps}, \bs_{\veps'}^{E_{\veps'}}, \frs_{\scO})$ for the complex spanned by morphisms $\langle \xs, \phi\rangle$, where $\frs_{\qs}(\xs)$ is torsion, and $\phi$ is a multiple of $\phi^{\scO}$. 

\begin{prop}
	\label{prop:top-degree-cycle} Suppose that $\bs_{\veps}$ and $\bs_{\veps'}$ are two collections of meridional beta curves on a link Heegaard splitting $(\Sigma,\ws,\zs)$ for $(Y,L)$, and $\cS$ and $\cS'$ are two collections of adapted shadows. Then the Alexander grading $(0,\dots, 0)$ subspace of
\[
\ve{\HF}^-_{\phi^\scO}(\bs^{E_{\veps}}_{\veps}, \bs^{E_{\veps'}}_{\veps'}, \frs_\scO)
\]
is isomorphic to $\Lambda_g\otimes \bF\llsquare U\rrsquare$, where $\Lambda_g$ denotes the exterior algebra on $g$ generators.
\end{prop}

This follows from \cite{ZemExactTriangle}*{Remark~8.16}, which addresses the case where the shadows are identical. The general case, where the choice of shadows may be different, can be derived by using associativity of the triangle maps to reduce to the aforementioned case.

\subsection{Building surgery hypercube hypercubes}
\label{sec:hypercube-hypercubes}

The Alexander grading from Section~\ref{sec:Alexander-grading} and the computations in Section~\ref{sec:homology-groups} allow us to efficiently build link surgery complexes for complicated meridional systems of Heegaard diagrams, and also to build hypercubes of such complexes.

We first consider the construction of the link surgery complex for more general systems of Heegaard diagrams than those considered in Section~\ref{sec:meridional-complete-systems}.

\begin{define} 
	\item
	\begin{enumerate}
		\item We say that a Heegaard diagram $(\Sigma,\as, (\bs_{\veps})_{\veps\in \bE_\ell}, \ws, \zs)$ is a \emph{generalized meridional system of Heegaard diagrams} if there is a meridional system of Heegaard diagrams $(\Sigma,\as,(\bs_{\veps}')_{\veps\in \bE_\ell}, \ws, \zs)$, in the sense of Definition~\ref{def:meridional-surgery-diagram}, such that $\bs_{\veps}$ and $\bs_{\veps}'$ are related by a sequence of handleslides and isotopies in the complement of $\ws\cup \zs$.
		\item We say a collection of Heegaard diagrams $(\Sigma, (\as_{\mu})_{\mu\in \bE_m}, (\bs_{(\nu, \veps)})_{(\nu,\veps)\in \bE_n\times \bE_\ell},\ws,\zs)$ is a \emph{cube of generalized meridional systems} if there is a genuine meridional system $(\Sigma, \as, (\bs'_{\veps})_{\veps\in \bE_\ell}, \ws, \zs)$ so that each $\as_\nu$ is handleslide equivalent to $\as$ in the complement of $\ws\cup \zs$, and each $\bs_{(\nu,\veps)}$ is handleslide equivalent to $\bs'_{\veps}$ in the complement of $\ws\cup \zs$. We also assume that each $\as_\nu$ is disjoint from the meridional disk regions $D_1,\dots, D_\ell$.
	\end{enumerate} 
	\end{define}

Suppose $\Lambda$ is a Morse framing on a link $L\subset Y$, and  $(\Sigma,\as, (\bs_{\veps})_{\veps\in \bE_\ell}, \ws, \zs)$ is a generalized meridional system of Heegaard diagrams for $(Y,L)$. For each $\veps\in \bE_\ell$, we pick an adapted set of link shadows $\cS_\veps$ for $(\Sigma,\as,\bs_{\veps}, \ws, \zs)$. We assume these link shadows intersect each of the canonical special meridional disk regions $D_1,\dots, D_{\ell}$ in a single arc. These shadows give local systems on each $\bs_\veps$, as well as Alexander multi-gradings on each $\ve{\CF}^-_{\phi^{\scO}}(\bs_{\veps}^{E_{\veps}}, \bs_{\veps'}^{E_{\veps'}}, \frs_{\scO})$, where $\scO$ is an orientation on $L$. 

For each orientation $\scO$ on $L$, we build a hypercube $\scB_{\Lambda}^{\scO}$. If $\veps<\veps'$ and $|\veps'-\veps|_{L^1}=1$, then we pick $\theta^{\scO}_{\veps,\veps'}$ in this hypercube to be a cycle representing the top $\gr_{\qs}$-graded element in $\ve{\HF}^-_{\phi^{\scO}}(\bs_{\veps}^{E_{\veps}}, \bs_{\veps'}^{E_{\veps'}},\frs_{\scO})$, where $\qs$ is a compatible complete collection of basepoints in $\ws\cup \zs$. This homology class is described in Proposition~\ref{prop:top-degree-cycle}: If $\scO$ is positive on the component of $L$ which is incremented from $\veps$ to $\veps'$, then $\theta^{\scO}_{\veps,\veps'}$ is the cycle $\theta_\sigma^+$, and if $\scO$ is negative on this component, then $\theta^{\scO}_{\veps,\veps'}$ is the cycle $\theta_\tau^+$. Next, we fill the higher length chains in the hypercube $\scB_{\Lambda}^{\scO}$ by using the same strategy as \cite{MOIntegerSurgery}*{Lemma~8.6}. This procedure works by Propositions~\ref{prop:grading-preserving} and~\ref{prop:top-degree-cycle}. We then define the full hypercube $\scB_{\Lambda}$ by summing the chains of each $\scB^{\scO}$ for all $\scO\in \fro(L)$. Assuming the diagram $(\Sigma,\as,\scB_{\Lambda},\ws,\zs)$ is weakly admissible at each complete collection of basepoints $\qs\subset \ws\cup \zs$, we get a model for the link surgery complex $\cC_{\Lambda}(L):= \ve{\CF}^-(\as, \scB_{\Lambda})$.

Essentially the same construction works for a cube of generalized meridional systems. In this case, we obtain two hypercubes $\cA$ (of dimension $m$) and $\scB_{\Lambda}$ (of dimension $\ell+n$). Pairing these gives a hypercube of dimension $n+m+\ell$, which we can view as a hypercube of link surgery complexes of dimension $n+m$.

Note that if $\cA=\begin{tikzcd} \as_0\ar[r, "\theta"] & \as_1 \end{tikzcd}$, then $\ve{\CF}^-(\cA, \scB_{\Lambda})$ is a 1-dimensional hypercube, i.e. a mapping cone of a map $\Psi_{\a_0\to \a_1}^{\scB_{\Lambda}}$. We define this map to be the transition map for changing the alpha curves, assuming the diagram $(\Sigma,\as_1,\as_0,\scB_{\Lambda},\ws,\zs)$ is weakly admissible at each complete collection $\qs\cup \ws\cup \zs$. If the above diagram is not weakly admissible, we pick third collection $\as'$ so that $(\Sigma,\as',\as_i,\scB_{\Lambda}, \ws, \zs)$ is admissible for $i=0,1$, and we set
\[
\Psi_{\a_0\to \a_1}^{\scB_{\Lambda}}:=\Psi_{\a'\to \a_1}^{\scB_{\Lambda}}\circ \Psi_{\a_0\to \a'}^{\scB_{\Lambda}}
\]

 Similarly, if $\scB_{\Lambda}$ has dimension $\ell+1$, then we can view $\scB_{\Lambda}$ as a mapping cone $\scB_{\Lambda}=\begin{tikzcd} \scB_{\Lambda,0}\ar[r, "\theta"] & \scB_{\Lambda,1} \end{tikzcd}$. We then view $\ve{\CF}^-(\as, \scB_{\Lambda})$ as the mapping cone of a chain map $\Psi_{\a}^{\scB_{\Lambda,0}\to \scB_{\Lambda,1}}$, and we define this map to be the transition map for changing the beta attaching curves and the shadows. 
 
 By considering hypercubes of dimension $\ell+2$ (for example $\dim(\cA)=2$ and $\dim(\scB_{\Lambda})=\ell$; or $\dim(\cA)=1$ and $\dim(\scB_{\Lambda})=\ell+1$; or $\dim(\cA)=0$ and $\dim(\scB_{\Lambda})=\ell+2$) we see that these transition maps are functorial up to chain homotopy. In summary, we have the following:
 
 \begin{prop}
 \label{prop:distinguished-rectangles-alpha-beta}
 	\item
 	  \begin{enumerate}
 	  	\item If $(\Sigma,\as,\scB_{\Lambda},\ws,\zs)$ and $(\Sigma,\as',\scB_{\Lambda}',\ws,\zs)$ are generalized meridional systems of Heegaard diagrams with the same Heegaard surface, then there is a homotopy equivalence
 	\[
 	\Psi_{(\a, \scB_{\Lambda})\to (\a', \scB_{\Lambda}')}\colon \ve{\CF}^-(\as,\scB_{\Lambda})\to \ve{\CF}^-(\as', \scB_{\Lambda}').
 	\]
 	which is well-defined up to chain homotopy.
 	\item These maps satisfy
 	 	\[
 	\Psi_{(\a, \scB_{\Lambda})\to (\a'', \scB_{\Lambda}'')}\simeq  \Psi_{(\a', \scB_{\Lambda}')\to (\a'', \scB_{\Lambda}'')}\circ \Psi_{(\a, \scB_{\Lambda})\to (\a', \scB_{\Lambda}')}.
 	\]
 	\end{enumerate}
 	\end{prop}
 
\section{Naturality maps for meridional systems}
\label{sec:naturality}

  In this section, we define the transition maps and prove naturality for meridional basic systems. After defining the maps $\Psi_{\scH\to \scH'}$, we prove the following:

\begin{thm}
\label{thm:naturality-type-D} If $\scH$ and $\scH'$ are two meridional complete systems of surgery diagrams for $(Y,L,\Lambda)$, then there is a transition map
\[
\Psi_{\scH\to \scH'}\colon \cC_{\Lambda}(\scH)\to \cC_{\Lambda}(\scH'),
\]
well-defined up to chain homotopy, such that the following hold:
\begin{enumerate}
\item $\Psi_{\scH\to \scH}\simeq \id$.
\item $\Psi_{\scH'\to \scH''}\circ \Psi_{\scH\to \scH'}\simeq \Psi_{\scH\to \scH''}.$
\end{enumerate}
The same statements hold at the level of type-$D$ modules and morphisms over the algebra $\cL_\ell$ where $\ell=|L|$. 
\end{thm}

Many of the details in our proof follow from relatively standard arguments. The main new technical contribution in our present work concerns the interaction between the twisted coefficients, knot shadows, and our naturality maps.

\subsection{Overview of the naturality maps}

We now sketch the construction of the naturality maps. If $\scH=(\Sigma,\as,\scB_{\Lambda},\ws,\zs)$ is a meridional link surgery diagram for $(Y,L,\ws,\zs)$, then we write $\cC_{\Lambda}(\scH)$ for the link surgery hypercube 
\[
\cC_{\Lambda}(\scH):=\ve{\CF}^-(\as,\scB_{\Lambda}).
\]

We think of $\scH$ as being constructed from a Heegaard diagram for the sutured manifold $(M(L),\g^\circ)$ considered in Section~\ref{sec:meridional-diagrams}. We will apply the naturality theorem from \cite{JTNaturality} to this sutured manifold.

Suppose $\scH=(\Sigma,\as,\scB_{\Lambda},\ws,\zs)$ and $\scH'=(\Sigma,\as',\scB_{\Lambda}',\ws,\zs)$ are two meridional link surgery diagrams which share the same Heegaard surface. Suppose first that $(\Sigma,\as',\as,\scB_{\Lambda},\scB_{\Lambda}',\qs)$ is weakly admissible for all complete collections $\qs\subset \ws\cup \zs$.  Then we define
\[
\Psi_{\scH\to \scH'}:=\Psi_{\a \to \a'}^{\scB'_{\Lambda}}\circ \Psi_{\a}^{\scB_{\Lambda}\to \scB_{\Lambda}'}.
\]
Here, $\Psi_{\a}^{\scB_{\Lambda}\to \scB'_{\Lambda}}$ denotes the holomorphic polygon map $\mu_2^{\Tw}(-, \theta)$, where $\theta\colon \scB_{\Lambda}\to \scB_{\Lambda}'$ is the canonical hypercube morphism constructed in Section~\ref{sec:hypercube-hypercubes}. If $(\Sigma,\as',\as,\scB_{\Lambda}, \scB_{\Lambda}',\qs)$ is not admissible, we then we define
\[
\Psi_{\scH\to \scH'}:=\Psi_{\scH''\to \scH'}\circ \Psi_{\scH\to \scH''}
\]
for any diagram $\scH''$ sharing the same Heegaard surface, for which the pairs $(\scH,\scH'')$ and $(\scH'',\scH')$ satisfy the above admissibility hypotheses. Proposition~\ref{prop:distinguished-rectangles-alpha-beta} above verifies commutation under the distinguished rectangles  of type (1), which is to say commutation of maps for alpha equivalences and beta equivalences, from \cite{JTNaturality}*{Definition~2.29}. 

In Section~\ref{sec:stabilizations}, we define maps for stabilizations of the underlying Heegaard surface, and we show that these maps commute up to chain homotopy with the above maps for attaching curve equivalences. This will verify commutation under distinguished rectangles of type-(2) from \cite{JTNaturality}*{Definition~2.29}, namely commutation of alpha/beta equivalences and stabilizations. 
Commutation under the distinguished rectangles of type-(3), namely commutation of alpha/beta equivalences and diffeomorphisms, follows immediately from Proposition~\ref{prop:distinguished-rectangles-alpha-beta}, above.
Commutation under distinguished rectangles of type-(4), namely commutation of maps for two disjoint stabilizations, from \cite{JTNaturality}*{Definition~2.29} follows from the same argument as in the case of ordinary naturality; for this see \cite{JTNaturality}*{Section~9.2}. Commutation of distinguished rectanges of type (5), namely commutation of stabilizations and diffeomorphisms, can be derived from Lemma~\ref{lem:stabilization}.

 We consider diffeomorphisms which are isotopic to the identity rel boundary in Section~\ref{sec:continuity}, and we show that such diffeomorphisms induce the identity morphism on the link surgery complex in Proposition~\ref{prop:continuity}. This verifies the \emph{continuity} axiom from \cite{JTNaturality}*{Definition~3.23}.
 
 Finally \emph{handleswap invariance} is verified in Section~\ref{sec:handleswap}. 

\subsection{Almost complex structures}

Given a collection of attaching curves $\gs_0\dots, \gs_n$ on $\Sigma$, the holomorphic polygon counting maps count pseudoholomorphic curves in $\Sym^{g+k}(\Sigma)$ for a family of almost complex structures, indexed by points $x\in K_{n}$ and points $z\in \bD$. If $\bI=\{0,\dots, n\}$, we will write $J^{\bI}=(J^{\bI}_{x})_{x\in K_{n}}$ for this family. 

If $\bJ\subset \bI$ is an ordered subset with $1<|\bJ|<|\bI|$, we can restrict a family $(J_x)^{\bI}_{x\in K_n}$ to obtain a family $J^{\bJ}=(J^{\bJ}_x)_{x\in K_{|\bJ|-1}}$ of almost complex structures. We refer to $J^{\bJ}$ as the \emph{restriction} of $J^{\bI}$. 

It is helpful conceptually to think of a family $J^{\bI}$ as inducing a \emph{simplex} of almost complex structures. We can think of $J^{\bI}$ as being an almost complex structure associated to the $n$-simplex. If $\bJ\subset \bI$ is a $k$-simplex with $k>0$, then there an induced restriction of $J^{\bI}$ to the simplex $\bJ$. See Figure~\ref{fig:43} for a schematic of a 3-simplex of almost complex structures, and the restrictions to the four codimension 1 faces.

\begin{figure}[h]
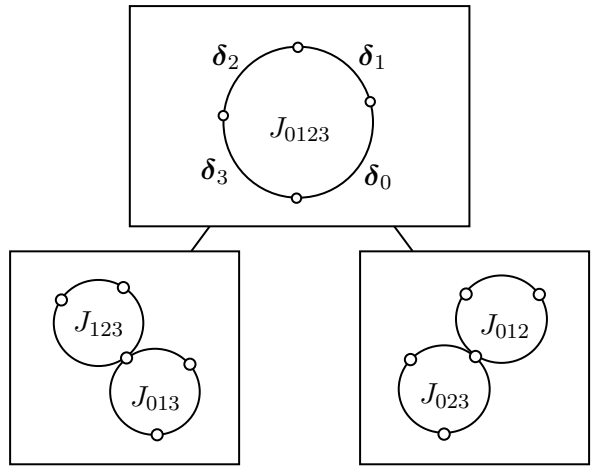
\caption{A  3-simplex of almost complex structures is a family $J_{0123}$ of almost complex structures indexed of $K_3\iso [0,1]$, which can be used to compute holomorphic polygons on a diagram $(\Sigma,\ds_0,\ds_1,\ds_2,\ds_3)$. Here we show the codimension 1 strata of $K_3$, and the four 2-simplices of almost complex structures, $J_{012}$, $J_{013}$, $J_{023}$ and $J_{123}$.}
\label{fig:43}
\end{figure}

\begin{define}
 An $n$-cube of almost complex structures on $\Sym^{g+k}(\Sigma)$ consists of the following: for each increasing sequence $\bI=\{\veps_0,\dots, \veps_j\}$ in $\bE_n$ with $j>0$, there is a choice of a domain dependent family of almost complex structures
 \[
J^{\bI}= (J^{\bI}_{z,x})_{z\in \bD ,x\in K_{j}}. 
 \]
  Furthermore, these sequences are compatible with respect to restriction, in the sense that if $\bJ\subset \bE_n$ and $\bI\subset \bE_n$ are increasing sequences of points with $\bJ\subset \bI$, then the restriction of $J^{\bI}$ to $\bJ$ is equal to $J^{\bJ}$.
\end{define}

\begin{rem}
Recall that for transversality of holomorphic curves, one picks families of almost complex structures which depend on a choice of points $z$ in the source curve and $x$ in the associahedron $K_j$. For our purposes, the source curve is a unit complex disk $\bD$. We typically suppress the dependence on $z$ in the notation.
\end{rem}

\begin{rem} We do not assign almost complex structures to $0$-simplices. The correct notion of a 0-simplex of almost complex structures would be a family of almost complex structures with respect to which boundary degenerations can be counted. 
\end{rem}

Of course, an $n$-cube of almost complex structures can be thought of as $n!$ simplices of almost complex structures which agree on their common faces. The definition naturally extends to general simplicial complexes:

\begin{define} Let $X$ be a simplicial complex. A \emph{$X$-complex of almost complex structures} consists of a collection of $k$-simplices of almost complex structures, for each $k>0$, so that the restrictions to common simplices agree.
\end{define}

Implicitly in the previous sections, we were picking $n$-cubes of almost complex structures when we considered $n$-dimensional hypercubes of attaching curves. In this section, we consider the dependence of hypercubes of attaching curves and hypercubes of chain complexes on these families of almost complex structures.

If $\cA$ is an $n$-dimensional hypercube of attaching curves, we view a choice of generic $n$-cube of almost complex structures as being part of the data of $\cA$.  We will write $(\cA,J_{\cA})$ when we wish to emphasize  the family.

In order to define morphisms, we will introduce the notion of a \emph{formal endormophism} $1\colon \as\to \as$. We interpret a holomorphic  $n$-gon which has $1$ as an input as being a holomorphic $(n-1)$-gon with an extra marked point along its boundary, corresponding to the input labeled $1$. We typically do not put any constraint on this marked point. For our purposes, the purpose of this marked point will be to give an extra parameter with respect to which the almost complex structures may vary. We will consider the case where our Heegaard diagrams $(\Sigma,\gs_1,\dots, \gs_n)$ have a repeated Lagrangian; that is, some $\gs_i$ may be equal to $\gs_j$ for $i\neq j$. However, we are only going to consider the case where these repeated Lagrangians are adjacent, and the morphisms between them are the formal endomorphisms $1$.

Given two hypercubes $(\cA,J_{\cA})$ and $(\cA',J_{\cA}')$ which have the same set of underlying attaching cubes, we can therefore define a notion of a morphism of twisted complexes
\[
\theta\colon (\cA,J_{\cA})\to (\cA',J_{\cA'})
\]
If $\dim(\cA)=n$, then $\theta$ consists of a choice of $(n+1)$-dimensional cube of almost complex structures, as well as a collection of Floer chains. We typically will assume that the components of $\theta$ which preserve the cube point are the formal endomorphisms $1$.

\begin{example} Let 
\[
\cA=\Cone(\begin{tikzcd} \as_0\ar[r, "\xs"]& \as_1\end{tikzcd}) \quad \text{and} \quad \cA'= \Cone(\begin{tikzcd} \as_0\ar[r, "\xs'"]& \as_1\end{tikzcd}).
\]
We assume that $\cA$ uses the almost complex structure $J$ and $\cA'$ uses the almost complex structure $J'$. Consider a morphism of the following form:
\[
F=\begin{tikzcd} \as_0 \ar[d, "1"]\ar[r, "\xs"] \ar[dr, "\eta"] & \as_1 \ar[d, "1"]\\ 
\as_0\ar[r, "\xs'"] & \as_1
\end{tikzcd}.
\]
Additionally, the morphism $F$ includes the data of two non-cylindrical almost complex structures $J_{1'10}$ and $J_{1'0'0}$ on $[0,1]\times \R$, as well as a third cylindrical almost complex structure $J''$ on $[0,1]\times \R$. The almost complex structure $J_{1'10}$ restricts to $J$ and $J''$, while $J_{1'0'0}$ restricts to $J'$ and $J''$. We think of $J_{1'10}$ as being a 3-simplex of almost complex structures, but where the restriction to the edge $1'1$ containing no non-trivial information. The map $\mu_2(1,\xs)$ counts $J_{1'10}$-holomorphic disks of Maslov index 0 and with input $\xs$. We can think of such disks as having a marked point along the $\as_1$ boundary, which is mapped to $\bT_{\a_1}$, but which has no other constraint. Up to conformal equivalence, we can identify the domain with $[0,1]\times \R$ so that this marked point is mapped to  $(0,0)$. In particular the marked point gives no constraint, and we can think of $J_{1'10}$  as an almost complex structure for counting holomorphic disks which is not invariant under the $\R$-action on $[0,1]\times \R$. Similarly $\mu_2(\xs',1)$ counts $J_{1'0'0}$-holomorphic disks with input $\xs'$. The hypercube condition is that
\[
\mu_2^{J_{1'10}}(1,\xs)+\mu_2^{J_{1'0'0}}(\xs',1)=\d^{J''}(\eta). 
\]
\end{example}

We now discuss almost complex structures for the complexes $\ve{\CF}^-(\cA,\cB)$, when $\cA$ and $\cB$ are hypercubes. Let $\dim(\cA)=n$ and $\dim(\cB)=m$. In this case, we need a family of almost complex structures indexed by the  \emph{abstract join}, $C(n,m)$, of an $n$-cube and an $m$-cube. 
To ease the notation, let us write $C(n)$ for the cube, viewed as a simplicial complex with simplices consisting of increasing sequences $(\veps_1,\dots, \veps_j)$ in $\bE_n$. If $X$ and $Y$ are simplicial complexes, we write $\join(X,Y)$ for their join. We recall that the simplicies of $\join(X,Y)$ consist of pairs $(s,s')$, where $s$ and $s'$ are (potentially empty) simplices of $X$ and $Y$. In our present case, $C(n,m)$ is a simplicial complex whose vertices are identified with points of  $\bE_n\sqcup \bE_m$. The $k$-simplices in $C(n,m)$ consist of sequences $(\veps_1,\dots, \veps_i, \nu_1,\dots, \nu_j)$ where $(\veps_1,\dots, \veps_i)$ is an increasing sequence in $\bE_n$ and $(\nu_1,\dots, \nu_j)$ is an increasing sequence in $\bE_m$. (The case that one of $(\veps_1,\dots, \veps_i)$ or $(\nu_1,\dots, \nu_j)$ is the empty sequence is allowed).
As before, we do not assign almost complex structures to 0-simplices. Each 1-simplex is a complex structure for counting holomorphic disks.  More generally, each $k$-simplex is an almost complex structure for counting holomorphic $(k+1)$-gons.

To construct the complex $\ve{\CF}^-(\cA,\cB)$, we need to pick a $C(n,m)$-complex of almost complex structures $J$, which have $J_{\cA}$ and $J_{\cB}$ as restrictions. We now show that up to homotopy equivalence, the choice of extension of $J_{\cA}$ and $J_{\cB}$ does not affect the complex $\ve{\CF}^-_{J}(\cA,\cB).$

\begin{lem} Fix hypercubes of attaching curves $\cA$, $\cB$ with cubes of almost complex structures $J_\cA$ and $J_{\cB}$. The complex $\ve{\CF}^-_{J}(\cA, \cB)$ is independent of the choice of extension $J$ of $J_{\cA}$ and $J_{\cB}$, up to chain homotopy equivalence.
\end{lem} 
\begin{proof}
Let $J_0$ and $J_1$ be two such families of almost complex structures, which both extend $J_\cA$ and $J_{\cB}$. 

We consider the formal identity morphism $1\colon (\cA,J_{\cA})\to (\cA, J_{\cA})$. This requires a choice of an $(n+1)$-cube of almost complex structures. We choose this family to be a product, in the sense that we assume that each simplex in this family for counting almost complex structures where $1$ is an input is the pullback of a family of almost complex structures under the natural forgetful map $K_{j+1}\to K_{j}$ obtained by forgetting this marked point corresponding to $1$. Call this family $J_{\cA\times I}$.

We now pick a $C(n+1,m)$-complex of almost complex structures $\tilde{J}$ which restricts to $J_{\cA\times I}$ on $C(n+1)$, and which restricts to $J_0$ on $\join(C(n)\times \{0\},C(m))$, and restricts to $J_1$ on $\join(C(n)\times \{1\},C(m))$. We use $\tilde{J}$ to define a map $\mu^{\Tw}(1,-)$ from $\ve{\CF}^-_{J_0}(\cA,\cB)$ to $\ve{\CF}^-_{J_1}(\cA,\cB)$.  Write $F$ for this map. It is straightforward to see that $F$ is a chain map.

To see that this is a homotopy equivalence, observe that we can perform a similar construction to define a map $G$ in the opposite direction. We observe that each component of $G\circ F$ which preserves the cube points is chain homotopic to the identity, as can be seen by a simple argument involving 2-parameter families of almost complex structures for counting holomorphic disks. Let $H\colon \ve{\CF}^-_{J_0}(\cA,\cB)\to \ve{\CF}^-_{J_0}(\cA,\cB)$ be the map obtained by summing these homotopies, each of which preserves the cube point. Then $G\circ F+\d(H)$ is equal to $\bI+L$, where $L$ strictly lowers the cube point. This is a chain isomorphism since its inverse is given by $\bI+L+L^2+\cdots$. Therefore $(\bI+L+L^2+\cdots)\circ G$ is a homotopy inverse of $F$.
\end{proof}

\begin{lem} Suppose that $(\cA,J_{\cA})$ and $(\cA',J_{\cA'})$ are hypercubes of handleslide equivalent attaching curves, such that furthermore $\cA$ and $\cA'$ have the same dimension and share the same attaching curves. Then there is a formal endomorphism $\bI\colon (\cA,J_{\cA})\to (\cA',J_{\cA'})$ whose cube-coordinate preserving components consist of the formal endomorphism $1$, and such that $\bI$ is a cycle.
\end{lem}
\begin{proof} The proof follows from the same inductive argument which shows that for every weakly admissible cube of attaching curves $(\as_{\veps})_{\veps\in \bE_n}$, there is a collection of chains $\theta_{\veps,\veps'}$ (ranging over all $\veps<\veps'$) which make the diagram $(\as_{\veps}, \theta_{\veps,\veps'})_{\veps\in \bE_n}$ into a hypercube of attaching curves. See \cite{MOIntegerSurgery}*{Lemma~8.6}. 
\end{proof}

\subsection{Stabilizations}
\label{sec:stabilizations}

We now discuss the naturality maps for index $(1,2)$-stabilizations. If $\scH$ is a meridional basic system of Heegaard diagrams, then we can form a new meridional basic system $\scH'$ by taking the connected sum of the Heegaard surface with a 2-torus. On this torus, we place a small alpha and beta curve, $\a_0$ and $\b_0$, respectively. If $\cA$ and $\cB$ are the two hypercubes for $\scH$, then we form hypercubes $\cA'$ and $\cB'$ by adjoining to each $\as_{\veps}$ in $\cA$ a copy of a small translation of $\a_0$. Similarly to each $\bs_{\nu}$ in $\cB$, we adjoin a copy of $\b_0$. We assume that each pair of translates of $\a_0$ intersect pairwise in two points, and similarly for the translates of the $\b_0$. We define the morphism $\theta_{\veps',\veps}'\in \ve{\CF}^-(\as'_{\veps'},\as_{\veps}')$ of $\cA'$ to be
\[
\theta_{\veps',\veps}':=\theta_{\veps',\veps}\times \theta^+,
\]
where $\theta^+$ is the top degree intersection point of the two copies of $\a_0$. We define $\cB'$ similarly.

We define a hypercube morphism
\[
\varsigma\colon \ve{\CF}^-(\cA,\cB)\to \ve{\CF}^-(\cA',\cB').
\]
The map $\varsigma$ has no higher length components, and is given  by the equation
\[
\varsigma(\xs)=\xs\otimes p,
\]
where $\{p\}=\a_0\cap \b_0$, extended equivariantly over the variables. The map $\varsigma$ has an inverse given by the equation $\varsigma^{-1}(\xs\otimes p)=\xs$. 

Some care must be taken regarding almost complex structures. If for each $\as_\veps\in \cA$ and $\bs_\nu\in \cB$ the diagram $(\Sigma,\as_{\veps},\bs_\nu)$ represents a rational homology 3-sphere, then the map may be defined as above for any almost complex structure which is sufficiently stretched along the connected sum neck. This is because the argument that $\varsigma$ is a chain map involves a limiting argument which is applied to each neck individually. When $(\Sigma,\as_{\veps},\bs_{\nu})$ is a rational homology 3-sphere, only finitely many classes of polygons are relevant so we may apply the argument to all classes at once. For the more general case, we can either consider a stronger version of admissibility to ensure that only finitely many polygon classes contribute (cf. \cite{OSDisks}*{Definition~4.10}), or we can consider infinitely stretched necks, as in \cite{HHSZExact}*{Section~10}.

\begin{lem}
\label{lem:stabilization} For suitable almost complex structures, the maps $\varsigma$ and $\varsigma^{-1}$ are chain maps.
\end{lem}
\begin{proof} This follows from the following well-known neck-stretching argument. We focus on the case that $(\Sigma,\as_{\veps},\bs_\nu)$ represents a rational homology 3-sphere, so that only finitely many classes contribute to the hypercube structure map. The argument using infinitely stretched necks, for general 3-manifolds, is a straightforward modification.

 We consider classes of polygons on the Heegaard diagram 
\[
\cD'=(\Sigma\# \bT^2, \as_{\veps_j}',\dots, \as_{\veps_1}',\bs_{\nu_1}',\dots, \bs_{\nu_k}',\ws)
\]
 and also on
\[
\cD=(\Sigma, \as_{\veps_j},\dots, \as_{\veps_1}, \bs_{\nu_1},\dots, \bs_{\nu_k}, \ws).
\] 
See Figure~\ref{fig:naturality_30} for the stabilization region of $\cD'$. Let 
\[
\psi\in \pi_2(\theta_{\veps_j,\veps_{j-1}},\dots, \theta_{\veps_2,\veps_1}, \xs,\theta_{\nu_1,\nu_2},\dots, \theta_{\nu_{k-1}, \nu_k})
\]
 be a class of polygons of Maslov index $j+k-3$. Then for sufficiently stretched almost complex structure, we have a mod 2 equivalence of counts (see \cite[Proposition 6.5, Example 6.7]{HHSZNaturality}):
\[
\# \cM_{J}(\psi)\equiv \sum_{\substack{\psi_0\in \pi_2(\theta^+,\dots, \theta^+, p,\theta^+,\dots, \theta^+)\\ n_p(\psi)=n_{p_0}(\psi)}} \#\cM_{J(T)}(\psi\# \psi_0) \pmod 2. 
\]
In the above, the sum is taken over classes $\psi_0$ on a diagram obtained by performing small translations to $(\bT^2,\a_0,\dots, \a_0, \b_0,\dots, \b_0)$.
\end{proof}

\begin{figure}[h]
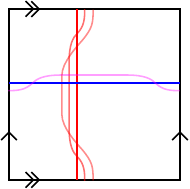
\caption{Stabilizing a Heegaard multi-diagram.}
\label{fig:naturality_30}
\end{figure}

\subsection{Continuity}
\label{sec:continuity}

We now address \emph{continuity} of the naturality maps:
\begin{prop}\label{prop:continuity} If $\phi\colon (Y,L)\to (Y,L)$ is a diffeomorphism of strongly framed links which is isotopic to the identity through diffeomorphisms of strongly framed links, then $\cX(\phi)\simeq \id$.
\end{prop}

To prove Proposition~\ref{prop:continuity}, we first address the claim when the diffeomorphism is the flow of a time dependent Hamiltonian vector field, for some symplectic form on $\Sigma$. Later we will generalize this to arbitrary isotopies on $\Sigma$. We must describe continuation maps in the context of the surgery hypercube. An alternate strategy would be to consider small translate theorems (cf. \cite{HHSZExact}*{Section~11});  however, these are challenging in our present setting due to admissibility issues.

We consider link surgery hypercube data $(\Sigma,\as,\scB_{\Lambda},\ws,\zs)$ and $(\Sigma,\as,\scB'_{\Lambda}, \ws, \zs)$, where $\scB'_{\Lambda}$ is obtained by applying small Hamiltonian translations to $\as$ and $\scB'_{\Lambda}$, respectively. We assume that the translations are \emph{small} in the following sense. We assume that there is a collection of points $\ve{r}$ on $\Sigma$ so that $\ve{r}$ contains exactly one point from each component of
\[
\Sigma\setminus (\as\cup \bigcup_{\bs\in \scB_{\Lambda}} \bs),
\]
and the Hamiltonian translations to do not cross over $\ve{r}$. 

We will write $H$ for the time dependent Hamiltonian.  We let $\Phi_t$  be the flow of $H$ on $\Sigma$. We assume that $\Phi_t$ is fixed for $t<-1$ and $t>1$, and write $\Phi=\Phi_t$ for $t>1$.

We now describe a version of the continuation map
\[
\Gamma\colon \ve{\CF}^-(\as, \scB_{\Lambda})\to \ve{\CF}^-(\as', \scB_{\Lambda}').
\]

The version of the continuation map $\Gamma$ that we describe will be slightly specialized to our present case. In particular, our analysis assumes that the differential on each  $\ve{\CF}^-_{\phi^{\scO}}(\bs_{\veps},\bs_{\veps'}, \frs_{\scO})$ vanishes. This can easily be achieved by assuming that $\scB_{\Lambda}$ is constructed using the Heegaard diagrams described in Section~\ref{sec:meridional-complete-systems}.  

The map $\Gamma$ counts holomorphic polygons somewhat similarly to the hypercube differential on $\ve{\CF}^-(\as, \scB_{\Lambda})$. We consider a unit complex disk $D$ with $n+2$ boundary punctures $(n\ge 0)$, two of which are distinguished. One of the distinguished punctures is called the \emph{distinguished input}. The other distinguished puncture is the \emph{output}. We may identify $D$ with $[0,1]\times \R$ (minus several boundary points) so that the distinguished input is $-\infty$ and the distinguished ouput is $+\infty$. 

We are interested in counting holomorphic polygons $u$ which have domain $[0,1]\times \R$ (minus several boundary punctures) and codomain $\Sym^{g+|\ws|-1}(\Sigma)$. Suppose that $\theta_{\veps_0,\veps_1},\dots, \theta_{\veps_{n-1}, \veps_n}$ is a sequence of composable morphisms in $\scB_{\Lambda}'$, and $\xs\in \ve{\CF}^-(\as, \scB_{\Lambda})$ and $\ys\in \ve{\CF}^-(\as', \scB_{\Lambda}')$. (Note that these spaces are canonically isomorphic as groups).  We let $t_1<\dots<t_n$ denote an increasing sequence of real numbers. Let $c_0,\dots, c_n$ denote the connected components of $\{1\}\times (\R\setminus \{t_1,\dots, t_n\})$. Writing $\ve{t}=(t_1,\dots, t_n)$, let $D_{\ve{t}}$ denote
\[
D_{\ve{t}}=[0,1]\times \R\setminus (\{1\}\times \{t_1,\dots, t_n\}).
\]
We write $K_{n+1}^{(0,n)}$ for the space of such punctured strips. Note that, of course, $K_{n+1}^{(0,n)}$ can be identified with the ordinary associahedron $K_{n+1}$.

For a fixed sequence $\ve{t}=(t_1,\dots, t_n)$,  we define the moduli space of \emph{$H$-shifted holomorphic $(n+2)$-gons} with parameter $\ve{t}$, denoted $\cM^{H}_{\ve{t}}(\phi)$ as the space of holomorphic maps
\[
u\colon D_{\ve{t}}\to \Sym^{g+|\ws|-1}(\Sigma)
\]
which have the property that 
\[
u(0,t)\in \bT_{\Phi_t(\a)} \text{ for all } t\in \R,
\]
and
\[
u(1,t)\in \bT_{\Phi_t(\b_{\veps_i})} \text{ for all }  t\in c_i.
\]
We assume that at puncture $t_i$, the holomorphic curve is asymptotic to the image of $\theta_{\veps_{i-1}, \veps_i}$ under $\Phi_{t_i}$. In the above, $\phi$ denotes a homology class of such polygons.

We define the moduli space of $H$-shifted polygons to be the parametrized moduli space
\[
\bigcup_{\substack{\ve{t}=(t_1,\dots, t_n) \\ t_1<\cdots <t_n}} \cM_{\ve{t}}^H(\phi).
\]

We define 
\[
\Gamma\colon \ve{\CF}^-(\as, \scB_{\Lambda})\to \ve{\CF}^-(\as', \scB_{\Lambda}')
\]
 to count such $H$-shifted holomorphic $(n+2)$-gons, ranging over sequences $\theta_{\veps_0,\veps_1},\dots, \theta_{\veps_{n-1}, \veps_n}$ of composable morphisms in $\scB_{\Lambda}$ and over homology classes $\phi$ which have $\mu(\phi)=-n$. The reader may compare this to the construction in \cite{SeidelFukaya}*{Section~10c}.

We first address admissibility, in a manner similar to \cite{OSDisks}*{Lemma~7.4}:

\begin{lem}
\label{lem:continuation-map-admissibility} Suppose that the diagram $(\Sigma,\as,\scB_{\Lambda},\qs)$ is weakly admissible for each complete collection $\qs\subset \ws\cup \zs$. Then for each $N$, only finitely many holomorphic curves have algebraic contribution to the map $\Gamma$ modulo the ideal $(U_1,\dots, U_\ell)^N$.
\end{lem}
\begin{proof} There is a bijection
\[
I\colon \pi_2^H(\xs, \theta_{\veps_0,\veps_1},\dots, \theta_{\veps_{n-1},\veps_n}, \ys)\to \pi_2(\xs,\theta_{\veps_0,\veps_1},\dots, \theta_{\veps_{n-1},\veps_n},\ys)
\]
gotten by sending an $H$-skewed polygon class $\phi$ to the class $\phi'$ determined by the equation
\[
\phi'(s,t)=(\Phi_t^{-1}\circ\phi)(s,t).
\]
Note that isotopy $\Phi_t$ fixes the points $\ve{r}\subset \Sigma$ and therefore 
\[
n_{\qs}(I(\phi))=n_{\qs}(\phi)
\]
for each class of $H$-skewed polygons $\phi$. By weak admissibility, there are only finitely many classes in $\pi_2(\xs,\theta_{\veps_0,\veps_1},\dots, \theta_{\veps_{n-1},\veps_n})$ with total multiplicity $n_{\qs}(\phi)<N$, and therefore the same holds for the classes in $\pi_2^H$. For each $\Spin^c$ structure contributing to the surgery complex, there is a complete collection of basepoints $\ve{q}$ such that the module morphism associated to the homology class $\phi$ (output by the holomorphic polygon maps) lies in $(U_1,\dots, U_\ell)^N$ if $n_{\ve{q}}(\phi)\ge N$. The main claim therefore follows.
\end{proof}

The diffeomorphism $\phi\colon (Y,L)\to (Y,L)$ induces a canonical map
\[
\cX^{\can}(\phi)\colon \ve{\CF}^-_J(\as, \scB_{\Lambda})\to \ve{\CF}^-_{\phi_*(J)}(\as', \scB_{\Lambda}')
\]
for each collection of familiy of almost complex structures $J$. The map $\cX^{\can}(\phi)$ is given by $\xs\mapsto \phi(\xs)$, extended equivariantly over all of the variables. We now compare the continuation map $\Gamma$ with the map $\cX^{\can}(\phi)$:

\begin{lem}
\label{lem:Gamma-Psi-can}
If $\phi=\Phi_1$ is the time 1 flow of a 1-parameter family of Hamiltonian vector fields, then the map $\Gamma$ is chain homotopic to the canonical diffeomorphism map $\cX^{\can}(\phi)$.
\end{lem}
\begin{proof} For the proof it  is helpful to enlarge our notion of pseudoholomorphic curves slightly. We recall the Cauchy-Riemann equation for maps from $[0,1]\times \R$ to $\Sym^{g+|\ws|-1}(\Sigma)$:
\[
\d_t u-J \d_s u=0.
\]
It is helpful to consider a perturbation of this equation by an inhomogeneous term. If $X(t)$ is the vector field on $\Sym^{g+|\ws|-1}(\Sigma)$ is the time dependent vector field from the time dependent Hamiltonian used to construct $\Phi_1$, then we consider solutions $u'$ to the equation
\[
\d_t u'-J \d_s u'=X_t(u'(s,t)).
\]
(See \cite{SeidelFukaya}*{Section~8.f} for more on the inhomogeneous Cauchy-Riemann equation).

Note that if $u\colon D_{\ve{t}}\to \Sym^{g+|\ws|-1}(\Sigma)$ is a holomorphic polygon mapping to $\bT_{\a}$ and the $\bT_{\b_{\veps}}$, for example one counted by the differential on $\ve{\CF}^-(\as,\scB_{\Lambda})$, then the map $u'(s,t)=\Phi_t\circ u(s,t)$ satisfies the above inhomogeneous Cauchy-Riemann equation. In this way, we obtain a diffeomorphism between the $H$-shifted moduli space $\cM_{\ve{t}}^H(\psi)$ (defined using the inhomogeneous Cauchy-Riemann equations) and the corresponding ordinary moduli space $\cM(\psi)$ (defined using the homogeneous Cauchy-Riemann equation, as we use to define the differential on $\ve{\CF}^-_{J}(\as,\scB_{\Lambda})$).
 Note that this quickly implies that $\Gamma$ is the identity, if we use the inhomogeneous Cauchy-Riemann equations. The map counts index $-n$ curves. For the non-skewed moduli spaces of $(n+2)$-gons $\cM(\psi)$, ordinary transversality implies that there will be representatives only if $n=0$ and the classes $\psi$ are the constant disks. Under the identification of $\cM(\psi)$ with $\cM_{\ve{t}}^H(\psi)$, the constant disks move an intersection point $\ve{x}$ to its image under the map $\phi$.

Computing the map $\Gamma$ for a different choice of almost complex structures (or rather, counting solutions to the ordinary Cauchy-Riemann equations rather than the equations which are perturbed by an inhomogeneous term) will give a chain homotopic map, and therefore we conclude that $\Gamma\simeq \cX^{\can}(\phi)$.
\end{proof}

\begin{rem} In Lipshitz's cylindrical reformulation \cite{LipshitzCylindrical}, the above argument has the following analog. We take an ordinary family of almost complex structures $J$ on the spaces $\Sigma\times D_{\ve{t}}$ for counting holomorphic $(n+2)$-gons. We define an automorphism $\Psi$ of $\Sigma\times D_{\ve{t}}$ via the formula $(x,s,t)\mapsto (\Phi_t(x), s,t)$. Then there is an identification between $J$-holomorphic curves for counting ordinary holomorphic polygons on $(\Sigma,\as,\scB_{\Lambda})$, and counting $\Psi_*(J)$ holomorphic curves with dynamic boundary conditions.
\end{rem}

We now prove another result about the map $\Gamma$: 

\begin{prop}
\label{prop:continuity-techical}
 There is a chain homotopy $\Gamma\simeq \Psi_{\a\to \a'}^{\scB'_{\Lambda}}\circ \Psi_{\a}^{\scB_{\Lambda}\to \scB_{\Lambda}'}$, where the latter composition denotes the composition of the polygon counting maps.
\end{prop}

The proof of Proposition~\ref{prop:continuity-techical} requires several steps. The first step is to deform the relative heights in the polygon counts of the alpha and beta dynamic regions of the boundary. We let $\Gamma_{\a\to \a'}^{\scB_{\Lambda}'}$ denote the continuation map where only $\as$ changed under the isotopy $\Phi_t$. We let $\Gamma_{\a}^{\scB_{\Lambda}\to \scB_{\Lambda}'}$ be the continuation map where only $\scB_{\Lambda}$ is affected by the isotopy.

\begin{lem}\label{lem:homotopy-Gamma-Gamma-a-b} There is a chain homotopy
\[
\Gamma\simeq \Gamma_{\a\to \a'}^{\scB_{\Lambda}'}\circ \Gamma_{\a}^{\scB_{\Lambda}\to \scB_{\Lambda}'}.
\]
\end{lem}
\begin{proof} The chain  homotopy is obtained by counting 1-dimensional moduli spaces where the alpha deformation region has height $t\in (0,\infty)$ relative to the beta deformation region, as in Figure~\ref{fig_naturality_39.pdf_tex}. Counting the ends of moduli spaces quickly gives the chain homotopy. 
\end{proof}

\begin{figure}[h]
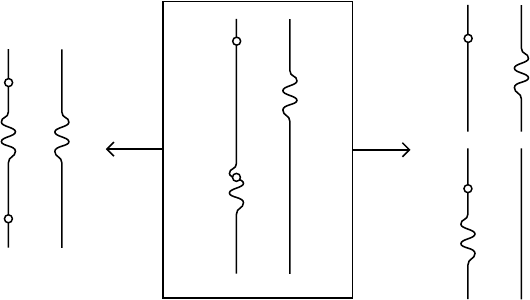
\caption{The moduli spaces considered in Lemma~\ref{lem:homotopy-Gamma-Gamma-a-b}. In the center we illustrate the moduli space where the height of the dynamic regions takes values in $t\in (0,1)$. On the left and right we illustrate two important types of codimension 1 ends of this moduli space.}
\label{fig_naturality_39.pdf_tex}
\end{figure}

To prove Proposition~\ref{prop:continuity-techical}, we deform the moduli spaces used to define $\Gamma$ by introducing a parameter $r\in (0,1)$ which corresponds to a scaling factor on the speed of the deformation $\Phi_t$; that is, we replace $\Phi_t$ with $\Phi_{t/r}$ as $r\to 0$.  Before counting the ends of the moduli spaces which are parametrized by $r\in (0,1)$, we introduce a new type of curve which appears in the compactification.

\begin{define} Let $\bs_{0},\dots, \bs_n$ be a sequence of attaching curves, and let $\theta_{0,1},\dots, \theta_{n-1,n}$ be intersection points where $\theta_{i,i+1}\in \bT_{\b_i}\cap \bT_{\b_{i+1}}$. Let $\Phi_t$ be the flow of a time-dependent Hamiltonian which is locally constant in $t$ for $t\in (-\infty,-1]\cup [1,\infty)$. We write $\phi=\Phi_1$. Let $\ys\in \bT_{\b_0}\cap \bT_{\phi(\b_n)}$. An \emph{$H$-skewed holomorphic $n$-gon} consists of a $J$-holomorphic map
\[
u\colon \bH\setminus \ve{t}\to \Sym^{g+\ws-1}(\Sigma),
\]
where $\bH=(-\infty,0]\times \R$  and $\ve{t}$ is a collection of $n$ boundary punctures. We assume the following are also satisfied. If $c_0,\dots, c_n$ denote the $n+1$ components of $\d H$ (ordered in ascending $t$-values) then we assume that if the point $(0,t)\in c_i$ then
\[
u(0,t)\in \bT_{\Phi_t(\b_i)}. 
\]
We assume that if the $i$-th boundary puncture of $H$ occurs at height $t\in \R$, then $u$ is asymptotic at that puncture to $\Phi_t(\theta_{i,i+1})$.
\end{define}

\begin{figure}[h]
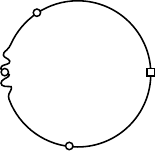
\caption{A schematic of an $H$-skewed holomorphic polygon. Inputs are denoted with circles, and the output is denoted with a square. The region of the source $u$ has dynamic boundary conditions is denoted with a squiggle. We are identifying the half plane with a unit disk so that the output puncture is identified with $-1$.}
\label{fig:37}
\end{figure}

Suppose now that $\scB_{\Lambda}$ is the hypercube built in the construction of the surgery formula, and furthermore assume that $\scB_{\Lambda}$ is algebraically rigid. We now show that by counting $H$-skewed holomorphic $n$-gons, we can construct the natural homotopy equivalence $\theta^H \colon \scB_{\Lambda}\to \scB_{\Lambda}'$, where $\scB_{\Lambda}'$ denotes the image of $\scB_{\Lambda}$ under $\Phi$. More precisely, we define $\theta^H_{\veps,\veps'}$ to be the sum over all increasing sequences $\veps=\veps_0<\cdots <\veps_n=\veps'$ of skewed holomorphic $n$-gons with inputs from $\scB_{\Lambda}$ and with Maslov index $1-n$. The case that $n=1$ is allowed.

Note that in the context of the hypercube $\scB_{\Lambda}$, we additionally need to specify the module maps between the underlying spaces of the local systems $E_{\veps}$.
These are given as the composition of the monodromy maps for each edge of the polygon, and the module-morphisms for input.  Note that since the skewed polygons used to define $\theta^H$ all have boundary only on beta curves, the total Alexander grading shift of these maps is always 0 since 
\[
\# (\d_{\b}(\psi)\cap S_i)=\#(\d(\psi)\cap S_i)=0.
\]
 In particular, the module morphism for a skewed monogon is always the identity from $E_{\veps}$ to $E_{\veps}$.

If $\scB_{\Lambda}$ is the beta cube from a cube of generalized meridional systems of Heegaard diagrams of a link $(Y,L)$, we will say that $\scB_{\Lambda}$ is \emph{algebraically rigid} if it has the property that for each complete collection $\ps\subset \ws\cup \zs$, and each pair $\bs_{\veps},\bs_{\veps'}$ with $\veps<\veps'$, there are exactly $2^{g(\Sigma)+|L|-1}$ intersection points $\xs\in \bT_{\b_{\veps}}\cap \bT_{\b_{\veps'}}$ for which $\frs_{\ps}(\xs)$ is torsion. For such diagrams, there are no intersection points with $\frs_{\ps}(\xs)$ torsion which have greater $\gr_{\ps}$-grading than the canonical generator.

\begin{lem}
\label{lem:thetaH-monogons} If $\scB_{\Lambda}$ is algebraically rigid and $H$ is a time dependent Hamiltonian so that its flow $\Phi_t$ is sufficiently small (to achieve weak admissibility as in Lemma~\ref{lem:continuation-map-admissibility}), then the morphism
\[
\theta^H\colon \scB_{\Lambda}\to \scB_{\Lambda}'
\]
is a cycle (i.e. a morphism of hypercubes of Lagrangians), and furthermore coincides with the canonical equivalence described in Section~\ref{sec:homology-groups} up to chain homotopy. 
\end{lem}
\begin{proof} We first show that $\theta^H$ is a hypercube morphism. This is proven by counting the ends of generic 1-parameter families of $H$-skewed $n$-gons. The compactification of such 1-parameter families consist of a pair consisting of one $H$-skewed $k$-gon with $1\le k\le n$ and one non-skewed $(n-k+2)$-gon. If a pair of such curves appears in the boundary of our 1-dimensional moduli space, generically both curves have to be part of 0-dimensional moduli spaces. That is, the skewed $k$-gon must have Maslov index $k-1$, and the non-skewed polygon should have index $1-n-k$.

Note that in such a degeneration, there is a \emph{lower level} and an \emph{upper level}. The lower level is the one which contains the final output puncture. See Figure~\ref{fig_naturality_38.pdf_tex}. We claim that the curves where the skewed polygon are on the lower level come in canceling pairs. To see this, we observe that if the skewed polygon occurs on the lower level, then the polygon on the upper level occurs on some $\Phi_t(\scB_{\Lambda})$. This diagram of Lagrangians is a hypercube for all $t$, and is furthermore algebraically rigid for all $t$, being the diffeomorphic image of something which is algebraically rigid. In particular, the upper level must consist of either a holomorphic bigon or a holomorphic triangle. The count of disks cancels because each chain in $\Phi_t(\scB_{\Lambda})$ is a cycle for all $t$. Furthermore, the count of triangles cancel because 
\[
\mu_2(\Phi_t(\theta_{\veps,\veps+e_i}^{\circ}), \Phi_t(\theta_{\veps+e_i,\veps+e_i+e_j}^{\circ'}))=\mu_2(\Phi_t(\theta_{\veps,\veps+e_j}^{\circ'}), \Phi_t(\theta_{\veps+e_j,\veps+e_i+e_j}^{\circ})).
\]

\begin{figure}[h]
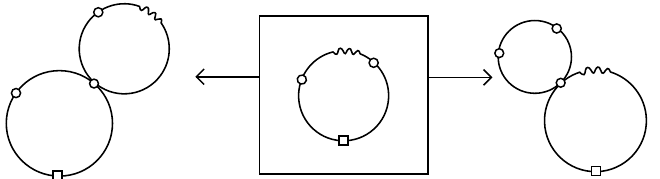
\caption{Examples of codimension 1 degenerations (left and right) of 1-parameter family of $H$-skewed polygons (center). In the figure, the lower levels are the bottom-most levels. In the left figure, the lowest level is an ordinary holomorphic polygon. In the right figure, the lowest level has dynamic boundary conditions.}
\label{fig_naturality_38.pdf_tex}
\end{figure}

The remaining configurations, for which the upper level is a skewed polygon, are exactly the hypercube relations for $\theta^H$.

We now address the claim that $\theta^H$ is homotopic to the canonical equivalence described in Section~\ref{sec:homology-groups}. It follows from Lipshitz's work \cite{LipshitzCylindrical}*{Proof of Proposition~11.4} that the length 1 components are chain homotopic to the top degree cycles $\theta_{\b_{\veps}, \Phi(\b_{\veps})}^+$. (Lipshitz proves the claim when $\bs_{\veps}$ and $\phi(\bs_{\veps})$ are standard translations of each other. The general case follows from this by using a continuation map argument). For the rest of the claim, it suffices to show that the chains of $\theta^H$ have the expected Maslov grading and have Alexander grading 0. The claim about Maslov gradings is proven by observing that the skewed polygon maps are homogeneously graded. Furthermore, the output chains will have Alexander grading 0 since we are using a single collection of link shadows for $L$, and each boundary component of the skewed polygons are beta, so the total intersection number of the boundary of skewed $n$-gon with a link shadow will be 0. 
\end{proof}

\begin{lem}
\label{lem:continuation=polygons}
There are chain homotopies 
\[
\Gamma_{\a\to \a'}^{\scB_{\Lambda'}}\simeq \Psi_{\a\to \a'}^{\scB_{\Lambda}'}\quad \text{and} \quad \Gamma_{\a}^{\scB_{\Lambda}\to \scB_{\Lambda}'}\simeq \Psi_{\a}^{\scB_{\Lambda}\to \scB_{\Lambda}'}.
\]
\end{lem}
\begin{proof} We focus on the proof that $ \Gamma_{\a}^{\scB_{\Lambda}\to \scB_{\Lambda}'}\simeq \Psi_{\a}^{\scB_{\Lambda}\to \scB_{\Lambda}'}$. The proof is an elaboration on the Lipshitz construction to relate the continuation map on Floer complexes with a holomorphic triangle map \cite{LipshitzCylindrical}*{Proposition~11.4}. The idea is again to introduce a new parameter $r\in (0,1)$ which rescales the dynamic dynamic region of the boundary; equivalently, we replace the isotopy $\Phi_{t}$ with $\Phi_{t/r}$. We call such polygons $H^r$-skewed, and we consider the 1-dimensional moduli space of such curves, viewed as a parametrized moduli space over $r\in (0,1)$. The limiting curves appearing as $r\to 1$ correspond exactly to the map $\Gamma_{\a}^{\scB_{\Lambda}\to \scB_{\Lambda}'}$. Curve breakings which occur at $r\in (0,1)$ either cancel with other breakings (if there is one level which contains both alpha and beta curves) or contribute to the chain homotopy (if there are two levels which contain alpha and beta curves). The ends which appear as $r\to 0$ correspond exactly to the map $\mu_2^{\Tw}(-,\theta^H)$, where $\theta^H\in \ve{\CF}^-(\scB_{\Lambda},\scB_{\Lambda}')$ is the morphism considered in Lemma~\ref{lem:thetaH-monogons}, which is gotten by counting $H$-skewed holomorphic polygons with boundary components on the beta curves. By definition, the map $\mu_2^{\Tw}(-,\theta^H)$ is equal to the map   $\Psi_{\a}^{\scB_{\Lambda}\to \scB_{\Lambda}'}$. The same argument works for the map $\Psi_{\a\to \a'}^{\scB_{\Lambda}\to \scB_{\Lambda}'}$.
\end{proof}

We recall a final lemma about diffeomorphism groups of compact manifolds and the subgroup generated by Hamiltonian isotopies.. Let $\Sigma$ be a compact surface with a (potentially empty) set of basepoints $\ps\subset \Sigma$. Write $\Diff_0(\Sigma,\ps)$ denote the group of diffeomorphisms which are the identity on some neighborhood of $\ps$, and which are smoothly isotopic to the identity through such diffeomorphisms. If $H$ is a Hamiltonian function (i.e. a smooth map $H\colon \Sigma\to \R$), write $X^{H}$ for the vector field given by
\[
d H=\iota_{X^H}(\omega).
\]
Write $\Theta^{H}_t\colon \Sigma\to \Sigma$ for the time $t$ flow of the vector field $X^H$. If $H$ is zero in a neighborhood of $\ve{p}$, then $\Theta^H_t\in \Diff_0(\Sigma,\ps)$.

\begin{lem}
\label{lem:Hamiltonian-isotopy} Suppose that $(\Sigma,\ps)$ is a pointed, compact surface as above and let $\phi\in \Diff_0(\Sigma,\ps)$. Then there is a sequence of symplectic forms $\omega_1,\dots, \omega_n$ on $\Sigma$ and Hamiltonian functions $H_1,\dots, H_{n}$ which vanish on a neighborhood of $\ps$, so that
\[
\phi=\Theta_1^{H_n}\circ \cdots \circ \Theta^{H_n}_1.
\]
\end{lem}
\begin{proof} We consider the subgroup $G\subset \Diff_0(\Sigma,\ps)$ generated by the time 1 flows of such Hamiltonian vector fields. We observe that this is a normal subgroup by the following argument. Suppose $H$ is a Hamiltonian function and $\omega$ is a symplectic form on $\Sigma$. Then
\[
\Phi \circ \Theta^H_t\circ \Phi^{-1}
\]
is the time $t$ flow of the vector field $(\Phi^{-1})^*(X^H)$. Furthermore, $(\Phi^{-1})^*(X^H)$ is the Hamiltonian vector field for the function $H\circ \Phi^{-1}$ with respect to the symplectic form $(\Phi^{-1})^*(\omega)$. Therefore $G$ is a normal subgroup. It follows from Thurston's work \cite{Thurston_Simple}*{Theorem~1} applied to the manifold $\Sigma\setminus \ps$ that  $\Diff_0(\Sigma,\ps)$ is a simple group, and therefore $G=\Diff_0(\Sigma,\ps)$. See \cite{Banyaga_Diffeo_Book}*{Theorem~2.1.1} for a detailed account of Thurston's theorem.
\end{proof}

We can now prove Proposition~\ref{prop:continuity}: 

\begin{proof}[Proof of Proposition~\ref{prop:continuity}] Since the maps for diffeomorphisms of the Heegaard surface are functorial under composition, by using Lemma~\ref{lem:Hamiltonian-isotopy} we may reduce the claim to when $\phi$ is a small Hamiltonian isotopy for some symplectic structure on $\Sigma$. By definition, the map $\cX(\phi)$ is the composition of the canonical diffeomorphism map and a polygon counting map:
\[
\phi_*:=\Psi_{\a'\to \a}^{\scB_{\Lambda}}\circ \Psi_{\a'}^{\scB_{\Lambda}'\to \scB_{\Lambda}}\circ \cX^{\can}(\phi).
\]
By Lemmas~\ref{lem:Gamma-Psi-can}, ~\ref{lem:homotopy-Gamma-Gamma-a-b} and \ref{lem:continuation=polygons}, there is a chain homotopy
\[
\Psi_{\a'\to \a}^{\scB_{\Lambda}}\circ \Psi_{\a'}^{\scB_{\Lambda}'\to \scB_{\Lambda}}\simeq (\cX^{\can}(\phi))^{-1}
\]
so $\cX(\phi)\simeq (\cX(\phi)^{\can})^{-1}\circ \cX(\phi)^{\can}=\id$. This completes the proof.
\end{proof}

\subsection{Simple handleswaps}
\label{sec:handleswap}
We now address simple handleswap invariance. 

\begin{figure}[h]
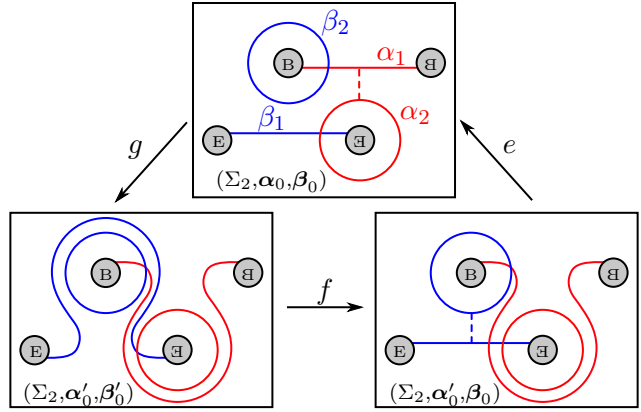
\caption{A simple handleswap, on the level of Heegaard diagrams.}
\label{fig_naturality_31}
\end{figure}

\begin{prop} Simple handleswap loops induce trivial monodromy on the link surgery complexes.
\end{prop}
\begin{proof} The proof for the link surgery formula is an extension of the proof in the setting of the ordinary Heegaard Floer complexes \cite{JTNaturality}*{Theorem~9.30}. To simplify the notation we consider a parallel situation involving general hypercubes of attaching curves, though the same analysis easily implies the analogous result for the link surgery formula. We consider two hypercubes of attaching curves $\cA$ and $\cB$  on a pointed Heegaard surface $(\Sigma,\ws)$ and a distinguished point $p\in \Sigma$, not on any of the attaching curves. We assume that the curves of $\cA$ are pairwise handleslide equivalent, and similarly for $\cB$. We stabilize $\Sigma$ with a genus 2 surface $\Sigma_2$. In $\Sigma_2$, we consider four sets of attaching curves, which we denote $\as_0,$ $\as_0'$, $\bs_0$ and $\bs_0'$, as shown in Figure~\ref{fig_naturality_31}. Write $\cA_0$, $\cA_0'$, $\cB_0$ and $\cB_0'$ for the hypercubes on $\Sigma \# \Sigma_0$ obtained by stabilizing. The handleswap loop gives a diffeomorphism map
\[
\phi_*\colon \ve{\CF}^-(\cA_0,\cB_0)\to \ve{\CF}^-(\cA_0,\cB_0).
\]
This map is the composition of a canonical diffeomorphism map, followed by two polygon counting maps
\begin{equation}
\begin{tikzcd}[column sep=2cm]
\ve{\CF}^-(\cA'_0,\cB_0')\ar[r, "{\mu_2^{\Tw}(-,\theta_{\cB_0',\cB_0})}"] &\ve{\CF}^-(\cA_0',\cB_0) \ar[r, "{\mu_2^{\Tw}(\theta_{\cA_0,\cA_0'},-)}"]& \ve{\CF}^-(\cA_0,\cB_0)
\end{tikzcd}
\label{eq:composition-b00-a00-handleswap}
\end{equation}
See \cite{ZemBordered}*{Section~13.1} for a related discussion of this construction in the context of basepoint moving diffeomorphisms.

We now stretch the neck on the stabilization region. Using the stabilization results from \cite{HHSZNaturality}*{Proposition~6.5}, we obtain identifications
\[
\ve{\CF}^-(\cA,\cB)\iso \ve{\CF}^-(\cA_0,\cB_0)\iso \ve{\CF}^-(\cA_0',\cB_0)\iso \ve{\CF}^-(\cA_0',\cB_0'),
\]
where here we are identifying attaching curves with their small translations in the various complexes.
Furthermore, the stabilization result of \cite{HHSZNaturality}*{Proposition~6.5} identifies the two maps in Equation~\eqref{eq:composition-b00-a00-handleswap} with the corresponding polygon counts on the unstabilized diagrams. The unstabilized Heegaard diagrams encode a small isotopy of the alpha and beta curves, so by continuity, Proposition~\ref{prop:continuity}, the induced map is chain homotopic to the identity.
\end{proof}

\section{Proof of the equivariant surgery formula}
\label{sec:proof-equivariant-surgery-formula}

In this section, we prove the equivariant surgery formula, which we restate below:

\begin{thm}
\label{thm:equivariant-surgery-theorem-text-body}
 Let $(Y,L,\Lambda)$ be a Morse framed link, and let $\phi\colon(Y,L)\to (Y,L)$ be a diffeomorphism of strongly framed links. Let $\Phi\colon Y_{\Lambda}(L)\to Y_{\Lambda}(L)$ denote the diffeomorphism induced by $\phi$. Then there is a homotopy equivalence
\[
\Gamma\colon \ve{\CF}^-(Y_{\Lambda}(L))\to \cC_{\Lambda}(Y,L)
\]
which intertwines $\CF(\Phi)$ with the the map $\cC(\phi)$, defined to be the composition of two maps:
\[
\begin{tikzcd}[column sep=.7cm]
\cX_{\Lambda}(Y,L)^{\cL_\ell}\boxtimes {}_{\cL_\ell} \cD^\ell
	 \ar[r, "\cX(\phi)\boxtimes \bI"] &
\cX_{\Lambda}(Y,L)^{\cL_\ell}\boxtimes {}_{\cL_\ell} [\scE_\phi]^{\cL_\ell}\boxtimes {}_{\cL_\ell} \cD^\ell
	\ar[r, "\bI\boxtimes \Omega_\phi"]
&\cX_{\Lambda}(Y,L)^{\cL_\ell}\boxtimes {}_{\cL_\ell} \cD^\ell.
\end{tikzcd}
\]
\end{thm}

\begin{proof}
Given a meridional surgery diagram  
\[
\cH=(\Sigma, \as, \bs_{\Lambda}, (\bs_{\veps})_{\veps\in \bE_\ell}, \cS)
\] 
for $(Y,L,\Lambda)$, there is an induced Heegaard diagram \[
\cH_{Y_{\Lambda}(L)}:=(\Sigma,\as,\bs_\Lambda,\ws)
\]
 for $Y_{\Lambda}(L)$, as well as a cube of Heegaard diagrams 
\[
\cH_{(Y,L)}= (\Sigma,\as, \{\bs_{\veps}\}_{\veps\in \bE_\ell}, \ws, \zs, \cS).
\]
We can fill the cube of beta attaching curves $\cH_{(Y,L)}$ to construct a hypercube $\scB_{\Lambda}$ which can compute the link surgery formula.

Note that given a meridional surgery diagram $\cH$, there is a canonical morphism
\[
\Theta\colon \bs_{\Lambda}\to \scB_{\Lambda}
\]
which is given by a pair $\langle \theta_{\Lambda,0}, \Delta\rangle$. Here, $\theta_{\Lambda,0}$ denotes the canonical intersection point of $\bT_{\b_{\Lambda}}\cap \bT_{\b_{0,\dots, 0}}$ and $\Delta$ is the map from $\bF[U_1,\dots, U_\ell]$ to $E_0\otimes \cdots \otimes E_0$ which sends $1$ to $1\otimes \cdots \otimes 1$, and is $\bF[U_1,\dots, U_\ell]$-equivariant. 

We obtain a polygon counting map
\[
\Gamma:=\mu_2^{\Tw}(-,\Theta) \colon \ve{\CF}^-(\cH_{Y_{\Lambda}(L)})\to \cC_{\Lambda}(\cH_{(Y,L)}).
\]
It follows from \cite{ZemExactTriangle}*{Section~6} that $\Theta$ is a cycle, so that $\Gamma$ is a chain map, and furthermore that $\Gamma$ is homotopy equivalence of chain complexes.

Applying the diffeomorphism $\phi$, we obtain another meridional surgery diagram $\phi \cH$. We obtain, tautologically, a hypercube
\begin{equation}
\begin{tikzcd}[column sep=1.5cm]
\ve{\CF}^-(\cH_{Y_{\Lambda}(L)})
	\ar[r, "\Gamma"]
	\ar[d, "\Phi_{\can}"]
&
\cC_{\Lambda}(\cH_{(Y,L)})
	\ar[d, "\cC^{\can}(\phi)"]
\\
\ve{\CF}^-((\phi\cH)_{Y_{\Lambda}(L)})
	\ar[r, "\Gamma"]
&
\cC_{\Lambda}((\phi\cH)_{(Y,r_JL)}).
\end{tikzcd}
\label{eq:canonical-cube}
\end{equation}
(Recall that $J\subset L$ is the image of the components of $L$ on which $\phi$ is orientation reversing). 

Note that on the right hand side of Equation~\eqref{eq:canonical-cube}, the map $\Phi_{\can}$ permutes the indices of $L$ by a permutation $\varrho$. Therefore, if $\cX_{\Lambda}(\cH_{(Y,L)})^{\cL_\ell}$ and $\cX_{\Lambda}((\phi\cH)_{(Y,r_J L)})^{\cL_\ell}$ are the two associated type-$D$ modules, then we can write $\cC^{\can}(\phi)$ as a composition of the following type-$D$ module maps
\begin{equation}
\begin{tikzcd}[row sep=.4cm]
\cX(\cH_{(Y,L)})^{\cL_\ell}\boxtimes {}_{\cL_\ell} \cD^{\ell}
	\ar[d, "\cX^{\can}(\phi)\boxtimes \bI"] 
\\ \cX((\phi\cH)_{(Y,r_J L)})^{\cL_\ell}\boxtimes {}_{\cL_\ell} [\bI_{\varrho^{-1}}]^{\cL_\ell}\boxtimes {}_{\cL_\ell} \cD^{\ell}
	\ar[d, "\bI\boxtimes \varrho_*"]
\\\cX((\phi\cH)_{(Y,r_J L)})^{\cL_\ell}\boxtimes {}_{\cL_\ell} \cD^{\ell}.
\end{tikzcd}
\label{eq:decompose-phi-can}
\end{equation}

Next, we stack the hypercube in Equation~\eqref{eq:canonical-cube} with the following hypercube
\begin{equation}
\begin{tikzcd}[column sep=1.5cm]
\ve{\CF}^-((\phi\cH)_{Y_{\Lambda}(L)})
	\ar[r, "\Gamma"]
	\ar[d, equals]
&
\cC_{\Lambda}((\phi\cH)_{(Y,r_JL)})
	\ar[d, "\bI\boxtimes\varpi_J"]
\\
\ve{\CF}^-((r_J\phi\cH)_{Y_{\Lambda}(L)})
	\ar[r, "\Gamma"]
&
\cC_{\Lambda}((r_J\phi\cH)_{(Y,L)}).
\end{tikzcd}
\label{eq:hypercube-change-orientation}
\end{equation}
This hypercube is obtained by switching the roles of $w_i$ and $z_i$ on the top and bottom components, and switching the orientation of the shadows in the sublink $J\subset L$. Note that this has no effect on the left hand complexes, since $w_i$ and $z_i$ are immediately adjacent on these complexes (and disks are counted with factors of $U_i^{n_{w_i}(\psi)}$). The map $\Gamma$ on the top and bottom rows count the same holomorphic polygons. Compare Lemma~\ref{lem:change-orientation}. 

Next, we observe that $r_J \phi \cH$ is a meridional Heegaard surgery diagram for $(Y, L, \Lambda)$. Therefore, by Proposition~\ref{prop:meridional-diagram-moves}, there is a sequence of Heegaard moves for meridional surgery diagrams from $r_J \phi \cH$ to $\cH$.

We observe that if $\cH^1$ and $\cH^2$ are meridional systems of Heegaard diagrams which differ by one of the moves in Proposition~\ref{prop:meridional-diagram-moves}, then there is a hypercube of chain complexes
\[
\begin{tikzcd}[column sep=1.5cm]
\ve{\CF}^-(\cH^1_{Y_{\Lambda}(L)})
	\ar[r, "\Gamma"]
	\ar[dr,dashed]
	\ar[d, "\Psi_{\cH^1_{Y_{\Lambda}(L)}\to \cH^2_{Y_{\Lambda}(L)}}",swap]
&
\cC_{\Lambda}(\cH_{(Y,L)}^1)
	\ar[d, "\Psi_{\cH_{(Y,L)}^1\to \cH_{(Y,L)}^2}"]
\\
\ve{\CF}^-(\cH^2_{Y_{\Lambda}(L)})
	\ar[r, "\Gamma"]
&
\cC_{\Lambda}(\cH_{(Y,L)}^2).
\end{tikzcd}
\]
Stacking these hypercubes for a sequence of moves from $r_J\phi \cH$ to $\cH$ and then compressing gives a hypercube of the following form:
\begin{equation}
\begin{tikzcd}[column sep=1.5cm]
\ve{\CF}^-((r_J\phi\cH)_{Y_{\Lambda}(L)})
	\ar[r, "\Gamma"]
	\ar[dr,dashed]
	\ar[d,swap, "\Psi_{(r_J\phi\cH)_{Y_{\Lambda}(L)}\to \cH_{Y_{\Lambda}(L)}}"]
&
\cC_{\Lambda}((r_J \phi \cH)_{(Y,L)})
	\ar[d, "\Psi_{(r_J \phi \cH)_{(Y,L)}\to \cH_{(Y,L)}}"]
\\
\ve{\CF}^-(\cH_{Y_{\Lambda}(L)})
	\ar[r, "\Gamma"]
&
\cC_{\Lambda}(\cH_{(Y,L)}).
\end{tikzcd}
\label{eq:hypercube-system-moves}
\end{equation}

We stack the hypercubes in Equations~\eqref{eq:canonical-cube}, \eqref{eq:hypercube-change-orientation} and  \eqref{eq:hypercube-system-moves}. Along the left side of the diagram, we obtain the naturality map
\[
\CF^-(\Phi)\colon \ve{\CF}^-(Y_{\Lambda}(L))\to \ve{\CF}^-(Y_{\Lambda}(L)),
\]
and along the right side we obtain the map $\cC(\phi):=(\bI\boxtimes \Omega_\phi)\circ (\cX(\phi)\boxtimes \bI)$ in the statement.
\end{proof}

\bibliographystyle{custom}
\def\MR#1{}
\bibliography{biblio}

\end{document}

%% file: fig_naturality_41.pdf_tex
\begingroup%
  \makeatletter%
  \providecommand\color[2][]{%
    \errmessage{(Inkscape) Color is used for the text in Inkscape, but the package 'color.sty' is not loaded}%
    \renewcommand\color[2][]{}%
  }%
  \providecommand\transparent[1]{%
    \errmessage{(Inkscape) Transparency is used (non-zero) for the text in Inkscape, but the package 'transparent.sty' is not loaded}%
    \renewcommand\transparent[1]{}%
  }%
  \providecommand\rotatebox[2]{#2}%
  \newcommand*\fsize{\dimexpr\f@size pt\relax}%
  \newcommand*\lineheight[1]{\fontsize{\fsize}{#1\fsize}\selectfont}%
  \ifx\svgwidth\undefined%
    \setlength{\unitlength}{164.53739196bp}%
    \ifx\svgscale\undefined%
      \relax%
    \else%
      \setlength{\unitlength}{\unitlength * \real{\svgscale}}%
    \fi%
  \else%
    \setlength{\unitlength}{\svgwidth}%
  \fi%
  \global\let\svgwidth\undefined%
  \global\let\svgscale\undefined%
  \makeatother%
  \begin{picture}(1,0.42888651)%
    \lineheight{1}%
    \setlength\tabcolsep{0pt}%
    \put(0,0){\includegraphics[width=\unitlength,page=1]{fig_naturality_41.pdf}}%
    \put(0.18615536,0.18210157){\color[rgb]{0,0,0}\makebox(0,0)[t]{\lineheight{1.25}\smash{\begin{tabular}[t]{c}$K$\end{tabular}}}}%
    \put(0,0){\includegraphics[width=\unitlength,page=2]{fig_naturality_41.pdf}}%
    \put(0.81039696,0.18292529){\color[rgb]{0,0,0}\makebox(0,0)[t]{\lineheight{1.25}\smash{\begin{tabular}[t]{c}$rK$\end{tabular}}}}%
    \put(0,0){\includegraphics[width=\unitlength,page=3]{fig_naturality_41.pdf}}%
  \end{picture}%
\endgroup%

%% file: fig_naturality_28.pdf_tex
\begingroup%
  \makeatletter%
  \providecommand\color[2][]{%
    \errmessage{(Inkscape) Color is used for the text in Inkscape, but the package 'color.sty' is not loaded}%
    \renewcommand\color[2][]{}%
  }%
  \providecommand\transparent[1]{%
    \errmessage{(Inkscape) Transparency is used (non-zero) for the text in Inkscape, but the package 'transparent.sty' is not loaded}%
    \renewcommand\transparent[1]{}%
  }%
  \providecommand\rotatebox[2]{#2}%
  \newcommand*\fsize{\dimexpr\f@size pt\relax}%
  \newcommand*\lineheight[1]{\fontsize{\fsize}{#1\fsize}\selectfont}%
  \ifx\svgwidth\undefined%
    \setlength{\unitlength}{111.97942016bp}%
    \ifx\svgscale\undefined%
      \relax%
    \else%
      \setlength{\unitlength}{\unitlength * \real{\svgscale}}%
    \fi%
  \else%
    \setlength{\unitlength}{\svgwidth}%
  \fi%
  \global\let\svgwidth\undefined%
  \global\let\svgscale\undefined%
  \makeatother%
  \begin{picture}(1,0.58443796)%
    \lineheight{1}%
    \setlength\tabcolsep{0pt}%
    \put(0,0){\includegraphics[width=\unitlength,page=1]{fig_naturality_28.pdf}}%
    \put(0.84442977,0.31291067){\makebox(0,0)[lt]{\lineheight{1.25}\smash{\begin{tabular}[t]{l}$S_i$\end{tabular}}}}%
    \put(0.33239167,0.26964357){\makebox(0,0)[t]{\lineheight{1.25}\smash{\begin{tabular}[t]{c}$w_i$\end{tabular}}}}%
    \put(0.54663629,0.26964371){\makebox(0,0)[t]{\lineheight{1.25}\smash{\begin{tabular}[t]{c}$z_i$\end{tabular}}}}%
    \put(0,0){\includegraphics[width=\unitlength,page=2]{fig_naturality_28.pdf}}%
    \put(0.46000718,0.38759612){\color[rgb]{0,0,1}\makebox(0,0)[lt]{\lineheight{1.25}\smash{\begin{tabular}[t]{l}$\b_{0,i}$\end{tabular}}}}%
  \end{picture}%
\endgroup%

%% file: fig_naturality_42.pdf_tex
\begingroup%
  \makeatletter%
  \providecommand\color[2][]{%
    \errmessage{(Inkscape) Color is used for the text in Inkscape, but the package 'color.sty' is not loaded}%
    \renewcommand\color[2][]{}%
  }%
  \providecommand\transparent[1]{%
    \errmessage{(Inkscape) Transparency is used (non-zero) for the text in Inkscape, but the package 'transparent.sty' is not loaded}%
    \renewcommand\transparent[1]{}%
  }%
  \providecommand\rotatebox[2]{#2}%
  \newcommand*\fsize{\dimexpr\f@size pt\relax}%
  \newcommand*\lineheight[1]{\fontsize{\fsize}{#1\fsize}\selectfont}%
  \ifx\svgwidth\undefined%
    \setlength{\unitlength}{152.55770946bp}%
    \ifx\svgscale\undefined%
      \relax%
    \else%
      \setlength{\unitlength}{\unitlength * \real{\svgscale}}%
    \fi%
  \else%
    \setlength{\unitlength}{\svgwidth}%
  \fi%
  \global\let\svgwidth\undefined%
  \global\let\svgscale\undefined%
  \makeatother%
  \begin{picture}(1,0.92428078)%
    \lineheight{1}%
    \setlength\tabcolsep{0pt}%
    \put(0,0){\includegraphics[width=\unitlength,page=1]{fig_naturality_42.pdf}}%
    \put(0.51599335,0.67379281){\color[rgb]{0,0,1}\makebox(0,0)[lt]{\lineheight{1.25}\smash{\begin{tabular}[t]{l}$\b_0$\end{tabular}}}}%
    \put(0.7999499,0.30319516){\color[rgb]{0.37647059,0,1}\makebox(0,0)[lt]{\lineheight{1.25}\smash{\begin{tabular}[t]{l}$\b_1$\end{tabular}}}}%
    \put(0.51656734,0.37076368){\makebox(0,0)[lt]{\lineheight{1.25}\smash{\begin{tabular}[t]{l}$\theta_\sigma^{+}$\end{tabular}}}}%
    \put(0.47585094,0.08897463){\makebox(0,0)[rt]{\lineheight{1.25}\smash{\begin{tabular}[t]{r}$\theta_\tau^+$\end{tabular}}}}%
    \put(0,0){\includegraphics[width=\unitlength,page=2]{fig_naturality_42.pdf}}%
    \put(0.43365371,0.19661367){\makebox(0,0)[rt]{\lineheight{1.25}\smash{\begin{tabular}[t]{r}$w$\end{tabular}}}}%
    \put(0.55314476,0.19661367){\makebox(0,0)[lt]{\lineheight{1.25}\smash{\begin{tabular}[t]{l}$z$\end{tabular}}}}%
    \put(0,0){\includegraphics[width=\unitlength,page=3]{fig_naturality_42.pdf}}%
    \put(0.66852718,0.48465066){\color[rgb]{1,0,0}\makebox(0,0)[lt]{\lineheight{1.25}\smash{\begin{tabular}[t]{l}$\b_\lambda$\end{tabular}}}}%
    \put(0,0){\includegraphics[width=\unitlength,page=4]{fig_naturality_42.pdf}}%
    \put(0.77716137,0.13009547){\makebox(0,0)[lt]{\lineheight{1.25}\smash{\begin{tabular}[t]{l}$\cS$\end{tabular}}}}%
  \end{picture}%
\endgroup%

%% file: fig_naturality_32.pdf_tex
\begingroup%
  \makeatletter%
  \providecommand\color[2][]{%
    \errmessage{(Inkscape) Color is used for the text in Inkscape, but the package 'color.sty' is not loaded}%
    \renewcommand\color[2][]{}%
  }%
  \providecommand\transparent[1]{%
    \errmessage{(Inkscape) Transparency is used (non-zero) for the text in Inkscape, but the package 'transparent.sty' is not loaded}%
    \renewcommand\transparent[1]{}%
  }%
  \providecommand\rotatebox[2]{#2}%
  \newcommand*\fsize{\dimexpr\f@size pt\relax}%
  \newcommand*\lineheight[1]{\fontsize{\fsize}{#1\fsize}\selectfont}%
  \ifx\svgwidth\undefined%
    \setlength{\unitlength}{194.70133676bp}%
    \ifx\svgscale\undefined%
      \relax%
    \else%
      \setlength{\unitlength}{\unitlength * \real{\svgscale}}%
    \fi%
  \else%
    \setlength{\unitlength}{\svgwidth}%
  \fi%
  \global\let\svgwidth\undefined%
  \global\let\svgscale\undefined%
  \makeatother%
  \begin{picture}(1,0.40038525)%
    \lineheight{1}%
    \setlength\tabcolsep{0pt}%
    \put(0,0){\includegraphics[width=\unitlength,page=1]{fig_naturality_32.pdf}}%
    \put(0.4473107,0.00507722){\makebox(0,0)[lt]{\lineheight{1.25}\smash{\begin{tabular}[t]{l}$\lambda$\end{tabular}}}}%
    \put(0.06895707,0.02536257){\makebox(0,0)[rt]{\lineheight{1.25}\smash{\begin{tabular}[t]{r}$w_1$\end{tabular}}}}%
    \put(0.96544055,0.01558012){\makebox(0,0)[t]{\lineheight{1.25}\smash{\begin{tabular}[t]{c}$w_2$\end{tabular}}}}%
    \put(0.50877322,0.3831199){\makebox(0,0)[t]{\lineheight{1.25}\smash{\begin{tabular}[t]{c}$w'$\end{tabular}}}}%
    \put(0.69257744,0.21161757){\makebox(0,0)[lt]{\lineheight{1.25}\smash{\begin{tabular}[t]{l}$\lambda_2$\end{tabular}}}}%
    \put(0.3236261,0.19133221){\makebox(0,0)[rt]{\lineheight{1.25}\smash{\begin{tabular}[t]{r}$\lambda_1$\end{tabular}}}}%
    \put(0,0){\includegraphics[width=\unitlength,page=2]{fig_naturality_32.pdf}}%
  \end{picture}%
\endgroup%

%% file: fig_naturality_40.pdf_tex
\begingroup%
  \makeatletter%
  \providecommand\color[2][]{%
    \errmessage{(Inkscape) Color is used for the text in Inkscape, but the package 'color.sty' is not loaded}%
    \renewcommand\color[2][]{}%
  }%
  \providecommand\transparent[1]{%
    \errmessage{(Inkscape) Transparency is used (non-zero) for the text in Inkscape, but the package 'transparent.sty' is not loaded}%
    \renewcommand\transparent[1]{}%
  }%
  \providecommand\rotatebox[2]{#2}%
  \newcommand*\fsize{\dimexpr\f@size pt\relax}%
  \newcommand*\lineheight[1]{\fontsize{\fsize}{#1\fsize}\selectfont}%
  \ifx\svgwidth\undefined%
    \setlength{\unitlength}{226.77165354bp}%
    \ifx\svgscale\undefined%
      \relax%
    \else%
      \setlength{\unitlength}{\unitlength * \real{\svgscale}}%
    \fi%
  \else%
    \setlength{\unitlength}{\svgwidth}%
  \fi%
  \global\let\svgwidth\undefined%
  \global\let\svgscale\undefined%
  \makeatother%
  \begin{picture}(1,0.25)%
    \lineheight{1}%
    \setlength\tabcolsep{0pt}%
    \put(0,0){\includegraphics[width=\unitlength,page=1]{fig_naturality_40.pdf}}%
  \end{picture}%
\endgroup%

%% file: fig_naturality_33.pdf_tex
\begingroup%
  \makeatletter%
  \providecommand\color[2][]{%
    \errmessage{(Inkscape) Color is used for the text in Inkscape, but the package 'color.sty' is not loaded}%
    \renewcommand\color[2][]{}%
  }%
  \providecommand\transparent[1]{%
    \errmessage{(Inkscape) Transparency is used (non-zero) for the text in Inkscape, but the package 'transparent.sty' is not loaded}%
    \renewcommand\transparent[1]{}%
  }%
  \providecommand\rotatebox[2]{#2}%
  \newcommand*\fsize{\dimexpr\f@size pt\relax}%
  \newcommand*\lineheight[1]{\fontsize{\fsize}{#1\fsize}\selectfont}%
  \ifx\svgwidth\undefined%
    \setlength{\unitlength}{164.12870776bp}%
    \ifx\svgscale\undefined%
      \relax%
    \else%
      \setlength{\unitlength}{\unitlength * \real{\svgscale}}%
    \fi%
  \else%
    \setlength{\unitlength}{\svgwidth}%
  \fi%
  \global\let\svgwidth\undefined%
  \global\let\svgscale\undefined%
  \makeatother%
  \begin{picture}(1,0.41670575)%
    \lineheight{1}%
    \setlength\tabcolsep{0pt}%
    \put(0,0){\includegraphics[width=\unitlength,page=1]{fig_naturality_33.pdf}}%
    \put(0.15072673,0.21490765){\makebox(0,0)[lt]{\lineheight{1.25}\smash{\begin{tabular}[t]{l}$1$\end{tabular}}}}%
    \put(0.82331272,0.31941623){\makebox(0,0)[t]{\lineheight{1.25}\smash{\begin{tabular}[t]{c}$1$\end{tabular}}}}%
    \put(0.82331272,0.07184082){\makebox(0,0)[t]{\lineheight{1.25}\smash{\begin{tabular}[t]{c}$-1$\end{tabular}}}}%
    \put(0.31022258,0.37062562){\makebox(0,0)[lt]{\lineheight{1.25}\smash{\begin{tabular}[t]{l}$\g_0$\end{tabular}}}}%
    \put(0.9051915,0.18292007){\makebox(0,0)[lt]{\lineheight{1.25}\smash{\begin{tabular}[t]{l}$\g_0'$\end{tabular}}}}%
  \end{picture}%
\endgroup%

%% file: fig_naturality_43.pdf_tex
\begingroup%
  \makeatletter%
  \providecommand\color[2][]{%
    \errmessage{(Inkscape) Color is used for the text in Inkscape, but the package 'color.sty' is not loaded}%
    \renewcommand\color[2][]{}%
  }%
  \providecommand\transparent[1]{%
    \errmessage{(Inkscape) Transparency is used (non-zero) for the text in Inkscape, but the package 'transparent.sty' is not loaded}%
    \renewcommand\transparent[1]{}%
  }%
  \providecommand\rotatebox[2]{#2}%
  \newcommand*\fsize{\dimexpr\f@size pt\relax}%
  \newcommand*\lineheight[1]{\fontsize{\fsize}{#1\fsize}\selectfont}%
  \ifx\svgwidth\undefined%
    \setlength{\unitlength}{223.48488154bp}%
    \ifx\svgscale\undefined%
      \relax%
    \else%
      \setlength{\unitlength}{\unitlength * \real{\svgscale}}%
    \fi%
  \else%
    \setlength{\unitlength}{\svgwidth}%
  \fi%
  \global\let\svgwidth\undefined%
  \global\let\svgscale\undefined%
  \makeatother%
  \begin{picture}(1,0.81435296)%
    \lineheight{1}%
    \setlength\tabcolsep{0pt}%
    \put(0,0){\includegraphics[width=\unitlength,page=1]{fig_naturality_43.pdf}}%
    \put(0.37592107,0.49167791){\color[rgb]{0,0,0}\makebox(0,0)[rt]{\lineheight{1.25}\smash{\begin{tabular}[t]{r}$\ds_3$\end{tabular}}}}%
    \put(0.60259544,0.48310331){\color[rgb]{0,0,0}\makebox(0,0)[lt]{\lineheight{1.25}\smash{\begin{tabular}[t]{l}$\ds_0$\end{tabular}}}}%
    \put(0.59311663,0.68508144){\color[rgb]{0,0,0}\makebox(0,0)[lt]{\lineheight{1.25}\smash{\begin{tabular}[t]{l}$\ds_1$\end{tabular}}}}%
    \put(0.39452433,0.68508144){\color[rgb]{0,0,0}\makebox(0,0)[rt]{\lineheight{1.25}\smash{\begin{tabular}[t]{r}$\ds_2$\end{tabular}}}}%
    \put(0.49454323,0.56330009){\color[rgb]{0,0,0}\makebox(0,0)[t]{\lineheight{1.25}\smash{\begin{tabular}[t]{c}$J_{0123}$\end{tabular}}}}%
    \put(0,0){\includegraphics[width=\unitlength,page=2]{fig_naturality_43.pdf}}%
    \put(0.25159263,0.10980766){\color[rgb]{0,0,0}\makebox(0,0)[t]{\lineheight{1.25}\smash{\begin{tabular}[t]{c}$J_{013}$\end{tabular}}}}%
    \put(0.15566234,0.23534931){\color[rgb]{0,0,0}\makebox(0,0)[t]{\lineheight{1.25}\smash{\begin{tabular}[t]{c}$J_{123}$\end{tabular}}}}%
    \put(0.73962024,0.1141055){\color[rgb]{0,0,0}\makebox(0,0)[t]{\lineheight{1.25}\smash{\begin{tabular}[t]{c}$J_{023}$\end{tabular}}}}%
    \put(0.84043177,0.23116103){\color[rgb]{0,0,0}\makebox(0,0)[t]{\lineheight{1.25}\smash{\begin{tabular}[t]{c}$J_{012}$\end{tabular}}}}%
    \put(0,0){\includegraphics[width=\unitlength,page=3]{fig_naturality_43.pdf}}%
  \end{picture}%
\endgroup%

%% file: fig_naturality_30.pdf_tex
\begingroup%
  \makeatletter%
  \providecommand\color[2][]{%
    \errmessage{(Inkscape) Color is used for the text in Inkscape, but the package 'color.sty' is not loaded}%
    \renewcommand\color[2][]{}%
  }%
  \providecommand\transparent[1]{%
    \errmessage{(Inkscape) Transparency is used (non-zero) for the text in Inkscape, but the package 'transparent.sty' is not loaded}%
    \renewcommand\transparent[1]{}%
  }%
  \providecommand\rotatebox[2]{#2}%
  \newcommand*\fsize{\dimexpr\f@size pt\relax}%
  \newcommand*\lineheight[1]{\fontsize{\fsize}{#1\fsize}\selectfont}%
  \ifx\svgwidth\undefined%
    \setlength{\unitlength}{90.7382039bp}%
    \ifx\svgscale\undefined%
      \relax%
    \else%
      \setlength{\unitlength}{\unitlength * \real{\svgscale}}%
    \fi%
  \else%
    \setlength{\unitlength}{\svgwidth}%
  \fi%
  \global\let\svgwidth\undefined%
  \global\let\svgscale\undefined%
  \makeatother%
  \begin{picture}(1,0.99917983)%
    \lineheight{1}%
    \setlength\tabcolsep{0pt}%
    \put(0,0){\includegraphics[width=\unitlength,page=1]{fig_naturality_30.pdf}}%
    \put(0.5936227,0.45079619){\color[rgb]{0,0,1}\makebox(0,0)[lt]{\lineheight{1.25}\smash{\begin{tabular}[t]{l}$\b_0$\end{tabular}}}}%
    \put(0.4241537,0.34181526){\color[rgb]{1,0,0}\makebox(0,0)[lt]{\lineheight{1.25}\smash{\begin{tabular}[t]{l}$\a_0$\end{tabular}}}}%
    \put(0,0){\includegraphics[width=\unitlength,page=2]{fig_naturality_30.pdf}}%
  \end{picture}%
\endgroup%

%% file: fig_naturality_39.pdf_tex
\begingroup%
  \makeatletter%
  \providecommand\color[2][]{%
    \errmessage{(Inkscape) Color is used for the text in Inkscape, but the package 'color.sty' is not loaded}%
    \renewcommand\color[2][]{}%
  }%
  \providecommand\transparent[1]{%
    \errmessage{(Inkscape) Transparency is used (non-zero) for the text in Inkscape, but the package 'transparent.sty' is not loaded}%
    \renewcommand\transparent[1]{}%
  }%
  \providecommand\rotatebox[2]{#2}%
  \newcommand*\fsize{\dimexpr\f@size pt\relax}%
  \newcommand*\lineheight[1]{\fontsize{\fsize}{#1\fsize}\selectfont}%
  \ifx\svgwidth\undefined%
    \setlength{\unitlength}{254.32005071bp}%
    \ifx\svgscale\undefined%
      \relax%
    \else%
      \setlength{\unitlength}{\unitlength * \real{\svgscale}}%
    \fi%
  \else%
    \setlength{\unitlength}{\svgwidth}%
  \fi%
  \global\let\svgwidth\undefined%
  \global\let\svgscale\undefined%
  \makeatother%
  \begin{picture}(1,0.56594638)%
    \lineheight{1}%
    \setlength\tabcolsep{0pt}%
    \put(0,0){\includegraphics[width=\unitlength,page=1]{fig_naturality_39.pdf}}%
    \put(0.23472279,0.29482007){\makebox(0,0)[lt]{\lineheight{1.25}\smash{\begin{tabular}[t]{l}$\d$\end{tabular}}}}%
    \put(0.69752034,0.29482007){\makebox(0,0)[lt]{\lineheight{1.25}\smash{\begin{tabular}[t]{l}$\d$\end{tabular}}}}%
  \end{picture}%
\endgroup%

%% file: fig_naturality_37.pdf_tex
\begingroup%
  \makeatletter%
  \providecommand\color[2][]{%
    \errmessage{(Inkscape) Color is used for the text in Inkscape, but the package 'color.sty' is not loaded}%
    \renewcommand\color[2][]{}%
  }%
  \providecommand\transparent[1]{%
    \errmessage{(Inkscape) Transparency is used (non-zero) for the text in Inkscape, but the package 'transparent.sty' is not loaded}%
    \renewcommand\transparent[1]{}%
  }%
  \providecommand\rotatebox[2]{#2}%
  \newcommand*\fsize{\dimexpr\f@size pt\relax}%
  \newcommand*\lineheight[1]{\fontsize{\fsize}{#1\fsize}\selectfont}%
  \ifx\svgwidth\undefined%
    \setlength{\unitlength}{74.46182852bp}%
    \ifx\svgscale\undefined%
      \relax%
    \else%
      \setlength{\unitlength}{\unitlength * \real{\svgscale}}%
    \fi%
  \else%
    \setlength{\unitlength}{\svgwidth}%
  \fi%
  \global\let\svgwidth\undefined%
  \global\let\svgscale\undefined%
  \makeatother%
  \begin{picture}(1,0.96957904)%
    \lineheight{1}%
    \setlength\tabcolsep{0pt}%
    \put(0,0){\includegraphics[width=\unitlength,page=1]{fig_naturality_37.pdf}}%
  \end{picture}%
\endgroup%

%% file: fig_naturality_38.pdf_tex
\begingroup%
  \makeatletter%
  \providecommand\color[2][]{%
    \errmessage{(Inkscape) Color is used for the text in Inkscape, but the package 'color.sty' is not loaded}%
    \renewcommand\color[2][]{}%
  }%
  \providecommand\transparent[1]{%
    \errmessage{(Inkscape) Transparency is used (non-zero) for the text in Inkscape, but the package 'transparent.sty' is not loaded}%
    \renewcommand\transparent[1]{}%
  }%
  \providecommand\rotatebox[2]{#2}%
  \newcommand*\fsize{\dimexpr\f@size pt\relax}%
  \newcommand*\lineheight[1]{\fontsize{\fsize}{#1\fsize}\selectfont}%
  \ifx\svgwidth\undefined%
    \setlength{\unitlength}{316.51650256bp}%
    \ifx\svgscale\undefined%
      \relax%
    \else%
      \setlength{\unitlength}{\unitlength * \real{\svgscale}}%
    \fi%
  \else%
    \setlength{\unitlength}{\svgwidth}%
  \fi%
  \global\let\svgwidth\undefined%
  \global\let\svgscale\undefined%
  \makeatother%
  \begin{picture}(1,0.27117892)%
    \lineheight{1}%
    \setlength\tabcolsep{0pt}%
    \put(0,0){\includegraphics[width=\unitlength,page=1]{fig_naturality_38.pdf}}%
    \put(0.32610925,0.16414823){\makebox(0,0)[lt]{\lineheight{1.25}\smash{\begin{tabular}[t]{l}$\d$\end{tabular}}}}%
    \put(0.67549886,0.16414823){\makebox(0,0)[lt]{\lineheight{1.25}\smash{\begin{tabular}[t]{l}$\d$\end{tabular}}}}%
  \end{picture}%
\endgroup%

%% file: fig_naturality_31.pdf_tex
\begingroup%
  \makeatletter%
  \providecommand\color[2][]{%
    \errmessage{(Inkscape) Color is used for the text in Inkscape, but the package 'color.sty' is not loaded}%
    \renewcommand\color[2][]{}%
  }%
  \providecommand\transparent[1]{%
    \errmessage{(Inkscape) Transparency is used (non-zero) for the text in Inkscape, but the package 'transparent.sty' is not loaded}%
    \renewcommand\transparent[1]{}%
  }%
  \providecommand\rotatebox[2]{#2}%
  \newcommand*\fsize{\dimexpr\f@size pt\relax}%
  \newcommand*\lineheight[1]{\fontsize{\fsize}{#1\fsize}\selectfont}%
  \ifx\svgwidth\undefined%
    \setlength{\unitlength}{237.69413795bp}%
    \ifx\svgscale\undefined%
      \relax%
    \else%
      \setlength{\unitlength}{\unitlength * \real{\svgscale}}%
    \fi%
  \else%
    \setlength{\unitlength}{\svgwidth}%
  \fi%
  \global\let\svgwidth\undefined%
  \global\let\svgscale\undefined%
  \makeatother%
  \begin{picture}(1,0.64671548)%
    \lineheight{1}%
    \setlength\tabcolsep{0pt}%
    \put(0,0){\includegraphics[width=\unitlength,page=1]{fig_naturality_31.pdf}}%
    \put(0.7834963,0.40520737){\makebox(0,0)[lt]{\lineheight{1.25}\smash{\begin{tabular}[t]{l}$e$\end{tabular}}}}%
    \put(0.48755043,0.17439329){\makebox(0,0)[lt]{\lineheight{1.25}\smash{\begin{tabular}[t]{l}$f$\end{tabular}}}}%
    \put(0.21305974,0.40554246){\makebox(0,0)[rt]{\lineheight{1.25}\smash{\begin{tabular}[t]{r}$g$\end{tabular}}}}%
    \put(0,0){\includegraphics[width=\unitlength,page=2]{fig_naturality_31.pdf}}%
    \put(0.03293986,0.08188928){\makebox(0,0)[lt]{\lineheight{1.25}\smash{\begin{tabular}[t]{l}$\Sss{\mathrm{E}}$\end{tabular}}}}%
    \put(0.26768671,0.08188884){\makebox(0,0)[t]{\lineheight{1.25}\smash{\begin{tabular}[t]{c}\reflectbox{$\Sss{\mathrm{E}}$}\end{tabular}}}}%
    \put(0.15399453,0.20450122){\makebox(0,0)[t]{\lineheight{1.25}\smash{\begin{tabular}[t]{c}$\Sss{\mathrm{B}}$\end{tabular}}}}%
    \put(0.3798261,0.20450107){\makebox(0,0)[t]{\lineheight{1.25}\smash{\begin{tabular}[t]{c}\reflectbox{$\Sss{\mathrm{B}}$}\end{tabular}}}}%
    \put(0,0){\includegraphics[width=\unitlength,page=3]{fig_naturality_31.pdf}}%
    \put(0.62004559,0.08188938){\makebox(0,0)[t]{\lineheight{1.25}\smash{\begin{tabular}[t]{c}$\Sss{\mathrm{E}}$\end{tabular}}}}%
    \put(0.84670732,0.08188906){\makebox(0,0)[t]{\lineheight{1.25}\smash{\begin{tabular}[t]{c}\reflectbox{$\Sss{\mathrm{E}}$}\end{tabular}}}}%
    \put(0.73304084,0.20450139){\makebox(0,0)[t]{\lineheight{1.25}\smash{\begin{tabular}[t]{c}$\Sss{\mathrm{B}}$\end{tabular}}}}%
    \put(0.95887215,0.20450107){\makebox(0,0)[t]{\lineheight{1.25}\smash{\begin{tabular}[t]{c}\reflectbox{$\Sss{\mathrm{B}}$}\end{tabular}}}}%
    \put(0,0){\includegraphics[width=\unitlength,page=4]{fig_naturality_31.pdf}}%
    \put(0.3224631,0.41477022){\makebox(0,0)[lt]{\lineheight{1.25}\smash{\begin{tabular}[t]{l}$\Sss{\mathrm{E}}$\end{tabular}}}}%
    \put(0.5572099,0.41476967){\makebox(0,0)[t]{\lineheight{1.25}\smash{\begin{tabular}[t]{c}\reflectbox{$\Sss{\mathrm{E}}$}\end{tabular}}}}%
    \put(0.4435177,0.53738227){\makebox(0,0)[t]{\lineheight{1.25}\smash{\begin{tabular}[t]{c}$\Sss{\mathrm{B}}$\end{tabular}}}}%
    \put(0.6693494,0.53738194){\makebox(0,0)[t]{\lineheight{1.25}\smash{\begin{tabular}[t]{c}\reflectbox{$\Sss{\mathrm{B}}$}\end{tabular}}}}%
    \put(0,0){\includegraphics[width=\unitlength,page=5]{fig_naturality_31.pdf}}%
    \put(0.58297174,0.55288219){\color[rgb]{1,0,0}\makebox(0,0)[lt]{\lineheight{1.25}\smash{\begin{tabular}[t]{l}$\a_1$\end{tabular}}}}%
    \put(0.62046307,0.4595215){\color[rgb]{1,0,0}\makebox(0,0)[lt]{\lineheight{1.25}\smash{\begin{tabular}[t]{l}$\a_2$\end{tabular}}}}%
    \put(0.49573185,0.60147418){\color[rgb]{0,0,1}\makebox(0,0)[lt]{\lineheight{1.25}\smash{\begin{tabular}[t]{l}$\b_2$\end{tabular}}}}%
    \put(0.39437892,0.44930084){\color[rgb]{0,0,1}\makebox(0,0)[lt]{\lineheight{1.25}\smash{\begin{tabular}[t]{l}$\b_1$\end{tabular}}}}%
    \put(0,0){\includegraphics[width=\unitlength,page=6]{fig_naturality_31.pdf}}%
    \put(0.32844961,0.35332721){\makebox(0,0)[lt]{\lineheight{1.25}\smash{\begin{tabular}[t]{l}$\Ss{(\Sigma_2,\as_0,\bs_0)}$\end{tabular}}}}%
    \put(0.02235455,0.01563575){\makebox(0,0)[lt]{\lineheight{1.25}\smash{\begin{tabular}[t]{l}$\Ss{(\Sigma_2,\as_0',\bs_0')}$\end{tabular}}}}%
    \put(0.60340782,0.01563575){\makebox(0,0)[lt]{\lineheight{1.25}\smash{\begin{tabular}[t]{l}$\Ss{(\Sigma_2,\as_0',\bs_0)}$\end{tabular}}}}%
  \end{picture}%
\endgroup%